\documentclass[11pt]{article}
\usepackage[margin=1in]{geometry}
\usepackage{amsmath,hyperref}
\usepackage{amsfonts}
\usepackage{bbm}
\usepackage{amssymb}
\usepackage{amsthm,bm}
\usepackage{xcolor}
\usepackage{enumitem}
\usepackage{float}
\usepackage{graphicx}
\usepackage{multirow}
\usepackage{caption}
\usepackage{subcaption}
\usepackage{epsfig}
\usepackage{cite}
\usepackage{indentfirst}
\usepackage{algorithmicx,algorithm,algpseudocode}

\hypersetup{
  colorlinks = true,
  allcolors = green!50!black,
  urlcolor = red!50!black,
}

\newtheoremstyle{dotless}{12pt}{12pt}{}{}{\bfseries}{}{1em}{\thmname{#1}\thmnumber{~#2}\thmnote{~(#3)}}
\theoremstyle{dotless}
\newtheorem{theorem}{Theorem}[section]
\newtheorem{corollary}{Corollary}[section]
\newtheorem{lemma}{Lemma}[section]
\newtheorem{definition}{Definition}[section]
\newtheorem{proposition}{Proposition}[section]
\newtheorem{remark}{Remark}[section]

\newcommand{\email}[1]{\protect\href{mailto:#1}{#1}}
\newcommand\funding[1]{\protect\\ \hspace*{1.8em}{\bfseries Funding:} #1}

\usepackage{amsopn}
\DeclareMathOperator{\diag}{diag}
\DeclareMathOperator{\rank}{rank}
\DeclareMathOperator{\dom}{dom}
\DeclareMathOperator{\prox}{prox}

\title{Multiplicative Noise Removal: Nonlocal Low-Rank Model and Its Proximal Alternating Reweighted Minimization Algorithm\thanks{Submitted to the editors January 30, 2020.
\funding{The research of J. Lu is partially supported by the Natural Science Foundation of China under grant 61972265 and 11871348. The research of L. Shen is partially supported by the National Science Foundation under grant DMS-1913039. The research of C. Xu is partially supported by the Natural Science Foundation of China under grant 61872429. The research of Y. Xu is partially supported by the National Science Foundation under grant DMS-1912958 and by the Natural Science Foundation of China under grant 11771464.}}}

\author{Xiaoxia Liu\thanks{Shenzhen Key Laboratory of Advanced Machine Learning and Applications, College of Mathematics and Statistics, Shenzhen University, Shenzhen, 518060, P.R. China.  (\email{xliu@szu.edu.cn},\email{jianlu@szu.edu.cn},\email{chenxuszu@sina.com})}
\and Jian Lu\footnotemark[2] \thanks{Corresponding author.}
\and Lixin Shen\thanks{Department of Mathematics, Syracuse University, Syracuse, NY 13244, USA. (\email{lshen03@syr.edu})}
\and Chen Xu\footnotemark[2] \and Yuesheng Xu\thanks{Department of Mathematics and Statistics, Old Dominion University, Norfolk, VA 23529, USA. (\email{y1xu@odu.edu})}}

\date{}

\begin{document}

\maketitle

\begin{abstract}
The goal of this paper is to develop a novel numerical method for efficient multiplicative noise removal. The nonlocal self-similarity of natural images implies that the matrices formed by their nonlocal similar patches are low-rank. By exploiting this low-rank prior with application to multiplicative noise removal, we propose a nonlocal low-rank model for this task and develop a proximal alternating reweighted minimization (PARM) algorithm to solve the optimization problem resulting from the model. Specifically, we utilize a generalized nonconvex surrogate of the rank function to regularize the patch matrices and develop a new nonlocal low-rank model, which is a nonconvex nonsmooth optimization problem having a patchwise data fidelity and a generalized nonlocal low-rank regularization term. To solve this optimization problem, we propose the PARM algorithm, which has a proximal alternating scheme with a reweighted approximation of its subproblem. A theoretical analysis of the proposed PARM algorithm is conducted to guarantee its global convergence to a critical point. Numerical experiments demonstrate that the proposed method for multiplicative noise removal significantly outperforms existing methods such as the benchmark SAR-BM3D method in terms of the visual quality of the denoised images, and the PSNR (the peak-signal-to-noise ratio) and SSIM (the structural similarity index measure) values.
\end{abstract}

\section{Introduction}
We consider in this paper the problem of multiplicative noise removal. To effectively restore images degraded by multiplicative noise, we develop a method which consists of an optimization model and an iterative algorithm to solve the minimization problem. Based on the nonlocal self-similarity of natural images, we propose a nonlocal low-rank model for multiplicative noise removal. The resulting model is a nonconvex nonsmooth minimization problem. We develop a proximal alternating reweighted minimization (PARM) algorithm with a convergence guarantee to efficiently solve the problem.

Multiplicative noise (i.e., speckle noise) widely occurs in coherent imaging systems due to the interference of coherent waves scattered from distributed targets. For example, images obtained from synthetic aperture radar (SAR) \cite{oliver2004understanding}, ultrasound imaging \cite{wagner1983statistics} and laser imaging \cite{schmitt1999speckle} are naturally contaminated with multiplicative noise. Removing multiplicative noise from such images is inevitable in many areas of applications.

Methods employed for multiplicative noise removal in the literature include the total variation (TV) regularization based models, patch-based methods, and nonlocal low-rank based methods. TV regularization has been widely used to preserve edges in the restored images. In a TV regularization based model, the objective function is the sum of a data fidelity term and a TV regularization term. The data fidelity term measures the closeness between the desired image and the observed noisy image, while the TV regularization term measures the total variation of a desired image or an image in its transformed domain. The AA model\cite{aubert2008variational} used the Bayesian maximum a posteriori probability (MAP) estimation to derive the data fidelity term in terms of the desired image. However, this data fidelity term is nonconvex and the resulting optimization problem is challenging to solve. To overcome this challenge, the DZ model\cite{dong2013convex} modified the data fidelity term by adding a quadratic term. As a consequence, the objective function of the DZ model becomes convex under some mild conditions. The I-DIV model\cite{steidl2010removing} used the so-called I-divergence as the data fidelity term. By performing the logarithmic transformation, the SO model\cite{shi2008nonlinear}, the HNM model\cite{huang2009new}, and the Exp model\cite{lu2016multiplicative} led to convex, even strictly convex, data fidelity terms. The $m$V model\cite{yun2012new} and the TwL-$m$V model\cite{kang2013two} used convex or strongly convex data fidelity terms via the $m$th root transformation. The TV regularization based models have good performance in denoising. However, they tend to over-smooth image textures and generate unexpected artifacts.

The patch-based methods make use of the redundancy of image patches to yield a restored image with fine details. Sparse representations of image patches have been studied in the patch-based methods for multiplicative noise removal.  In the learned dictionary method \cite{huang2012multiplicative}, an optimal over-complete dictionary was learned from the patches of the logarithmic transformed noisy image and then an image was restored via a variational model based on the learned dictionary and a TV regularization. The SAR-BM3D method \cite{parrilli2011nonlocal} is another remarkable approach relying on a sparse representation, which takes advantage of the nonlocal self-similarity of natural images \cite{buades2005a}.  Nonlocal similar patches, collected as 3D groups, were identified based on a probabilistic similarity measure for multiplicative noise, and then were denoised by using jointly nonlocal filtering and a local linear minimum-mean-square-error shrinkage in a wavelet domain. We remark that those methods constrain the sparsity priors in either a fixed dictionary or a fixed wavelet domain, which limits their capability in multiplicative noise removal.

Recently, the nonlocal low-rank based methods were extensively exploited in image processing. It is recognized that natural images are of nonlocal self-similarity. Matrices formed by nonlocal similar patches are low-rank, and hence the desired image can be restored by low-rank estimations of nonlocal similar patch matrices. To regularize the rank of the matrices formed by  nonlocal similar patches, different approximations of the rank function including the weighted nuclear norm and the log-det function were adopted, see, e.g., \cite{dong2014compressive,gu2014weighted,huang2017mixed,huang2018rank,wei2017nonlocal}.

Existing studies have shown impressive empirical performance of nonlocal low-rank based methods. However, theoretical analysis of the existing methods is missing and there is little work on applications of nonlocal low-rank based methods to multiplicative noise removal. To address this issue, we propose to develop a new nonlocal low-rank based method that is theoretically and practically suitable for multiplicative noise removal. The proposed method includes a novel nonlocal low-rank model and an efficient iterative algorithm to solve the proposed model with a convergence guarantee.  We explore the underlying low-rank prior of the patch matrices and propose a nonlocal low-rank model for multiplicative noise removal. The resulting optimization problem is nonconvex and nonsmooth, which is challenging to design efficient and theoretically convergence-guaranteed algorithms to solve. In fact, the well-known alternating direction method of multipliers (ADMM) algorithm is not applicable to this optimization problem, and the alternating minimization (AM) algorithm and the augmented Lagrange multiplier (ALM) algorithm may not converge \cite{bolte2014proximal,yu2018global}. To address this challenge on developing an efficient convergent algorithm, we propose a proximal alternating minimization scheme with a reweighted approximation of its subproblem and further use the Kurdyka-\L ojasiewicz theory \cite{attouch2013convergence,bolte2014proximal} to prove its global convergence to a critical point.  The experiments demonstrate that the proposed nonlocal low-rank based method is well suitable for multiplicative noise removal.

The main contributions of this work are:
\begin{itemize}
\item We propose a nonlocal low-rank model for multiplicative noise removal. This model is formulated in the log-transformed domain of images.  The objective function of the model as the sum of a fidelity term and a regularization term is nonconvex and nonsmooth. Its  fidelity term is adapted from the corresponding one in the Exp model~\cite{lu2016multiplicative} to patches, and is strictly convex under certain conditions. Its regularization term is the application of the composition of the rank operator with the patch extraction operator onto the underlying image. Due to the difficulties caused by the composition and the rank function in solving this model, we propose to split this composition by introducing an auxiliary variable and to approximate the rank function using a smooth concave function.

\item We develop a proximal alternating reweighted minimization (PARM) algorithm for solving the proposed nonlocal low-rank model. The key in the PARM algorithm is to deal with the concave function that is used to approximate the rank function in the model. We propose to approximate this concave function by its affine approximation (i.e., the reweighted approximation) in each iteration of the PARM algorithm. This approach could be useful for a wide range of nonlocal low-rank models.

\item We provide a theoretical analysis of the PARM algorithm which guarantees its global convergence to a critical point, in contrast to the practically used algorithms such as in \cite{dong2014compressive,wu2016speckle} which are lack of convergence analysis.

\item We give a detailed description on the implementation of the PARM algorithm including parameter settings, patch sizes, and search windows. We also test the proposed method for various images at different noise levels. Furthermore, we conduct the performance comparison of the proposed method with many existing ones for multiplicative noise removal, with respect to the visual quality of the denoised images, and the PSNR (the peak-signal-to-noise ratio) and SSIM (the structural similarity index measure) values.
\end{itemize}

This paper is organized into six sections. In section~\ref{Section:Model}, we present the nonlocal low-rank model for multiplicative noise removal. The proposed PARM algorithm to solve the resulting nonconvex nonsmooth optimization problem is presented in section~\ref{Section:Algorithm}. Section~\ref{Section:Convergence} is devoted to the convergence analysis of the proposed algorithm. In section~\ref{Section:Experiment}, we demonstrate the efficiency of the new method numerically by experiment results. Section~\ref{Section:Conclusions} concludes this paper.

\section{Nonlocal Low-Rank Model for Multiplicative Noise Removal}\label{Section:Model}

We propose in this section a nonlocal low-rank model for multiplicative noise removal by exploiting low-rank priors of the nonlocal similar patch matrices extracted from the underlying images.

Throughout this paper, matrices are bold capital, vectors are bold lowercase and scalars or entries are not bold. Given $\bm{x}, \bm{y} \in \mathbb{R}^d$,  $\langle \bm{x}, \bm{y}\rangle :=\sum_{i=1}^d\langle x_i,y_i\rangle$ is the standard inner product and  $\|\bm{x}\|_2:=\sqrt{\langle \bm{x},\bm{x}\rangle}$ is the standard $\ell_2$ norm. Let $\mathbb{S}^d_+$ denote the set of symmetric positive definite matrices of size $d\times d$ and let $\bm{I}_d$ denote the identity matrix of size $d\times d$. Given $\bm{x}, \bm{y} \in \mathbb{R}^d$ and $\bm{H} \in \mathbb{S}^d_+$, $\langle \bm{x}, \bm{y}\rangle_{\bm{H}} := \langle \bm{x},\bm{Hy}\rangle$ is the $\bm{H}$-weighted inner product and $\|\bm{x}\|_{\bm{H}} :=\sqrt{\langle \bm{x}, \bm{x}\rangle_{\bm{H}}}$ is the $\bm{H}$-weighted $\ell_2$ norm. Given $\bm{X}, \bm{Y} \in \mathbb{R}^{m\times n}$,  $\langle \bm{X},\bm{Y}\rangle_{F}:=\operatorname{tr}(\bm{X}^{\top}\bm{Y})$ is the Frobenius inner product and  $\|\bm{X}\|_{F}:=\sqrt{\langle \bm{X},\bm{X}\rangle_{F}}$ is the Frobenius norm.

Multiplicative noise removal in this paper refers to reducing multiplicative noise in an $L$-look image obtained by the multi-look averaging technique.  An $L$-look image $\bm{v}\in\mathbb{R}^N$ in the intensity format degraded by multiplicative noise can be modeled as
\begin{displaymath}
\bm{v}=\bm{u\eta},
\end{displaymath}
where $\bm{u}\in\mathbb{R}^N$ is the desired image to be restored, $\bm{\eta}\in\mathbb{R}^N$ is the multiplicative noise and the multiplication operation is a componentwise operation. The multiplicative noise in each pixel follows a Gamma distribution \cite{goodman1976some}, whose probability distribution function is defined as
\begin{displaymath}
p(\eta_i)=\frac{L^L \eta^{L-1}_i}{\Gamma(L)}e^{-L\eta_i},\quad i=1,\dots,N,
\end{displaymath}
which has mean $1$ and variance of $1/L$.  A list of TV regularization based models for multiplicative noise removal is presented  in Table \ref{Table:TVModels}.

\begin{table}[htbp]
\small
\centering
\caption{TV regularization based models for multiplicative noise removal.}\label{Table:TVModels}
\begin{tabular}{llp{1.8cm}p{2.7cm}}
\hline
\textbf{Name} & \qquad\qquad\qquad\qquad\textbf{Model $\Phi$} & \textbf{Transform.} &  \textbf{Properties of $\Phi$}\\
\hline
AA\cite{aubert2008variational} & $\underset{\bm{u}\in \mathbb{R}^N_+}{\min}\; \langle \log \bm{u} +\frac{\bm{v}}{\bm{u}},\mathbbm{1}\rangle+\lambda \|\bm{u}\|_{TV}$ & -- & nonconvex\\
DZ\cite{dong2013convex} & $\underset{\bm{u}\in \mathbb{R}^N_+}{\min}\;\langle \log \bm{u} +\frac{\bm{v}}{\bm{u}},\mathbbm{1}\rangle+\rho  \|\sqrt{\frac{\bm{u}}{\bm{v}}}-\mathbbm{1} \|_2^2+\lambda \|\bm{u}\|_{TV}$ & -- & strictly convex if $\rho\geq \frac{2\sqrt{6}}{9}$\\
I-DIV\cite{steidl2010removing} & $\underset{\bm{u}\in \mathbb{R}^N_+}{\min}\; \langle \bm{u} -\bm{v}\log\bm{u},\mathbbm{1}\rangle+\lambda \|\bm{u}\|_{TV}$ & --  & convex\\
SO\cite{shi2008nonlinear} & $\underset{\bm{u}\in \mathbb{R}^N}{\min}\;\langle \bm{x}+\frac{\bm{v}}{e^{\bm{x}}},\mathbbm{1}\rangle+\lambda \|\bm{x}\|_{TV}$ & $\bm{x}=\log \bm{u}$  & strictly convex\\
HNM\cite{huang2009new} & $\underset{\bm{x}\in \mathbb{R}^N,\bm{w}\in \mathbb{R}^N}{\min}\; \langle \bm{x}+\frac{\bm{v}}{e^{\bm{x}}},\mathbbm{1}\rangle+\rho \|\bm{x}-\bm{w}\|_2^2+\lambda\|\bm{w}\|_{TV}$ & $\bm{x}=\log \bm{u}$ & convex\\
Exp\cite{lu2016multiplicative}   & $\underset{\bm{x}\in \mathbb{R}^N}{\min}\; \langle \bm{x}+\frac{\bm{v}}{e^{\bm{x}}},\mathbbm{1}\rangle+\rho \|\sqrt{\frac{e^{\bm{x}}}{\bm{v}}}-\gamma\mathbbm{1}\|_2^2+\lambda \|\bm{x}\|_{TV}$ & $\bm{x}=\log \bm{u}$ & strictly convex if $\rho\gamma^4\leq \frac{4096}{27}$\\
$m$V\cite{yun2012new}   & $\underset{\bm{x}\in \sqrt[m]{U}}{\min}\; \langle m\log \bm{x} +\frac{\bm{v}}{\bm{x}^m},\mathbbm{1}\rangle+\lambda\|\bm{x}\|_{TV}$ & $\bm{x}=\sqrt[m]{\bm{u}}$  & convex if $m$ is sufficiently large \\
TwL-$m$V\cite{kang2013two} & $\underset{a>0,\bm{x}\in \sqrt[m]{U}}{\min}\;\frac{1}{s}\langle a,\bm{x}^s\rangle-\frac{1}{s}\langle m\log a -\frac{s\bm{v}}{\bm{x}^m},\mathbbm{1}\rangle+\lambda\|\bm{x}\|_{TV}$ & $\bm{x}=\sqrt[m]{\bm{u}}$ & strongly convex with respect to $\bm{x}$\\
\hline
\end{tabular}

{\scriptsize 1. $\mathbb{R}_+=(0,+\infty)$; 2. $U=(0,C]^N$, $C\in \mathbb{R}_+$; 3. $\lambda>0$, $\rho>0$, $\gamma\geq 1$, and $s\geq 1$; 4. $\mathbbm{1}$ denotes the vector whose entries are all ones; 5. The division, multiplication, logarithmic, exponential, square root operations are componentwise operations. }
\end{table}

In the following, we present our nonlocal low-rank model for multiplicative noise removal step by step. According to the nonlocal self-similarity of natural images, for an image patch, we can find nonlocal similar patches across the image or within a local window\cite{buades2005a}. We begin with collecting similar patches using block matching\cite{dabov2007image,parrilli2011nonlocal} and formulating patch matrices. Suppose that $\hat{\bm{u}}\in \mathbb{R}^N$ is an estimated clean image in the intensity format and that $J$ is the number of nonlocal similar patch groups to be collected. For the reference patch $\hat{\bm{u}}_j\in \mathbb{R}^{m_j}$ with size $\sqrt{m_j}\times \sqrt{m_j}$ in the $j$th patch group, we search within a local window for a total of $n_j$ patches that are similar to the reference patch, assuming $m_j \leq n_j$, $j=1,2,\dots,J$.  To fully exploit the statistics of $L$-look images, we measure the similarity between two patches $\hat{\bm{u}}_j\in \mathbb{R}^{m_j}$ and $\hat{\bm{u}}_j'\in \mathbb{R}^{m_j}$ using the block similarity measure introduced in \cite{parrilli2011nonlocal}
\begin{displaymath}
d(\hat{\bm{u}}_j,\hat{\bm{u}}_j')=(2L-1)\sum_{i=1}^{m_j} \log\left (\sqrt{\frac{(\hat{{u}}_j)_i}{(\hat{{u}}_j')_i}}+\sqrt{\frac{(\hat{{u}}_j')_i}{(\hat{{u}}_j)_i}}\right ).
\end{displaymath}

Following the above, for each group we construct a patch matrix from all the patches in the given group through an extraction operator. Define $\bm{R}_{jl}\in\mathbb{R}^{m_j\times N}$ be a binary matrix (i.e., its entries are either $1$ or $0$) such that $\bm{R}_{jl}\hat{\bm{u}}$ is the $l$th  patch in the $j$th nonlocal similar patch group of the given estimated image $\hat{\bm{u}}$, $l=1,\dots, n_j$, $j=1,\dots,J$. Then we define a linear operator $R_j:\mathbb{R}^{N}\to \mathbb{R}^{m_j\times n_j}$, $m_j\leq n_j$, as follows,
\begin{displaymath}
R_j(\bm{x})=\begin{bmatrix}\bm{R}_{j1}\bm{x}&\bm{R}_{j2}\bm{x}& \cdots &\bm{R}_{jn_j}\bm{x}\end{bmatrix}.
\end{displaymath}
Here, $R_j(\bm{x})$ is called the $j$th patch matrix of the (transformed) image $\bm{x}\in\mathbb{R}^N$. After the patch matrix is extracted, the patch matrix can be further processed using, for example, normalization with mean zero, and the corresponding extraction operator $R_j$ can be defined accordingly. Intuitively, the patch matrix $R_j(\bm{x})$ with similar structures should be a low-rank matrix if $\bm{x}$ is close to the clean image $\hat{\bm{u}}$, for example, up to a transformation.

Taking advantage of the low-rank prior of image patch matrices $R_j(\bm{x})$'s, the objective function of a patch-based nonlocal low-rank model consists of a data fidelity term to restore the desired image and a nonlocal low-rank regularization term as follows
\begin{equation}\label{Model1}
\min_{\bm{x}}\; \tau f(\bm{x})+ \sum_{j=1}^{J}\lambda_j \rank(R_j(\bm{x})),
\end{equation}
where $\bm{x}\in\mathbb{R}^{N}$ is the desired (transformed) image to be restored, $f:\mathbb{R}^{N}\to (-\infty,+\infty]$ is the data fidelity term  that measures the closeness between the observed image and the desired image, $R_j:\mathbb{R}^{N}\to \mathbb{R}^{m_j\times n_j}$, $m_j\leq n_j$,  is the (normalized) extraction of $j$th nonlocal similar patch matrix, and $\lambda_j>0$ is the regularization parameter, $j=1,\dots,J$.

Model \eqref{Model1} regularizes low-rank priors on image patch matrices, but it is not a feasible model from both theoretical and practical perspectives. First, model \eqref{Model1} as a composition optimization is not easy to solve. Second, the rank function is discontinuous and nonconvex, and minimizing a problem involving the rank function is NP-hard\cite{recht2010guaranteed}; therefore, it is challenging to solve model \eqref{Model1}. To tackle the above challenges, we plan to  relax model~\eqref{Model1} in the following ways. We first apply the variable splitting method to model \eqref{Model1} to address the composition optimization problem, adopt a nonconvex surrogate of the rank function to replace the rank function, and preferably utilize a patchwise data fidelity term.

First, we apply the variable splitting method to relax model \eqref{Model1}.  By introducing auxiliary (splitting) variables $\bm{Y}_j\in \mathbb{R}^{m_j\times n_j}$ such that $\bm{Y}_j=R_j(\bm{x})$ and then relaxing these equalities of the splitting variables,  we obtain the following model
\begin{displaymath}
\min_{\bm{x},\bm{Y}_1, \ldots,\bm{Y}_J}\; \tau f(\bm{x})+ \sum_{j=1}^J \left \{ \frac{\mu_j}{2}\|\bm{Y}_j-R_j(\bm{x})\|_{F}^2+\lambda_j \rank(\bm{Y}_j)\right \},
\end{displaymath}
where  $\mu_j>0$ is a parameter.

Second, we utilize a nonconvex relaxation of the rank function to characterize the low-rank prior of patch matrices. By introducing a function $g:[0,\infty)\to \mathbb{R}$ such that $g$ is monotonically increasing, a generalized relaxation of the rank function is defined as
\begin{displaymath}
\|\bm{Y}\|_{*,g}=\sum_{i=1}^{m} g(\sigma_i(\bm{Y})),
\end{displaymath}
where $\bm{Y}\in \mathbb{R}^{m\times n}$, $m\leq n$, and $\sigma_i(\bm{Y})$ is the $i$th largest singular value of $\bm{Y}$. Here, we give two special cases of the function $g$. If $g(t)=\|t\|_0$ as the $\ell_0$ norm, then  $\|\bm{Y}\|_{*,g}$ exactly reduces to  the rank function. If $g(t)=t$ as a linear function, then $\|\bm{Y}\|_{*,g}=\|\bm{Y}\|_*$ is exactly the nuclear norm, which is the tightest convex surrogate of the rank function.  However, the rank minimization is NP-hard, while the nuclear norm minimization may over-shrink the  singular values with large values\cite{gu2014weighted}.

To better approximate the rank function, we are interested in its nonconvex relaxation $\|\cdot\|_{*,g}$ with the function $g$ to be monotonically increasing, concave and smooth. For example, a decent choice of  $g:[0,\infty)\to \mathbb{R}$ is the logarithmic function defined as
\begin{equation}\label{eq:g}
g(t)=\log(t+\varepsilon),
\end{equation}
where $\varepsilon>0$.

Third, we propose a patchwise data fidelity term to restore images degraded by multiplicative noise.  Let $\bm{v}\in \mathbb{R}^N$ be the given noisy image and let $\bm{x}\in \mathbb{R}^N$ be the unknown clean log-transformed image to be restored. We extend the pixelwise data fidelity term of the Exp model\cite{lu2016multiplicative} as shown in Table~\ref{Table:TVModels} to a patchwise data fidelity term that is in terms of patch matrices  $R_j(\bm{x})$'s as follows
\begin{displaymath}
f(\bm{x})=\sum_{j=1}^J \mu_j\left (\left \langle R_j (\bm{x})+\frac{R_j(\bm{v})}{e^{R_j(\bm{x})}},R_j(\mathbbm{1}_N)\right \rangle_F+\rho \left \|\sqrt{\frac{e^{R_j(\bm{x})}}{R_j(\bm{v})}}-\gamma R_j(\mathbbm{1}_N)\right \|_F^2\right ),
\end{displaymath}
where  $\mu_j>0$ is a parameter,  $\mathbbm{1}_N$ denotes the vector of size $N\times 1$ with all ones, and  parameters $\rho> 0$ and $\gamma\geq 1$ depend on the noise level. The exponential operation, division operation and square root operation are componentwise operations. Note that it is followed from  \cite{lu2016multiplicative} that $f$ is strictly convex if $\rho\gamma^4\leq \frac{4096}{27}$.


The patchwise data fidelity term can be further viewed as a weighted pixelwise data fidelity term. Define $R_j^{\top}:\mathbb{R}^{m_j\times n_j}\to \mathbb{R}^N$ as $R_j^{\top}(\bm{Y})=\sum_{l=1}^{n_j}\bm{R}_{jl}^{\top}\bm{y}_i$, where $\bm{y}_i\in \mathbb{R}^{m_j}$ is the $i$th vector of $\bm{Y}$. Since $R_j$ and $R_j^{\top}$ are linear operators such that  $\langle R_j(\bm{x}), \bm{Y}\rangle_{F} =\langle \bm{x}, R_{j}^{\top}(\bm{Y})\rangle$ for all $\bm{x}\in \mathbb{R}^N$ and $\bm{Y}\in\mathbb{R}^{m_j\times n_j}$, where $\langle \cdot,\cdot\rangle$ is the standard inner product for vectors, then $f$ can be written as
\begin{align}
f(\bm{x})&=\sum_{j=1}^J \mu_j\left (\left \langle \bm{x}+\frac{\bm{v}}{e^{\bm{x}}}, (R_j^{\top}\circ R_j)\mathbbm{1}_N\right \rangle+\rho\left \|R_j\left (\sqrt{\frac{e^{\bm{x}}}{\bm{v}}}-\gamma \mathbbm{1}_N\right )\right \|_F^2\right )\nonumber\\
&=\langle \bm{x}+\frac{\bm{v}}{e^{\bm{x}}},\mathbbm{1}_N\rangle_{\bm{W}}+\rho \left \|\sqrt{\frac{e^{\bm{x}}}{\bm{v}}}-\gamma\mathbbm{1}_N\right \|_{\bm{W}}^2,\label{eq:f}
\end{align}
where  $\bm{W}=\sum_{j=1}^J\mu_j R_j^{\top}\circ R_j=\sum_{j=1}^J\mu_j\sum_{l=1}^{n_j} \bm{R}_{jl}^{\top}\bm{R}_{jl}$ is a diagonal matrix whose main diagonal entries indicate the weighted counts for each pixel. Since we assume that each pixel belongs to at least one nonlocal similar patch group, then $\bm{W}\in\mathbb{S}_+^N$ and  the $\bm{W}$-weighted inner product $\langle \cdot, \cdot\rangle_{\bm{W}}$  and the $\bm{W}$-weighted $\ell_2$ norm $\|\cdot\|_{\bm{W}}$ are well-defined. The proposed data fidelity term assigns more weights to the image pixels that belong to multiple patch groups. It helps develop efficient algorithms and cooperates well with the framework of our algorithm introduced in section~\ref{Section:Algorithm}.


Putting all the above discussion together, we come up with the following \emph{nonlocal low-rank model}
\begin{equation}\label{FullModel}
\min_{\bm{x},\bm{Y}_1,\dots,\bm{Y}_J}\; \tau f(\bm{x})+\sum_{j=1}^J \left ( \frac{\mu_j}{2}\|\bm{Y}_j-R_j(\bm{x})\|_{F}^2+ \lambda_j \sum_{i=1}^{m_j} g(\sigma_i(\bm{Y}_j))\right ),
\end{equation}
where  $\bm{x}\in\mathbb{R}^{N}$, $\bm{Y}_j\in \mathbb{R}^{m_j\times n_j}$, $f:\mathbb{R}^N\to(-\infty,+\infty]$ is defined as \eqref{eq:f}, $g:[0,\infty)\to\mathbb{R}$ is defined as \eqref{eq:g},  $R_j:\mathbb{R}^{N}\to \mathbb{R}^{m_j\times n_j}$  is the (normalized) extraction of $j$th nonlocal similar patch matrix,  $m_j\leq n_j$, $\tau>0$, $\mu_j>0$, $\lambda_j>0$, $j=1,\dots,J$.

Clearly, the objective function of model~\eqref{FullModel} is  nonconvex and nonsmooth. Existing algorithms are not directly applicable to this problem. It is challenging to design theoretically convergence-guaranteed and practically efficient algorithms to solve this nonconvex nonsmooth optimization problem. In the next section, we will propose an efficient algorithm for the nonlocal low-rank model~\eqref{FullModel} and analyze its convergence in section~\ref{Section:Convergence}.

\section{Proximal Alternating Reweighted Minimization Algorithm}\label{Section:Algorithm}

We present a proximal alternating reweighted minimization algorithm for solving the nonconvex nonsmooth optimization problem of model \eqref{FullModel}.

The nonlocal low-rank model, which has the form of model \eqref{FullModel}, regularizes the low-rank prior of patch matrices and  can also be applicable to many image restoration problems such as image denoising and compressive sensing if the patch matrix extraction $R_j$ and the data fidelity term $f$ are appropriately selected. In the following, we consider the nonlocal low-rank model in a general setting. The objective function of model~\eqref{FullModel}, denoted as $\Phi$, can be written as
\begin{equation}\label{eq:Phi}
\Phi(\bm{x},\bm{Y}_1,\dots,\bm{Y}_J)=\tau f(\bm{x})+ \sum_{j=1}^J \Phi_j(\bm{x},\bm{Y}_j),
\end{equation}
where \begin{equation}\label{eq:Phi_j}
\Phi_{j}(\bm{x},\bm{Y})=\frac{\mu_j}{2}\|\bm{Y}-R_j(\bm{x})\|_{F}^2 + \lambda_j \sum_{i=1}^{m_j} g(\sigma_i(\bm{\bm{Y}})),
\end{equation} and we assume
\begin{itemize}
\item[\textbf{(A1)}] $f:\mathbb{R}^N\to (-\infty,+\infty]$ is inf-bounded, proper and lower semicontinuous, i.e., $\inf f>-\infty$,
\begin{displaymath}
\dom f :=\{\bm{x}\in \mathbb{R}^N:f(\bm{x})<+\infty\}\neq \emptyset\quad\text{and}\quad f(\bm{a})\leq \underset{\bm{x}\to \bm{a}}{\lim\inf} f(\bm{x}),\quad \forall \bm{a}\in \mathbb{R}^N;
\end{displaymath}

\item[\textbf{(A2)}] $g:[0,\infty)\to \mathbb{R}$ is  monotonically increasing and concave (and nonconvex); and $g$ is continuously differentiable with an $L_g$-Lipschitz continuous gradient, i.e.,
\begin{displaymath}
|g'(t_1)-g'(t_2)|\leq L_g |t_1-t_2|, \quad \forall t_1\geq 0, t_2\geq 0;
\end{displaymath}

\item[\textbf{(A3)}] $\Phi(\bm{x},\bm{Y}_1,\dots,\bm{Y}_J)$ is coercive, i.e.,
\begin{displaymath}
\lim_{\|(\bm{x},\bm{Y}_1\dots,\bm{Y}_J)\|\to \infty}\Phi(\bm{x},\bm{Y}_1,\dots,\bm{Y}_J)=+ \infty.
\end{displaymath}
\end{itemize}

In the application of multiplicative noise removal, we utilize the nonlocal low-rank model \eqref{FullModel} with $f$ defined as \eqref{eq:f} and $g$ defined as  \eqref{eq:g}.  It is easy to verify that $f$ satisfies Assumption~(A1) and $g$ satisfies Assumption (A2). These together with the coercivity of $f$ and $g$ imply that $\Phi$ is inf-bounded and coercive. Hence, Assumption (A1)-(A3) hold for our proposed model.

In this general setting, no convexity or smoothness is assumed for $f$ and the objective function $\Phi$ of the nonlocal low-rank model \eqref{FullModel} is nonconvex and nonsmooth. For solving this nonconvex and nonsmooth optimization problem, the alternating minimization (AM) algorithm was adopted for compressive sensing\cite{dong2014compressive}  and the augmented Lagrange multiplier (ALM) algorithm was adopted for speckle noise removal\cite{wu2016speckle}. However, there is no guarantee that those methods will converge. Because the sequence generated by the AM algorithm may cycle indefinitely without converging if the minimum in each alternating step is not uniquely obtained\cite{bolte2014proximal}; and the sequence generated by the ALM algorithm may diverge even with bounded penalty parameters \cite{yu2018global}.  Therefore, we will propose an algorithm called the \textbf{Proximal Alternating Reweighted Minimization (PARM)} algorithm customized for model \eqref{FullModel} as shown below
\begin{align}
\bm{Y}_j^{k+1}&\in\underset{\bm{Y}_j}{\arg\!\min}\; \widetilde{\Phi}_{j}(\bm{x}^{k},\bm{Y}_j)+\frac{\alpha_{jk}}{2}\|\bm{Y}_j-\bm{Y}_j^k\|_{F}^2,\label{eq:Min_Yj}\\
\bm{x}^{k+1}&\in \underset{\bm{x}}{\arg\!\min}\; \Phi(\bm{x},\bm{Y}_1^{k+1},\dots,\bm{Y}_J^{k+1})+\frac{\beta_k}{2}\|\bm{x}-\bm{x}^k\|_{\bm{W}}^2,\label{eq:Min_X}
\end{align}
where $\widetilde{\Phi}_j(\bm{x}^{k},\bm{Y}_j)$ is a reweighted approximation of $\Phi_{j}(\bm{x}^{k},\bm{Y}_j)$ with respect to $\bm{Y}_j$, $\bm{W}=\sum_{j=1}^J\mu_j R_j^{\top}\circ R_j\in \mathbb{S}^N_+$, and $\alpha_{jk}> 0$ and $\beta_k> 0$ are parameters satisfying Assumption (A4).
\begin{itemize}
\item[\textbf{(A4)}] For the sequences $\{\alpha_{jk}\}_{k\in\mathbb{N}}$, $j=1,2,\dots, J$, and the sequence $\{\beta_k\}_{k\in\mathbb{N}}$, there exist positive constants $\alpha_{-}$,  $\alpha_{+}$,  $\beta_{-}$,  $\beta_{+}$ such that
\begin{align*}
\inf \{\alpha_{jk}:k\in\mathbb{N}, j=1,2,\dots,J\}\geq \alpha_{-},\quad &\text{ and }\quad \inf \{\beta_{k}:k\in\mathbb{N}\}\geq \beta_{-},\\
\sup \{\alpha_{jk}:k\in\mathbb{N}, j=1,2,\dots,J\}\leq \alpha_+, \quad &\text{ and }\quad \sup \{\beta_{k}:k\in\mathbb{N}\}\leq \beta_+.
\end{align*}
\end{itemize}
The convergence analysis of the PARM algorithm will be provided in the next section.

The proposed PARM algorithm has a proximal alternating scheme similar to the proximal alternating linearized minimization\cite{bolte2014proximal} for nonconvex and nonsmooth problems proposed by Bolte et al., in which a proximal term at the previous iterate is added to each subproblem. In \eqref{eq:Min_Yj}, we utilize  $\widetilde{\Phi}_{j}(\bm{x}^{k},\bm{Y}_j)$, a reweighted approximation of $\Phi_{j}(\bm{x}^{k},\bm{Y}_j)$, to approximate the nonconvex surrogate of the rank function, which yields a closed form for  \eqref{eq:Min_Yj}. In  \eqref{eq:Min_X}, the proximal term is in term of the $\bm{W}$-weighted norm, which is to be consistent with the patchwise data fidelity term $f$, for example, as defined in  \eqref{eq:f}. In fact, we will continue to use the $\bm{W}$-weighted norm to measure the variable $\bm{x}$ throughout the entire paper. Moreover, as an algorithm for nonlocal low-rank models applied to image restoration, the PARM algorithm can be intuitively interpreted as follows. Equation \eqref{eq:Min_Yj} can be viewed as a low-rank patch matrix estimation, which returns the nonlocal patch matrices $\bm{Y}_j$'s with a low-rank property, while  equation \eqref{eq:Min_X} can be viewed as the image restoration step, which aggregates all the estimated nonlocal patch matrices from   \eqref{eq:Min_Yj} to form the desired image $\bm{x}$.

Before further derive our PARM algorithm, we review some preliminaries on subdifferentials and proximity operators for nonconvex and nonsmooth functions.

\subsection{Preliminaries on subdifferentials and proximity operators}\label{Section:Notations}

For nonconvex and nonsmooth functions, we use the following definitions for subdifferentials and proximity operators.

\begin{definition}[Subdifferentials]\label{Def:Subdiff} Let $f:\mathbb{R}^d\to(-\infty,+\infty]$ be a proper and lower semicontinuous function.
\begin{enumerate}
\item[(1)] For a given $\bm{x}\in \dom f$, the Fr\'echet subdifferential of $f$ at $\bm{x}$, written $\hat{\partial} f(\bm{x})$, is the set of all vectors $\bm{u}\in \mathbb{R}^d$ which satisfy
\begin{displaymath}
\underset{\bm{y}\neq \bm{x}\; \bm{y}\to \bm{x}}{\lim\inf}\;\frac{f(\bm{y}) - f(\bm{x}) - \langle \bm{u}, \bm{y} - \bm{x}\rangle}{\|\bm{y}-\bm{x}\|_2} \geq 0.
\end{displaymath}
When $\bm{x}\notin \dom f$, we set $\hat{\partial} f(\bm{x})=\emptyset$.

\item[(2)] The subdifferential (or called the limiting-subdifferential) of $f$ at $\bm{x}\in \mathbb{R}^d$, written $\partial f(\bm{x})$, is defined through the following closure process
\begin{displaymath}
\partial f(\bm{x}):=\{\bm{u}\in \mathbb{R}^d: \exists \bm{x}_k\to \bm{x}, f(\bm{x}_k)\to f(\bm{x})\text{ and }\bm{u}_k\in \hat{\partial} f(\bm{x}_k)\to \bm{u} \text{ as }k\to \infty\}.
\end{displaymath}

\end{enumerate}
\end{definition}

\begin{definition}[Proximity operators]
Let $f:\mathbb{R}^d\to (-\infty,+\infty]$ be a proper and lower semicontinuous function such that $\inf_{\mathbb{R}^d} f> -\infty$. The proximity operator of $f$ at $\bm{x}\in \mathbb{R}^d$ is defined  as
\begin{displaymath}
\prox_{f}(\bm{x})= \underset{\bm{u}\in \mathbb{R}^d}{\arg\!\min}\; f(\bm{u})+\frac{1}{2}\|\bm{u}-\bm{x}\|_2^2.
\end{displaymath}
Note that $\prox_{f}(\bm{x})$ is a set-valued map. If $f$ is convex, then $\prox_{f}(\bm{x})$ is reduced to a single-valued map.
\end{definition}

The definitions above for subdifferentials and proximity operators are defined on vectors with respect to the standard $\ell_2$ norm. Without loss of generality, these definitions can be extended to  vectors with respect to the weighted $\ell_2$ norm and matrices with respect to the Frobenius norm.

Let $\bm{H}\in\mathbb{S}_+^d$.   The Fr\'echet subdifferential of  $f:\mathbb{R}^d\to (-\infty,+\infty]$ at a vector $\bm{x}\in \mathbb{R}^d$ with respect to $\bm{H}$  is denoted as $\hat{\partial}^{\bm{H}} f(\bm{x})$; its  subdifferential is denoted as $\partial^{\bm{H}} f(\bm{x})$; and its proximity operator is denoted as $\prox_{f}^{\bm{H}}(\bm{x})$.

For the function $f:\mathbb{R}^{m\times n}\to (-\infty,+\infty]$ at a matrix $\bm{X}\in \mathbb{R}^{m\times n}$ with respect to the Frobenius norm,   its Fr\'echet subdifferential is denoted as $\hat{\partial}^{F} f(\bm{X})$ or $\hat{\partial} f(\bm{X})$; its  subdifferential is denoted as $\partial^{F} f(\bm{X})$ or $\partial f(\bm{X})$; and its proximity operator is denoted as $\prox_{f}^{F}(\bm{X})$ or $\prox_{f}(\bm{X})$.

Now, we are ready to discuss in detail the proposed PARM algorithm in \eqref{eq:Min_Yj} and \eqref{eq:Min_X}.

\subsection{Patch matrix estimation via a reweighted scheme}

To estimate low-rank patch matrices, the minimization of $\Phi_{j}(\bm{x}^{k},\bm{Y}_j)$, as a generalized rank minimization of the patch matrix $\bm{Y}_j$, is approximated via a reweighted scheme, as shown in  \eqref{eq:Min_Yj}.

Since $g$ is concave on $[0,\infty)$ and continuously differentiable, by the definition of the supergradient, we have
\begin{equation} \label{Concave_g}
g(\sigma_i(\bm{Y}_j))\leq g(\sigma_i(\bm{Y}_j^{k}))+ (w_j^{k})_i(\sigma_i(\bm{Y}_j)-\sigma_i(\bm{Y}_j^k)),
\end{equation}
where $\bm{w}^{k}_j=[(w_j^{k})_1,\dots,(w_j^{k})_{m_j}]^{\top}$ and $(w_j^{k})_i=g'(\sigma_i(\bm{Y}_j^{k}))$, $i=1,2,\dots,m_j$. Then we replace the term $g(\sigma_i(\bm{Y}_j))$ in $\Phi_j(\bm{x}^{k},\bm{Y}_j)$ by the right hand side of the inequality \eqref{Concave_g} and have its reweighted approximation $\widetilde{\Phi}_j(\bm{x}^{k},\bm{Y}_j)$ as follows
\begin{align}\label{tildePhi_j}
\widetilde{\Phi}_j(\bm{x}^k,\bm{Y}_j)= \frac{\mu_j}{2}\|\bm{Y}_j-R_j (\bm{x}^{k})\|_{F}^2+\lambda_j \sum_{i=1}^{m_j} g(\sigma_i(\bm{Y}_j^{k}))+ (w_j^{k})_i(\sigma_i(\bm{Y}_j)-\sigma_i(\bm{Y}_j^k)).
\end{align}
Hence, the update of the low-rank patch matrix $\bm{Y}_j^{k+1}$ in  \eqref{eq:Min_Yj} at the $(k+1)$th step can be rewritten as follows
\begin{align}
\bm{Y}_j^{k+1}&\in\underset{\bm{Y}_j}{\arg\!\min} \;\frac{\mu_j}{2}\|\bm{Y}_j-R_j (\bm{x}^{k})\|_{F}^2+\lambda_j \sum_{i=1}^{m_j} (w_j^{k})_i\sigma_i(\bm{Y}_j)+\frac{\alpha_{jk}}{2}\|\bm{Y}_j-\bm{Y}_j^k\|_{F}^2\label{Min_Y_j^k+1}\\
&=\underset{\bm{Y}_j}{\arg\!\min}\; \lambda_j \sum_{i=1}^{m_j} (w_j^{k})_i\sigma_i(\bm{Y}_j)+\frac{\mu_j+\alpha_{jk}}{2}\left \|\bm{Y}_j-\frac{\mu_j R_j (\bm{x}^{k})+\alpha_{jk} \bm{Y}_j^k}{\mu_j+\alpha_{jk}}\right \|_{F}^2.\nonumber
\end{align}

By introducing the definition of the weighted nuclear norm of $\bm{Y}\in \mathbb{R}^{m\times n}$, $m\leq n$, with the weight vector $\bm{w}=[w_1,\dots,w_m]^{\top}$ and $w_i\geq 0$, $i=1,\dots,m$, as follows
\begin{displaymath}
\|\bm{Y}\|_{*,\bm{w}}=\sum_{i=1}^m w_i \sigma_i(\bm{Y}),
\end{displaymath}
where $\sigma_1(\bm{Y})\geq \sigma_2(\bm{Y})\geq \cdots\geq \sigma_m(\bm{Y})\geq 0$. It was proved in \cite{chen2013reduced} that $\|\cdot\|_{*,\bm{w}}$ is convex if and only if $w_1 \ge w_2 \ge \cdots \ge w_m \ge 0$. In other words, for $\|\cdot\|_{*,\bm{w}}$ being a convex function, the weights must increase with singular values. However, in order for large singular values to receive less penalty to help reducing the bias and smaller singular values to receive heavier penalty to help promoting sparsity, the opposite order of the weight is desirable, i.e.,   $0 \le w_1 \le w_2 \le \cdots \le w_m$. Under this order of the weights, the weighted nuclear norm is a nonconvex function and in general its proximity operator  may be a set-valued map.
Fortunately, the proximity operator is a single-value map, as shown in the following lemma.
\begin{lemma}[see {\cite[Theorem 2.3]{chen2013reduced}}]\label{ProxWeighted} For any $\lambda>0$, $\bm{Y}\in \mathbb{R}^{m\times n}$, $m\leq n$ and $\bm{w}=[w_1,\dots,w_m]^{\top}$ with $0 \leq w_1 \leq w_2 \leq \dots \leq w_m$,
\begin{displaymath}
\prox_{\lambda \|\cdot\|_{*,\bm{w}}}(\bm{Y})=\bm{U}S_{\lambda,\bm{w}}(\bm{\Sigma})\bm{V}^{\top},
\end{displaymath}
where $\bm{Y} = \bm{U\Sigma V}^{\top}$ is the singular value decomposition (SVD) of $\bm{Y}$ and  $S_{\lambda,\bm{w}}(\bm{\Sigma}) = \diag\{(\Sigma_{ii} - \lambda w_i)_+\}$ is the weighted singular value thresholding (WSVT) operator.
\end{lemma}

The assumption that $g$ is monotonically increasing and concave implies that $g'$ is nonnegative and monotonically decreasing. Then the weight vector $\bm{w}^{k}_j$ satisfies the ascending constraint, that is, $0\leq (w_j^k)_1\leq \dots\leq (w_j^k)_{m_j}$. Hence, by Lemma \ref{ProxWeighted}, the low-rank patch matrix $\bm{Y}_j^{k+1}$ can be uniquely achieved
\begin{align*}
\bm{Y}_j^{k+1}&=\prox_{\frac{\lambda_j}{\mu_j+\alpha_{jk}}\|\cdot\|_{*, \bm{w}^{k}_j}}\left (\frac{\mu_j R_j (\bm{x}^{k})+\alpha_{jk} \bm{Y}_j^k}{\mu_j+\alpha_{jk}}\right ),\\
&=\frac{1}{\mu_j+\alpha_{jk}}\prox_{\lambda_j\|\cdot\|_{*, \bm{w}^{k}_j}}\left (\mu_j R_j (\bm{x}^{k})+\alpha_{jk} \bm{Y}_j^k\right ),\\
&=\frac{1}{\mu_j+\alpha_{jk}}\bm{U}_j^{k+1}S_{\lambda_j,\bm{w}^{k}_j}(\widetilde{\bm{\Sigma}}^k_j)(\bm{V}^{k+1}_j)^{\top},
\end{align*}
where $\bm{U}_j^{k+1}\widetilde{\bm{\Sigma}}_j^{k} (\bm{V}_j^{k+1})^{\top}$ is the SVD of $\mu_j R_j (\bm{x}^{k})+\alpha_{jk} \bm{Y}_j^k$.

\begin{remark}
The ascending constraint on the weight vector $\bm{w}^{k}_j$ may not be automatically satisfied, if $g$ is not differentiable and  $(w_j^{k})_i$ is chosen as a supergradient of $g$ at $\sigma_i(\bm{Y}_j^k)$, i.e., $-(w_j^{k})_i\in \partial (-g)(\sigma_i(\bm{Y}_j^k))$, as defined in \cite{lu2014generalized}. For example, suppose that $g$ is not differentiable at $\sigma_{i_0}(\bm{Y}_j^k)$ and then  $\partial (-g)(\sigma_{i_0}(\bm{Y}_j^k))$ contains more than one element. If $\sigma_{i_0+1}(\bm{Y}_j^k)=\sigma_{i_0}(\bm{Y}_j^k)$, then the weights $-(w_j^k)_{i_0+1}$ and $-(w_j^k)_{i_0}$ that are selected from the same set $\partial (-g)(\sigma_{i_0}(\bm{Y}_j^k))=\partial(-g)(\sigma_{i_0+1}(\bm{Y}_j^k))$ may have $(w_j^k)_{i_0+1}<(w_j^k)_{i_0}$ rather than an ascending order. Thus, we have to carefully select the $(w_j^k)_{i}$ in the case where $g$ is not differentiable. For example, let $(w_j^k)_i=-\min \partial (-g)(\sigma_i(\bm{Y}_j^k))$.
\end{remark}

\subsection{Image restoration via W-weighted proximal regularization}

After obtaining the estimates of the low-rank patch matrices $\bm{Y}_j^{k+1}$'s from the generalized rank minimization in the previous step, we may have a situation where the same pixel may have several estimated values. That is because one pixel may belong to more than one nonlocal similar patch matrices, when we group nonlocal similar patches by block matching. Thus, at this image restoration step in \eqref{eq:Min_X} of the PARM algorithm,  we aggregate all the estimated patches to restore the entire image by minimizing the proximal regularization of $\Phi(\bm{x},\bm{Y}_1^{k+1},\dots,\bm{Y}_J^{k+1})$  with respect to $\bm{x}$.

Note that the term $\sum_{j=1}^J\frac{\mu_j}{2}\|\bm{Y}_j^{k+1}-R_j(\bm{x})\|_{F}^2$ in  $\Phi(\bm{x},\bm{Y}_1^{k+1},\dots,\bm{Y}_J^{k+1})$ can be written as
\begin{align*}
\sum_{j=1}^J\frac{\mu_j}{2}\|\bm{Y}_j^{k+1}-R_j(\bm{x})\|_{F}^2=& \sum_{j=1}^J \frac{\mu_j}{2}\|\bm{Y}_j^{k+1}-R_j(\bm{x}^k)\|_{F}^2+\sum_{j=1}^J\frac{\mu_j}{2}\|R_j(\bm{x})-R_j(\bm{x}^k)\|_{F}^2\\
& - \sum_{j=1}^{J}\mu_j\langle R_j(\bm{x})-R_j(\bm{x}^k), \bm{Y}_j^{k+1}-R_j(\bm{x}^k)\rangle_{F}.
\end{align*}

Recall that $R_j^{\top}:\mathbb{R}^{m_j\times n_j}\to \mathbb{R}^N$ is defined as $R_j^{\top}(\bm{Y})=\sum_{l=1}^{n_j}\bm{R}_{jl}^{\top}\bm{y}_i$, where $\bm{y}_i$ is the $i$th vector of $\bm{Y}$. Since $\langle R_j(\bm{x}), \bm{Y}\rangle_{F} =\langle \bm{x}, R_{j}^{\top}(\bm{Y})\rangle= \langle \bm{x}, \bm{W}^{-1}R_{j}^{\top}(\bm{Y})\rangle_{\bm{W}}$ and $\bm{W}=\sum_{j=1}^J\mu_j R_j^{\top}\circ R_j\in \mathbb{S}_+^N$, then the right hand side of the above equality can be written as
\begin{align*}
&\sum_{j=1}^J\frac{\mu_j}{2}\|\bm{Y}_j^{k+1}-R_j(\bm{x})\|_{F}^2\nonumber\\
=&\sum_{j=1}^J \frac{\mu_j}{2}\|\bm{Y}_j^{k+1}-R_j(\bm{x}^k)\|_{F}^2 +\frac{1}{2}\|\bm{x}-\bm{x}^k\|_{\bm{W}}^2- \langle \bm{x}-\bm{x}^k, \sum_{j=1}^{J}\mu_j \bm{W}^{-1} R_j^{\top}(\bm{Y}_j^{k+1})-\bm{x}^k\rangle_{\bm{W}} .
\end{align*}
The update of the estimated image $\bm{x}^{k+1}$ in  \eqref{eq:Min_X} at the $(k+1)$th step can be rewritten as follows
\begin{align}
\bm{x}^{k+1}&\in\underset{\bm{x}}{\arg\!\min}\;\tau f(\bm{x})- \langle \bm{x}-\bm{x}^k, \sum_{j=1}^{J}\mu_j \bm{W}^{-1} R_j^{\top}(\bm{Y}_j^{k+1})-\bm{x}^k\rangle_{\bm{W}} +\frac{\beta_k+1}{2}\|\bm{x}-\bm{x}^k\|_{\bm{W}}^2\label{Min_x^{k+1}}\\
&=\prox_{\frac{\tau}{\beta_k+1} f}^{\bm{W}}\left (\bm{x}^k +\frac{1}{\beta_k+1}\left (\sum_{j=1}^J\mu_j \bm{W}^{-1} R_j^{\top}(\bm{Y}_j^{k+1})-\bm{x}^k\right )\right ). \nonumber
\end{align}

The overall procedure of the PARM algorithm in   \eqref{eq:Min_Yj} and \eqref{eq:Min_X} is summarized in Algorithm \ref{Alg:Theoretical}.

\begin{algorithm}
\caption{Proximal alternating reweighted minimization algorithm for model \eqref{FullModel} }
\label{Alg:Theoretical}
\begin{algorithmic}[1]
\State{Set parameters $\tau$, $\mu_j$, $\lambda_j$, $\alpha_{jk}$, and $\beta_{k}$}
\State{Set extraction $R_j$ by block matching}
\State{Compute matrix $\bm{W}$}
\State{Initialize  $\bm{x}^0$, $\bm{Y}_j^0$, and $\bm{w}^{0}_j$}
\State{Set $k=0$}
\Repeat
	\For{$j$ from $1$ to $J$}
		\State{$[\bm{U}^{k+1}_j, \widetilde{\bm{\Sigma}}^k_j, \bm{V}_j^{k+1}]=\mathrm{SVD}\left (\mu_j R_j (\bm{x}^{k})+\alpha_{jk} \bm{Y}_j^k\right )$}\Comment{SVD}
		\State{$\bm{\Sigma}_j^{k+1}=\frac{1}{\mu_j+\alpha_{jk}} S_{\lambda_j,\bm{w}^{k}_j}(\widetilde{\bm{\Sigma}}^k_j)$}\Comment{WSVT}
		\State{$\bm{Y}_j^{k+1}=\bm{U}^{k+1}_j\bm{\Sigma}_{j}^{k+1}(\bm{V}_j^{k+1})^{\top}$}\Comment{Update $\bm{Y}_j^{k+1}$}
		\State{$(w^{k+1}_j)_i=g'((\Sigma_j^{k+1})_{ii})$}\Comment{Update $\bm{w}^{k+1}_j$}
	\EndFor
    	\State{$\bm{x}^{k+1}\in\prox_{\frac{\tau}{\beta_k+1}f}^{\bm{W}}\left (\bm{x}^k+ \frac{1}{\beta_k+1}\left (\sum_{j=1}^J\mu_j \bm{W}^{-1} R_j^{\top}(\bm{Y}_j^{k+1})-\bm{x}^k\right )\right )$}\Comment{Update $\bm{x}^{k+1}$}
    	\State{$k\leftarrow k+1$}
\Until{stopping criterion is satisfied}
\end{algorithmic}
\end{algorithm}

\subsection{The PARM algorithm for multiplicative noise removal}

To remove multiplicative noise, we apply the PARM algorithm in Algorithm~\ref{Alg:Theoretical} to solve the nonlocal low-rank model \eqref{FullModel} with $f$  defined as  \eqref{eq:f} and  $g$ defined as  \eqref{eq:g}.  Accordingly, using the definition of $f$ and $g$, Algorithm~\ref{Alg:Theoretical} can be specifically implemented as follows. In line 11 of Algorithm~\ref{Alg:Theoretical}, $(w_j^{k+1})_i=\frac{1}{(\Sigma_j^{k+1})_{ii}+\varepsilon}$; in line 13, the proximity operator of $\frac{\tau}{\beta_k+1}f$ with respect to $\bm{W}$  can be computed using Newton's method.  Given that $\rho\gamma^4\leq \frac{4096}{27}$, the function $f$ is strictly convex and hence $\prox_{\frac{\tau}{\beta_k+1}f}^{\bm{W}}$  is single-valued defined as
\begin{align*}
\bm{x}&=\prox_{\frac{\tau}{\beta_k+1}f}^{\bm{W}} (\widetilde{\bm{x}})=\underset{{\bm{x}}}{\arg\!\min}\; f(\bm{x})+\frac{\beta_k+1}{2\tau}\|\bm{x}-\widetilde{\bm{x}}\|_{\bm{W}}^2.
\end{align*}
Since $f$ is differentiable with respect to the $\bm{W}$-weighted $\ell_2$ norm with its gradient $\nabla^{\bm{W}}f(\bm{x})=\mathbbm{1}-\frac{\bm{v}}{e^{\bm{x}}}+\rho \left (\frac{e^{\bm{x}}}{\bm{v}}-\gamma\sqrt{\frac{e^{\bm{x}}}{\bm{v}}}\right )$, then  $\bm{x}$ is the unique solution of the following equation
\begin{displaymath}
\nabla^{\bm{W}}f(\bm{x})+\frac{\beta_k+1}{\tau}(\bm{x}-\widetilde{\bm{x}})=0,
\end{displaymath}
and this equation can be efficiently solved by Newton's method.

\section{Convergence Analysis}\label{Section:Convergence}

The aim of this section is to analyze the convergence of the PARM algorithm for model \eqref{FullModel}. The proof is motivated by the inexact descent convergence results for Kurdyka-\L ojasiewicz functions in \cite{attouch2013convergence,bolte2014proximal}.

In the sequel, we use the notation
\begin{displaymath}
\bm{Z}:=(\bm{x}, \bm{Y}_1,\dots, \bm{Y}_J)\quad\text{ and }\quad\|\bm{Z}\|:=\sqrt{\|\bm{x}\|^2_{\bm{W}}+\sum_{j=1}^J \|\bm{Y}_j\|_F^2},
\end{displaymath}
and we denote by $\Phi(\bm{Z})$ the objective function in model \eqref{FullModel}.

Here are three essential conditions to guarantee convergence of the sequence $\{\bm{Z}^k\}_{k\in\mathbb{N}}$ generated by the PARM algorithm.
\begin{enumerate}
\item[\textbf{(H1)}] \textbf{Sufficient descent condition:} There exists a positive constant $c_1$ such that for $\forall k\in \mathbb{N}$,
\begin{displaymath}
c_1 \|\bm{Z}^{k+1}-\bm{Z}^{k}\|^2\leq \Phi(\bm{Z}^{k})-\Phi(\bm{Z}^{k+1}).
\end{displaymath}
\item[\textbf{(H2)}] \textbf{Relative error condition:} There exists a positive constant $c_2$ such that  for $\forall k\in \mathbb{N}$,
\begin{displaymath}
\|\bm{A}^{k+1} \|\leq c_2\|\bm{Z}^{k+1}-\bm{Z}^k\|\quad\text{ and } \quad \bm{A}^{k+1}\in \partial \Phi(\bm{Z}^{k+1}).
\end{displaymath}
\item[\textbf{(H3)}] \textbf{Continuity condition:} There exists a subsequence $\{\bm{Z}^{k_t}\}_{t\in\mathbb{N}}$ and $\bm{Z}^*$ such that
\begin{displaymath}
\lim_{t\to\infty}\bm{Z}^{k_t}= \bm{Z}^*\quad\text{ and }\quad \lim_{t\to\infty}\Phi(\bm{Z}^{k_t})= \Phi(\bm{Z}^*).
\end{displaymath}
\end{enumerate}

In the following, we prove that the sequence $\{\bm{Z}^k\}_{k\in\mathbb{N}}$ satisfies Condition (H1)-(H3), and then conclude that $\{\bm{Z}^k\}_{k\in\mathbb{N}}$ converges to a critical point of $\Phi$ using the fact that $\Phi$ is a Kurdyka-\L ojasiewicz function.

\subsection{Sufficient descent condition}
We show that the objective function $\Phi$ in model~\eqref{FullModel} evaluated at $\bm{Z}^k$, denoted $\Phi(\bm{Z}^k)$, decreases sufficiently as $k$ increases.

\begin{proposition}[Sufficient descent condition]\label{Thm:PhiDecreasing} Suppose that the objective function $\Phi$ in model \eqref{FullModel} satisfies Assumption (A1)-(A3). Let $\{\bm{Z}^k\}_{k\in \mathbb{N}}$ be the sequence generated by the PARM algorithm provided that the parameters satisfy Assumption (A4). Then $\{\Phi(\bm{Z}^k)\}_{k\in\mathbb{N}}$ is strictly decreasing and, in particular, there exists a positive constant $c_1$ such that for $\forall k\in\mathbb{N}$,
\begin{equation}\label{eq:PhiDecreasing}
c_1 \|\bm{Z}^{k+1}-\bm{Z}^{k}\|^2\leq \Phi(\bm{Z}^{k})-\Phi(\bm{Z}^{k+1}).
\end{equation}
\end{proposition}

\begin{proof}

Let $\Phi_j(\bm{x},\bm{Y})$ be defined as \eqref{eq:Phi_j} and let $\widetilde{\Phi}_{j}(\bm{x}^{k},\bm{Y}_j)$ be defined as \eqref{tildePhi_j}. Then, according the concavity of $g$ illustrated in inequality~\eqref{Concave_g},  $\Phi_j(\bm{x}^k,\bm{Y}_j^{k+1})$ and  its reweighted approximation $\widetilde{\Phi}_{j}(\bm{x}^{k},\bm{Y}_j)$ have the following relationship
\begin{displaymath}
\Phi_j(\bm{x}^{k},\bm{Y}_j^{k+1})\leq \widetilde{\Phi}_{j}(\bm{x}^{k},\bm{Y}_j^{k+1})\quad\text{and}\quad\Phi_j(\bm{x}^{k},\bm{Y}_j^{k})=\widetilde{\Phi}_{j}(\bm{x}^{k},\bm{Y}_j^{k}).
\end{displaymath}
Thus, the objective function $\Phi$ in \eqref{eq:Phi} evaluated at $\bm{x}^k$ and $\bm{Y}_j^{k+1}$'s can be rewritten as
\begin{align*}
\Phi(\bm{x}^{k},\bm{Y}_1^{k+1},\dots,\bm{Y}_J^{k+1})=\tau f(\bm{x}^k)+ \sum_{j=1}^J \Phi_j(\bm{x}^k,\bm{Y}_j^{k+1})\leq \tau f(\bm{x}^k)+ \sum_{j=1}^J \widetilde{\Phi}_{j}(\bm{x}^k,\bm{Y}_j^{k+1}).
\end{align*}
By the update of $\bm{Y}_j^{k+1}$ in  \eqref{eq:Min_Yj}, we have
\begin{displaymath}
\widetilde{\Phi}_{j}(\bm{x}^k,\bm{Y}_j^{k+1})+\frac{\alpha_{jk}}{2}\|\bm{Y}_j^{k+1}-\bm{Y}_j^k\|_{F}^2\leq \widetilde{\Phi}_{j}(\bm{x}^k,\bm{Y}_j^{k})= \Phi_j(\bm{x}^k,\bm{Y}_j^{k}).
\end{displaymath}

Combining the two inequalities above,  we have the following inequality on\\  $\Phi(\bm{x}^{k},\bm{Y}_1^{k+1},\dots,\bm{Y}_J^{k+1})$ and $\Phi(\bm{Z}^{k})$
\begin{align*}
\Phi(\bm{x}^{k},\bm{Y}_1^{k+1},\dots,\bm{Y}_J^{k+1})&\leq \tau f(\bm{x}^k)+ \sum_{j=1}^J \Phi_j(\bm{x}^k,\bm{Y}_j^{k})-  \sum_{j=1}^J\frac{\alpha_{jk}}{2}\|\bm{Y}_j^{k+1}-\bm{Y}_j^k\|_{F}^2,\\
&=\Phi(\bm{Z}^{k})-  \sum_{j=1}^J\frac{\alpha_{jk}}{2}\|\bm{Y}_j^{k+1}-\bm{Y}_j^k\|_{F}^2.
\end{align*}

By the update of $\bm{x}^{k+1}$ in  \eqref{eq:Min_X}, we have
\begin{displaymath}
\Phi(\bm{Z}^{k+1})\leq \Phi(\bm{x}^{k},\bm{Y}_1^{k+1},\dots,\bm{Y}_J^{k+1})-\frac{\beta_k}{2}\|\bm{x}^{k+1}-\bm{x}^k\|_{\bm{W}}^2.
\end{displaymath}

Combining the two inequalities above,  we have that $\Phi(\bm{Z}^{k+1})$ and $\Phi(\bm{Z}^{k})$ satisfy the following inequality
\begin{displaymath}
\frac{\beta_k}{2}\|\bm{x}^{k+1}-\bm{x}^{k}\|^2_{\bm{W}}+\sum_{j=1}^J\frac{\alpha_{jk}}{2}\|\bm{Y}_j^{k+1}-\bm{Y}_j^k\|_{F}^2\leq \Phi(\bm{Z}^{k})-\Phi(\bm{Z}^{k+1}).
\end{displaymath}
Equation \eqref{eq:PhiDecreasing} holds with $c_1=\frac{1}{2}\min \{\beta_-,\alpha_{-}\}>0$ and $\{\Phi(\bm{Z}^k)\}_{k\in\mathbb{N}}$ is strictly decreasing. Here, $\beta_-$ and $\alpha_{-}$ are two positive parameters given in Assumption (A4).
\end{proof}

The sufficient descent condition proved in Proposition~\ref{Thm:PhiDecreasing} immediately yields the following corollary.

\begin{corollary} \label{Cor:SeqConv} Suppose that the objective function $\Phi$ in model \eqref{FullModel} satisfies Assumption (A1)-(A3). Let $\{\bm{Z}^k\}_{k\in \mathbb{N}}$ be the sequence generated by the PARM algorithm provided that the parameters satisfy Assumption (A4). Then
\begin{displaymath}
\lim_{k\to \infty} \|\bm{Z}^{k}-\bm{Z}^{k+1}\|=0.
\end{displaymath}

\end{corollary}

\begin{proof}
Summing inequality \eqref{eq:PhiDecreasing} from $k=0$ to $k=K-1$, we have
\begin{displaymath}
c_1\sum_{k=0}^{K-1}\|\bm{Z}^{k+1}-\bm{Z}^{k}\|^2\leq \Phi(\bm{Z}^{0})-\Phi(\bm{Z}^{K})\leq \Phi(\bm{Z}^{0})-\Phi_{\inf},
\end{displaymath}
where $\Phi_{\inf}=\inf_{\bm{Z}} \Phi(\bm{Z})>-\infty$.

Taking $K\to\infty$, we have
\begin{displaymath}
\sum_{k=0}^{\infty}\|\bm{Z}^{k+1}-\bm{Z}^{k}\|^2<\infty,
\end{displaymath}
which implies $\lim_{k\to\infty} \|\bm{Z}^{k+1}-\bm{Z}^{k}\|=0$.
\end{proof}

\subsection{Relative error condition}

Before proving that a subgradient of $\Phi$ at $\bm{Z}^{k+1}$ is upper bounded by the iterates gap, we first characterize the subdifferential of $\Phi$.

Recall that the variable $\bm{x}$ is measured in terms of the $\bm{W}$-weight $\ell_2$ norm and that the variables $\bm{Y}_j$'s are measured in terms of the Frobenius norm. Then using the notations introduced in subsection~\ref{Section:Notations} we define the subdifferential of $\Phi$ by
\begin{displaymath}
\partial \Phi (\bm{Z})=\left \{(\bm{A}_{\bm{x}},\bm{A}_{\bm{Y}_1},\dots,\bm{A}_{\bm{Y}_J}):\bm{A}_{\bm{x}}\in\partial_{\bm{x}}^{\bm{W}} \Phi (\bm{Z}),\bm{A}_{\bm{Y}_j}\in\partial_{\bm{Y}_j} \Phi (\bm{Z}),j=1,\dots,J\right \},
\end{displaymath}
where $\partial_{\bm{x}}^{\bm{W}} \Phi (\bm{Z})$ is the partial subdifferential of $\Phi$ with respect to the variable $\bm{x}$ and with respect to the  $\bm{W}$-weight $\ell_2$ norm and $\partial_{\bm{Y}_j} \Phi (\bm{Z})$ is the partial subdifferential of $\Phi$ with respect to the variable $\bm{Y}_j$ and with respect to the Frobenius norm.

By the definition of $\Phi$ in model \eqref{FullModel} and the fact that
\begin{displaymath}
\sum_{j=1}^J \frac{\mu_j}{2}\|\bm{Y}_j-R_j(\bm{x})\|_{F}^2= \frac{1}{2}\langle\bm{x},\bm{x}-2\sum_{j=1}^J \mu_j\bm{W}^{-1}R_j^{\top}(\bm{Y}_j)\rangle_{\bm{W}}+\sum_{j=1}^J\frac{\mu_j}{2}\|\bm{Y}_j\|_F^2,
\end{displaymath}
 we have
\begin{displaymath}
\partial_{\bm{x}}^{\bm{W}} \Phi (\bm{Z})=\tau\partial^{\bm{W}} f(\bm{x})+\bm{x}-\sum_{j=1}^J \mu_j\bm{W}^{-1}R_j^{\top}(\bm{Y}_j)
\end{displaymath}
and
\begin{displaymath}
\partial_{\bm{Y}_j} \Phi (\bm{Z})=\mu_j (\bm{Y}_j-R_j(\bm{x}))+\lambda_j\partial\left ( \sum_{i=1}^{m_j}  g\circ\sigma_i\right )(\bm{Y}_j).
\end{displaymath}

To compute the subdifferential of the singular value function  $\sum_{i=1}^{m_j}  g\circ\sigma_i$ and further characterize $\partial_{\bm{Y}_j} \Phi (\bm{Z})$, we introduce some definitions and a lemma on singular value functions in \cite{lewis2005nonsmooth1,lewis2005nonsmooth2}.

\begin{definition} A function $f:\mathbb{R}^n\to \mathbb{R}$ is absolutely symmetric if
\begin{displaymath}
f(x_1,x_2,\dots,x_n)=f(|x_{\pi(1)}|,|x_{\pi(2)}|,\dots,|x_{\pi(n)}|),
\end{displaymath}
for any permutation $\pi$.
\end{definition}

\begin{definition}
A function $F:\mathbb{R}^{m\times n}\to  \mathbb{R}$, $m\leq n$, is a singular value function if $F(\bm{X})=(f\circ \sigma)(\bm{X})$, where $f:\mathbb{R}^m\to \mathbb{R}$ is  an absolutely symmetric function, $\sigma(\bm{X})=[\sigma_1(\bm{X}),\dots,$ $\sigma_m(\bm{X})]^{\top}$ and $\sigma_i(\bm{X})$ is the $i$th largest singular value of $\bm{X}$.
\end{definition}

The function $\sum_{i=1}^{m} g\circ\sigma_i$ can be viewed as a singular value function of the form
\begin{displaymath}
\left (\sum_{i=1}^{m} g\circ\sigma_i\right )(\bm{Y})=(\widetilde{g}\circ \sigma)(\bm{Y}),
\end{displaymath}
where $\widetilde{g}:\mathbb{R}^{m}\to \mathbb{R}$ is defined as $\widetilde{g}(\bm{t})=\sum_{i=1}^{m} g(|t_i|)$ and is absolutely symmetric.

\begin{lemma}\label{Lemma:singularvalue}
The subdifferential of a singular value function $f\circ \sigma$ at $\bm{X}\in \mathbb{R}^{m\times n}$ is given by the formula
\begin{displaymath}
\partial (f\circ \sigma)(\bm{X})=\left \{\bm{U}\diag(\bm{d})\bm{V}^{\top}:\bm{d}\in \partial f(\sigma(\bm{X})), (\bm{U},\bm{V})\in \mathcal{M}(\bm{X})\right \},
\end{displaymath}
where $\mathcal{M}(\bm{X})=\left \{(\bm{U},\bm{V})\in \mathbb{R}^{m\times l}\times \mathbb{R}^{n\times l}:\bm{U}^{\top}\bm{U}=\bm{V}^{\top}\bm{V}=\bm{I},\bm{X}=\bm{U}\diag (\sigma(\bm{X}))\bm{V}^{\top}\right \}$.
\end{lemma}

By Lemma~\ref{Lemma:singularvalue}, the subdifferential of $\sum_{i=1}^{m}  g\circ\sigma_i$ at $\bm{Y}\in \mathbb{R}^{m\times n}$ can be computed as follows
\begin{align*}
\partial\left ( \sum_{i=1}^{m}  g\circ\sigma_i\right)(\bm{Y}) =\{\bm{U}\diag(\bm{d})\bm{V}^{\top}:d_i=c_i g'[\sigma_i(\bm{Y})], c_i\in\partial|\cdot|(\sigma_i(\bm{Y})), &i=1,\dots,m,\\
& (\bm{U},\bm{V})\in \mathcal{M}(\bm{Y})\},
\end{align*}
where
\begin{displaymath}
\partial |\cdot|(\sigma_i(\bm{Y}))=
\begin{cases} \{1\},\quad &\text{if }\sigma_i(\bm{Y})>0;\\
[-1,1],\quad &\text{if }\sigma_i(\bm{Y})=0.
\end{cases}
\end{displaymath}

Next, we are ready to derive a subgradient of $\Phi$ at $\bm{Z}^{k+1}$ using the lemma below and to prove that it is upper bounded.

\begin{lemma} \label{Lemma:A_Y} Suppose that the objective function $\Phi$ in model \eqref{FullModel} satisfies Assumption (A1)-(A3). Let $\{\bm{Z}^k\}_{k\in \mathbb{N}}$ be the sequence generated by the PARM algorithm provided that the parameters satisfy Assumption (A4). Let $\bm{U}^{k+1}_j\bm{\Sigma}^{k+1}_j(\bm{V}^{k+1}_j)^{\top}$ be the SVD of $\bm{Y}_j^{k+1}$. Then, for each $k$ and each $j$, there exists $\bm{c}_j^{k+1}\in \mathbb{R}^{m_j}$ such that
\begin{equation}\label{eq:A1}
(c_j^{k+1})_i\in \partial|\cdot| (\sigma_i(\bm{Y}^{k+1}_j)),\quad i=1,\dots,m_j,
\end{equation}
and
\begin{equation}\label{eq:A2}
\lambda_j \bm{U}_j^{k+1}\diag(\bm{d}_j^{k+1})(\bm{V}_j^{k+1})^{\top}=-\alpha_{jk}(\bm{Y}_j^{k+1}-\bm{Y}_j^k)-\mu_j(\bm{Y}_j^{k+1}-R_j(\bm{x}^{k})),
\end{equation}
where $\diag(\bm{d}_j^{k+1})=\diag(\bm{c}_j^{k+1})\diag(\bm{w}_j^{k})$.
\end{lemma}

\begin{proof}
According to the update of $\bm{Y}_j^{k+1}$ in  \eqref{Min_Y_j^k+1}, we have
\begin{displaymath}
0\in \mu_j(\bm{Y}_j^{k+1}-R_j(\bm{x}^{k}))+\lambda_j \partial \|\cdot\|_{*,\bm{w}_j^{k}}(\bm{Y}_j^{k+1}) +\alpha_{jk}(\bm{Y}_j^{k+1}-\bm{Y}_j^k).
\end{displaymath}
Since the weighted nuclear norm $\|\cdot\|_{*,\bm{w}}$ is a singular value function, then by Lemma~\ref{Lemma:singularvalue} the subdifferential of $\|\cdot\|_{*,\bm{w}}$ can be computed as follows
\begin{displaymath}
\partial \|\cdot\|_{*,\bm{w}}(\bm{Y}) =\{\bm{U}\diag (\bm{d})\bm{V}^{\top}:d_i=c_i w_i, c_i\in\partial|\cdot|(\sigma_i(\bm{Y})), i=1,\dots,m, (\bm{U},\bm{V})\in \mathcal{M}(\bm{Y})\}.
\end{displaymath}
Note that $(\bm{U}_j^{k+1},\bm{V}_j^{k+1})\in \mathcal{M}(\bm{Y}_j^{k+1})$.
Thus, there exists $\bm{c}^{k+1}_j\in \mathbb{R}^{m_j}$ such that \eqref{eq:A1}  holds and
\begin{displaymath}
-\alpha_{jk}(\bm{Y}_j^{k+1}-\bm{Y}_j^k)-\mu_j(\bm{Y}_j^{k+1}-R_j(\bm{x}^{k}))=\lambda_j \bm{U}_j^{k+1}\diag (\bm{d}_j^{k+1})(\bm{V}_j^{k+1})^{\top}\in \lambda_j\partial\|\cdot\|_{*,\bm{w}_j^{k}}(\bm{Y}_j^{k+1}),
\end{displaymath}
where $\diag (\bm{d}_j^{k+1})=\diag (\bm{c}_j^{k+1})\diag (\bm{w}_j^{k})$.
\end{proof}

\begin{proposition}[Relative error condition]\label{Thm:UpperBounded}  Suppose that the objective function $\Phi$ in model \eqref{FullModel} satisfies Assumption (A1)-(A3). Let $\{\bm{Z}^k\}_{k\in \mathbb{N}}$ be the sequence generated by the PARM algorithm provided that the parameters satisfy Assumption (A4). Let $\bm{U}^{k+1}_j\bm{\Sigma}^{k+1}_j(\bm{V}^{k+1}_j)^{\top}$ be the SVD of $\bm{Y}_j^{k+1}$ and let  $\bm{c}^{k+1}_j$ and $\bm{d}^{k+1}_j$ be in $\mathbb{R}^{m_j}$ satisfying  \eqref{eq:A1} and \eqref{eq:A2}.

Define $\bm{A}^{k+1}=(\bm{A}_{\bm{x}}^{k+1},\bm{A}_{\bm{Y}_1}^{k+1},\dots,\bm{A}_{\bm{Y}_J}^{k+1})$, where
\begin{equation}
\bm{A}_{\bm{x}}^{k+1}=\beta_k (\bm{x}^{k+1}-\bm{x}^k)\label{A_x}
\end{equation}
and
\begin{equation}
\bm{A}_{\bm{Y}_j}^{k+1}=\mu_j(R_j(\bm{x}^{k})-R_j(\bm{x}^{k+1}))+\lambda_j \bm{U}^{k+1}_j\diag (\widetilde{\bm{d}}_j^{k+1}-\bm{d}_j^{k+1})(\bm{V}^{k+1}_j)^{\top}-\alpha_{jk}(\bm{Y}_j^{k+1}-\bm{Y}_j^k),\label{A_Y}
\end{equation}
where $\diag (\widetilde{\bm{d}}_j^{k+1})=\diag (\bm{c}_j^{k+1})\diag (\bm{w}_j^{k+1})$.

Then the following assertions hold for $\forall k\in \mathbb{N}$,
\begin{enumerate}
\item[(a)] $\bm{A}^{k+1}\in \partial \Phi(\bm{Z}^{k+1})$;

\item[(b)] $\|\bm{A}^{k+1} \|\leq c_2\|\bm{Z}^{k+1}-\bm{Z}^k\|$, for some $c_2>0$.
\end{enumerate}
\end{proposition}

\begin{proof} (a)
According to the update of $\bm{x}^{k+1}$ in  \eqref{Min_x^{k+1}}, we have
\begin{displaymath}
0\in \tau\partial^{\bm{W}} f(\bm{x}^{k+1})+\bm{x}^k-\sum_{j=1}^{J} \mu_j \bm{W}^{-1}R_j^{\top}(\bm{Y}_j^{k+1})+(\beta_k+1) (\bm{x}^{k+1}-\bm{x}^k).
\end{displaymath}
Then the definition of $A_{\bm{x}}^{k+1}$ in  \eqref{A_x} implies
\begin{displaymath}
\bm{A}_{\bm{x}}^{k+1}\in \tau\partial^{\bm{W}} f(\bm{x}^{k+1})+\bm{x}^{k+1}-\sum_{j=1}^{J} \mu_j\bm{W}^{-1} R_j^{\top}(\bm{Y}_j^{k+1})=\partial_{\bm{x}}^{\bm{W}}\Phi (\bm{Z}^{k+1}).
\end{displaymath}

Also, for each $j$, the definition of $\bm{A}_{\bm{Y}_j}^{k+1}$ in  \eqref{A_Y} and  Lemma~\ref{Lemma:A_Y} imply
\begin{align*}
\bm{A}_{\bm{Y}_j}^{k+1}&=\mu_j(\bm{Y}_j^{k+1}-R_j(\bm{x}^{k+1}))+\lambda_j \bm{U}_j^{k+1}\diag (\widetilde{\bm{d}}_j^{k+1})(\bm{V}_j^{k+1})^{\top} \\
&\in  \mu_j(\bm{Y}_j^{k+1}-R_j(\bm{x}^{k+1}))+\lambda_j \partial \|\cdot\|_{*,\bm{w}_j^{k+1}}(\bm{Y}_j^{k+1})\\
&= \mu_j(\bm{Y}_j^{k+1}-R_j(\bm{x}^{k+1}))+\lambda_j\partial\left ( \sum_{i=1}^{m_j}  g\circ\sigma_i\right )(\bm{Y}_j^{k+1})\\
&=\partial_{\bm{Y}_j}\Phi (\bm{Z}^{k+1}).
\end{align*}

(b) It follows from the Cauchy-Schwarz inequality that
\begin{displaymath}
\|\bm{A}^{k+1} \| \leq \|\bm{A}_{\bm{x}}^{k+1}\|_{\bm{W}}+\sum_{j=1}^J\|\bm{A}_{\bm{Y}_j}^{k+1}\|_F,
\end{displaymath}
where $\|\bm{A}_{\bm{x}}^{k+1}\|_{\bm{W}}=\beta_k\|\bm{x}^{k+1}-\bm{x}^k\|_{\bm{W}}$ and
\begin{align*}
\|\bm{A}_{\bm{Y}_j}^{k+1}\|_F\leq &\mu_j\|R_j(\bm{x}^{k})-R_j(\bm{x}^{k+1})\|_F+\lambda_j \|\bm{U}^{k+1}_j\diag (\widetilde{\bm{d}}_j^{k+1}-\bm{d}_j^{k+1})(\bm{V}^{k+1}_j)^{\top}\|_F\\
&+\alpha_{jk}\|\bm{Y}_j^{k+1}-\bm{Y}_j^k\|_F.
\end{align*}
The right hand side of the above inequality can be computed term by term as follows. The square of the first term is bounded above by the square of the weighted iterates of the variable $\bm{x}$,
\begin{displaymath}
\mu_j^2\|R_j(\bm{x}^{k})-R_j(\bm{x}^{k+1})\|_F^2 \leq \mu_j \sum_{j=1}^J\mu_j\|R_j(\bm{x}^{k}-\bm{x}^{k+1})\|_F^2 =\mu_j \|\bm{x}^{k}-\bm{x}^{k+1}\|_{\bm{W}}^2.
\end{displaymath}
This implies that $\mu_j\|R_j(\bm{x}^{k})-R_j(\bm{x}^{k+1})\|_F\leq \sqrt{\mu_j} \|\bm{x}^{k}-\bm{x}^{k+1}\|_{\bm{W}}$.

Also, the second term is bounded above by the iterates of the variable $\bm{Y}_j$. Since $\|\widetilde{\bm{d}}_j^{k+1}-\bm{d}_j^{k+1}\|_2\leq \|\widetilde{\bm{d}}_j^{k+1}-\bm{d}_j^{k+1}\|_1$, then we have
\begin{align*}
\lambda_j \|\bm{U}^{k+1}_j\diag (\widetilde{\bm{d}}_j^{k+1}-\bm{d}_j^{k+1})(\bm{V}^{k+1}_j)^{\top}\|_F&=\lambda_j\|\widetilde{\bm{d}}_j^{k+1}-\bm{d}_j^{k+1}\|_2\\
&\leq\lambda_j \sum_{i=1}^{m_j} \left |({c}_j^{k+1})_i\right |\left |(w_j^{k+1})_i-(w_j^{k})_i\right |\\
&\leq \lambda_j \sum_{i=1}^{m_j}\left |g'(\sigma_i(\bm{Y}_j^{k+1}))-g'(\sigma_i(\bm{Y}_j^{k}))\right|.
\end{align*}
Using the condition that $g'$ is  $L_g$-Lipschitz continuous, we further obtain
\begin{align*}
\lambda_j \|\bm{U}^{k+1}_j\diag (\widetilde{\bm{d}}_j^{k+1}-\bm{d}_j^{k+1})(\bm{V}^{k+1}_j)^{\top}\|_F&\leq  \lambda_j \sum_{i=1}^{m_j} L_g|\sigma_i(\bm{Y}_j^{k+1})-\sigma_i(\bm{Y}_j^{k})|\\
&\leq  \lambda_j m_j L_g\|\bm{Y}_j^{k+1}-\bm{Y}_j^{k}\|_F,
\end{align*}
where the last line is followed from Theorem 3.3.16 in \cite{horn1994topics} and $\|\bm{Y}_j^{k+1}-\bm{Y}_j^{k}\|_2\leq \|\bm{Y}_j^{k+1}-\bm{Y}_j^{k}\|_F$.

Therefore, combining all the inequalities above, we obtain
\begin{align*}
\|\bm{A}^{k+1}\|&\leq  (\beta_k+M_{\mu})\|\bm{x}^{k+1}-\bm{x}^k\|_{\bm{W}} +\sum_{j=1}^J (\lambda_j m_j L_g +\alpha_{jk})\|\bm{Y}_j^{k+1}-\bm{Y}_j^{k}\|_F\\
&\leq c_2 \|\bm{Z}_j^{k+1}-\bm{Z}_j^{k}\|,
\end{align*}
where $M_{\mu}=\sum_{j=1}^J \sqrt{\mu_j}$ and $c_2=\max\{\beta_++M_{\mu},\lambda_1 m_1 L_g +\alpha_{+},\dots,\lambda_J m_J L_g +\alpha_{+}\}$.
\end{proof}

The relative error condition proved in Proposition~\ref{Thm:UpperBounded} immediately  yields the following corollary.
\begin{corollary} \label{Cor:SubdifZero}  Suppose that the objective function $\Phi$ in model \eqref{FullModel} satisfies Assumption (A1)-(A3). Let $\{\bm{Z}^k\}_{k\in \mathbb{N}}$ be the sequence generated by the PARM algorithm provided that the parameters satisfy Assumption (A4). Define $\bm{A}^{k+1}=(\bm{A}_{\bm{x}}^{k+1},\bm{A}_{\bm{Y}_1}^{k+1},\dots,\bm{A}_{\bm{Y}_J}^{k+1})$, where $\bm{A}_{\bm{x}}^{k+1}$ is defined as \eqref{A_x} and $\bm{A}_{\bm{Y}_j}^{k+1}$ is defined as \eqref{A_Y}, $j=1,\dots,J$.   Then
\begin{displaymath}
\lim_{k\to\infty}\|\bm{A}^{k+1}\|= 0.
\end{displaymath}
\end{corollary}
\begin{proof}
The result is immediately followed by Proposition~\ref{Thm:UpperBounded} and Corollary~\ref{Cor:SeqConv}.
\end{proof}

\subsection{Continuity condition}

We first show the existence of a limit point of $\{\bm{Z}^k\}_{k\in\mathbb{N}}$ using the boundedness of  $\{\bm{Z}^k\}_{k\in\mathbb{N}}$, and then prove a continuity condition for any convergent subsequence of $\{\bm{Z}^k\}_{k\in\mathbb{N}}$, which implies Condition (H3).

\begin{proposition}\label{Thm:Continuity}   Suppose that the objective function $\Phi$ in model \eqref{FullModel} satisfies Assumption~(A1)-(A3). Let $\{\bm{Z}^k\}_{k\in \mathbb{N}}$ be the sequence generated by the PARM algorithm provided that the parameters satisfy Assumption (A4). Let $\mathcal{S}$ denote the set of all limit points of the sequence $\{\bm{Z}^k\}_{k\in\mathbb{N}}$. Then the following assertions hold.

\begin{itemize}
\item[(a)] $\mathcal{S}\neq \emptyset$;

\item[(b)] If $\{\bm{Z}^{k_t}\}_{t\in\mathbb{N}}$ is a subsequence of $\{\bm{Z}^k\}_{k\in\mathbb{N}}$ such that $\lim_{t\to\infty}\bm{Z}^{k_t}= \bm{Z}^*\in \mathcal{S}$, then
\begin{displaymath}
\lim_{t\to\infty}\Phi(\bm{Z}^{k_t})= \Phi(\bm{Z}^*).
\end{displaymath}
\end{itemize}

\end{proposition}

\begin{proof} (a) We show that $\{\bm{Z}^k\}_{k\in\mathbb{N}}$ is bounded by contradiction.

Assume for the sake of contradiction that there exists a subsequence $\{\bm{Z}^{k_l}\}_{l\in\mathbb{N}}$ such that $\|\bm{Z}^{k_l}\|\to \infty$ as $l\to \infty$. According to Assumption (A3), $\Phi$ is coercive, and then $\Phi(\bm{Z}^{k_l})\to \infty$ as $l\to \infty$. However, since $\{\Phi(\bm{Z}^k)\}_{k\in\mathbb{N}}$ is strictly decreasing and lower bounded by $\Phi_{\inf}>-\infty$, then $\{\Phi(\bm{Z}^k)\}_{k\in\mathbb{N}}$ converges and $\{\Phi(\bm{Z}^{k_l})\}_{l\in\mathbb{N}}$ also converges, which yields a contradiction. Thus, $\{\bm{Z}^{k}\}_{k\in\mathbb{N}}$ is bounded and there exists a convergent subsequence of $\{\bm{Z}^k\}_{k\in\mathbb{N}}$.

(b) Let $\{\bm{Z}^{k_t}\}_{t\in\mathbb{N}}$ be  a subsequence such that $\bm{Z}^{k_t}\to \bm{Z}^*$ as $t\to\infty$.

Since $f$ is lower semicontinuous, then we have
\begin{displaymath}
\liminf_{t\to\infty} \;\tau f(\bm{x}^{k_t+1}) \geq \tau  f(\bm{x}^*).
\end{displaymath}
From the update of $\bm{x}^{k_t+1}$ referring to  \eqref{Min_x^{k+1}}, we obtain the following inequality
\begin{align*}
&\tau f(\bm{x}^{k_t+1})+ \langle \bm{x}^{k_t+1}-\bm{x}^{k_t},  \bm{x}^{k_t}-\sum_{j=1}^J\mu_j \bm{W}^{-1}R_j^{\top}(\bm{Y}_j^{k_t+1})\rangle_{\bm{W}} +\frac{\beta_{k_t}+1}{2}\|\bm{x}^{k_t+1}-\bm{x}^{k_t}\|_{\bm{W}}^2\\
\leq &\tau f(\bm{x}^*)+ \langle \bm{x}^{*}-\bm{x}^{k_t},  \bm{x}^{k_t}-\sum_{j=1}^J\mu_j \bm{W}^{-1}R_j^{\top}(\bm{Y}_j^{k_t+1})\rangle_{\bm{W}} +\frac{\beta_{k_t}+1}{2}\|\bm{x}^{*}-\bm{x}^{k_t}\|_{\bm{W}}^2.
\end{align*}
 Letting $t\to\infty$ on both sides of the above inequality, we get
\begin{align*}
&\limsup_{t \to\infty} \;\tau f(\bm{x}^{k_t+1})\\
\leq &\tau f(\bm{x}^*)+\limsup_{t\to \infty}\;\langle \bm{x}^{*}-\bm{x}^{k_t}, \bm{x}^{k_t}-\sum_{j=1}^J\mu_j \bm{W}^{-1}R_j^{\top}(\bm{Y}_j^{k_t+1})\rangle_{\bm{W}} +\frac{\beta_{k_t}+1}{2}\|\bm{x}^*-\bm{x}^{k_t}\|_{\bm{W}}^2\\
=&\tau f(\bm{x}^*),
\end{align*}
where we use the boundedness of the sequences $\{\bm{x}^{k_t+1}\}_{t\in\mathbb{N}}$, $\{\bm{Y}^{k_t+1}_j\}_{t\in\mathbb{N}}$ and $\{\beta_{k_t}\}_{t\in\mathbb{N}}$ and the result that $\lim_{t\to\infty}\|\bm{x}^{k_t+1}-\bm{x}^{k_t}\|_{\bm{W}}\to 0$  followed from Corollary~\ref{Cor:SeqConv}.
Hence, \\$\lim_{t \to\infty} \tau f(\bm{x}^{k_t+1})=f(\bm{x}^*)$.

Due to the continuity of $\frac{\mu_j}{2}\|\bm{Y}_j-R_j(\bm{x})\|_{F}^2$ with respect to $\bm{Y}_j$ and  $\bm{x}$ and the continuity of $g(t)$ with respect to $t$, we have
\begin{align*}
\lim_{t\to\infty} \Phi(\bm{Z}^{k_t+1})= \tau f(\bm{x}^{*})+\sum_{j=1}^{J}\left (\frac{\mu_j}{2}\|\bm{Y}^{*}_j-R_j(\bm{x}^{*})\|_{F}^2+\lambda_j\sum_{i=1}^{m_j}g(\sigma_i(\bm{Y}_j^*))\right )=\Phi(\bm{Z}^{*}).
\end{align*}
\end{proof}

\subsection{Convergence results}

In this subsection, we show the convergence of  the sequence $\{\bm{Z}^k\}_{k\in\mathbb{N}}$ generated by the PARM algorithm.

Let us first review a definition and a theorem on the Kurdyka-\L ojasiewicz (KL) property of a function in \cite{attouch2010proximal,attouch2013convergence}.
\begin{definition}[Kurdyka-\L ojasiewicz] Let $f:\mathbb{R}^d\to (-\infty,+\infty]$ be proper and lower semicontinuous.
\begin{enumerate}
\item[(a)] The function $f$ is called to have the Kurdyka-\L ojasiewicz (KL) property at $\tilde{\bm{x}}\in \dom \partial f$ if there exist $\eta\in (0,+\infty]$, a neighborhood $U$ of $\tilde{\bm{x}}$ and a continuous function $\varphi:[0,\eta)\to [0,\infty)$ such that
\begin{enumerate}
\item[(i)] $\varphi(0)=0$;
\item[(ii)] $\varphi$ is $C^1$ on $(0,\eta)$ and continuous at $0$;
\item[(iii)] for all $s\in (0,\eta)$, $\varphi'(s)>0$;
\item[(iv)] for all $\bm{x}\in U\cap \{\bm{x}\in \mathbb{R}^d: f(\tilde{\bm{x}})<f(\bm{x})<f(\tilde{\bm{x}})+\eta\}$, the following Kurdyka-\L ojasiewicz inequality holds
\begin{displaymath}
\varphi'(f(\bm{x})-f(\tilde{\bm{x}}))\text{dist}(\bm{0},\partial f(\bm{x}))\geq 1.
\end{displaymath}
\end{enumerate}

\item[(b)] The function $f$ is called a KL function if $f$ has the KL property at each point of $\dom \partial f$.
\end{enumerate}
\end{definition}

\begin{theorem}[see {\cite[Theorem 2.9]{attouch2013convergence}}]\label{Thm:KLConv} Let $f:\mathbb{R}^{d}\to (-\infty,+\infty]$ be a proper lower semicontinuous function. Consider a sequence $\{\bm{x}^k\}_{k\in\mathbb{N}}$ that satisfies Condition (H1)-(H3). If $f$ has the Kurdyka-\L ojasiewicz property at the limit point $\bm{x}^*$ specified in (H3), then the sequence  $\{\bm{x}^k\}_{k\in\mathbb{N}}$ converges to $\bm{x}^*$ as $k$ goes to $\infty$, and $\bm{x}^*$ is a critical point of $f$. Moreover, the sequence $\{\bm{x}^k\}_{k\in\mathbb{N}}$ has a finite length, i.e.,
\begin{displaymath}
\sum_{k=0}^{\infty}\|\bm{x}^{k+1}-\bm{x}^{k}\|<\infty.
\end{displaymath}

\end{theorem}

The KL theory is a powerful tool for nonconvex nonsmooth optimization problems and KL functions are ubiquitous. For example, for multiplicative noise removal, the objective function $\Phi$ in model \eqref{FullModel} with $f$ defined as  \eqref{eq:f} and $g$  as \eqref{eq:g} is a KL function. For more examples of KL functions see \cite{attouch2010proximal,attouch2013convergence}.

Next, equipped with Condition (H1)-(H3) discussed in the previous subsections, we can show  that any limit point of $\{\bm{Z}^k\}_{k\in\mathbb{N}}$ is a critical point of $\Phi$ in the following theorem.
\begin{theorem}
\label{Thm:Crit}  Suppose that the objective function $\Phi$ in model \eqref{FullModel}  satisfies Assumption (A1)-(A3). Let $\{\bm{Z}^k\}_{k\in \mathbb{N}}$ be the sequence generated by the PARM algorithm provided that the parameters satisfy Assumption (A4). Let $\mathcal{S}$ denote the set of all limit points of the sequence $\{\bm{Z}^k\}_{k\in\mathbb{N}}$ and let $\operatorname{crit}(\Phi)$ denote the set of all critical points of the function $\Phi$. Then $\emptyset\neq\mathcal{S}\subseteq \operatorname{crit}(\Phi)$, that is,  any limit point of $\{\bm{Z}^k\}_{k\in\mathbb{N}}$ is a critical point of $\Phi$.
\end{theorem}
\begin{proof}\ \  Let $\bm{Z}^*$ be in $\mathcal{S}\neq \emptyset$ and let  $\{\bm{Z}^{k_t}\}_{t\in\mathbb{N}}$ be a subsequence of $\{\bm{Z}^k\}_{k\in\mathbb{N}}$  such that $ \lim_{t\to\infty}\bm{Z}^{k_t}= \bm{Z}^*$. Then by Proposition~\ref{Thm:Continuity}, $\lim_{t\to\infty}\Phi(\bm{Z}^{k_t})= \Phi(\bm{Z}^*)$. Also, it follows from  Proposition~\ref{Thm:UpperBounded} and Corollary~\ref{Cor:SubdifZero} that $\bm{A}^{k_{t}}\in \partial \Phi(\bm{Z}^{k_{t}})$ and $\bm{A}^{k_{t}}\to \bm{0}$ as $t\to \infty$. Thus, by the definition of subdifferential in  Definition~\ref{Def:Subdiff}, we have $\bm{0}\in \partial\Phi (\bm{Z}^*)$.
\end{proof}

In addition to Condition (H1)-(H3), if $\Phi$ is a Kurdyka-\L ojasiewicz (KL) function, then a stronger convergence result can be achieved for the sequence $\{\bm{Z}^k\}_{k\in\mathbb{N}}$. That is, we can prove that the sequence $\{\bm{Z}^k\}_{k\in \mathbb{N}}$ itself converges  a critical point of $\Phi$ using the KL theory.

\begin{theorem}   Suppose that the objective function $\Phi$ in model \eqref{FullModel} satisfies Assumption (A1)-(A3). Let $\{\bm{Z}^k\}_{k\in \mathbb{N}}$ be the sequence generated by the PARM algorithm provided that the parameters satisfy Assumption (A4). If  $\Phi$ is a KL function, then the following assertions hold.

\begin{enumerate}

\item[(a)] The sequence $\{\bm{Z}^k\}_{k\in\mathbb{N}}$ has finite length, that is,
\begin{displaymath}
\sum_{k=0}^{\infty}\|\bm{Z}^{k+1}-\bm{Z}^{k}\|<\infty;
\end{displaymath}

\item[(b)] The sequence $\{\bm{Z}^k\}_{k\in\mathbb{N}}$ converges to a critical point of $\Phi$.

\end{enumerate}
\end{theorem}

\begin{proof}
It follows from Proposition~\ref{Thm:PhiDecreasing},  Proposition~\ref{Thm:UpperBounded} and Proposition~\ref{Thm:Continuity} that the sequence  $\{\bm{Z}^k\}_{k\in \mathbb{N}}$ satisfies Condition (H1)-(H3), respectively. Then the results (a) and (b) immediately follow from  Theorem~\ref{Thm:KLConv}.
\end{proof}

\section{Numerical Results}\label{Section:Experiment}

In this section, we first describe a practical version of Algorithm~\ref{Alg:Theoretical} and then test the proposed algorithms to solve the proposed nonlocal low-rank model for multiplicative noise removal. We compare our proposed method with six existing methods: the DZ method\cite{dong2013convex}, the HNW method \cite{huang2009new},  the I-DIV method\cite{steidl2010removing}, the TwL-mV method\cite{kang2013two}, the learned dictionary (Dict) method \cite{huang2012multiplicative} and the SAR-BM3D method\cite{parrilli2011nonlocal}. Numerical results show superior performance of the proposed method over the existing ones.

The experiments were implemented in Matlab 2016b running a 64 bit Ubuntu 18.04 system and executed on an eight-core Intel Xeon E5-2640v3 128GB CPU at 2.6 GHz, with four NVIDIA Tesla P100 16GB GPUs. The proposed algorithms were accelerated using graphics processing units (GPUs), as the estimation of each patch matrix can be computed in parallel.

\subsection{Practical version of PARM algorithm}

The PARM algorithm presented in Algorithm~\ref{Alg:Theoretical} converges theoretically as shown in section~\ref{Section:Convergence}, if the patch extraction operator $R_j$ is assumed to be fixed. The extraction operator $R_j$ plays an important role in improving the denoising performance because a better initialization of $R_j$ can yield to a better denoised image. In the case in which the optimal $R_j$ is not available, it is empirically challenging to find an appropriate choice of $R_j$ only with a noisy image given.
\begin{algorithm}
\caption{Practical version of the PARM algorithm}
\label{Alg:Practical}
\begin{algorithmic}[1]
\State{Set parameters $\tau$, $\mu_j$, $\lambda_j$, and $\beta_{k}$}
\State{Initialize  $\bm{x}^0$ and $\bm{w}_j^{0}$}
\For{$k$ from $0$ to $K-1$}
\State{Set extraction $\hat{R_j}$ by block matching}
\State{Compute matrix $\bm{\hat{W}}$}
	\For{$j$ from $1$ to $J$}
		\State{$[\bm{U}_j^{k+1}, \widetilde{\bm{\Sigma}}^{k}_j, \bm{V}_j^{k+1}]=SVD\left (\hat{R_j} (\bm{x}^{k})\right )$}\Comment{SVD}
		\State{$\bm{\Sigma}_j^{k+1}=S_{\lambda_j/\mu_j ,\bm{w}_j^{k}}(\widetilde{\bm{\Sigma}}^{k}_j)$}\Comment{WSVT}
		\State{$\bm{Y}_j^{k+1}=\bm{U}^{k+1}_j\bm{\Sigma}_j^{k+1}(\bm{V}_j^{k+1})^{\top}$} \Comment{Update $\bm{Y}_j^{k+1}$}
		\State{$(w_j^{k+1})_i=g'((\Sigma_j^{k+1})_{ii})$}\Comment{Update $\bm{w}_j^{k+1}$}
	\EndFor
    	\State{$\bm{x}^{k+1}\in \prox_{\frac{\tau}{\beta_k+1}f}^{\bm{\hat{W}}}\left (\bm{x}^k+ \frac{1}{\beta_k+1}\left (\sum_{j=1}^J\mu_j \bm{\hat{W}}^{-1} \hat{R_j}^{\top}(\bm{Y}_j^{k+1})-\bm{x}^k\right)\right )$}\Comment{Update $\bm{x}^{k+1}$}
    \EndFor
\end{algorithmic}
\end{algorithm}

Here, we provide a practical version of the PARM algorithm with dynamically updated patch extraction operator, denoted as $\hat{R_j}$. The operator $\hat{R_j}$ is recomputed at each step by block matching based on the update of the estimated image $e^{\bm{x}^{k}}$, and the weighted counts matrix, now denoted as $\bm{\hat{W}}$, is recomputed based on the newest $\hat{R_j}$. As a result of this dynamically updating scheme on $\hat{R_j}$, the patch matrix $\bm{Y}_j^{k+1}$ and $\bm{Y}_j^{k}$ may not refer to the same patch group.  That is because $\bm{Y}_j^{k+1}$ is associated with $\hat{R_j}(\bm{x}^{k})$, while $\bm{Y}_j^{k}$ is associated with $\hat{R_j}(\bm{x}^{k-1})$ using a different extraction operator $\hat{R_j}$. Hence, in this  practical version of the PARM algorithm, we set that $\bm{Y}_j^{k+1}$ is  updated  without using the previous update $\bm{Y}_j^k$ and its parameter $\alpha_{jk}$. The overall procedure of a practical version of the PARM algorithm for multiplicative noise removal is summarized as Algorithm~\ref{Alg:Practical}.

\subsection{Parameter settings}

First, we utilize block matching and normalization with mean zero to extract patch matrices using the following parameter settings for  block matching. In Algorithm~\ref{Alg:Theoretical}, the fixed  extraction  $R_j$ is initialized via block matching based on the estimated image from the SAR-BM3D method; and in  Algorithm~\ref{Alg:Practical}, the dynamically updated extraction $\hat{R_j}$ is computed at each step via block matching based on the update of $e^{\bm{x}^{k}}$. Besides this, both algorithms share the same parameter settings for  block matching,  including the search window, the patch size and the number of patches in each patch group  as presented in Table~\ref{Table:ParaRj}.
\begin{table}[htbp]
\centering
\small
\caption{Settings for block matching.}\label{Table:ParaRj}
\begin{tabular}{cccc}
\hline
$L$ & \textbf{Search window} & \textbf{Patch size} & \textbf{Patch number} \\
\hline
1 & 50 & $10\times 10$ &  150\\
3 & 50 & $9\times 9$ &   120 \\
5 & 50 & $8\times 8$ & 100  \\
\hline
\end{tabular}
\end{table}

Second, we set the model parameters and algorithm parameters for Algorithm~\ref{Alg:Theoretical} and Algorithm~\ref{Alg:Practical}, respectively. The model parameters $\tau$,  $\lambda_j$'s, $\mu_j$'s, $\rho$, $\gamma$ and $\varepsilon$ are adaptive to the noise level. The algorithm parameters ($\alpha_{jk}$'s and) $\beta_k$'s  influence the computational speed. The settings of the above parameters  are presented in Table~\ref{Table:ParaTheor} and Table~\ref{Table:ParaPrac}.
\begin{table}[htbp]
\centering
\small
\caption{Parameter settings for Algorithm~\ref{Alg:Theoretical}.}\label{Table:ParaTheor}
\begin{tabular}{ccccccccccccc}
\hline
\multirow{2}{*}{$L$ } & \multicolumn{2}{c}{\textbf{Standard images}} & & \multicolumn{2}{c}{\textbf{Remote images}} & & \multicolumn{6}{c}{\textbf{Common parameters}} \\
\cline{2-3}\cline{5-6}\cline{8-13}
&$\tau$ & $\lambda_j$ & & $\tau$ & $\lambda_j$ & & $\mu_j$ & $\rho$ & $\gamma$ & $\varepsilon$ & $\alpha_{jk}$ & $\beta_k$\\
\hline
1 & $\beta_k/50$ &  1.8 & & $\beta_k/100$ &  1 & & 1 & 0.01 & 4 & $10^{-10}$ & 0.001 & 1.001 \\
3 & $\beta_k/150$ & 1 & & $\beta_k/150$ & 0.45 & &  1 & 1.5 & 1.9 & $10^{-10}$ & 0.001 & 1.001\\
5 & $\beta_k/250$ & 0.6 & &  $\beta_k/200$& 0.15 & & 1 & 2 & 1.3 & $10^{-10}$ & 0.001 & 1.001 \\
\hline
\end{tabular}
\end{table}
\begin{table}[htbp]
\centering
\small
\caption{Parameter settings for Algorithm~\ref{Alg:Practical}.}\label{Table:ParaPrac}
\begin{tabular}{cccccccccccc}
\hline
\multirow{2}{*}{$L$ } & \multicolumn{2}{c}{\textbf{Standard images}} & & \multicolumn{2}{c}{\textbf{Remote images}} & & \multicolumn{5}{c}{\textbf{Common parameters}} \\
\cline{2-3}\cline{5-6}\cline{8-12}
 &$\tau$ & $\lambda_j$ & & $\tau$ & $\lambda_j$ & & $\mu_j$ & $\rho$ & $\gamma$ & $\varepsilon$ & $\beta_k$\\
\hline
1 & $\beta_k/50$ &  2.6 & & $\beta_k/50$ &  2.6 & & 1 & 0.01 & 4 & $10^{-10}$ & 1.001 \\
3 & $\beta_k/150$ & 1.3 & & $\beta_k/150$ &  1.2 & & 1 & 1.5 & 1.9 & $10^{-10}$ & 1.001\\
5 & $\beta_k/250$ & 0.8 & &  $\beta_k/250$& 0.7 & & 1 & 2 & 1.3 & $10^{-10}$ & 1.001 \\
\hline
\end{tabular}
\end{table}

Third, the initialization settings and the stopping criteria are set differently for  Algorithm~\ref{Alg:Theoretical} and Algorithm~\ref{Alg:Practical}. Algorithm~\ref{Alg:Theoretical} is initialized using the estimated image from the SAR-BM3D method and is  terminated if the relative error reaches a tolerance threshold as follows
\begin{displaymath}
\frac{\|\bm{x}^{k+1}-\bm{x}^{k}\|_{\bm{W}}}{\|\bm{x}^{k}\|_{\bm{W}}}< \max \left \{10^{-3}, \frac{\|\bm{x}^{1}-\bm{x}^{0}\|_{\bm{W}}}{\|\bm{x}^{0}\|_{\bm{W}}}\times 50\%\right \}.
\end{displaymath}
Algorithm~\ref{Alg:Practical} is initialized using the given noisy image and terminated by experience based on the number of iterations $K$. For $L=1, 3, 5$, $K$ is set to 65-70, 23-25, 18-20, respectively.

Lastly, the restored image is estimated by $\hat{\bm{u}}=e^{\hat{\bm{x}}}$, where $\hat{\bm{x}}$ is the log-transformed image obtained from the proposed algorithms.

\subsection{Numerical results tested on standard test images}

In this experiment, we use standard test images ``Monarch", ``Lena" and ``House" all of size $256\times 256$, as shown in Figure~\ref{TestImg:Standard}. To generate the observed images, we degrade the original test images by multiplicative Gamma noise at $L=1$, $L=3$ and $L=5$.

\begin{figure}[htbp]
	\centering
	\begin{subfigure}[t]{.26\textwidth}\centering
        \includegraphics[height=3.9cm]{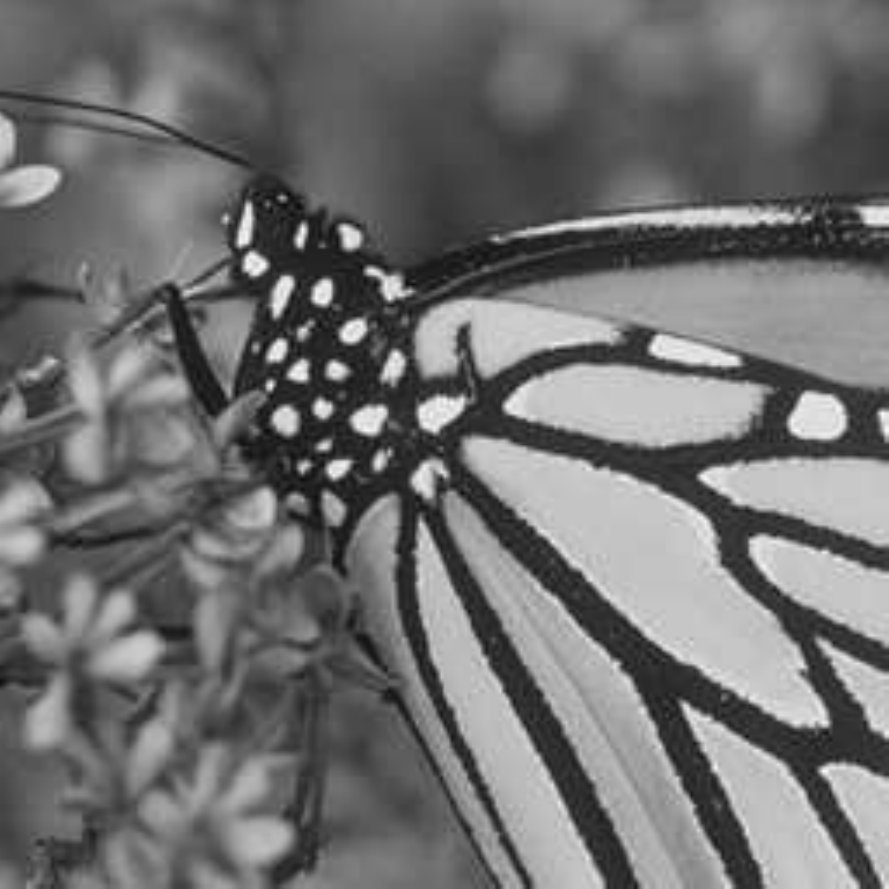}
        \caption{Monarch}
    \end{subfigure}
    	\begin{subfigure}[t]{.26\textwidth}\centering
        \includegraphics[height=3.9cm]{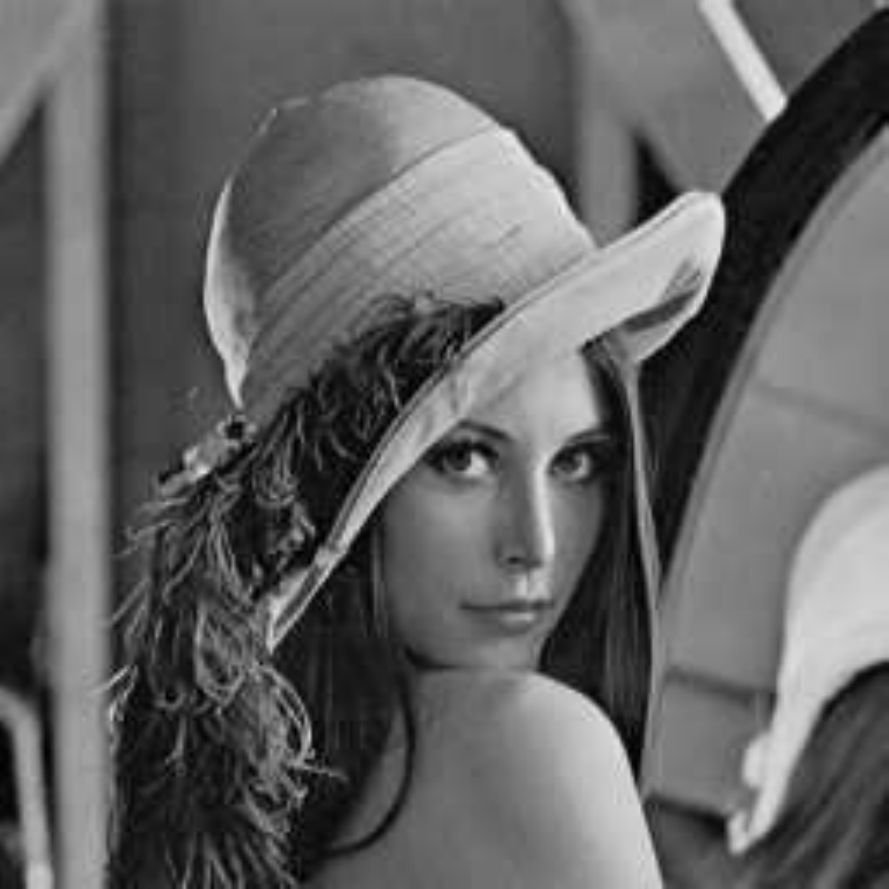}
        \caption{Lena}
    \end{subfigure}
    	\begin{subfigure}[t]{.26\textwidth}\centering
        	\includegraphics[height=3.9cm]{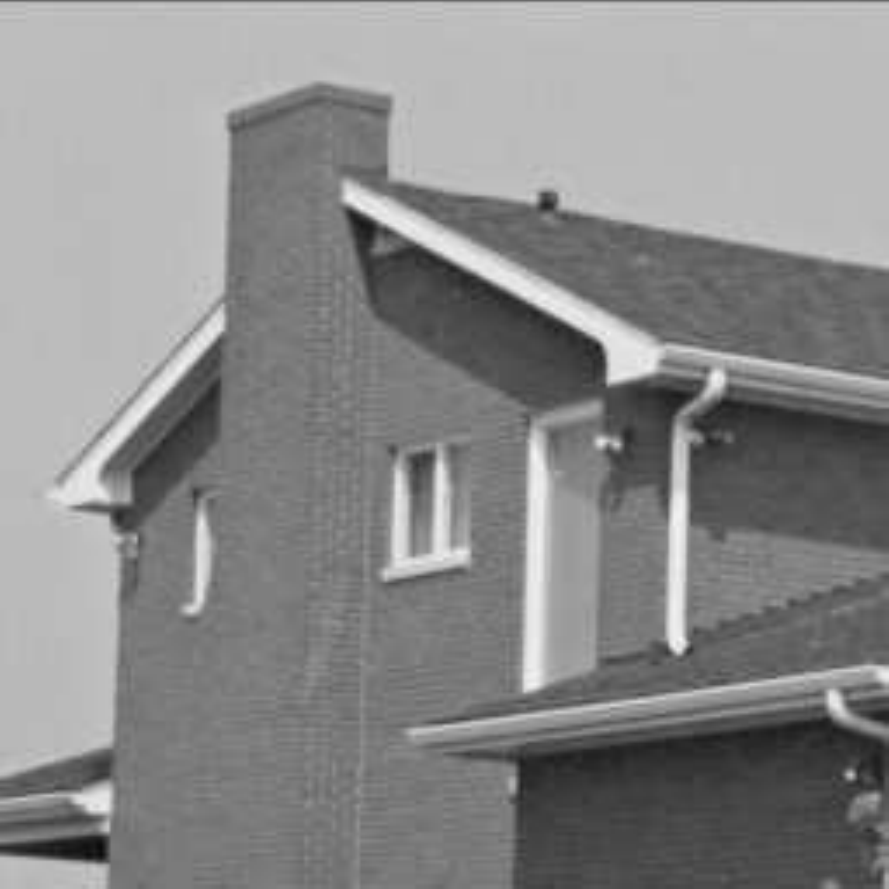}
        \caption{House}
    \end{subfigure}
	\caption{Standard test images.}\label{TestImg:Standard}
\end{figure}

The evaluation of the image quality is measured in the intensity format between the original image $\bm{u}\in \mathbb{R}^N$ and the estimated image $\hat{\bm{u}}\in \mathbb{R}^N$,  using the peak-signal-to-noise ratio (PSNR) defined as
\begin{displaymath}
\text{PSNR}=10\log_{10}\left ( \frac{255^2N }{\|\bm{u}-\hat{\bm{u}}\|_2^2}\right )
\end{displaymath}
and the structural similarity index measure (SSIM)\cite{zhou2004image}.

\begin{table}[htbp]
\centering
\caption{Numerical results tested on standard test images at different noise levels  by different methods.}\label{Table:Standard}
{\small
\begin{tabular}{lcccccccccc}
\hline
\textbf{Image} & $L$ & \textbf{Meas.} & \textbf{Alg 1} &\textbf{Alg 2} & \textbf{SAR-} &\textbf{DZ}& \textbf{HNW}&\textbf{I-DIV}  & \textbf{TwL-} & \textbf{Dict} \\
 &  &  &  & &  \textbf{BM3D} && & & \textbf{4V} &  \\
\hline
Monarch &1 & PSNR &  \textbf{21.94} & \underline{21.55} & 21.36 & 19.38 &19.73 & 19.91 & 19.26 & 19.50 \\
&  & SSIM & \underline{0.6926} & \textbf{0.6966} & 0.6404  & 0.5758 &0.5523 & 0.5883 & 0.5848 & 0.5726\\
 &3 & PSNR & \textbf{24.90} & \underline{24.69} & 24.48 & 22.66 &22.55 & 22.69 & 22.43 & 23.02 \\
&  & SSIM & \underline{0.8051} & \textbf{0.8102} & 0.7693  & 0.7156 &0.7049 & 0.7244 & 0.7096 & 0.7449\\
&5 & PSNR & \textbf{26.31} & \underline{26.24} & 25.78  & 24.04 &23.88 & 23.98 & 23.74 & 24.38\\
&  & SSIM & \underline{0.8524} & \textbf{0.8529} & 0.8232  & 0.7648 &0.7588 & 0.7723 & 0.7621 & 0.7740\\
&&&&&&&&&\\
Lena & 1 & PSNR & \textbf{23.74} & \underline{23.43} & 23.20  & 21.33 & 21.66 & 21.95 & 21.47 & 21.96\\
& & SSIM & \underline{0.6975} & \textbf{0.7082} & 0.6480  & 0.6027 &0.5551 & 0.5947 & 0.6123 & 0.6106\\
 &3 & PSNR &  \textbf{26.44} & \underline{26.29} & 26.00 & 24.06 &24.36 & 24.48 & 24.29 & 24.81 \\
&  & SSIM & \underline{0.7892} & \textbf{0.7944} & 0.7596  & 0.6907 &0.6911 & 0.7073 & 0.7128 & 0.7379\\
 &5 & PSNR & \textbf{27.85} & \underline{27.63} & 27.39  & 25.42 &25.66 & 25.79 & 25.64 & 25.77\\
&  & SSIM & \textbf{0.8308} & \underline{0.8306} & 0.8094  & 0.7469 &0.7455 & 0.7596 & 0.7604 & 0.7621\\
&&&&&&&&&\\
House & 1 & PSNR & \underline{23.42} & \textbf{23.90}  & 22.83  & 21.52 &21.57 & 21.99 & 21.72 & 21.70\\
& & SSIM & \underline{0.6726} & \textbf{0.7179}& 0.5916 & 0.6119 &0.4925 & 0.5860 & 0.6017 & 0.5801 \\
 &3 & PSNR & \underline{27.20} & \textbf{27.32} & 26.54  & 24.16 &24.26 & 24.51 & 24.25 & 23.84\\
&  & SSIM &  \textbf{0.7823} & \underline{0.7798} & 0.7139  & 0.6806 &0.6365 & 0.6938 & 0.6597 & 0.6602\\
&5 & PSNR & \underline{29.04} &\textbf{29.12} & 28.36  & 25.70 &25.73 & 25.84 & 25.79 & 24.56\\
&  & SSIM & \underline{0.8115} & \textbf{0.8163} & 0.7641  & 0.7339 &0.6995 & 0.7291 & 0.7197 & 0.6474\\
 \hline
\end{tabular}}
\end{table}
\begin{figure}[htbp]
	\centering
	\begin{subfigure}[t]{.1942\textwidth}
    \includegraphics[width=3.245cm]{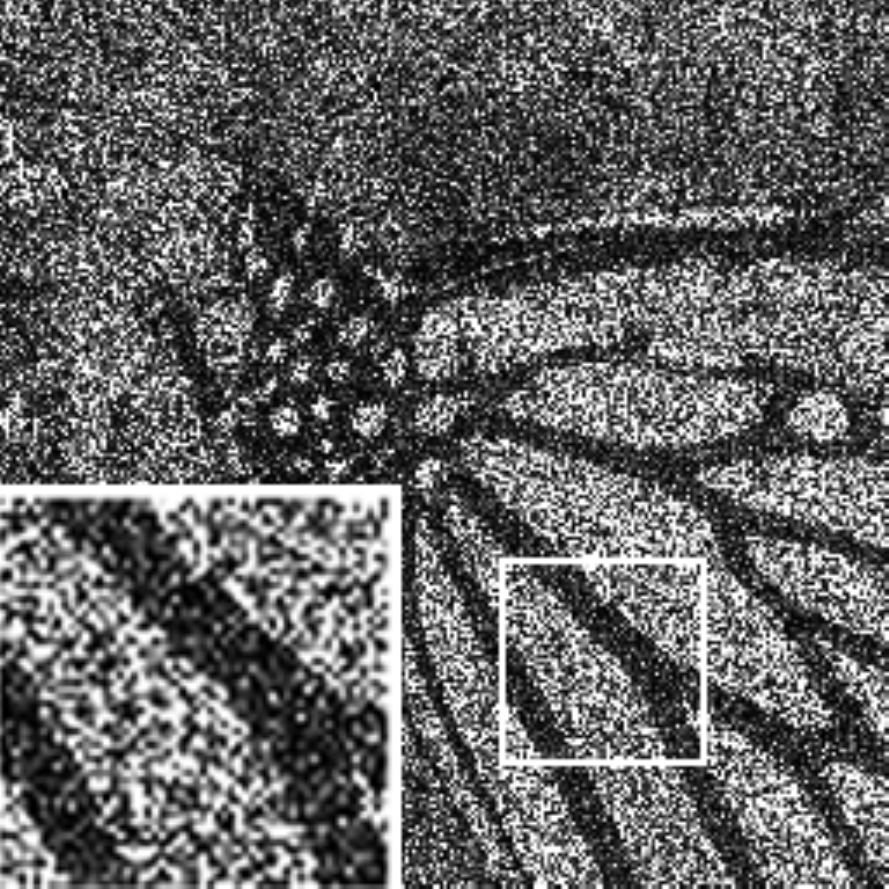}
    \caption{Noisy image (L=1)}
    \end{subfigure}	
    \begin{subfigure}[t]{.1942\textwidth}
    \includegraphics[width=3.245cm]{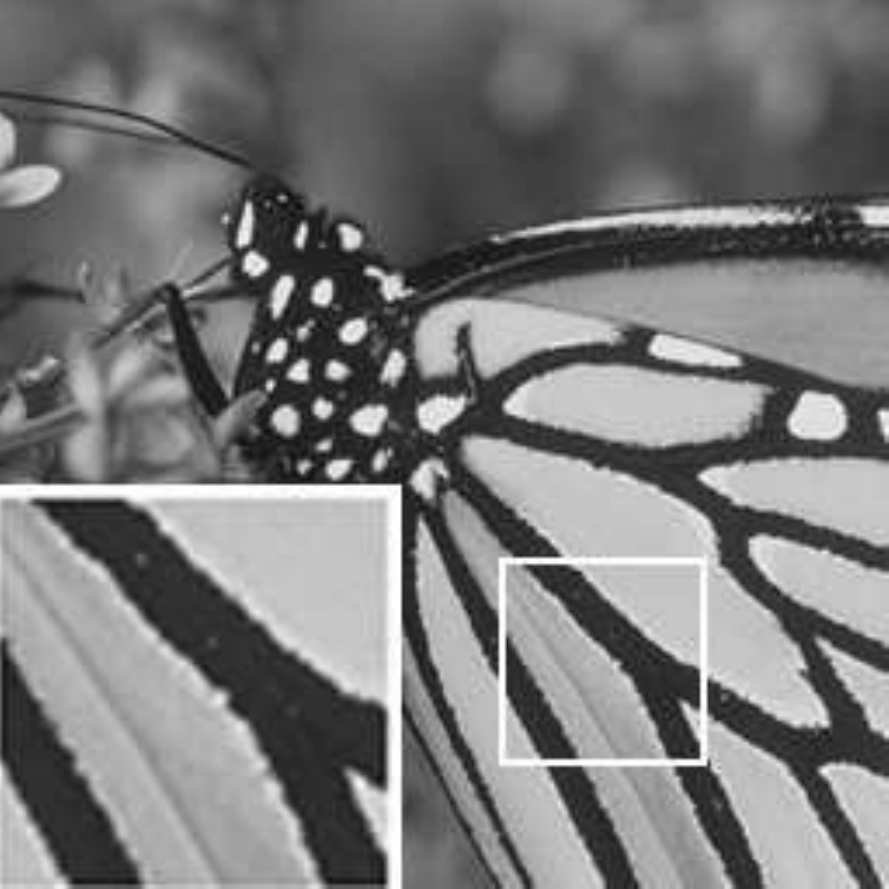}
    \caption{Ground truth}
    \end{subfigure}
    \begin{subfigure}[t]{.1942\textwidth}
    \includegraphics[width=3.245cm]{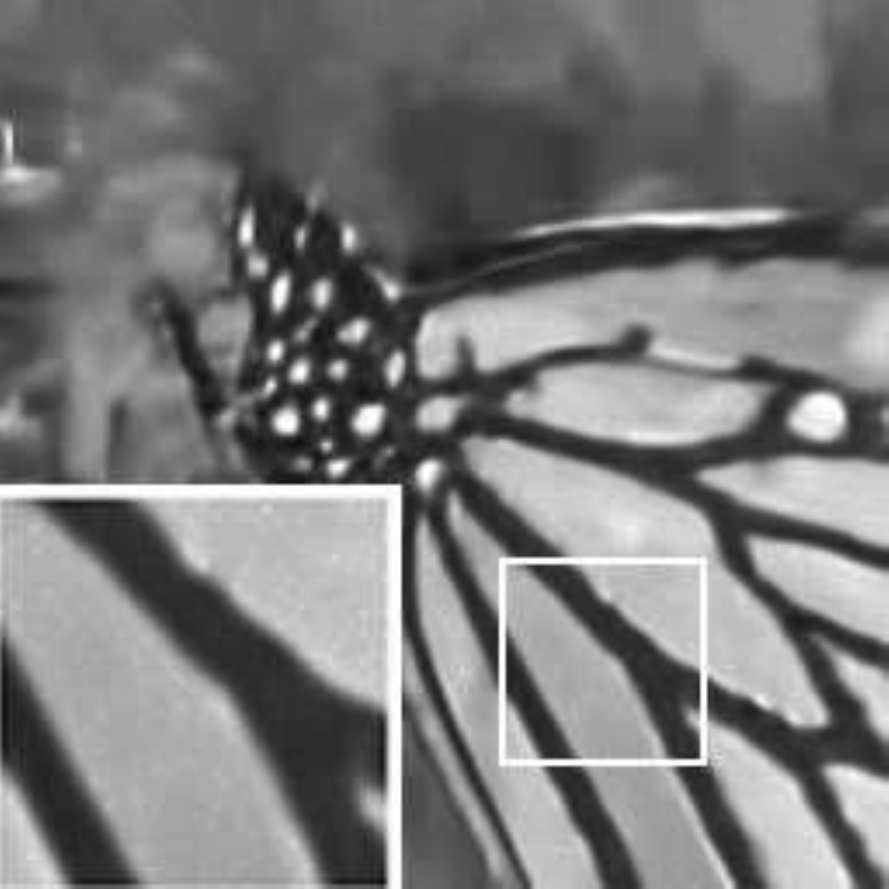}
    \caption{Alg 1}
    \end{subfigure}
    \begin{subfigure}[t]{.1942\textwidth}
    \includegraphics[width=3.245cm]{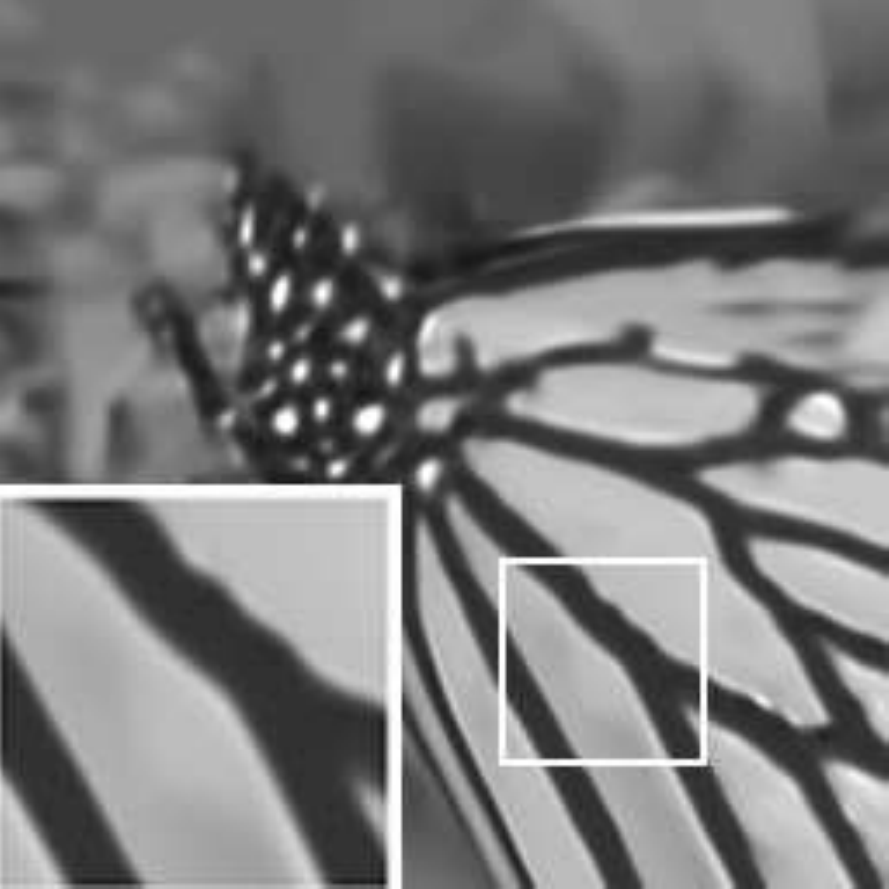}
    \caption{Alg 2}
    \end{subfigure}
    \begin{subfigure}[t]{.1942\textwidth}
    \includegraphics[width=3.245cm]{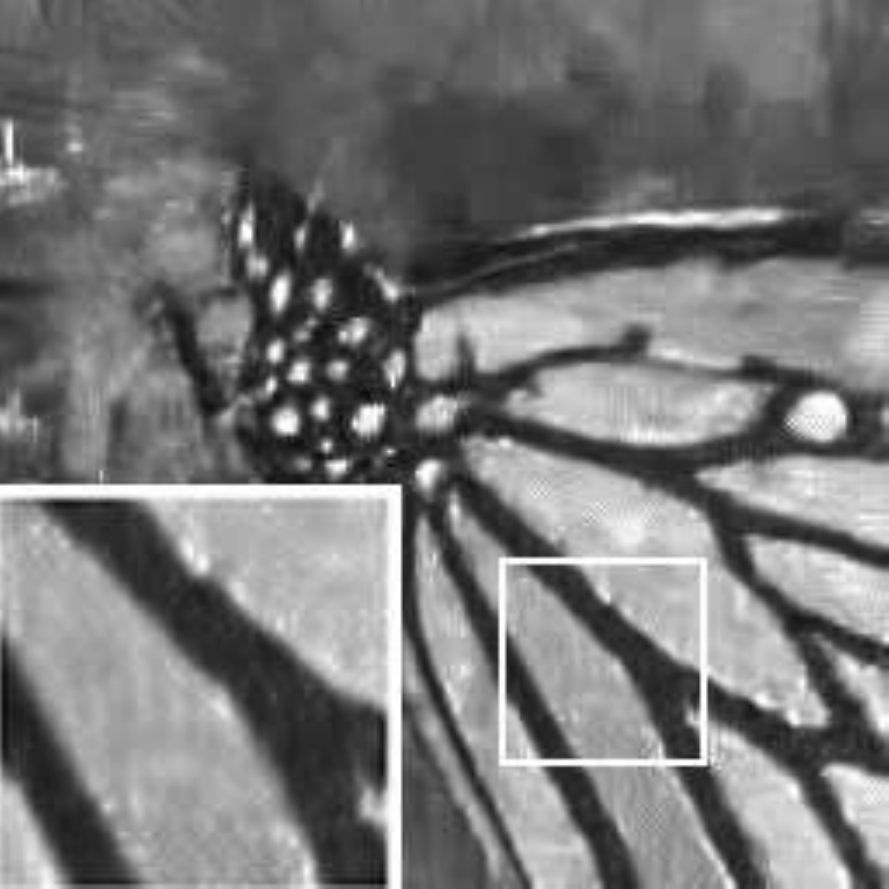}
	\caption{SAR-BM3D}
    \end{subfigure}\\
    \begin{subfigure}[t]{.1942\textwidth}
    \includegraphics[width=3.051cm]{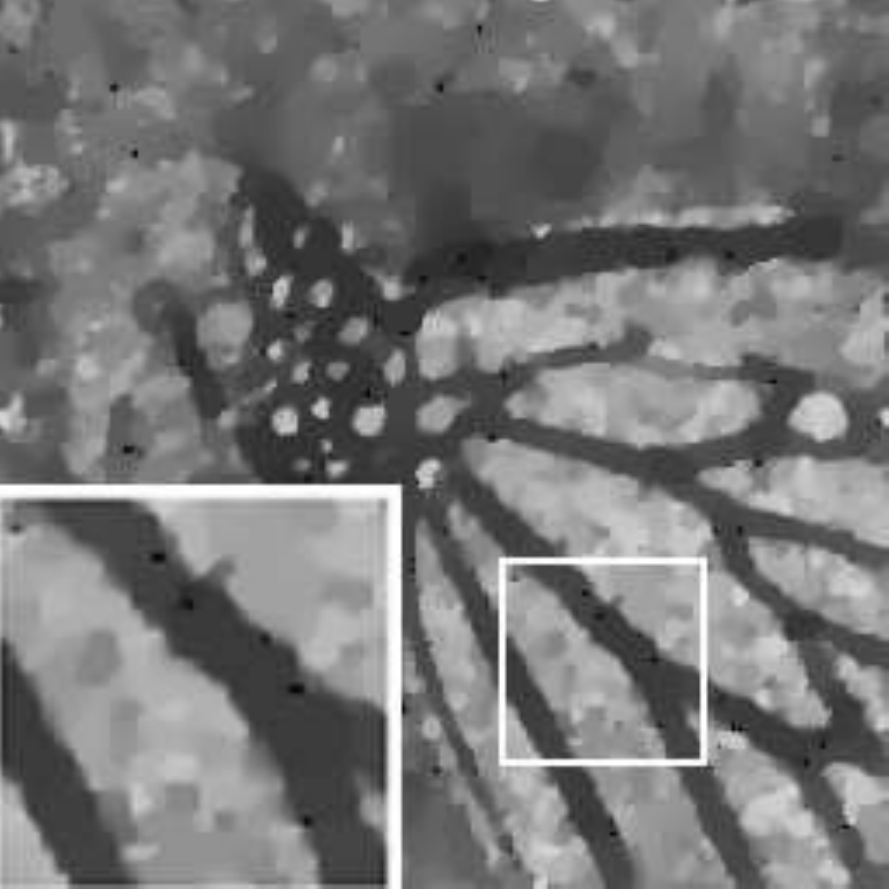}
	\caption{DZ}
    \end{subfigure}
    \begin{subfigure}[t]{.1942\textwidth}
    \includegraphics[width=3.245cm]{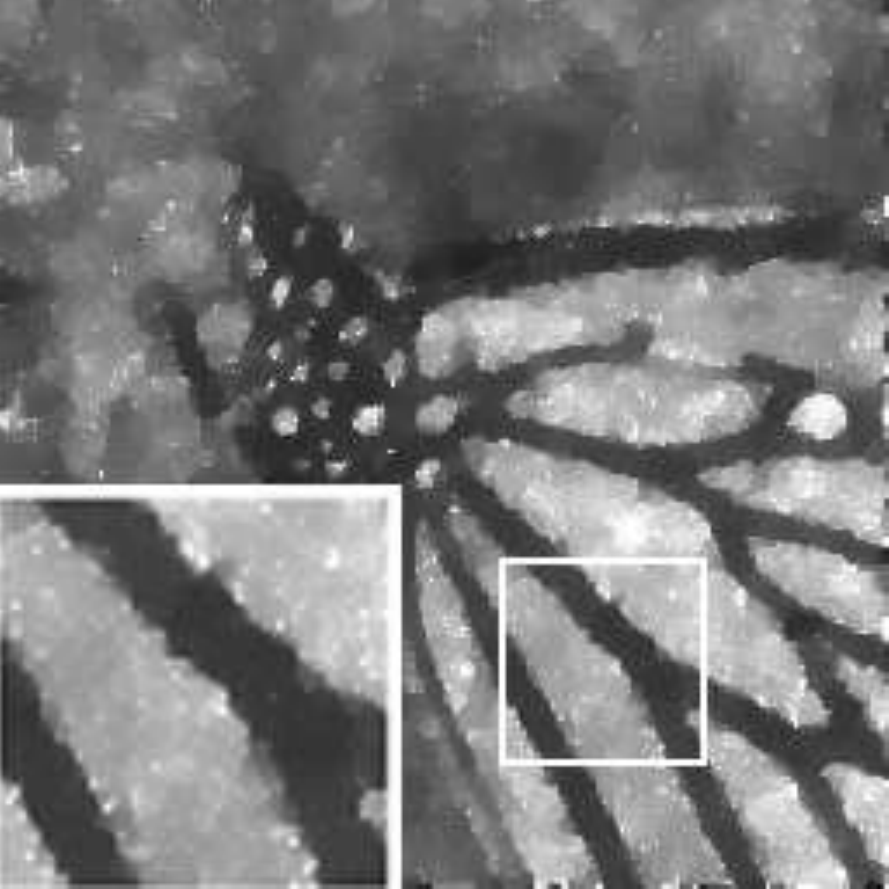}
	\caption{HNW}
    \end{subfigure}
    \begin{subfigure}[t]{.1942\textwidth}
    \includegraphics[width=3.245cm]{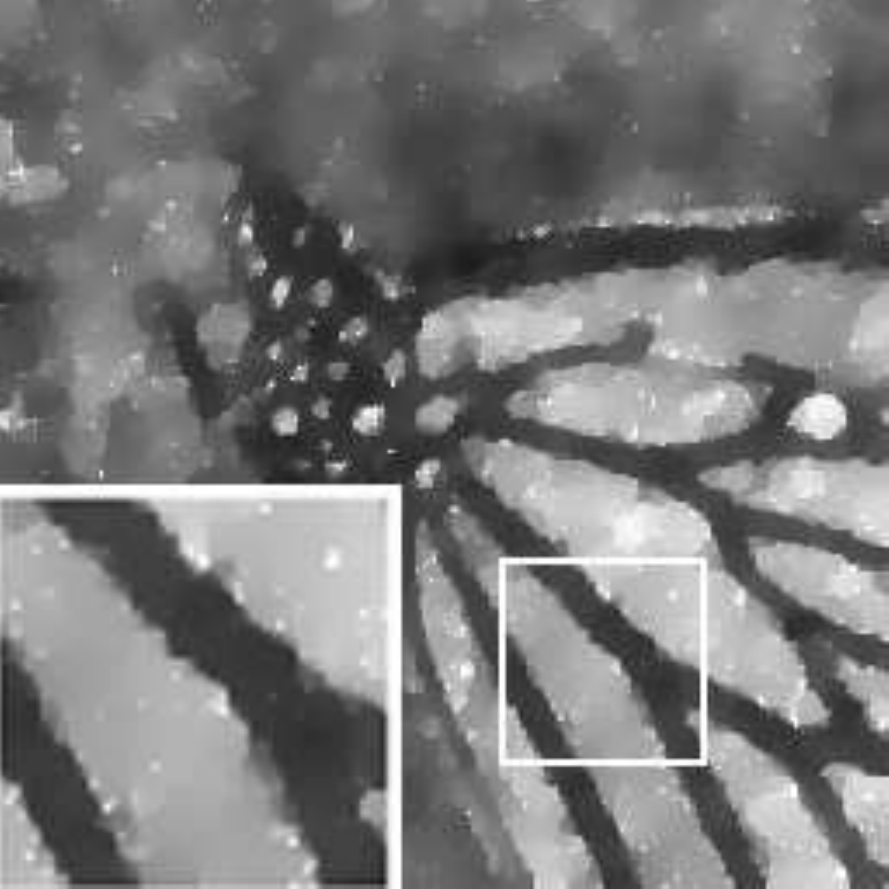}
    \caption{I-DIV}
    \end{subfigure}
    \begin{subfigure}[t]{.1942\textwidth}
    \includegraphics[width=3.245cm]{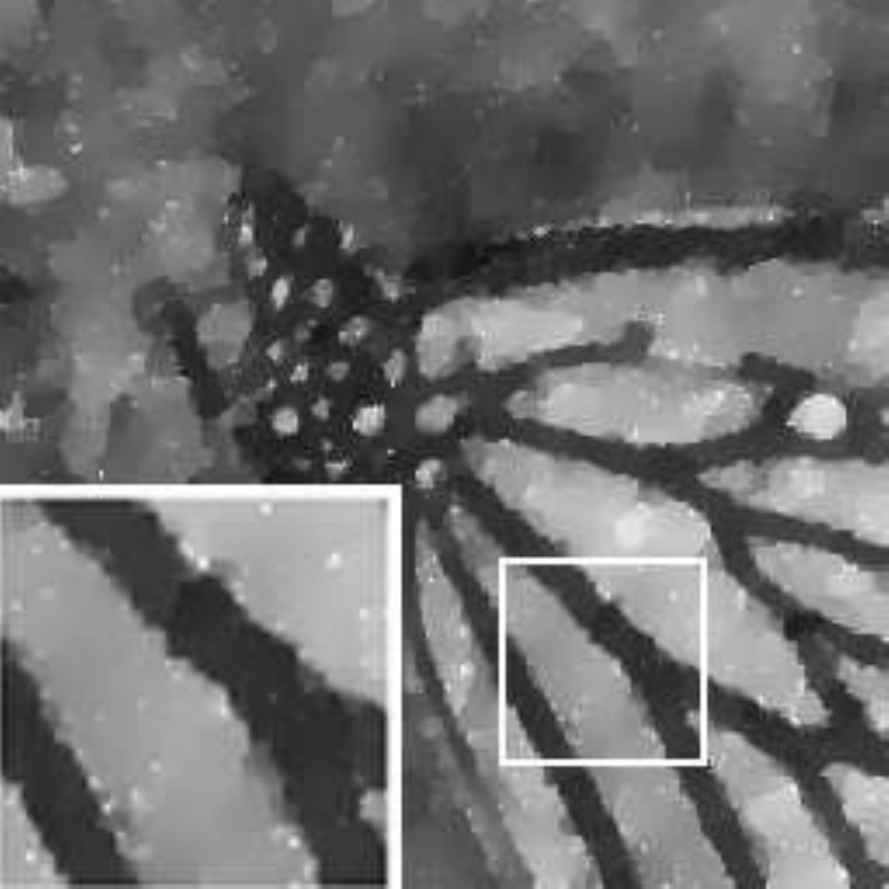}
    \caption{TwL-4V}
    \end{subfigure}
    \begin{subfigure}[t]{.1942\textwidth}
    \includegraphics[width=3.245cm]{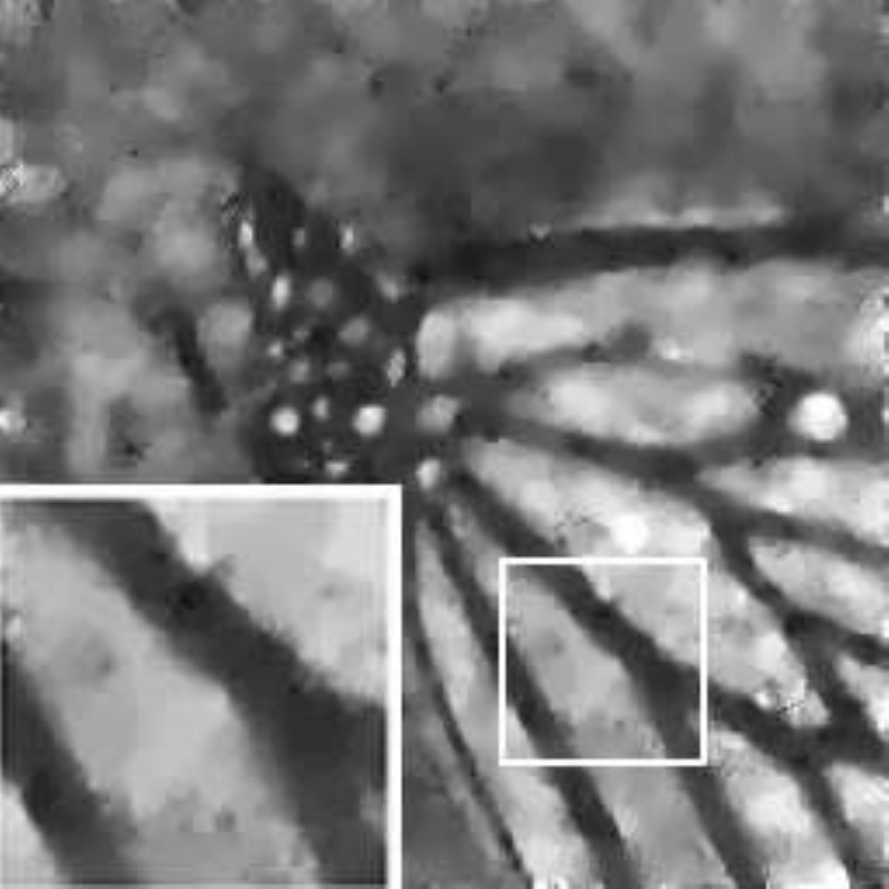}
	\caption{Dictionary}
    \end{subfigure}   	
	\caption{Comparison of denoised images restored from ``Monarch" at noise level $L=1$ by different methods. The (PSNR, SSIM) values for each denoised image: (c) Alg 1 (21.94dB, 0.6926); (d) Alg 2 (21.55dB, 0.6966); (e) SAR-BM3D (21.36dB, 0.6404); (f) DZ (19.38dB, 0.5758); (g) HNW (19.73dB, 0.5523); (h) I-DIV (19.91dB, 0.5883); (i) TwL-4V (19.26dB, 0.5848); (j) Dictionary (19.50dB, 0.5726).}\label{fig:Monarch}
\end{figure}
\begin{figure}[htbp]
	\centering
    \begin{subfigure}[t]{.1942\textwidth}
    \includegraphics[width=3.245cm]{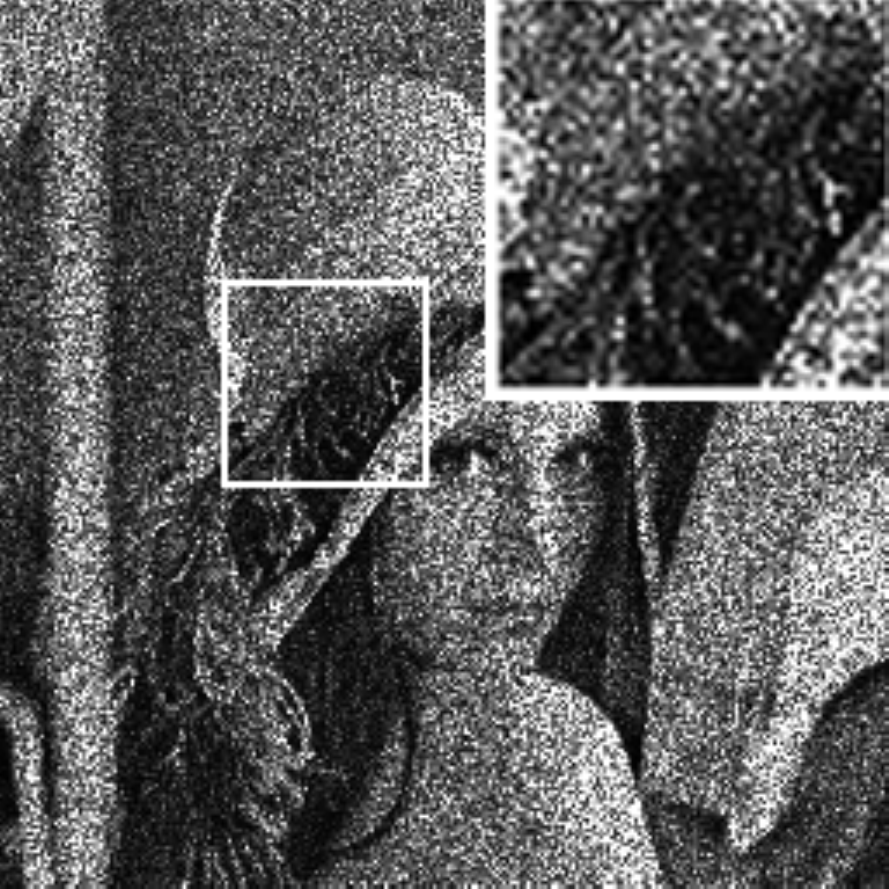}
    \caption{Noisy image (L=3)}
    \end{subfigure}
  	\begin{subfigure}[t]{.1942\textwidth}
    \includegraphics[width=3.245cm]{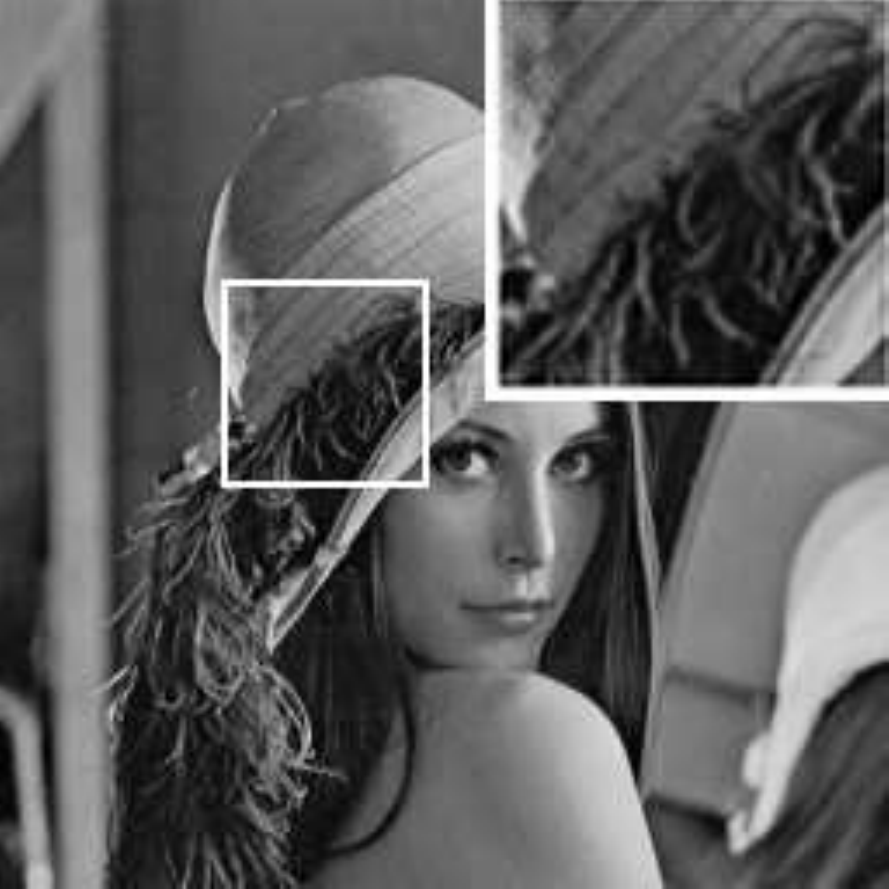}
    \caption{Ground truth}
    \end{subfigure}	
    \begin{subfigure}[t]{.1942\textwidth}
    \includegraphics[width=3.245cm]{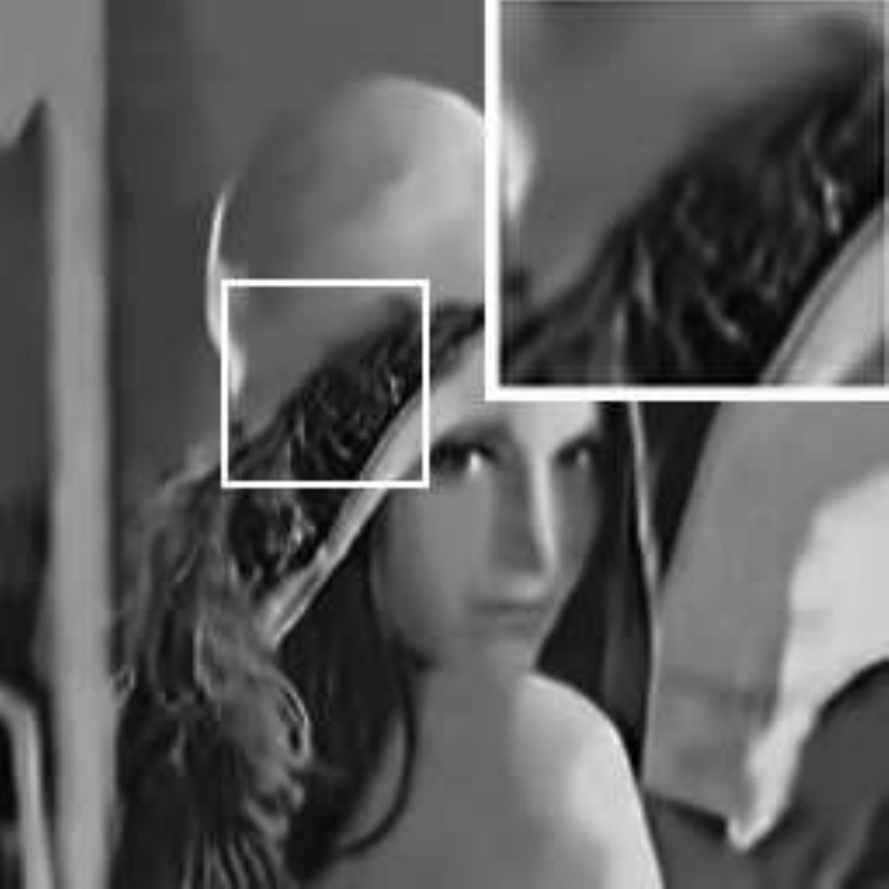}
    \caption{Alg 1}
    \end{subfigure}
    \begin{subfigure}[t]{.1942\textwidth}
    \includegraphics[width=3.245cm]{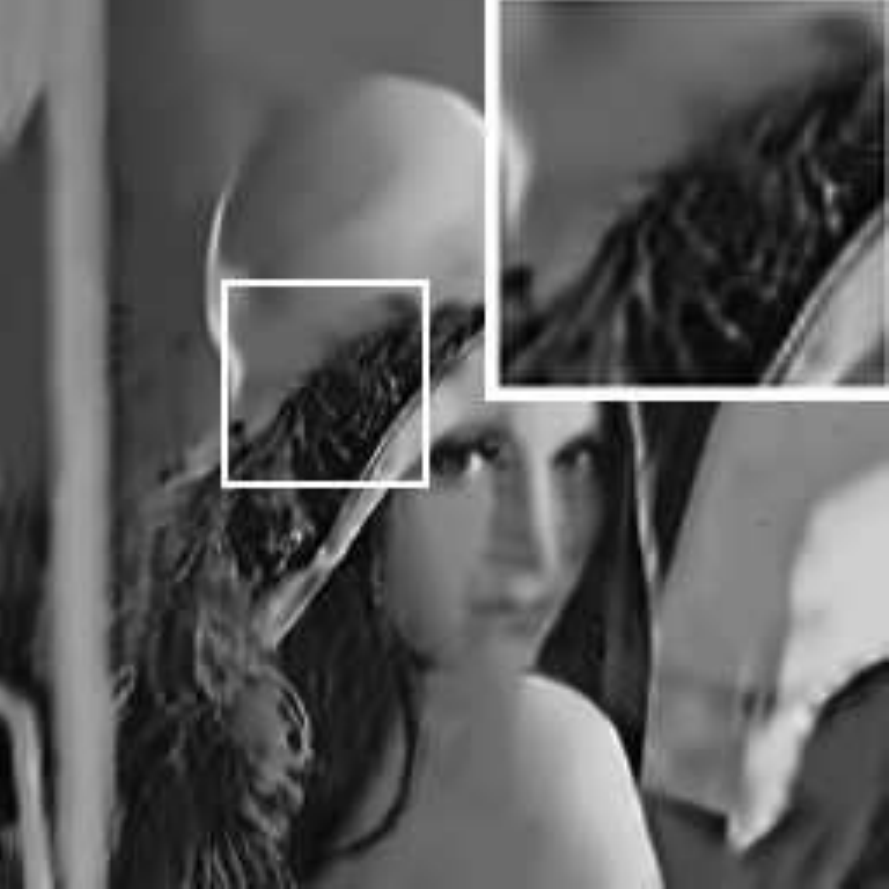}
    \caption{Alg 2}
    \end{subfigure}
    \begin{subfigure}[t]{.1942\textwidth}
    \includegraphics[width=3.245cm]{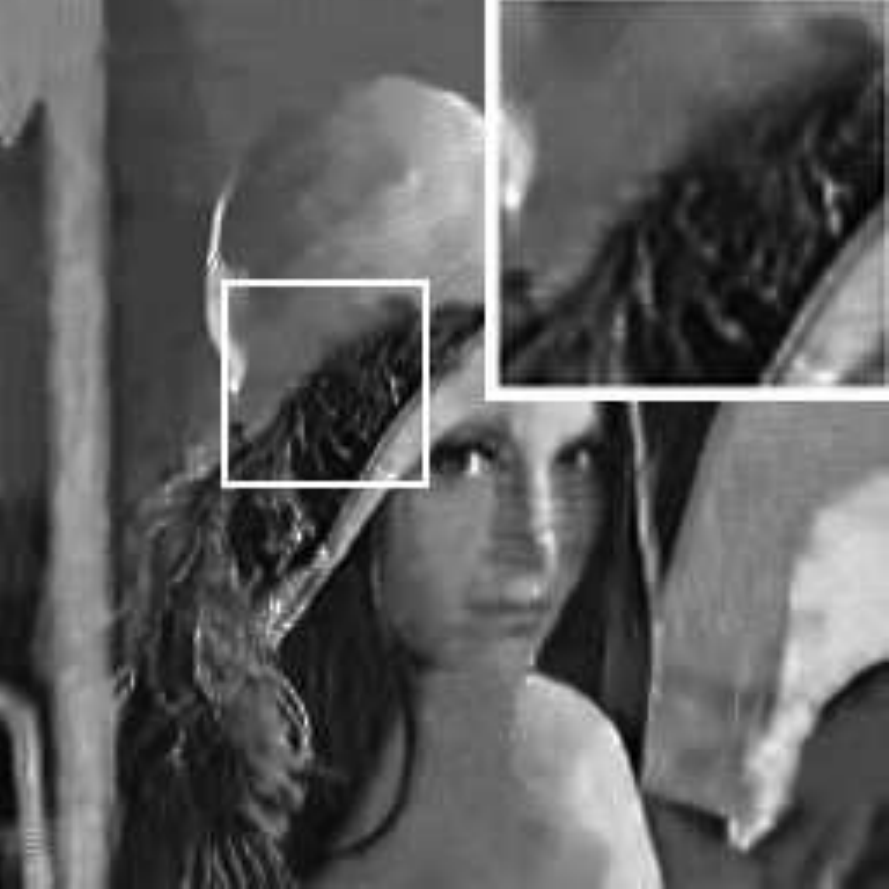}
	\caption{SAR-BM3D}
    \end{subfigure}   \\	
    \begin{subfigure}[t]{.1942\textwidth}
    \includegraphics[width=3.245cm]{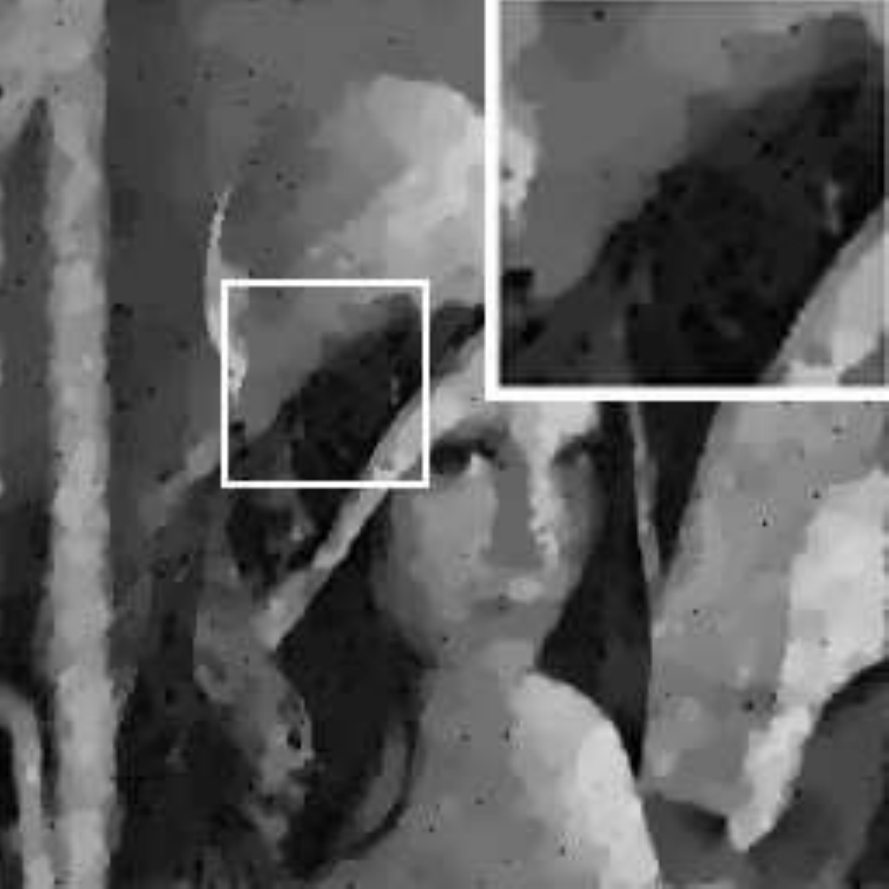}
	\caption{DZ}
    \end{subfigure}
    \begin{subfigure}[t]{.1942\textwidth}
    \includegraphics[width=3.245cm]{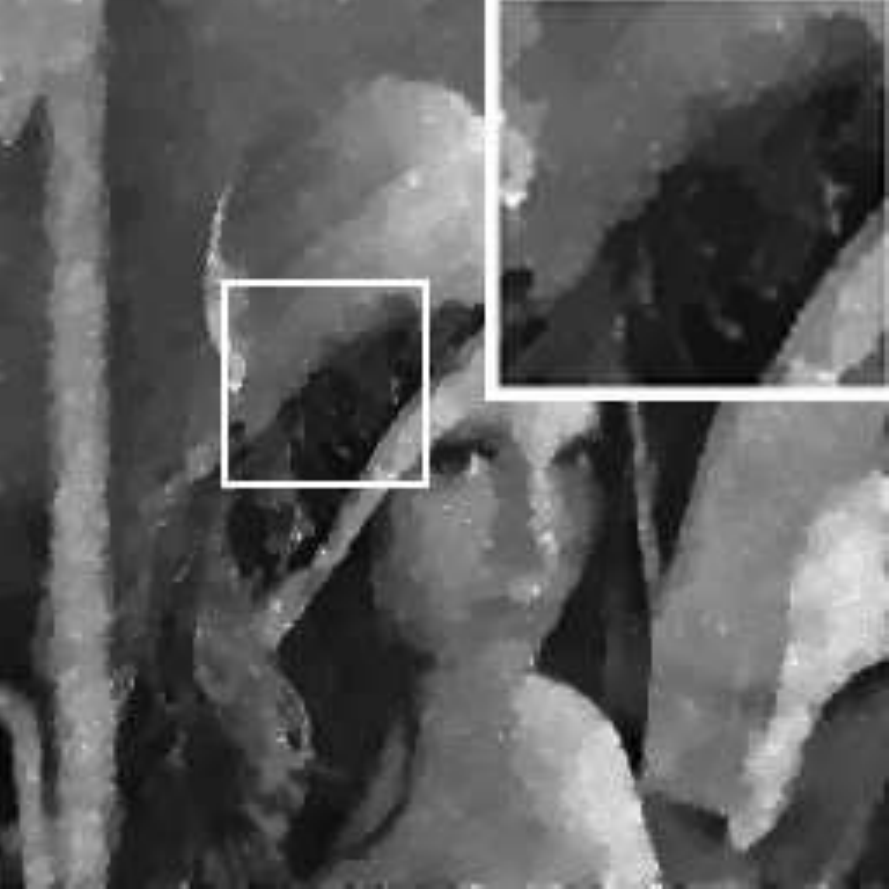}
	\caption{HNW}
    \end{subfigure}
    \begin{subfigure}[t]{.1942\textwidth}
    \includegraphics[width=3.245cm]{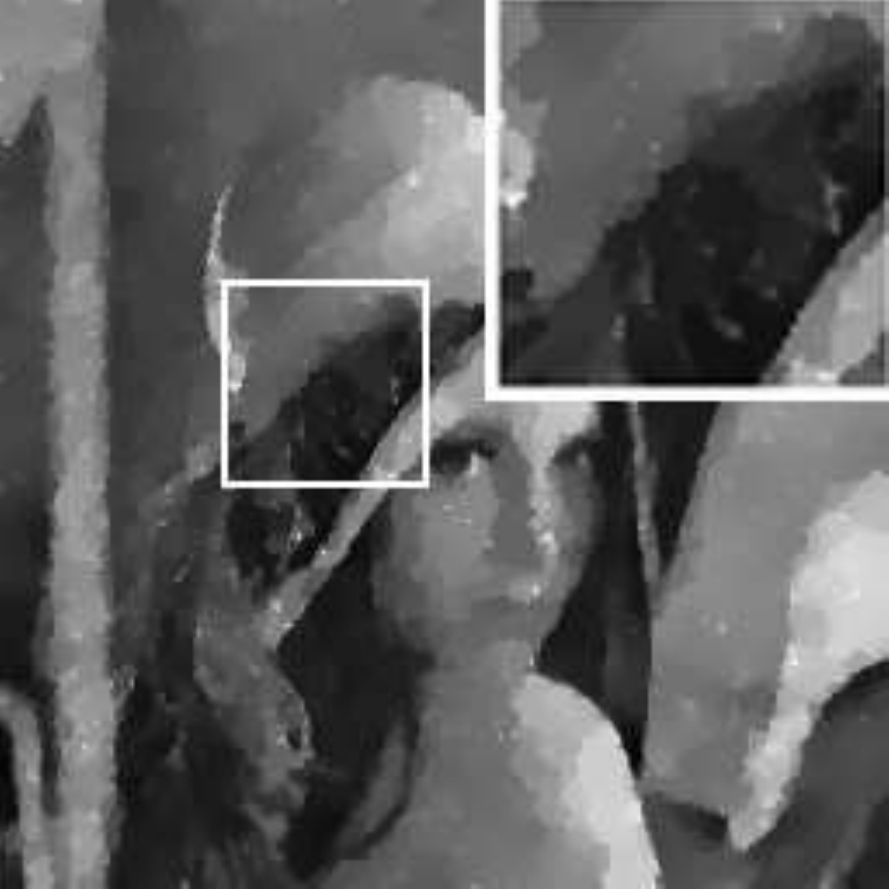}
    \caption{I-DIV}
    \end{subfigure}
    \begin{subfigure}[t]{.1942\textwidth}
    \includegraphics[width=3.245cm]{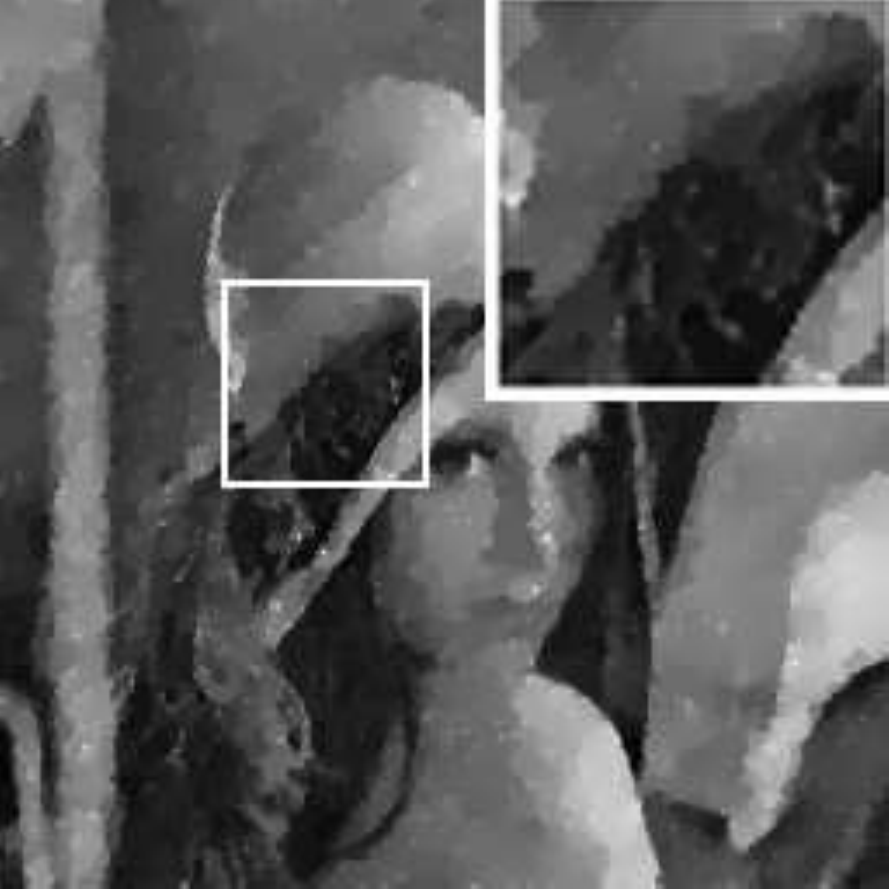}
    \caption{TwL-4V}
    \end{subfigure}
    \begin{subfigure}[t]{.1942\textwidth}
    \includegraphics[width=3.245cm]{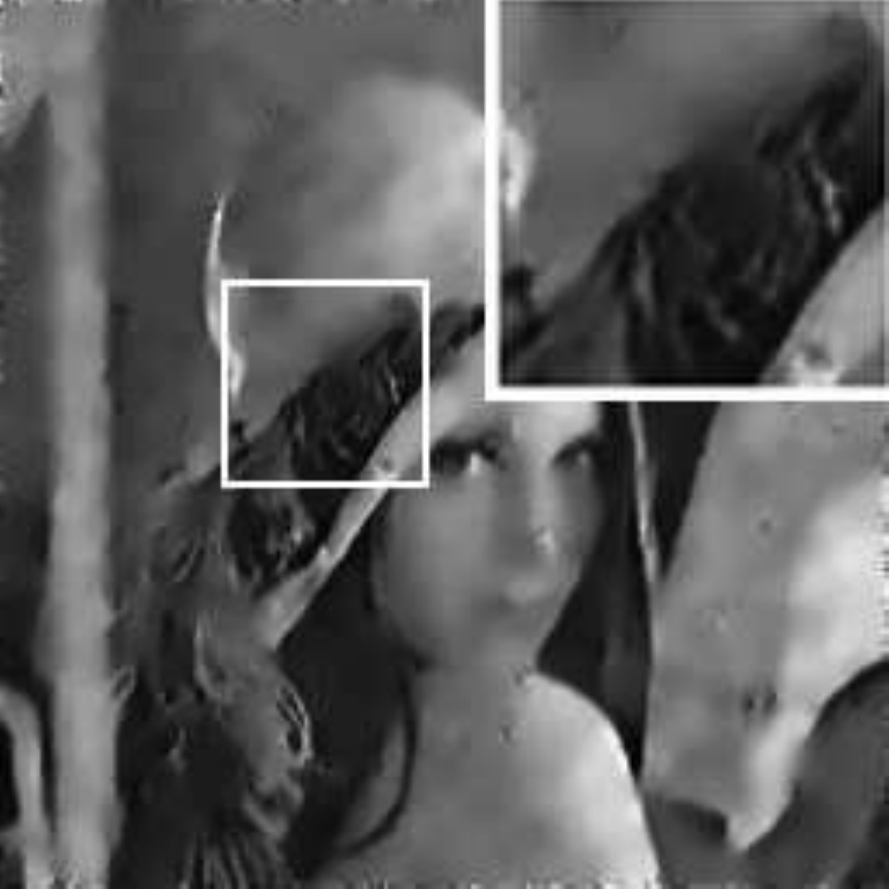}
	\caption{Dictionary}
    \end{subfigure}   	
	\caption{Comparison of denoised images restored from ``Lena" at noise level $L=3$ by different methods.  The (PSNR, SSIM) values for each denoised image: (c) Alg 1 (26.44dB, 0.7892); (d) Alg 2 (26.29dB, 0.7944); (e) SAR-BM3D (26.00dB, 0.7596); (f) DZ (24.06dB, 0.6907); (g) HNW (24.36dB, 0.6911); (h) I-DIV (24.48dB, 0.7073); (i) TwL-4V (24.29dB, 0.7128); (j) Dictionary (24.81dB, 0.7379).}\label{fig:Lena}
\end{figure}
\begin{figure}[htbp]
	\centering
    \begin{subfigure}[t]{.1942\textwidth}
    \includegraphics[width=3.245cm]{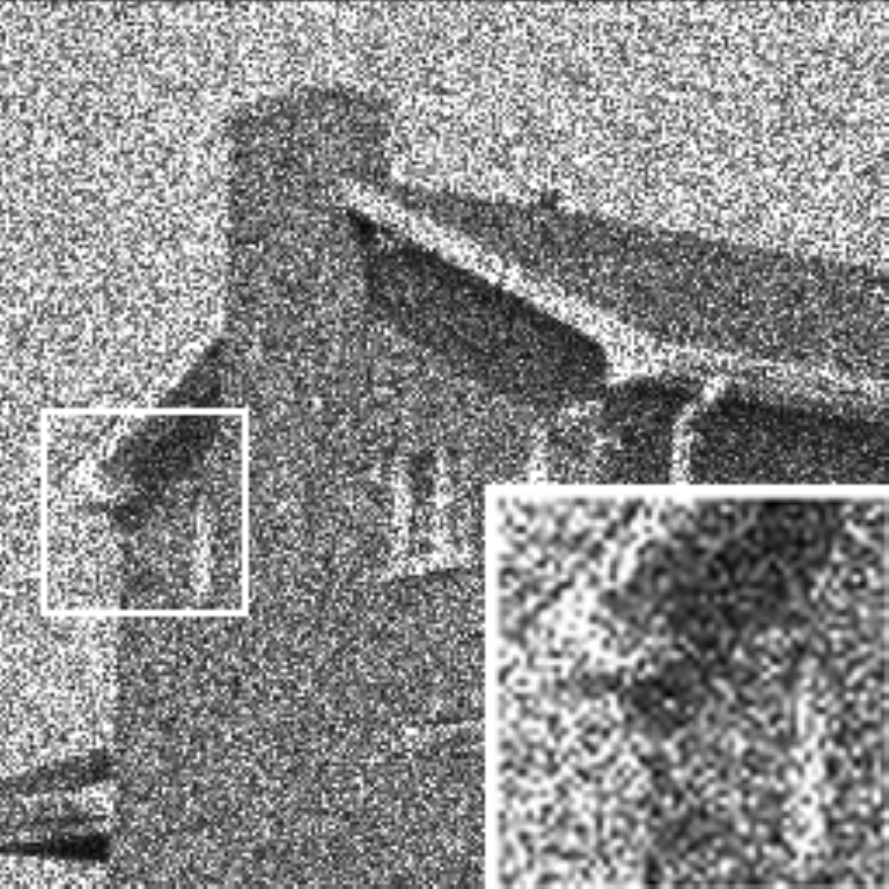}
    \caption{Noisy image (L=5)}
    \end{subfigure}
	\begin{subfigure}[t]{.1942\textwidth}
    \includegraphics[width=3.245cm]{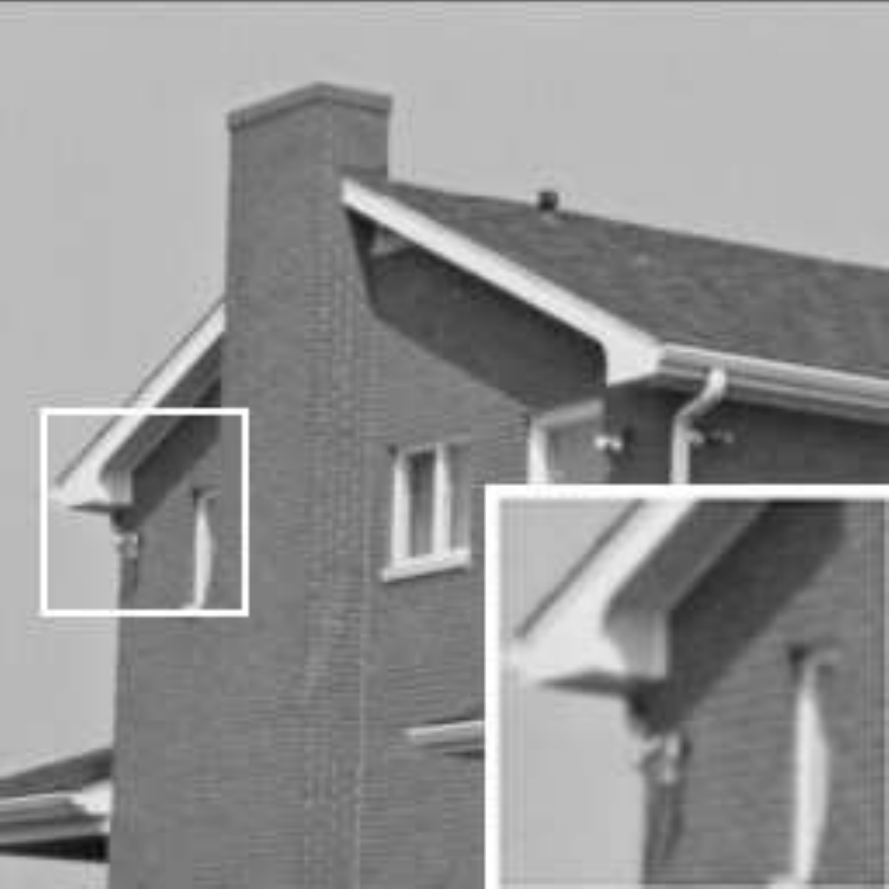}
    \caption{Ground truth}
    \end{subfigure}
    \begin{subfigure}[t]{.1942\textwidth}
    \includegraphics[width=3.245cm]{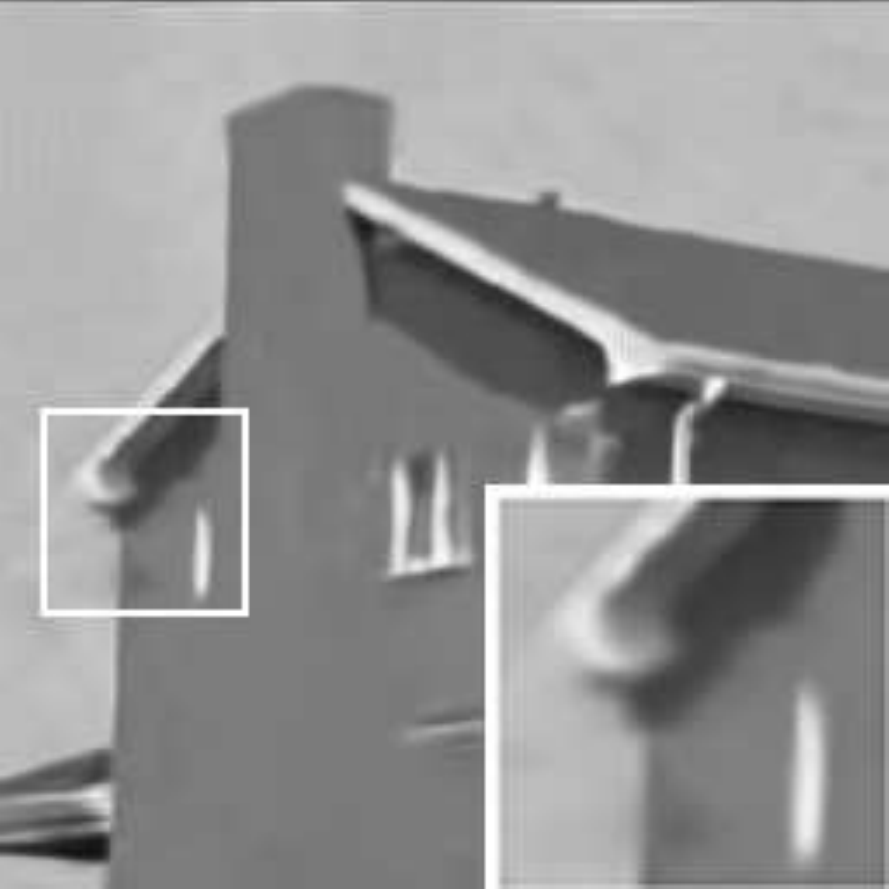}
    \caption{Alg 1}
    \end{subfigure}
    \begin{subfigure}[t]{.1942\textwidth}
    \includegraphics[width=3.245cm]{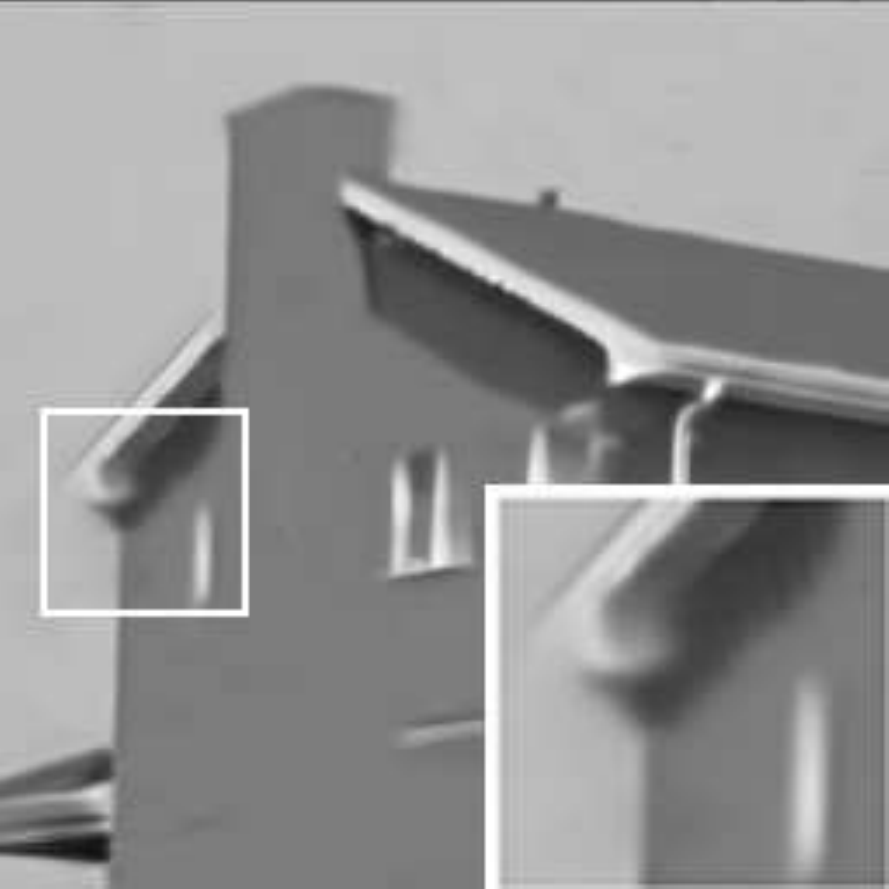}
    \caption{Alg 2 }
    \end{subfigure}
    \begin{subfigure}[t]{.1942\textwidth}
    \includegraphics[width=3.245cm]{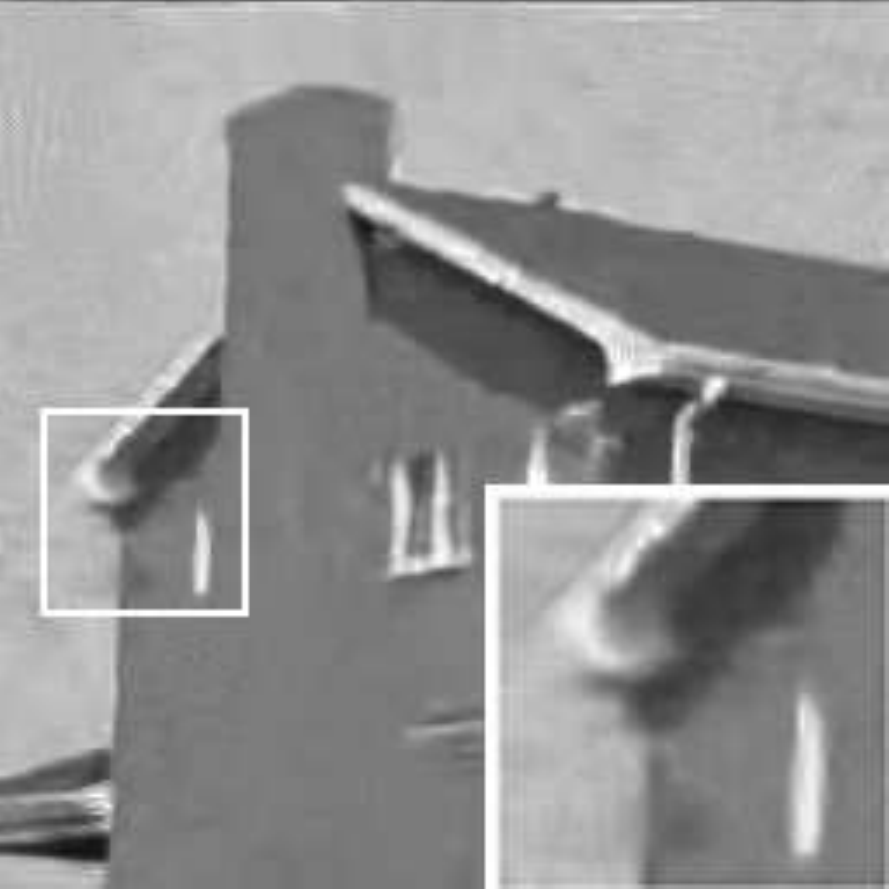}
	\caption{SAR-BM3D}
    \end{subfigure}\\
    \begin{subfigure}[t]{.1942\textwidth}
    \includegraphics[width=3.245cm]{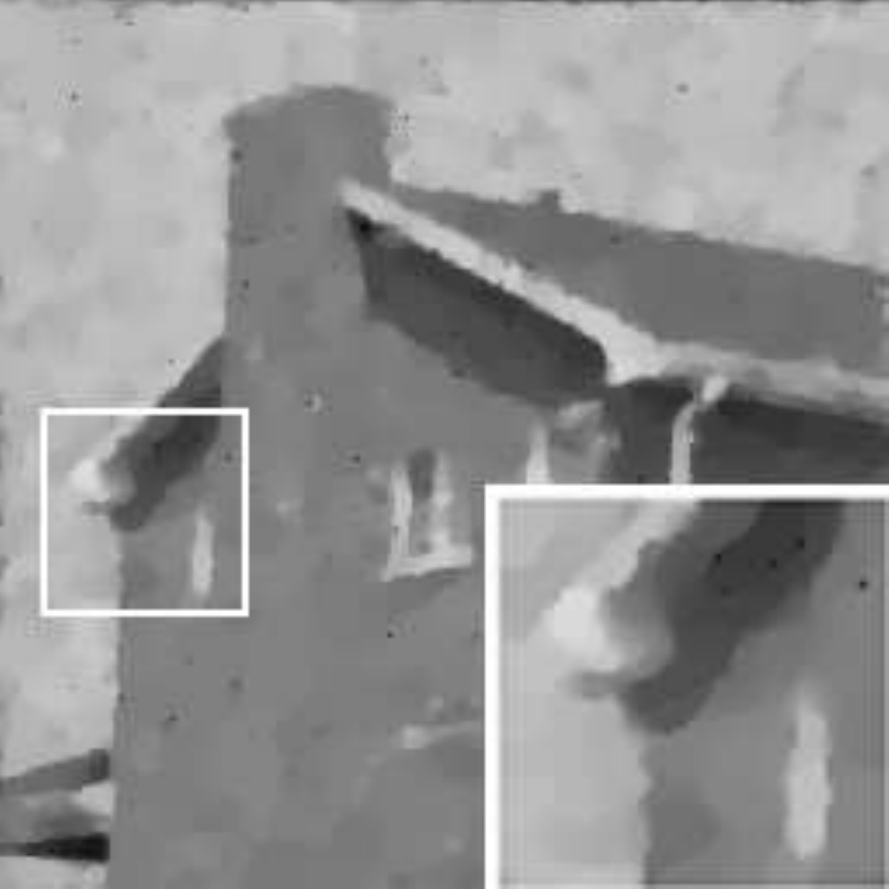}
	\caption{DZ}
	\end{subfigure}   	
    \begin{subfigure}[t]{.1942\textwidth}
    \includegraphics[width=3.245cm]{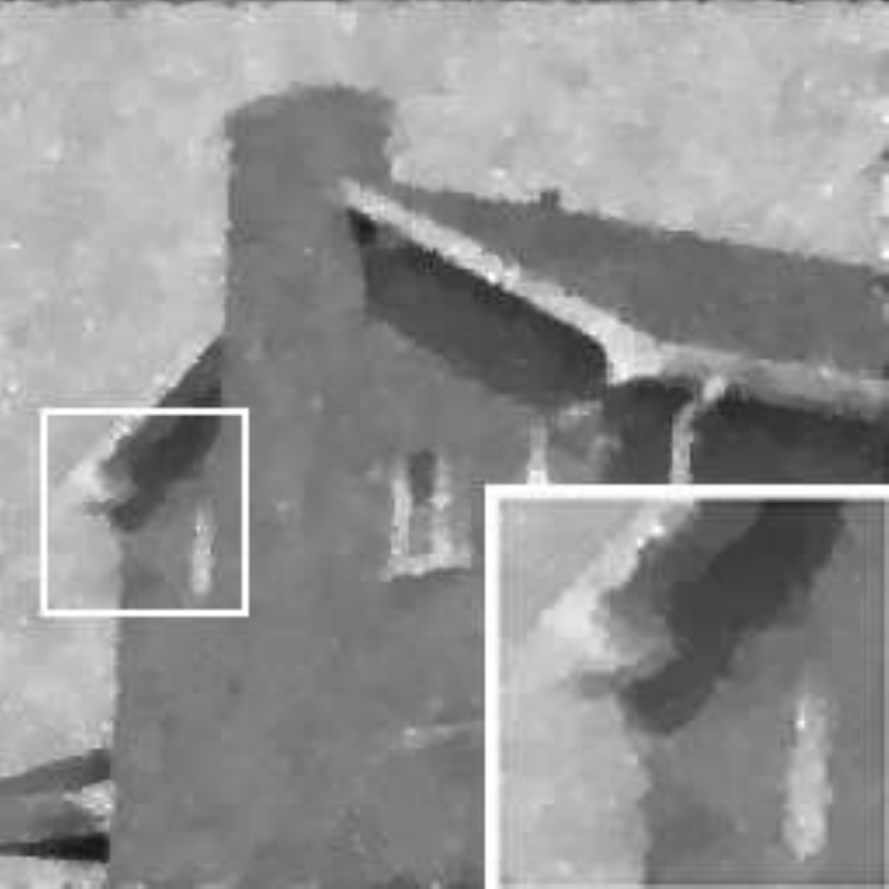}
	\caption{HNW}
    \end{subfigure}
    \begin{subfigure}[t]{.1942\textwidth}
    \includegraphics[width=3.245cm]{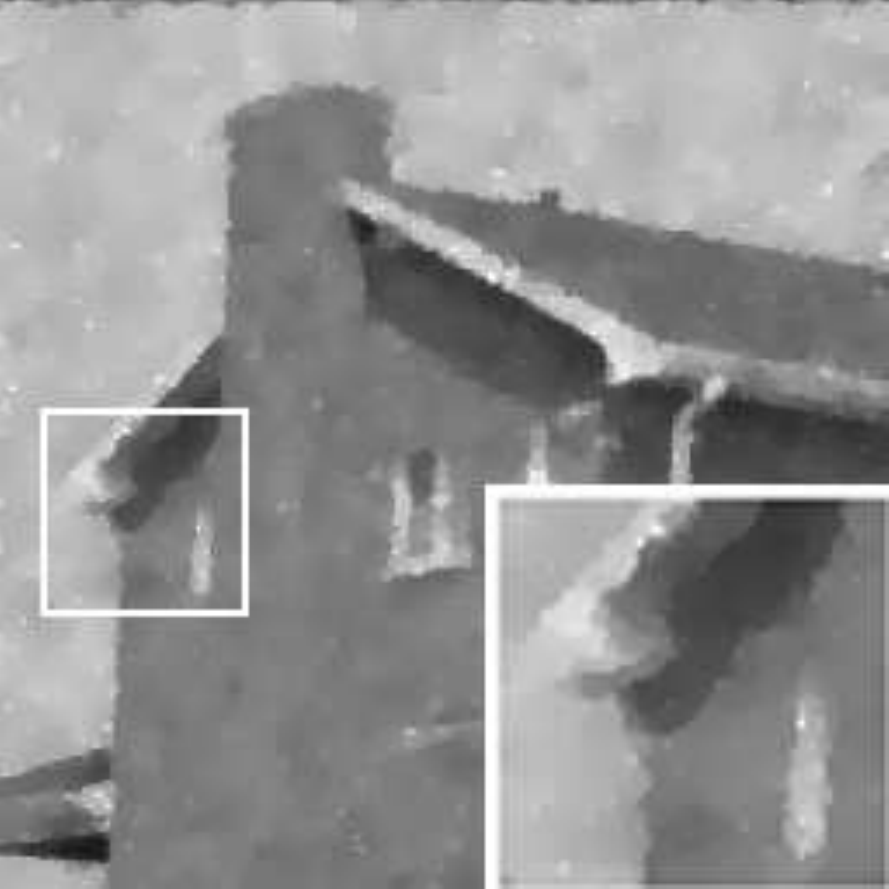}
    \caption{I-DIV}
    \end{subfigure}
    \begin{subfigure}[t]{.1942\textwidth}
    \includegraphics[width=3.245cm]{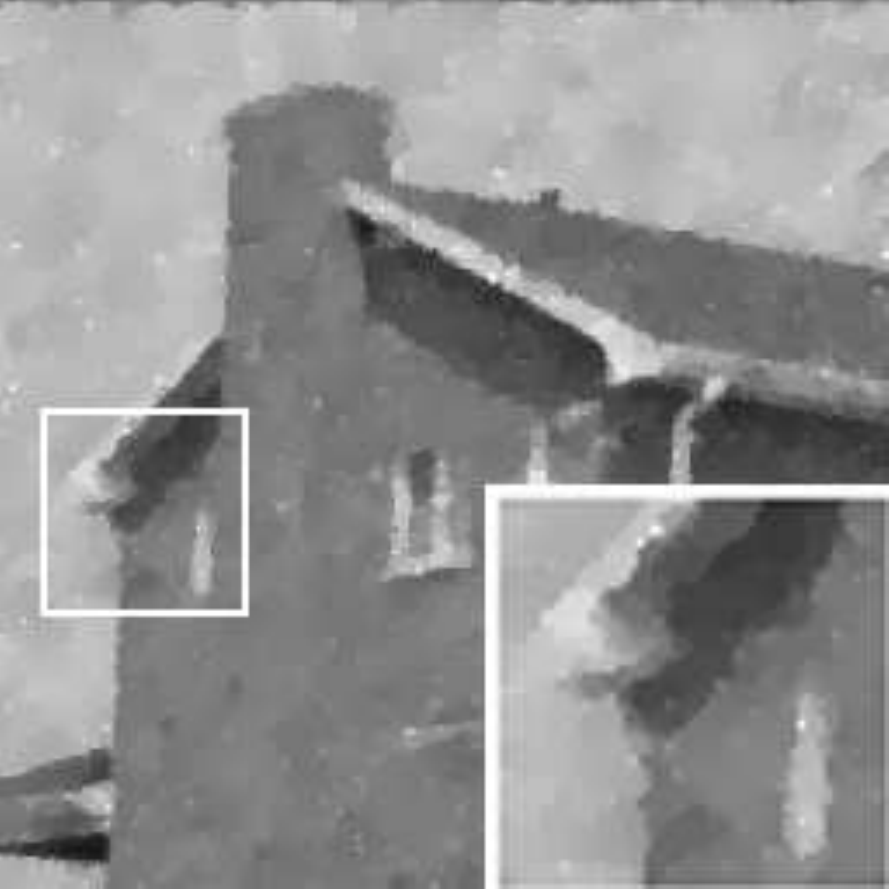}
    \caption{TwL-4V}
    \end{subfigure}
    \begin{subfigure}[t]{.1942\textwidth}
    \includegraphics[width=3.245cm]{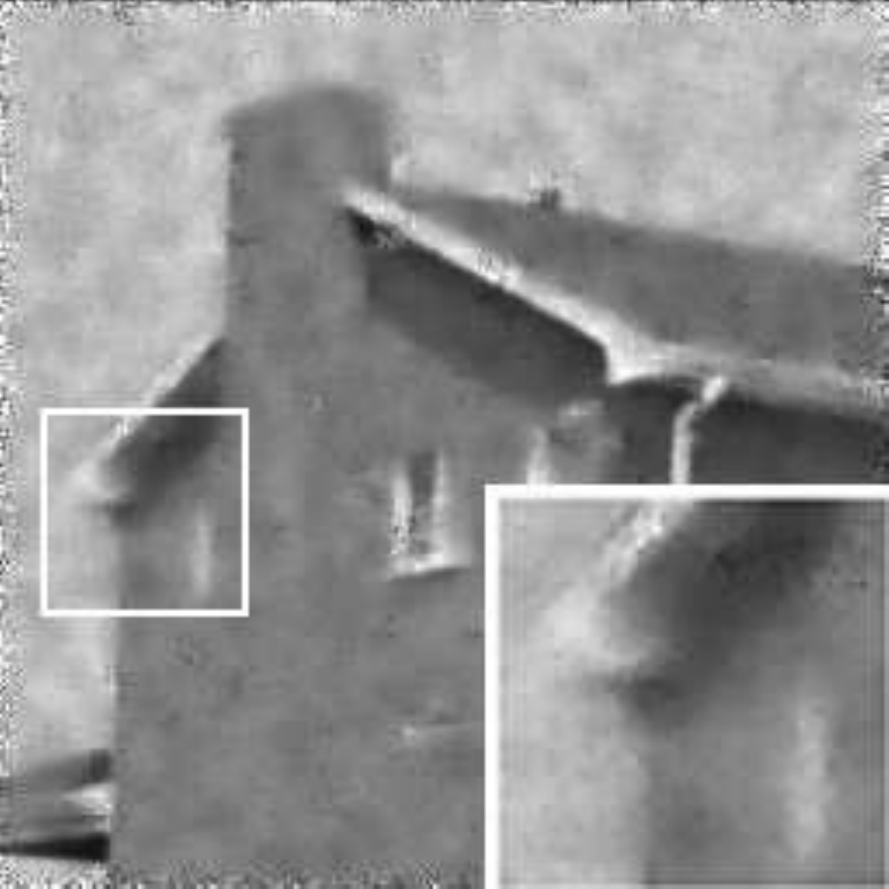}
	\caption{Dictionary}
    \end{subfigure}   	
	\caption{Comparison of denoised images restored from ``House" at noise level $L=5$ by different methods.  The (PSNR, SSIM) values for each denoised image: (c) Alg 1 (29.04dB, 0.8115); (d) Alg 2 (29.12dB, 0.8163); (e) SAR-BM3D (28.36dB, 0.7641); (f) DZ (25.70dB, 0.7339); (g) HNW (25.73dB, 0.6995); (h) I-DIV (25.84dB, 0.7291); (i) TwL-4V (25.79dB, 0.7197); (j) Ditionary (24.56dB, 0.6474).}\label{fig:House}
\end{figure}

Table~\ref{Table:Standard} reports the PSNR and SSIM values of the  denoised images tested on three standard test images. The best results for each case are marked in bold and the second-best results are underlined. Both Algorithm~\ref{Alg:Theoretical} and Algorithm~\ref{Alg:Practical} outperform all the other methods in terms of  PSNR and SSIM values. Compared with the benchmark SAR-BM3D method, Algorithm~\ref{Alg:Theoretical} achieves 0.54-0.59dB,  0.42-0.66dB  and 0.46-0.68dB improvements in PSNR when $L=1$, $L=3$ and $L=5$, respectively. Algorithm~\ref{Alg:Practical} with updated patch extraction also surpasses the SAR-BM3D method and it even surpasses Algorithm~\ref{Alg:Theoretical} in some of the cases, especially in terms of SSIM values.

Figure~\ref{fig:Monarch}-\ref{fig:House} present the denoised images tested on ``Monarch" at noise level $L=1$, ``Lena" at $L=3$ and ``House" at $L=5$. In terms of the visual quality,  Algorithm~\ref{Alg:Theoretical} and Algorithm~\ref{Alg:Practical}  perform better than other methods, because they reconstruct more details and more smooth textures, but less noise and fewer artifacts. For example, compared to the DZ method, the HNW method,  the I-DIV method, the TwL-mV method and the learned dictionary method, the proposed methods preserve more details of the hair of ``Lena" and generate more smooth textures on the wings of ``Monarch" and the sky of ``House". Compared to the benchmark SAR-BM3D method, the proposed methods generate fewer artifacts, resulting in better images in terms of  PSNR and SSIM values.

\subsection{Numerical results tested on remote sensing images}

In this experiment, we use remote sensing images ``Remote 1" and ``Remote 2" both of size $512\times 512$, and ``Remote 3" of size $540\times 632$ as shown in Figure~\ref{TestImg:Remote}. To generate the observed images, we degrade the original test images by multiplicative Gamma noise at $L=1$, $L=3$ and $L=5$. The image quality is evaluated using PSNR and SSIM values.

\begin{figure}[htbp]
	\centering
	\begin{subfigure}[t]{.26\textwidth}\centering
        \includegraphics[height=3.9cm]{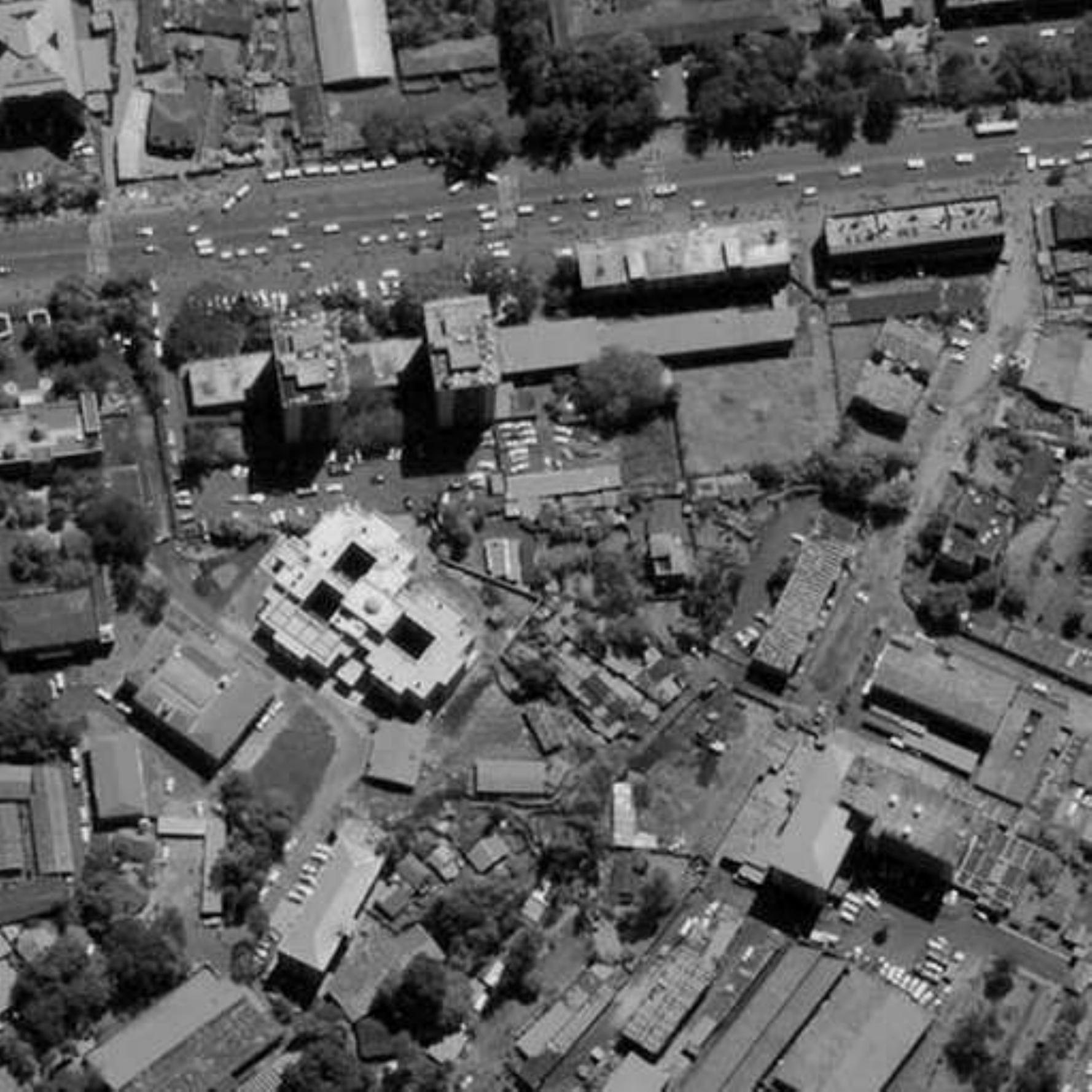}
        \caption{Remote 1}
    \end{subfigure}
    	\begin{subfigure}[t]{.26\textwidth}\centering
        \includegraphics[height=3.9cm]{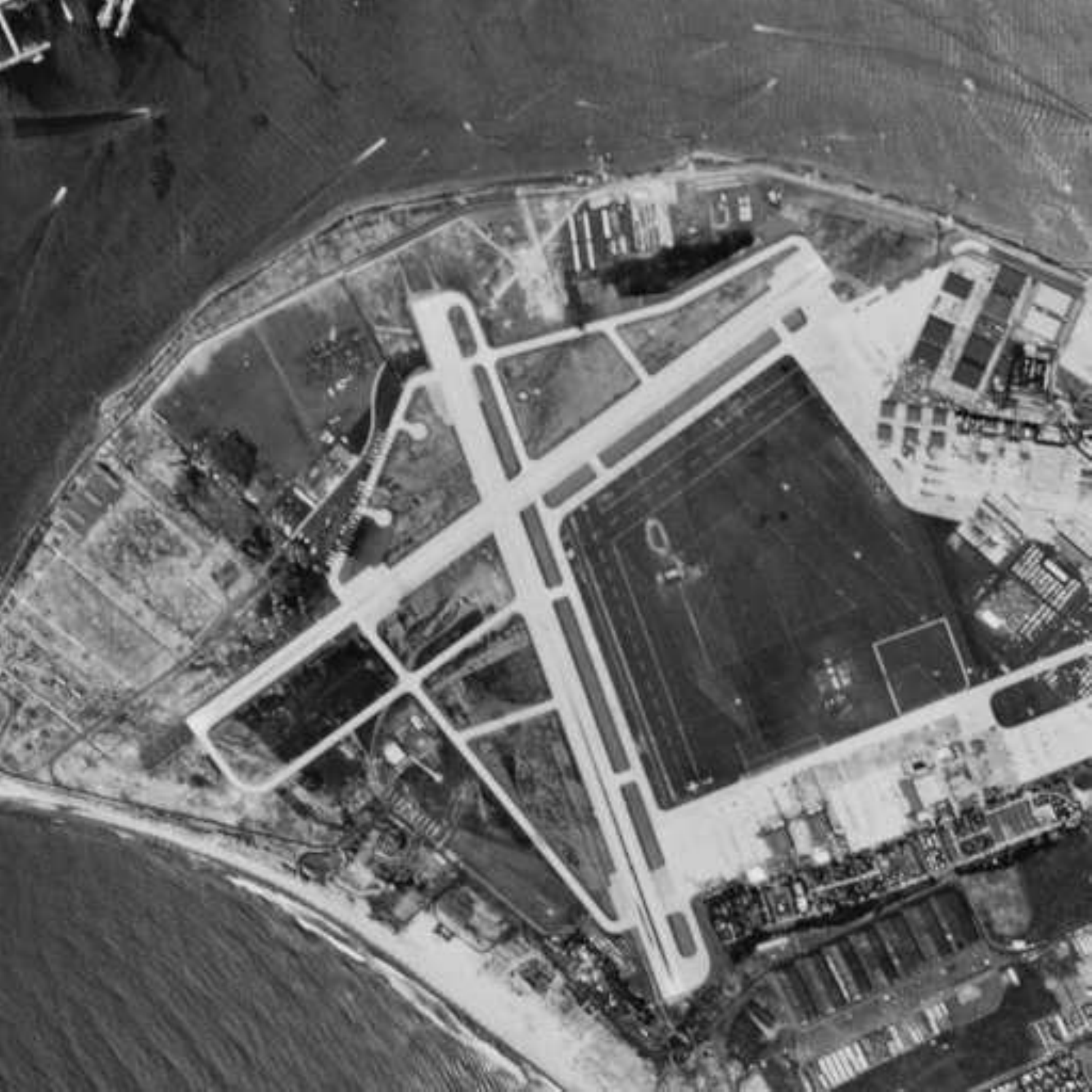}
        \caption{Remote 2}
    \end{subfigure}
    	\begin{subfigure}[t]{.3\textwidth}\centering
        	\includegraphics[height=3.9cm]{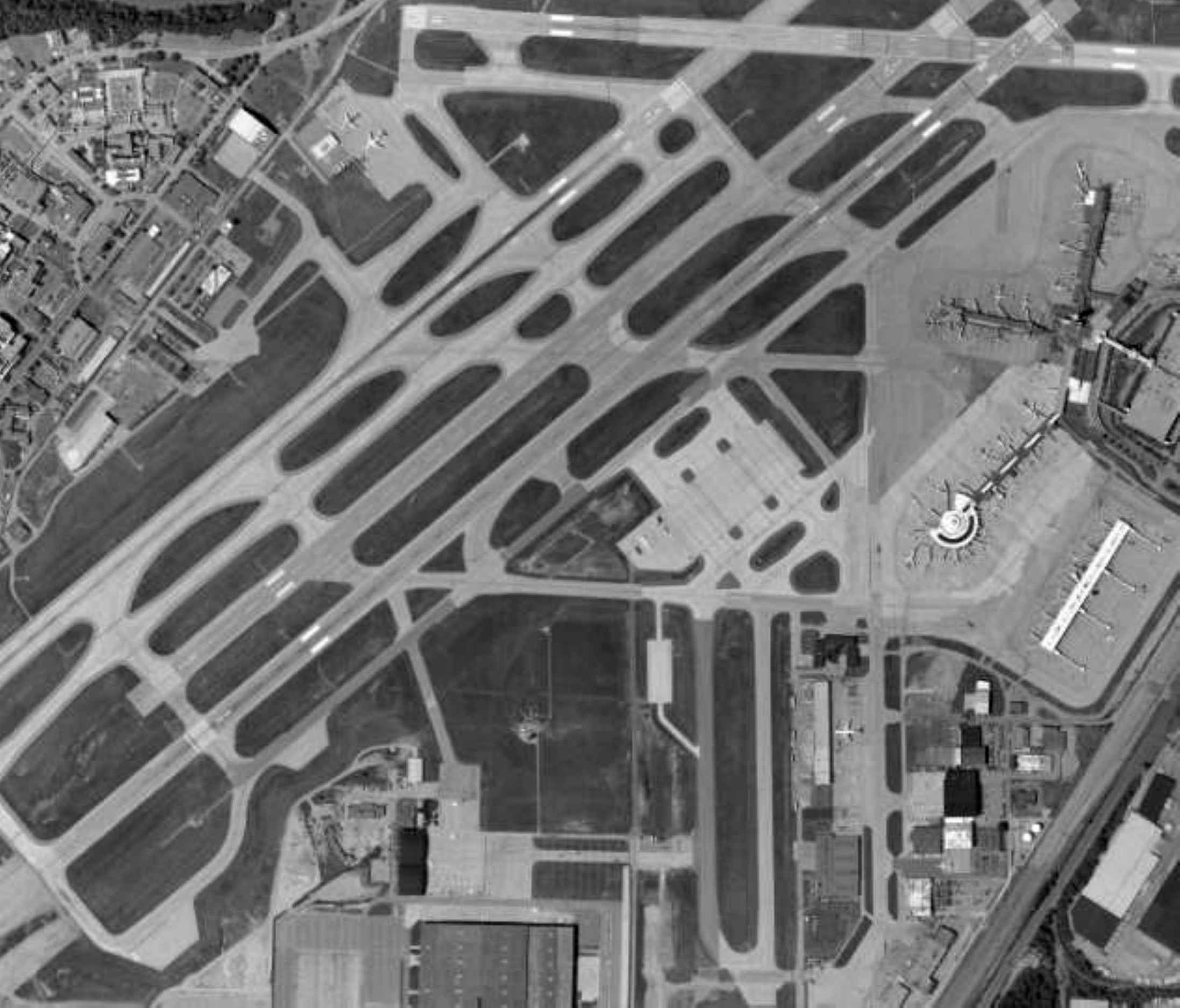}
        \caption{Remote 3}
    \end{subfigure}
	\caption{Remote sensing images.}\label{TestImg:Remote}
\end{figure}

\begin{table}[htbp]
\centering
\caption{Numerical results tested on remote sensing images at different noise levels  by different methods.}\label{Table:Remote}
{\small
\begin{tabular}{lcccccccccc}
\hline
\textbf{Image} & $L$ & \textbf{Meas.} & \textbf{Alg 1} &\textbf{Alg 2} & \textbf{SAR-} &\textbf{DZ}& \textbf{HNW}&\textbf{I-DIV}  & \textbf{TwL-} & \textbf{Dict} \\
 &  &  &  & &  \textbf{BM3D} && & & \textbf{4V} &  \\
 \hline
Remote 1 &1 & PSNR &\textbf{21.23} & 21.11 & \underline{21.12} & 20.47 &20.24 & 20.03 & 20.07 & 20.44\\
& &SSIM& \underline{0.5459} & \textbf{0.5510} & 0.5393 & 0.4950 & 0.4551 & 0.4709 & 0.4934 & 0.4867\\
 &3 & PSNR & \textbf{23.45} & \underline{23.43} & 23.39  & 22.51 &21.96 & 22.05 & 22.38 & 20.52\\
&  & SSIM & \textbf{0.6730} & \underline{0.6719} & 0.6716  & 0.6199 &0.5686 & 0.5935 & 0.6268 & 0.4953\\
&5 & PSNR & \underline{24.55} &\textbf{24.61} & 24.49   & 23.69 &22.90 & 23.17 & 23.55 & 20.93\\
&  & SSIM & \underline{0.7283} &\textbf{0.7326} & 0.7261  & 0.6800 &0.6274 & 0.6595 & 0.6824 & 0.5350\\
&&&&&&&&&\\
Remote 2 &1 & PSNR & \textbf{21.91} &\underline{21.88} & 21.68  & 20.37 &20.89 & 20.58 & 20.51 & 20.40\\
&  & SSIM & \textbf{0.5461} & \underline{0.5361} & 0.5334  & 0.4827 &0.4783 & 0.4791 & 0.4789 & 0.4665\\
&3 & PSNR & \textbf{24.13} & \underline{24.07} & 24.03  & 22.76 &22.71 & 22.49 & 22.66 & 22.34\\
&  & SSIM & \underline{0.6471} & \textbf{0.6474} & 0.6449  & 0.5758 &0.5805 & 0.5744 & 0.5845 & 0.5592\\
&5 & PSNR & \underline{25.32} & \textbf{25.37} & 25.21  & 23.98 &23.79 & 23.59 & 23.81 & 23.67\\
&  & SSIM & \textbf{0.6964} &\textbf{0.6964}& \underline{0.6939}  & 0.6302 &0.6294 & 0.6265 & 0.6364 & 0.6179 \\
&&&&&&&&&\\
Remote 3 &1 & PSNR & \textbf{22.16} &\textbf{22.16} & \underline{21.88}  & 20.93 &20.89 & 20.81 & 20.79 & 20.59\\
&  & SSIM & \underline{0.5895} & \textbf{0.6038} & 0.5565  & 0.5292 &0.4916 & 0.5131 & 0.5182 & 0.4955\\
&3 & PSNR & \textbf{24.66} & \underline{24.56} & 24.47 & 23.34 &22.81 & 23.02 & 23.18 & 22.14 \\
&  & SSIM & \textbf{0.7002} & \textbf{0.7002}  & \underline{0.6811} & 0.6236 &0.6077 & 0.6250 & 0.6218 & 0.5545\\
&5 & PSNR & \underline{25.80} & \textbf{25.83} & 25.65  & 24.45 &23.79 & 24.10 & 24.30 & 22.70\\
&  & SSIM & \underline{0.7427}& \textbf{0.7460} & 0.7316  & 0.6745 &0.6562 & 0.6713 & 0.6737 & 0.5736\\
 \hline
\end{tabular}}
\end{table}

\begin{figure}[htbp]
	\centering
	\begin{subfigure}[t]{.1942\textwidth}
    \includegraphics[width=3.245cm]{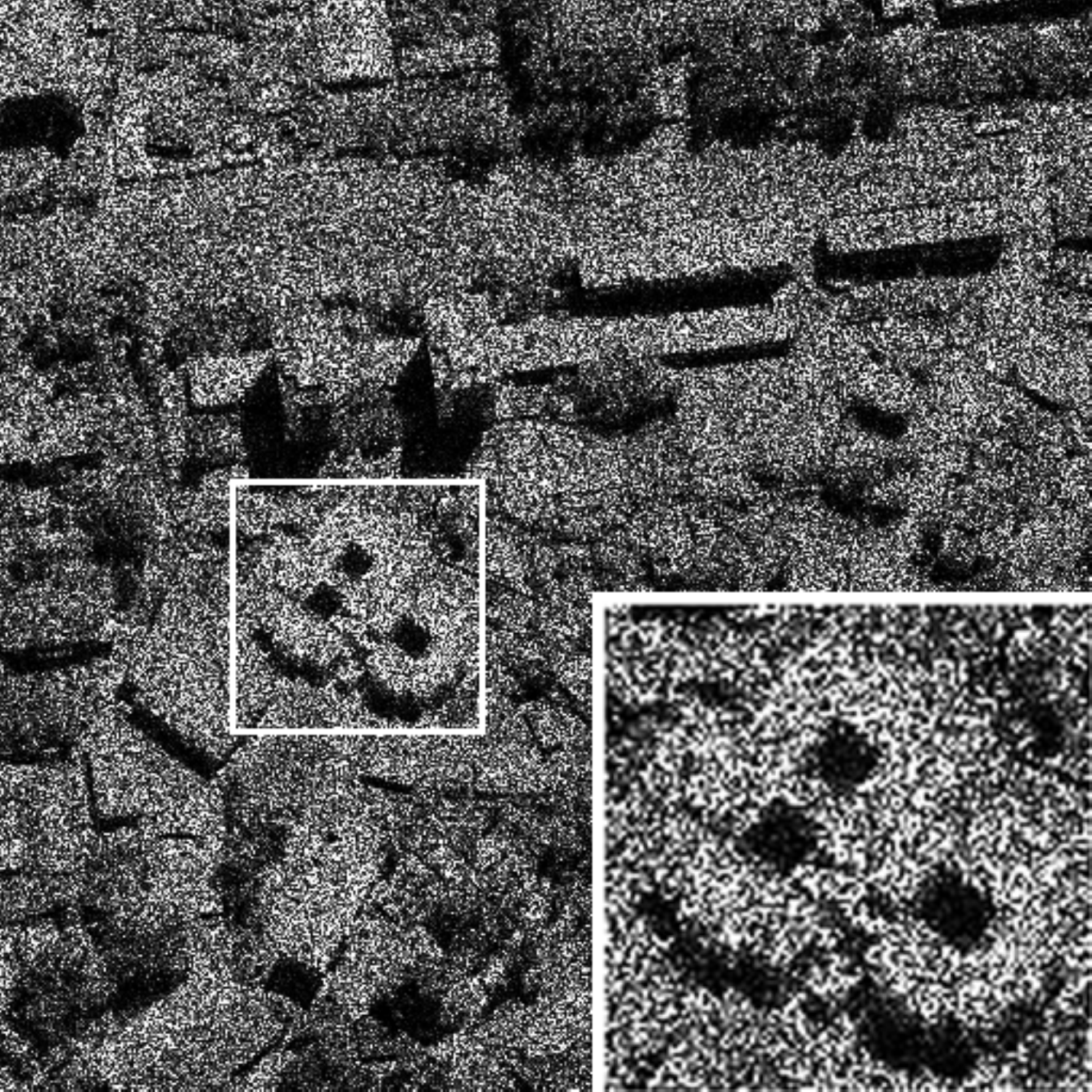}
    \caption{Noisy image (L=1)}
    \end{subfigure}	
    \begin{subfigure}[t]{.1942\textwidth}
    \includegraphics[width=3.245cm]{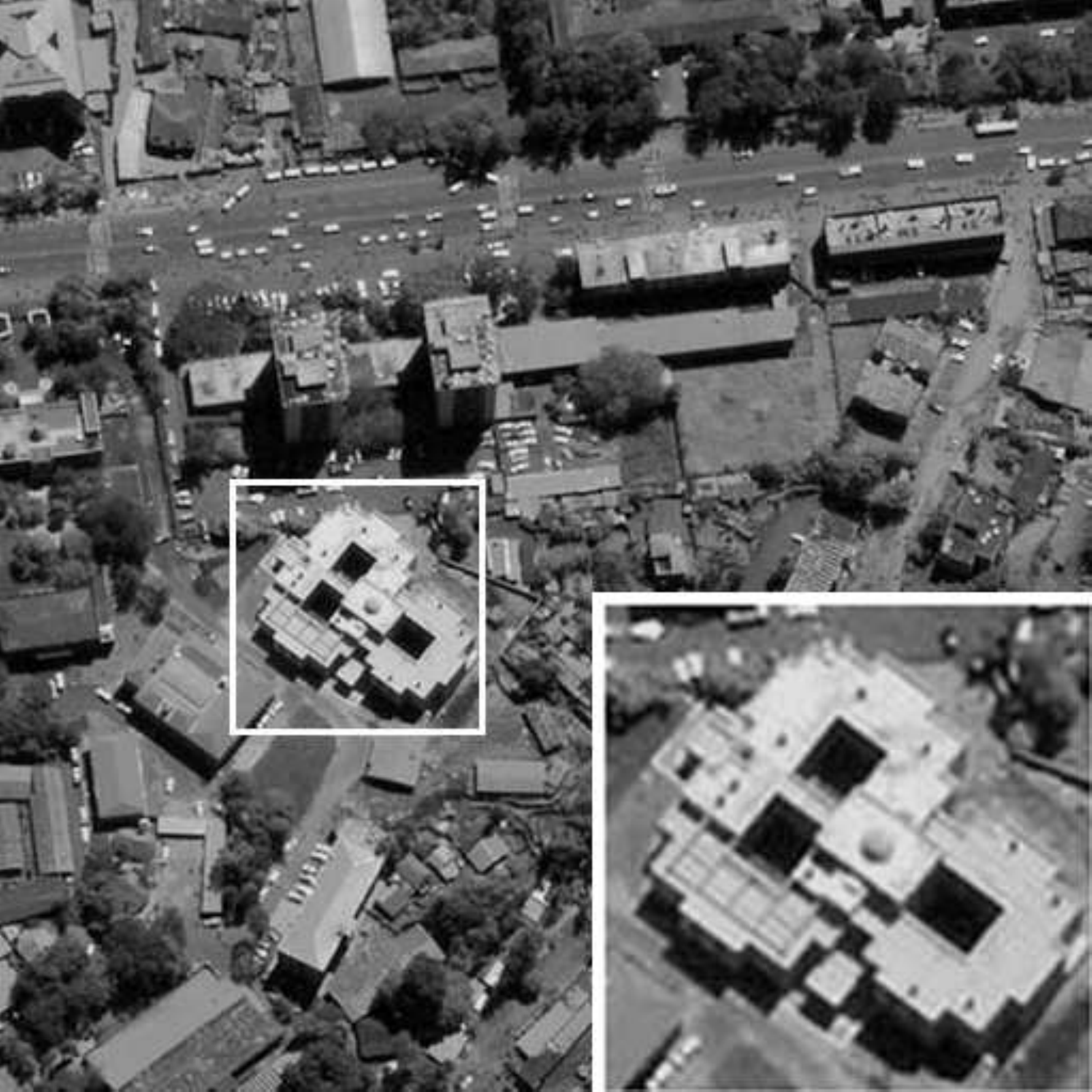}
    \caption{Ground truth}
    \end{subfigure}
    \begin{subfigure}[t]{.1942\textwidth}
    \includegraphics[width=3.245cm]{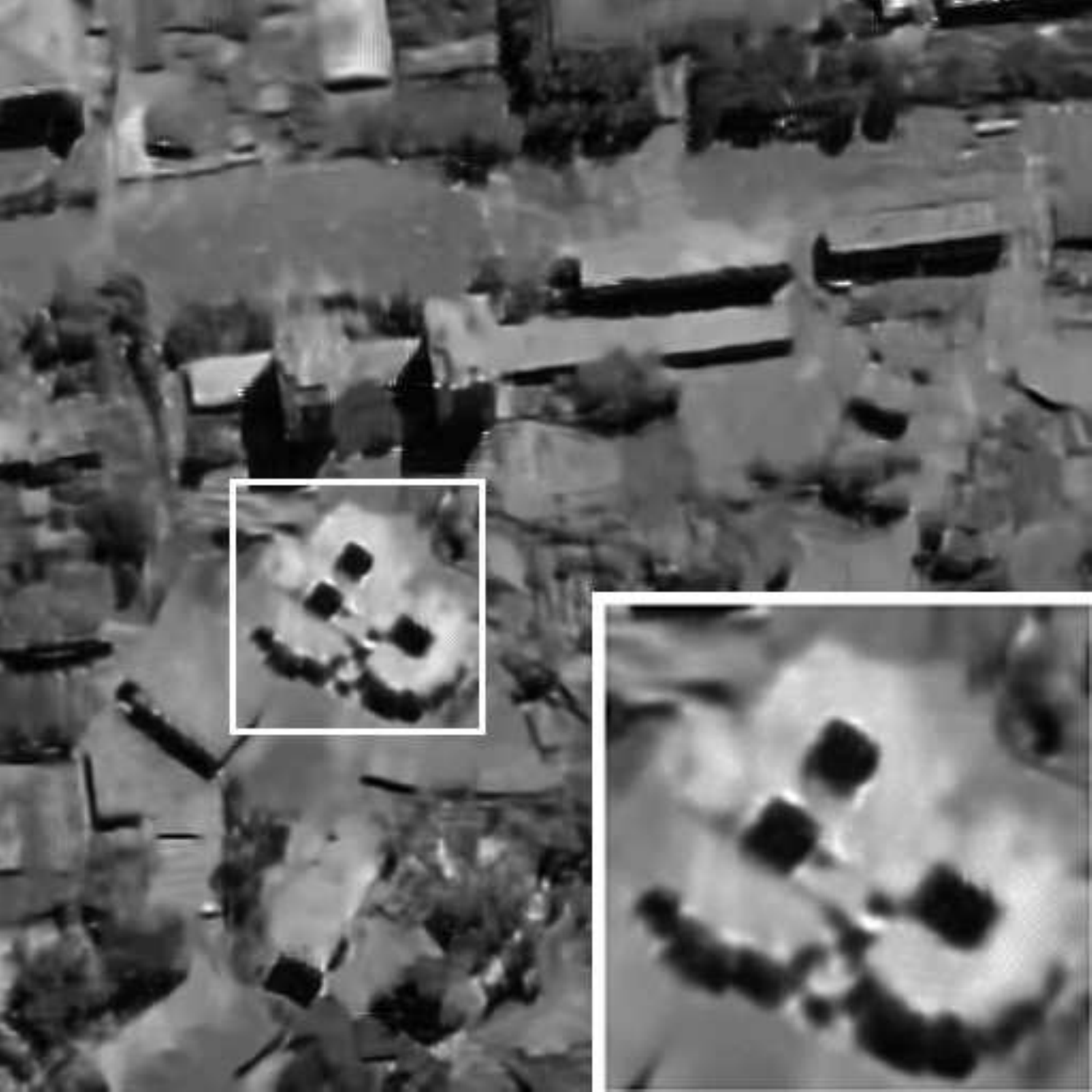}
    \caption{Alg 1}
    \end{subfigure}
    \begin{subfigure}[t]{.1942\textwidth}
    \includegraphics[width=3.245cm]{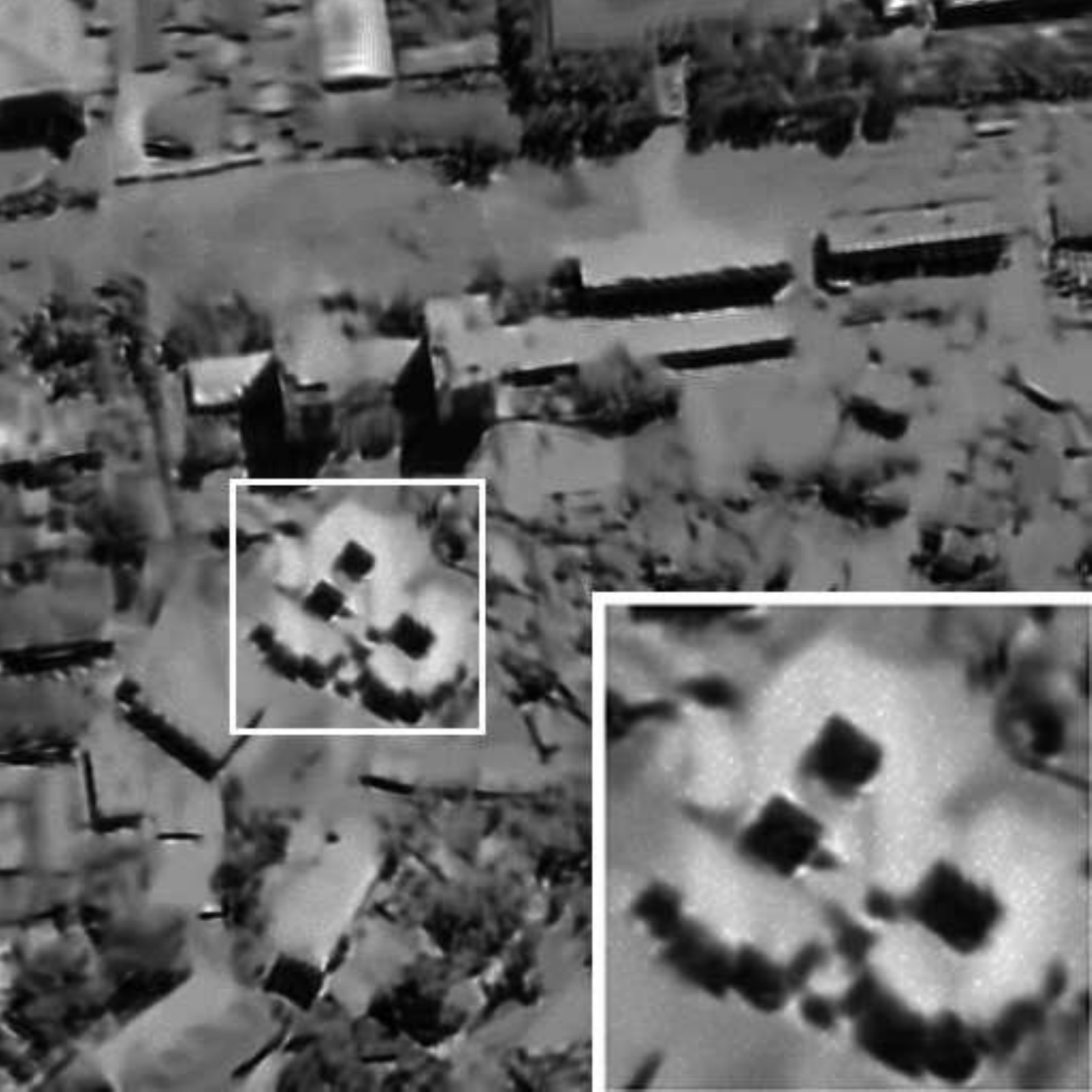}
    \caption{Alg 2}
    \end{subfigure}
    \begin{subfigure}[t]{.1942\textwidth}
    \includegraphics[width=3.245cm]{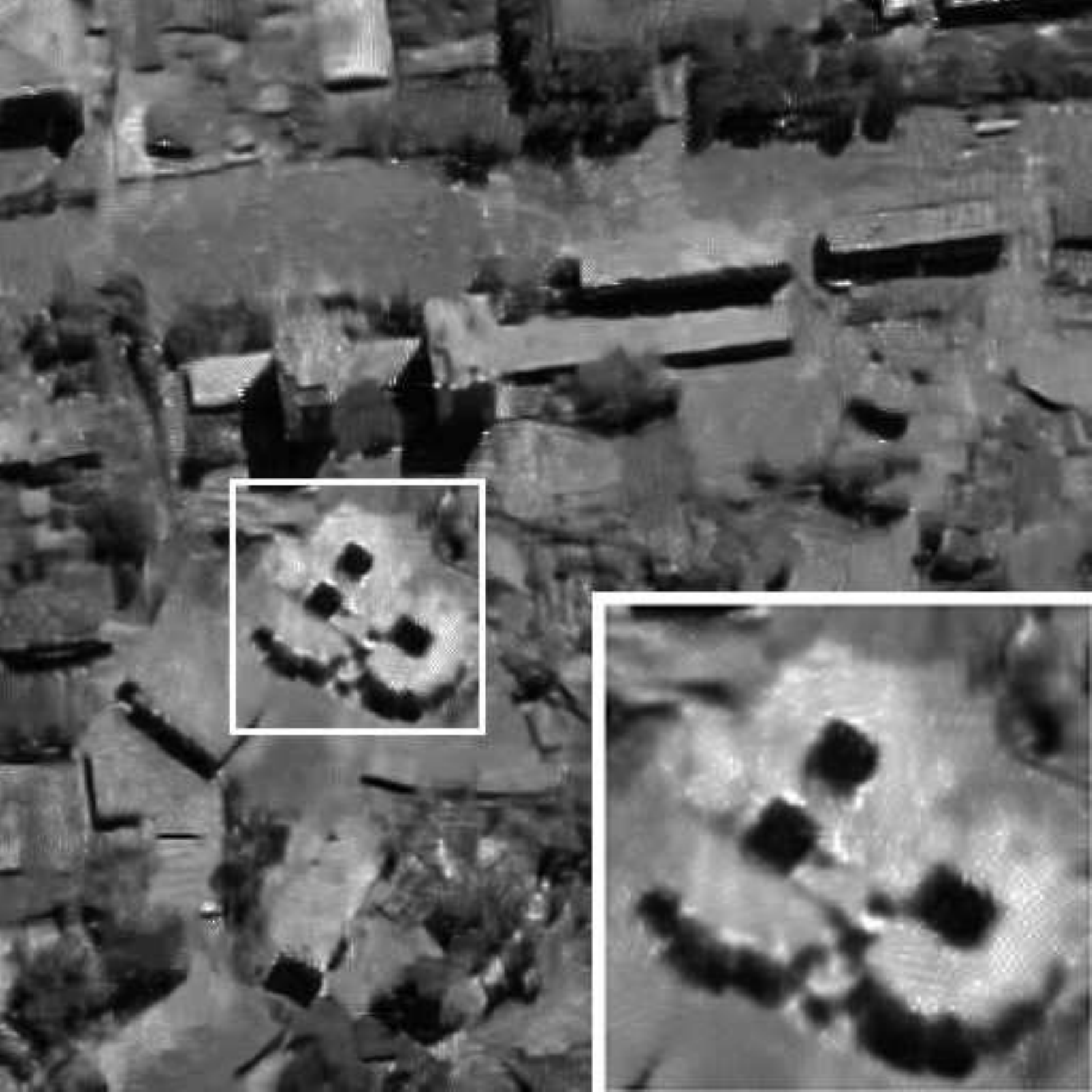}
	\caption{SAR-BM3D}
    \end{subfigure}\\
    \begin{subfigure}[t]{.1942\textwidth}
    \includegraphics[width=3.245cm]{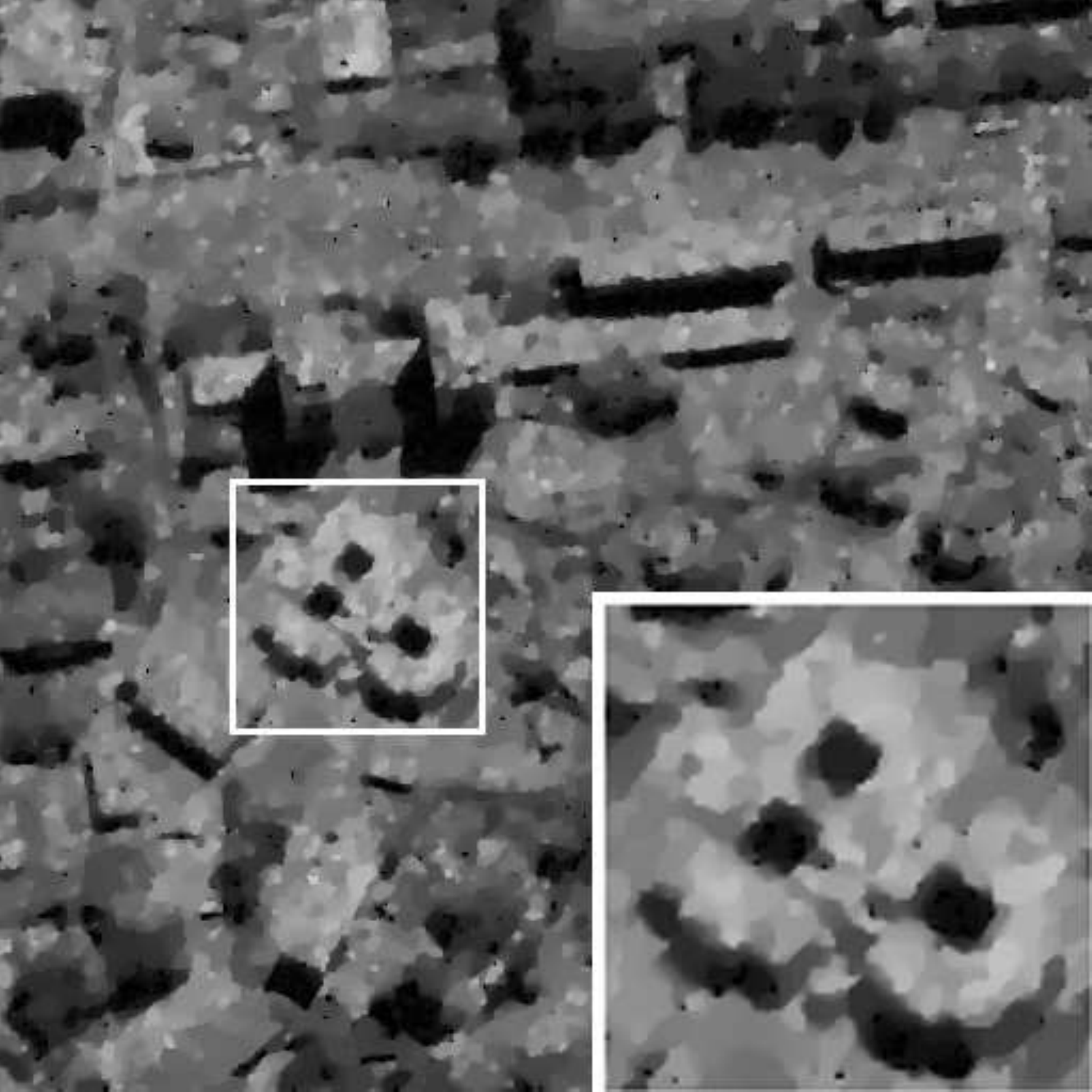}
	\caption{DZ}
    \end{subfigure}
    \begin{subfigure}[t]{.1942\textwidth}
    \includegraphics[width=3.245cm]{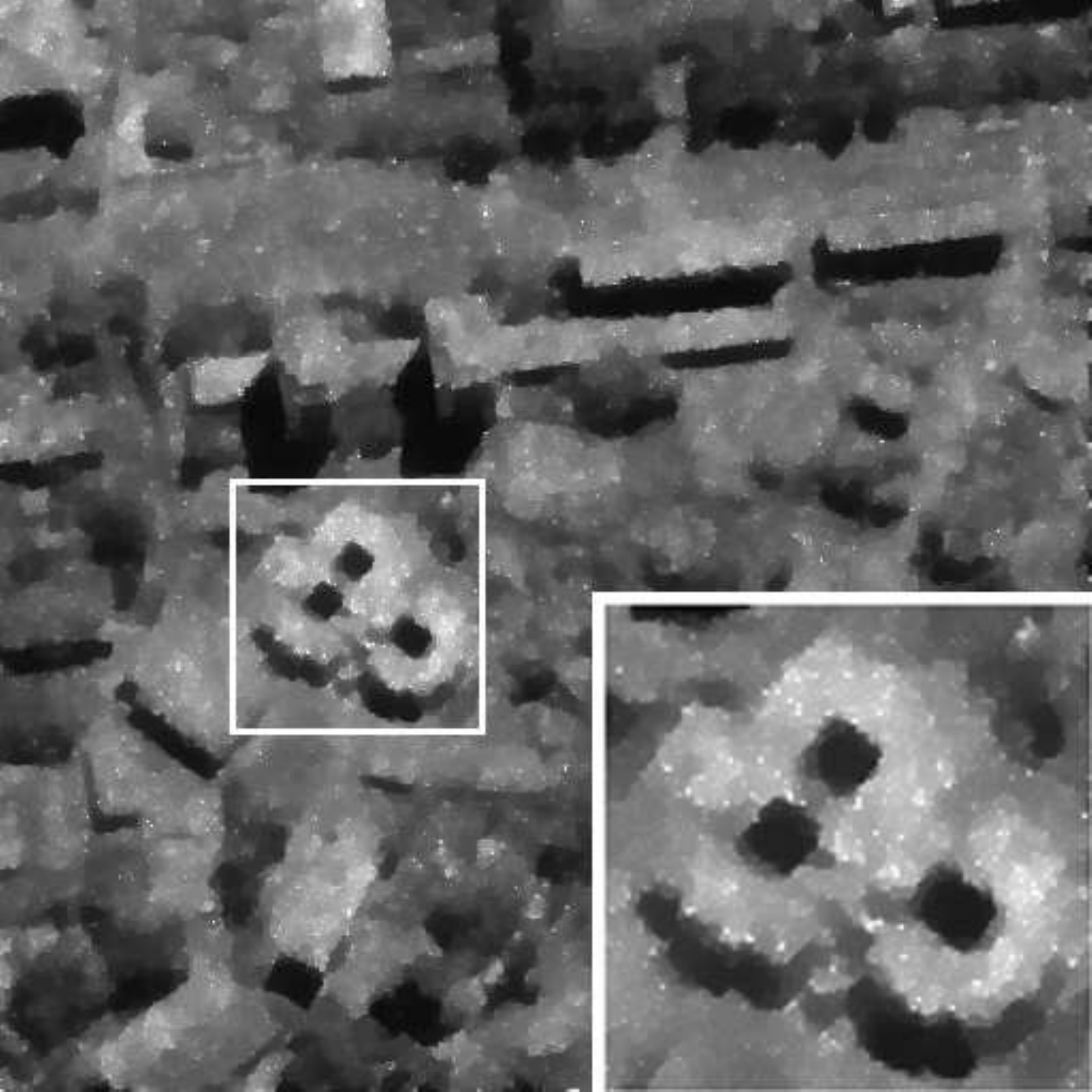}
	\caption{HNW}
    \end{subfigure}
    \begin{subfigure}[t]{.1942\textwidth}
    \includegraphics[width=3.245cm]{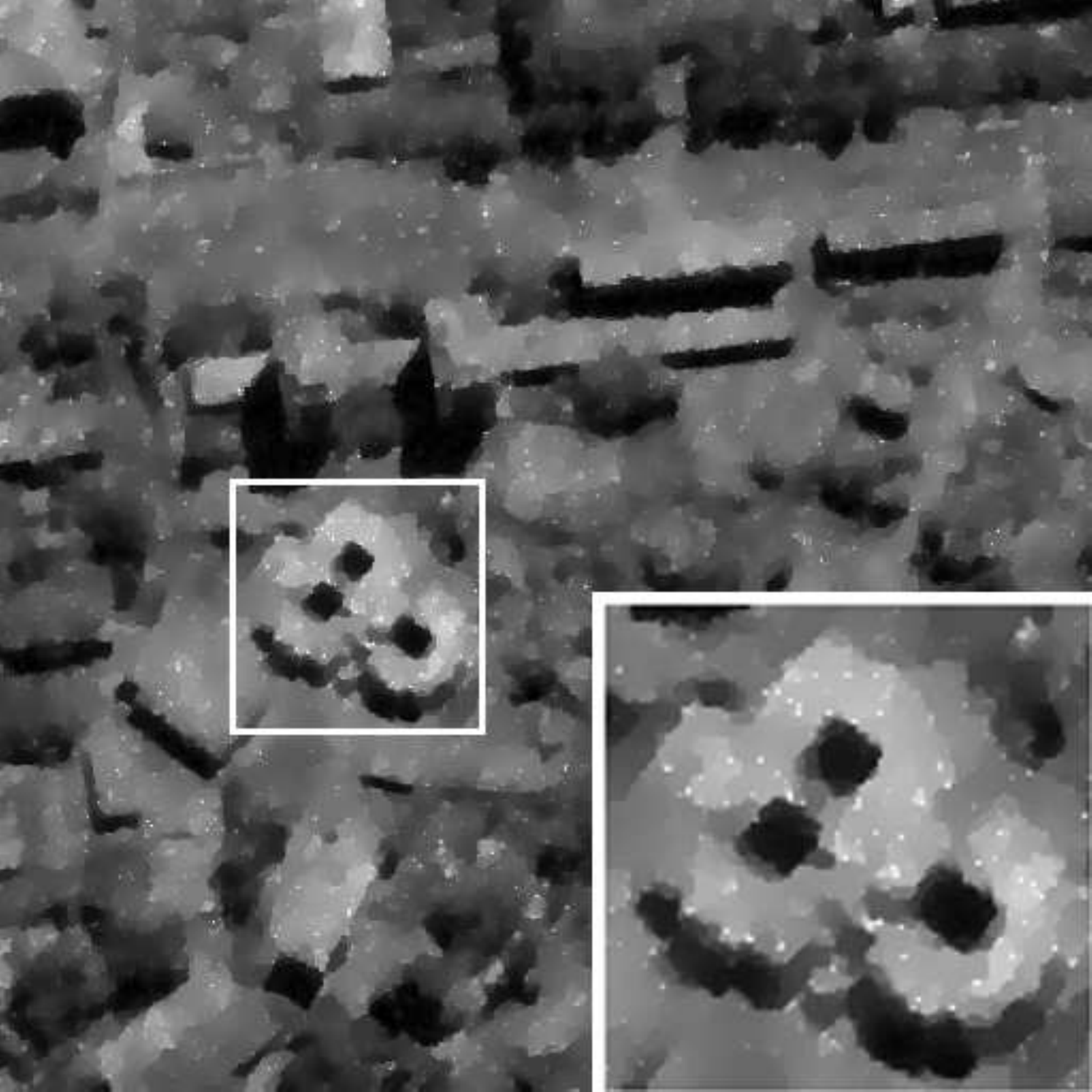}
    \caption{I-DIV}
    \end{subfigure}
    \begin{subfigure}[t]{.1942\textwidth}
    \includegraphics[width=3.245cm]{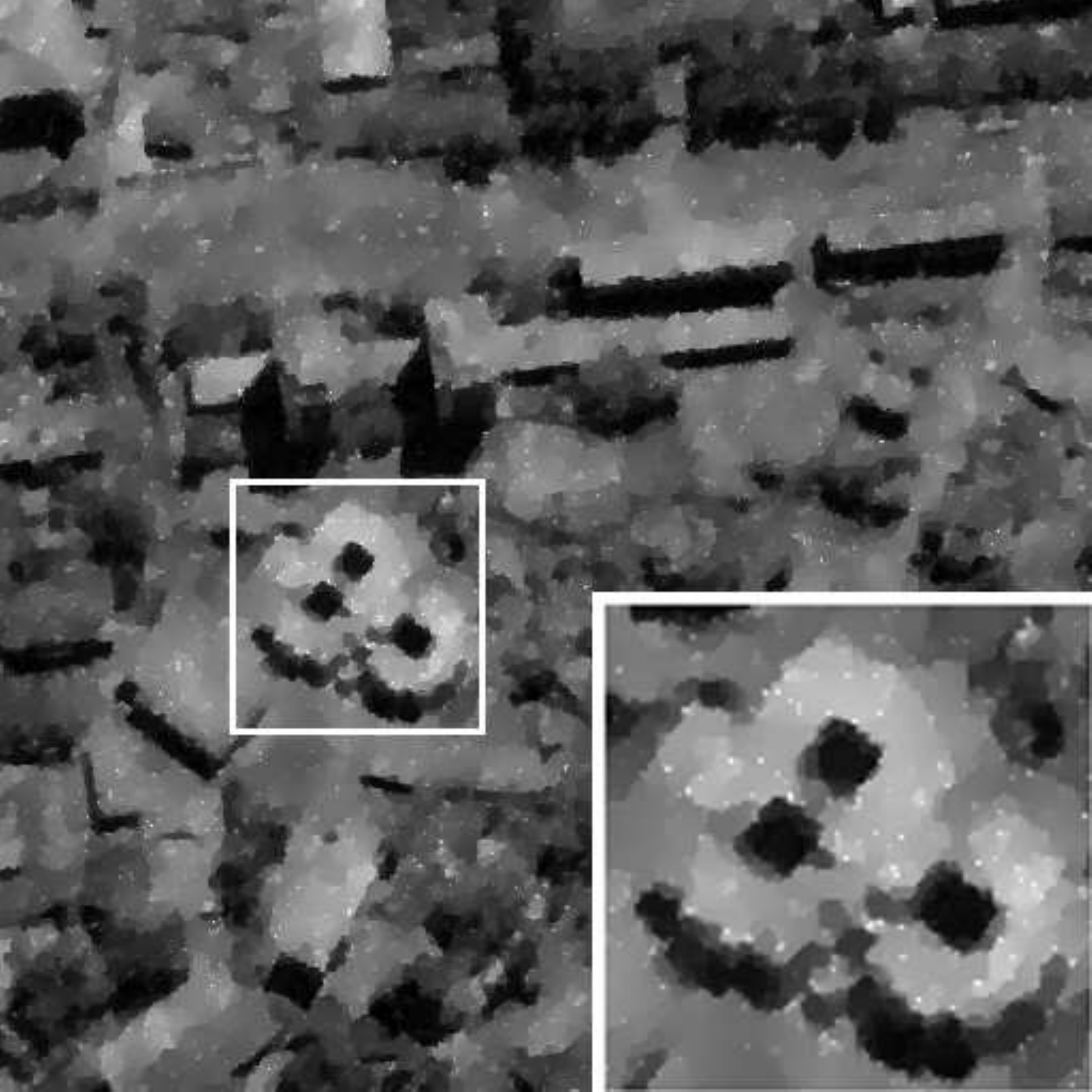}
    \caption{TwL-4V}
    \end{subfigure}
    \begin{subfigure}[t]{.1942\textwidth}
    \includegraphics[width=3.245cm]{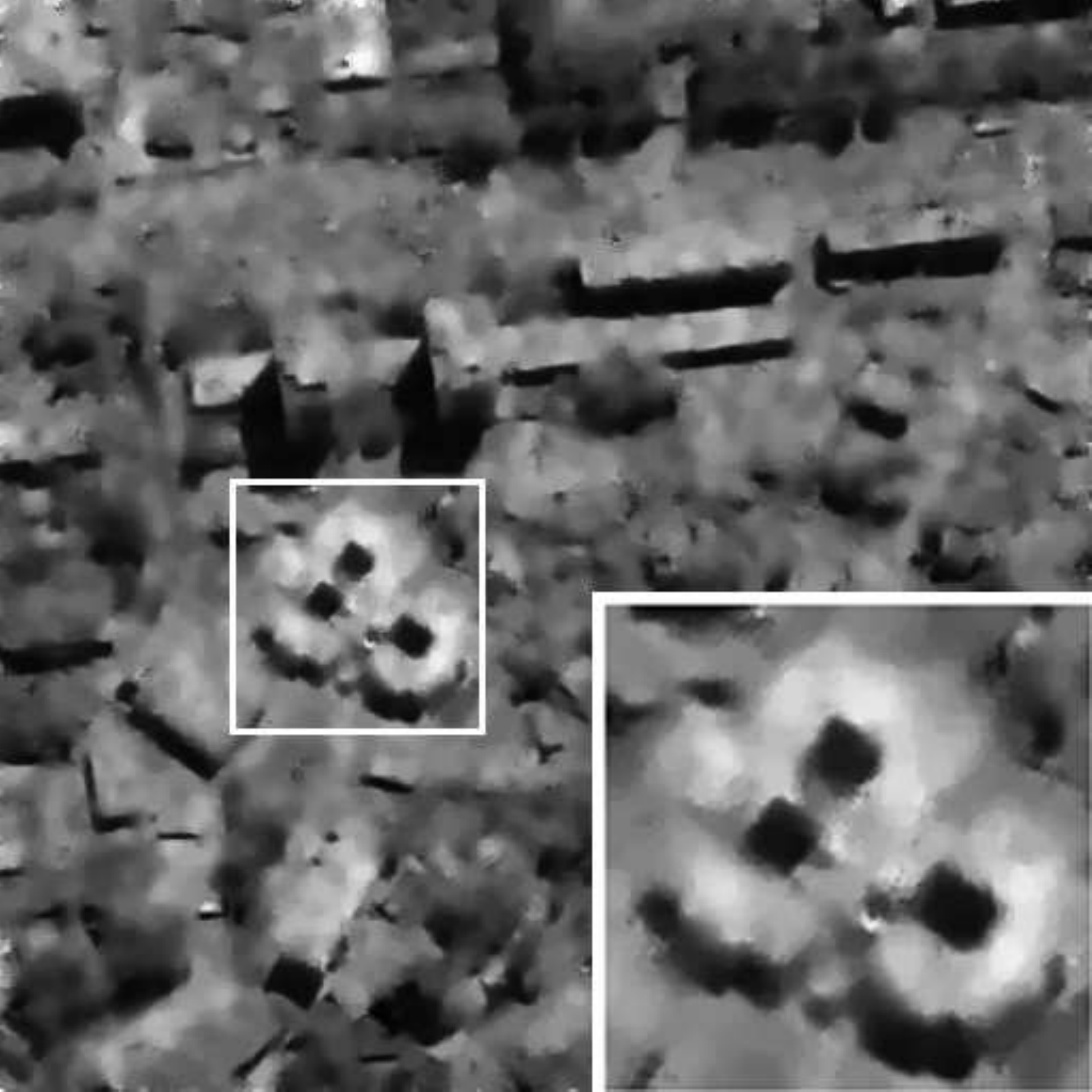}
	\caption{Dictionary}
    \end{subfigure}   	
	\caption{Comparison of denoised images restored from ``Remote 1" at noise level $L=1$ by different methods. The (PSNR, SSIM) values for each denoised image: (c) Alg 1 (21.23dB, 0.5459); (d) Alg 2 (21.11dB, 0.5510); (e) SAR-BM3D (21.12dB, 0.5393); (f) DZ (20.47dB, 0.4950); (g) HNW (20.24dB, 0.4551); (h) I-DIV (20.03dB, 0.4709); (i) TwL-4V (20.07dB, 0.4934); (j) Dictionary (20.44dB, 0.4867).}\label{fig:Remote1}
\end{figure}
\begin{figure}[htbp]
	\centering
	\begin{subfigure}[t]{.1942\textwidth}
    \includegraphics[width=3.245cm]{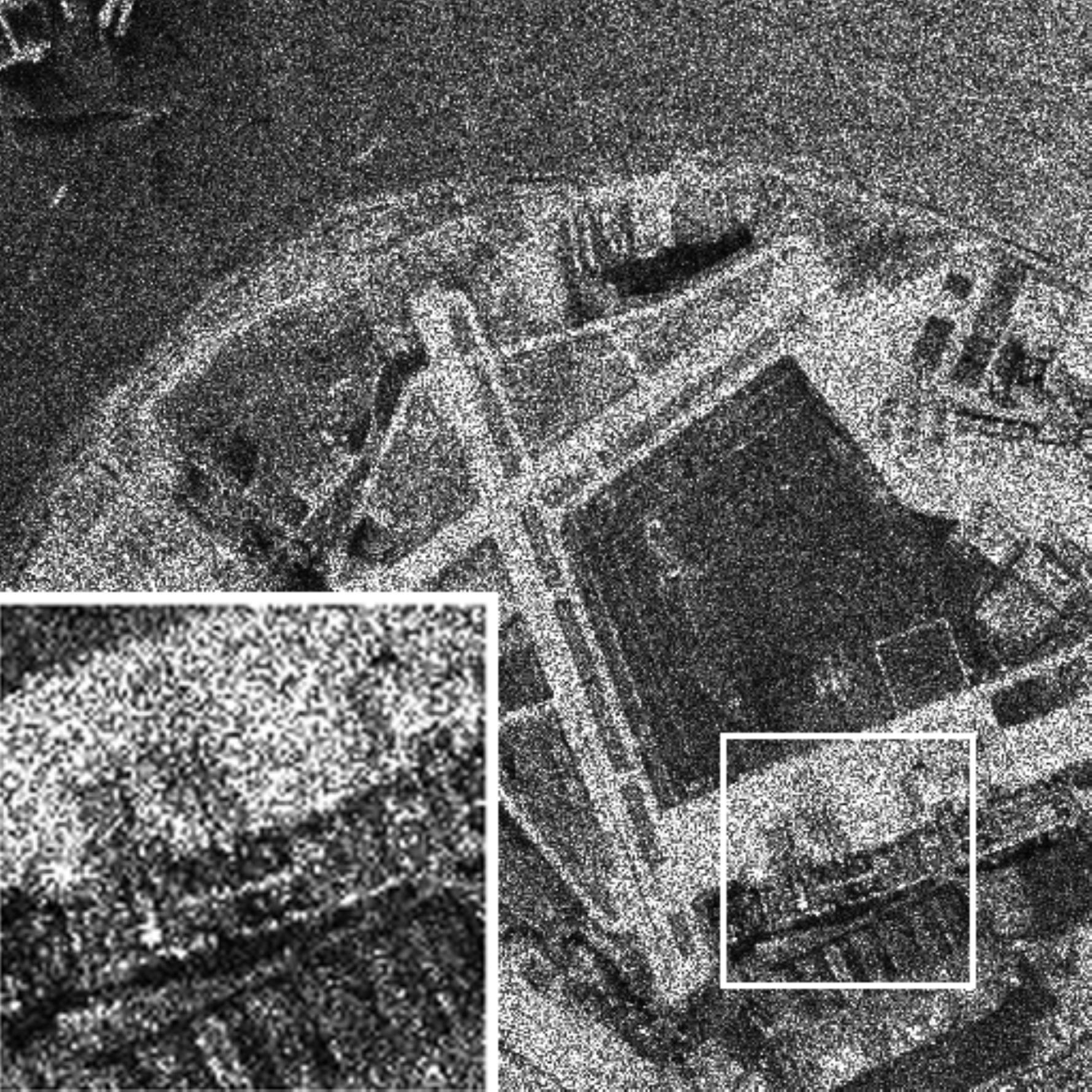}
    \caption{Noisy image (L=3)}
    \end{subfigure}
	\begin{subfigure}[t]{.1942\textwidth}
    \includegraphics[width=3.245cm]{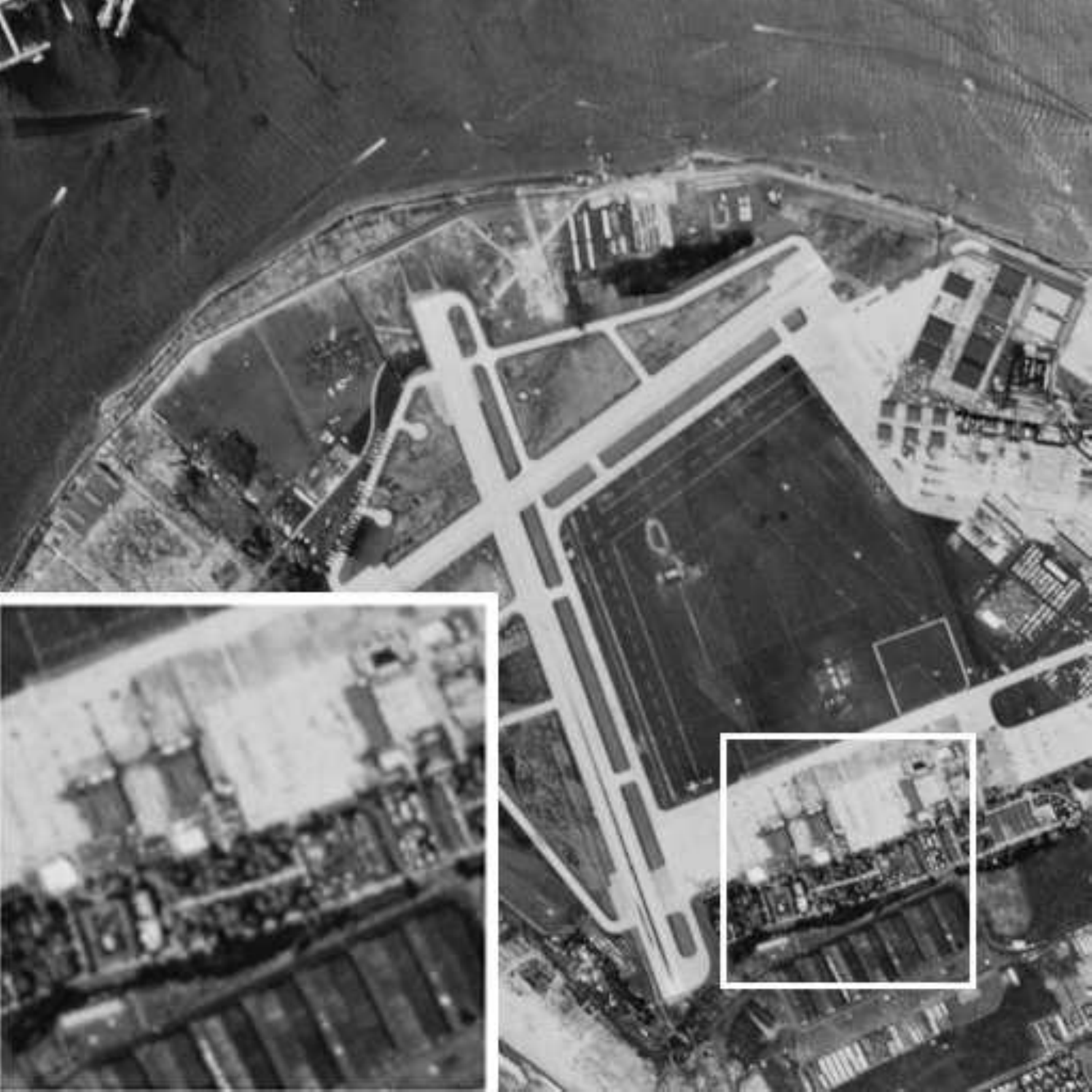}
    \caption{Ground truth}
    \end{subfigure}
    \begin{subfigure}[t]{.1942\textwidth}
    \includegraphics[width=3.245cm]{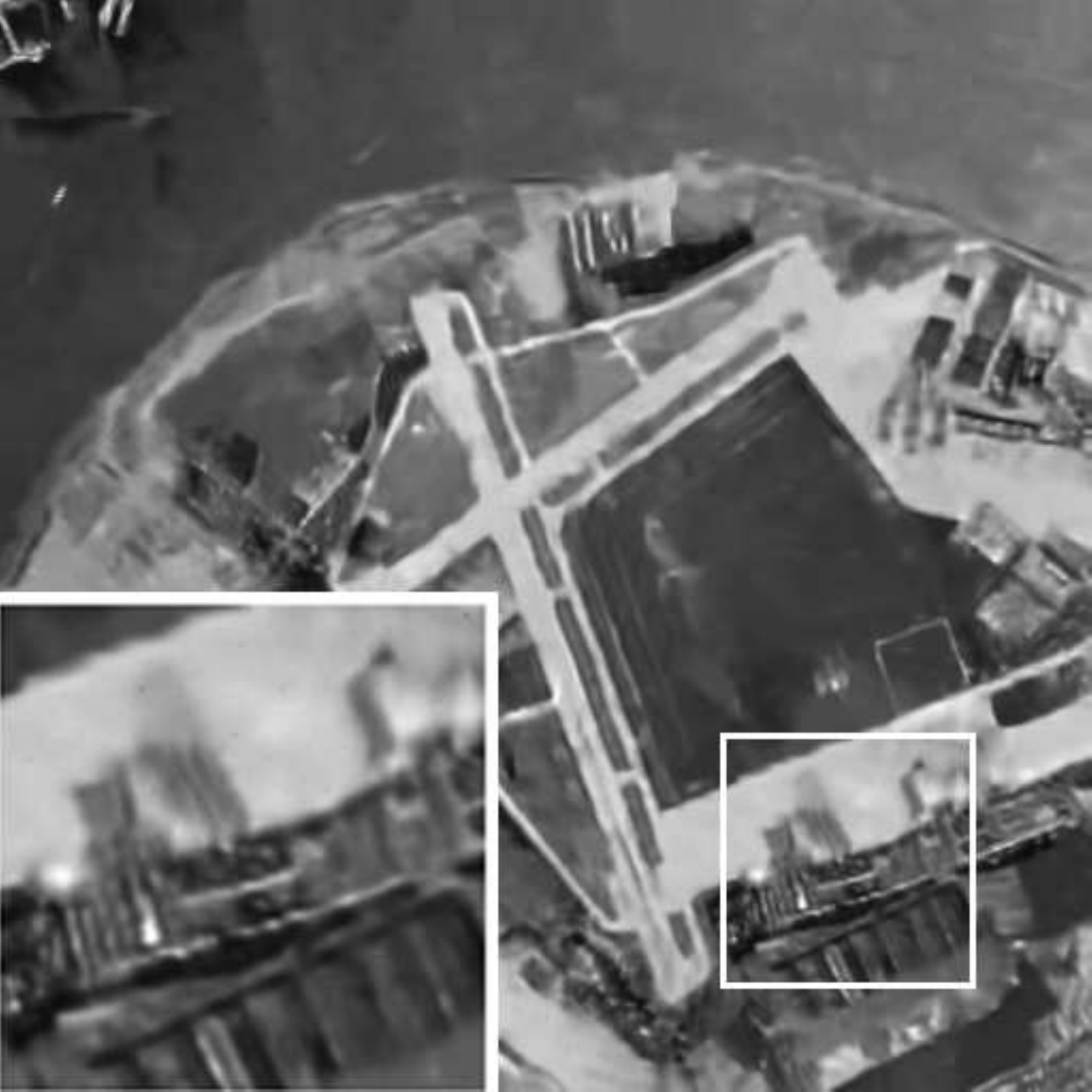}
    \caption{Alg 1}
    \end{subfigure}
    \begin{subfigure}[t]{.1942\textwidth}
    \includegraphics[width=3.245cm]{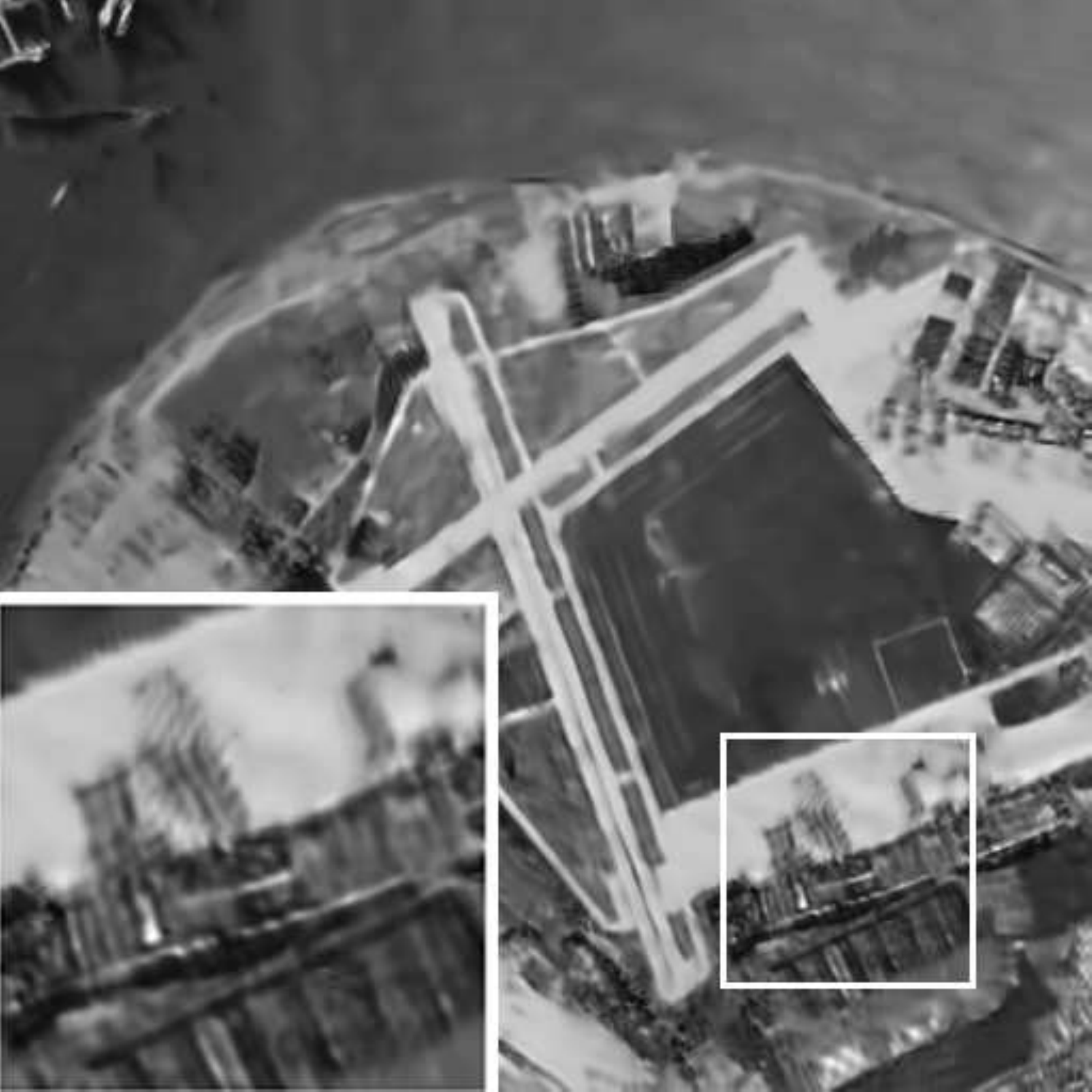}
    \caption{Alg 2}
    \end{subfigure}
    \begin{subfigure}[t]{.1942\textwidth}
    \includegraphics[width=3.245cm]{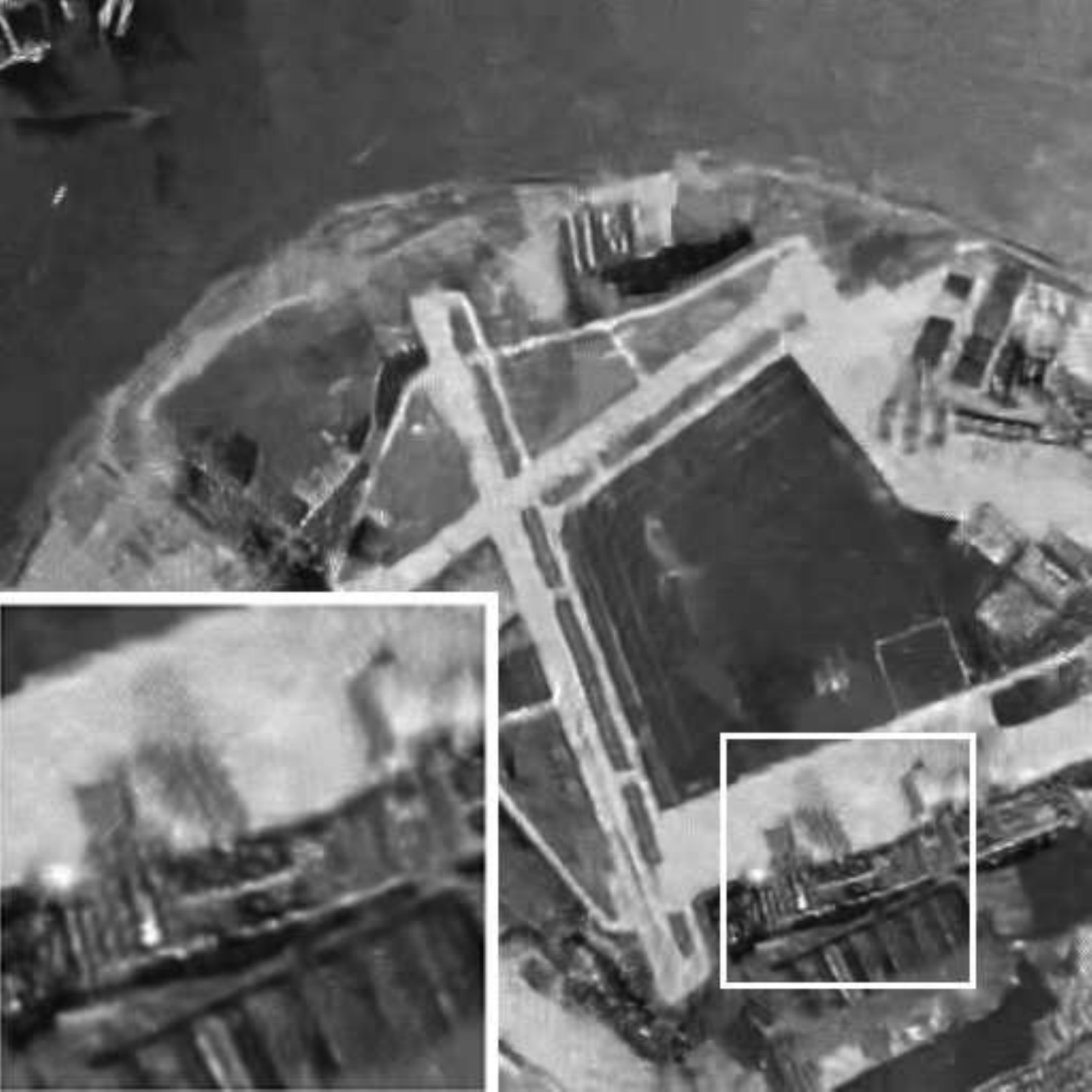}
	\caption{SAR-BM3D}
    \end{subfigure} \\
    \begin{subfigure}[t]{.1942\textwidth}
    \includegraphics[width=3.245cm]{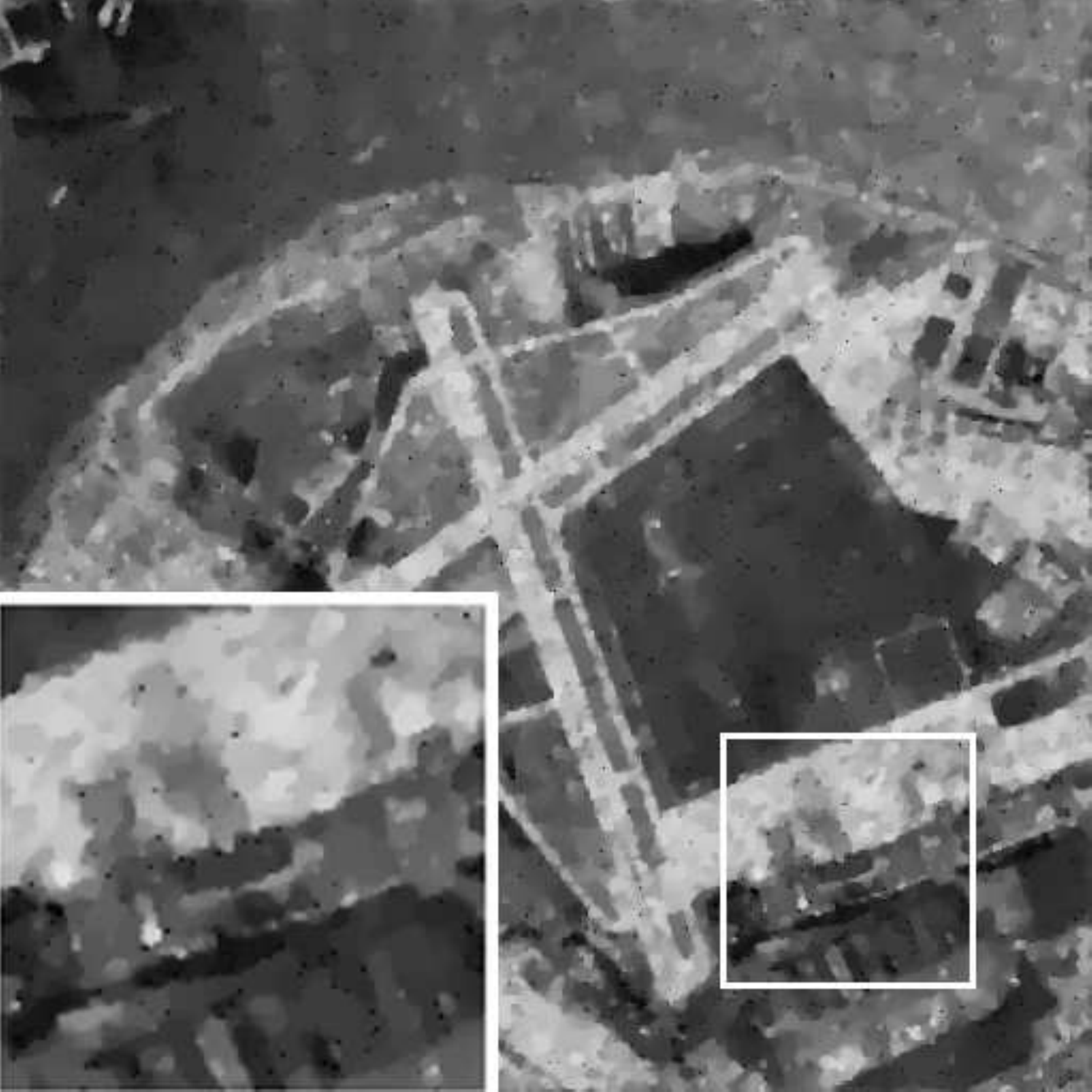}
	\caption{DZ}
    \end{subfigure}   	
    \begin{subfigure}[t]{.1942\textwidth}
    \includegraphics[width=3.245cm]{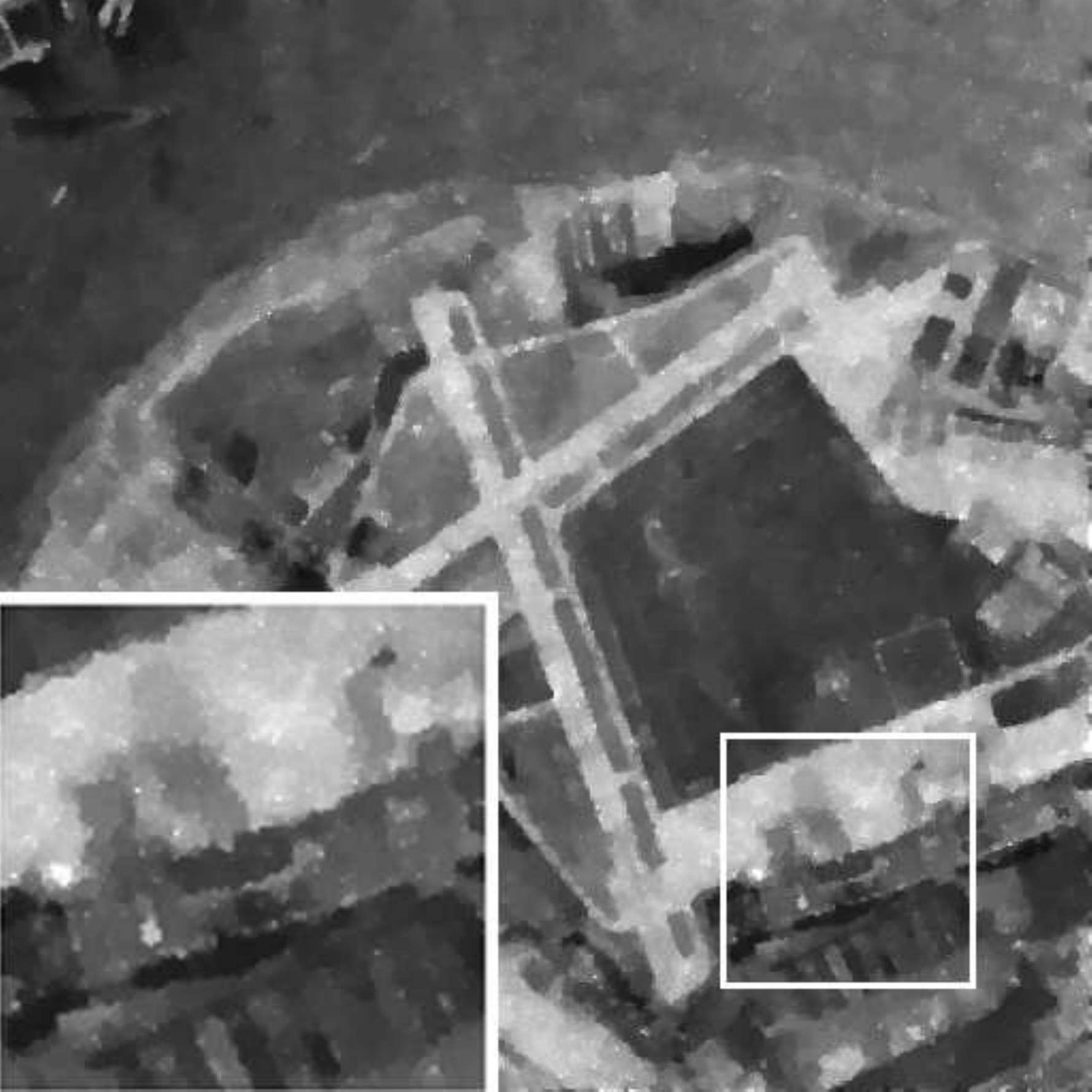}
	\caption{HNW}
    \end{subfigure}
    \begin{subfigure}[t]{.1942\textwidth}
    \includegraphics[width=3.245cm]{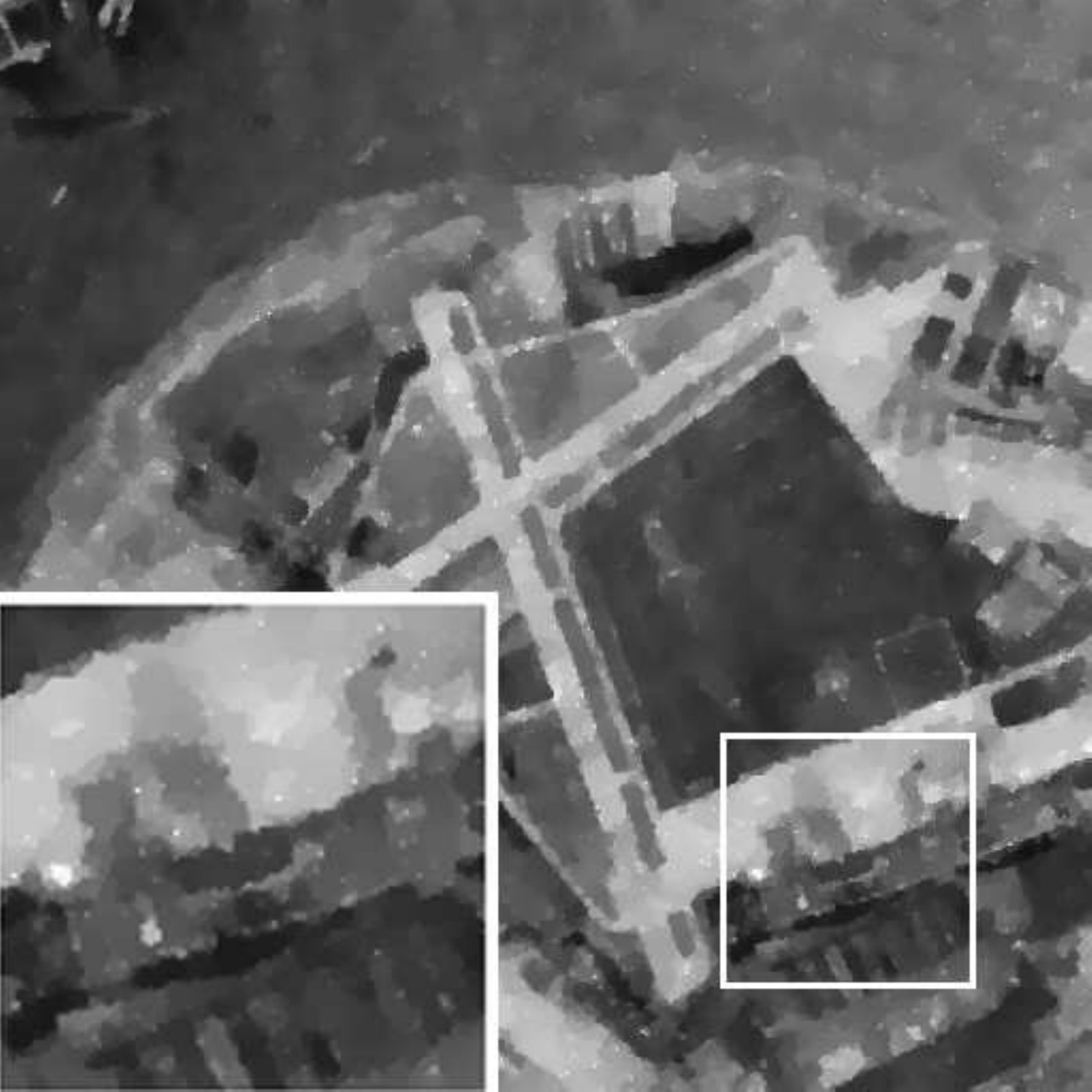}
    \caption{I-DIV}
    \end{subfigure}
    \begin{subfigure}[t]{.1942\textwidth}
    \includegraphics[width=3.245cm]{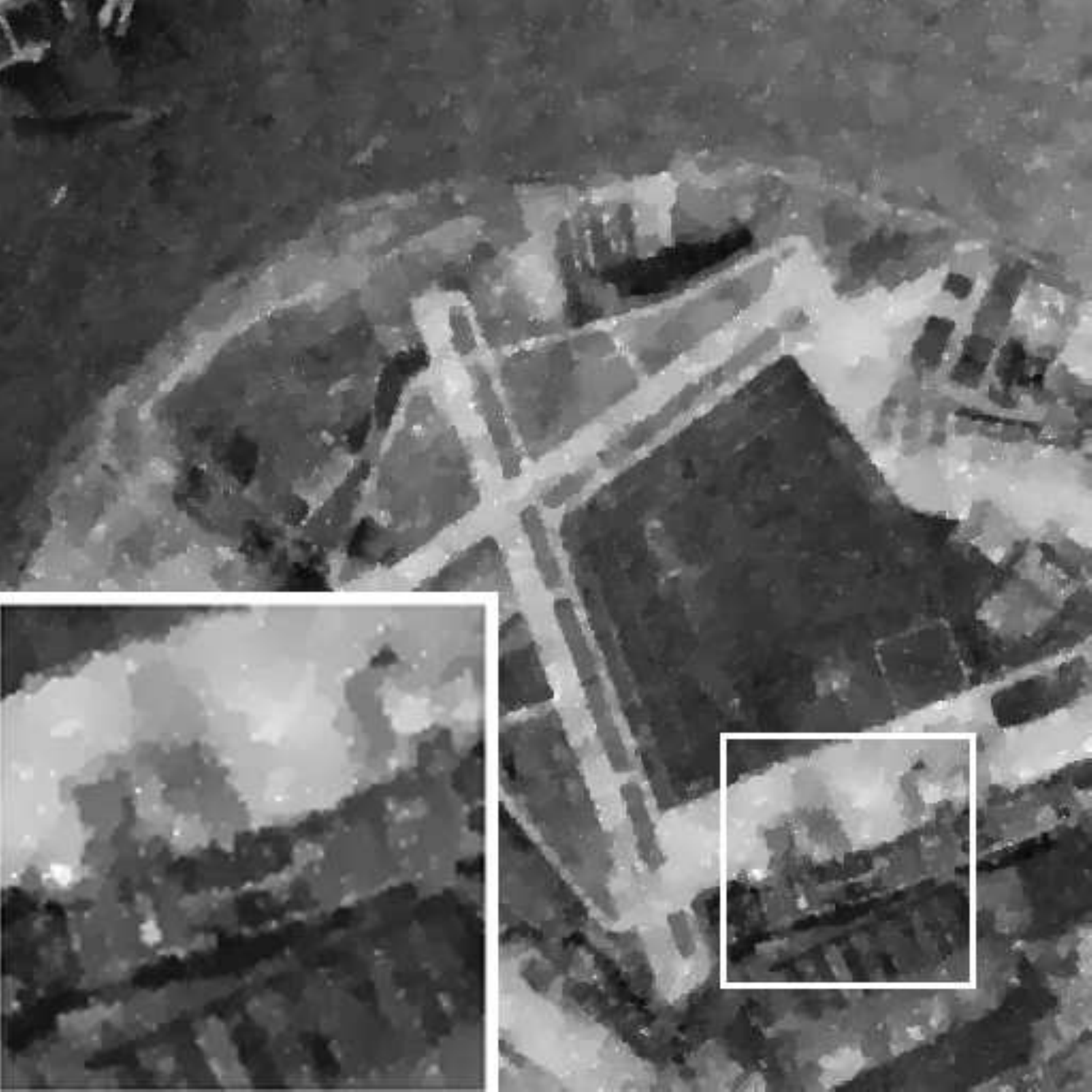}
    \caption{TwL-4V}
    \end{subfigure}
    \begin{subfigure}[t]{.1942\textwidth}
    \includegraphics[width=3.245cm]{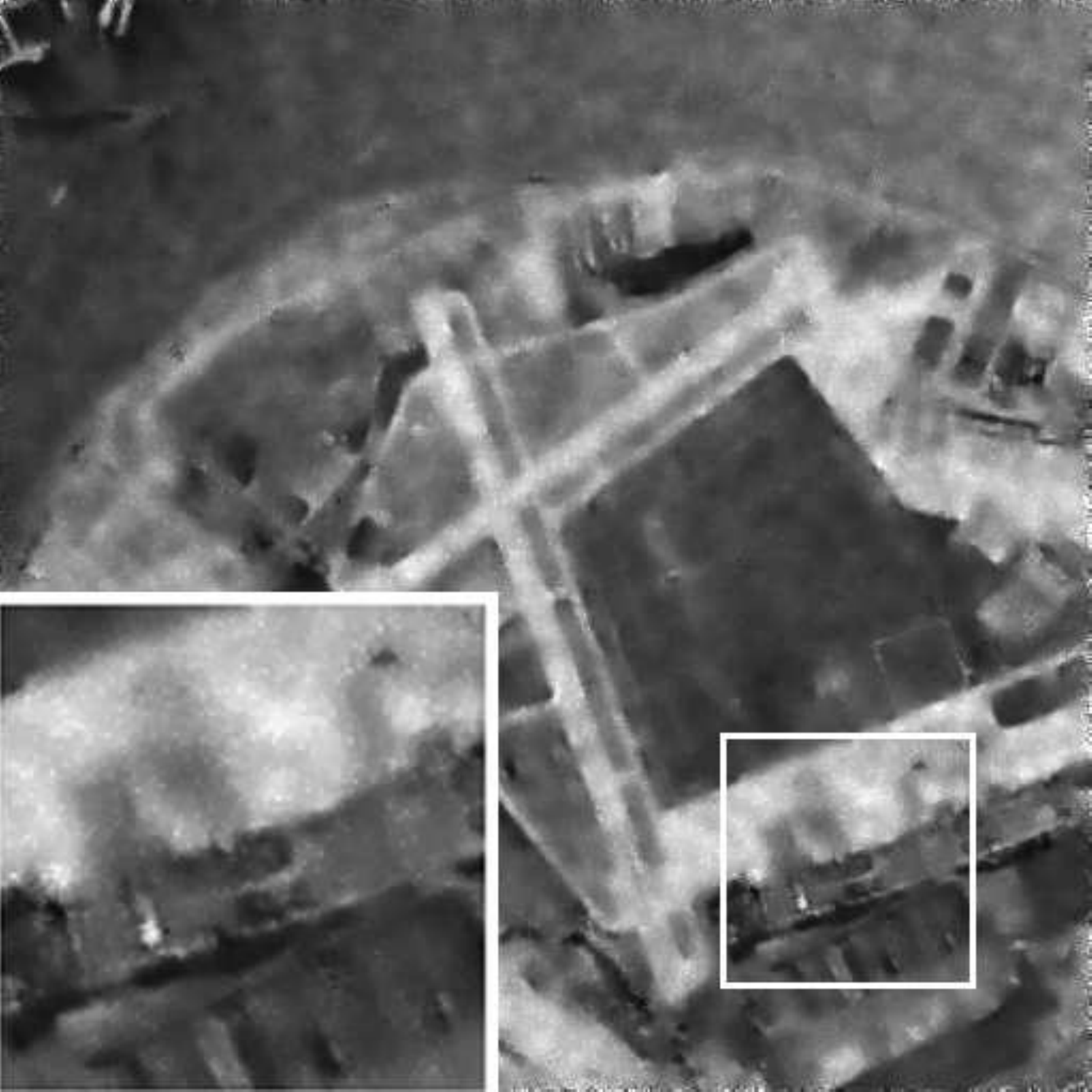}
	\caption{Dictionary}
    \end{subfigure}   	
	\caption{Comparison of denoised images restored from ``Remote 2" at noise level $L=3$ by different methods. The (PSNR, SSIM) values for each denoised image: (c) Alg 1 (24.13dB, 0.6471); (d) Alg 2 (24.07dB, 0.6474); (e) SAR-BM3D (24.03dB, 0.6449); (f) DZ (22.76dB, 0.5758); (g) HNW (22.71dB, 0.5805); (h) I-DIV (22.49dB, 0.5744); (i) TwL-4V (22.66dB, 0.5845); (j) Dictionary (22.34dB, 0.5592).}\label{fig:Remote2}
\end{figure}
\begin{figure}[htbp]
	\centering
	\begin{subfigure}[t]{.1942\textwidth}
    \includegraphics[width=3.245cm]{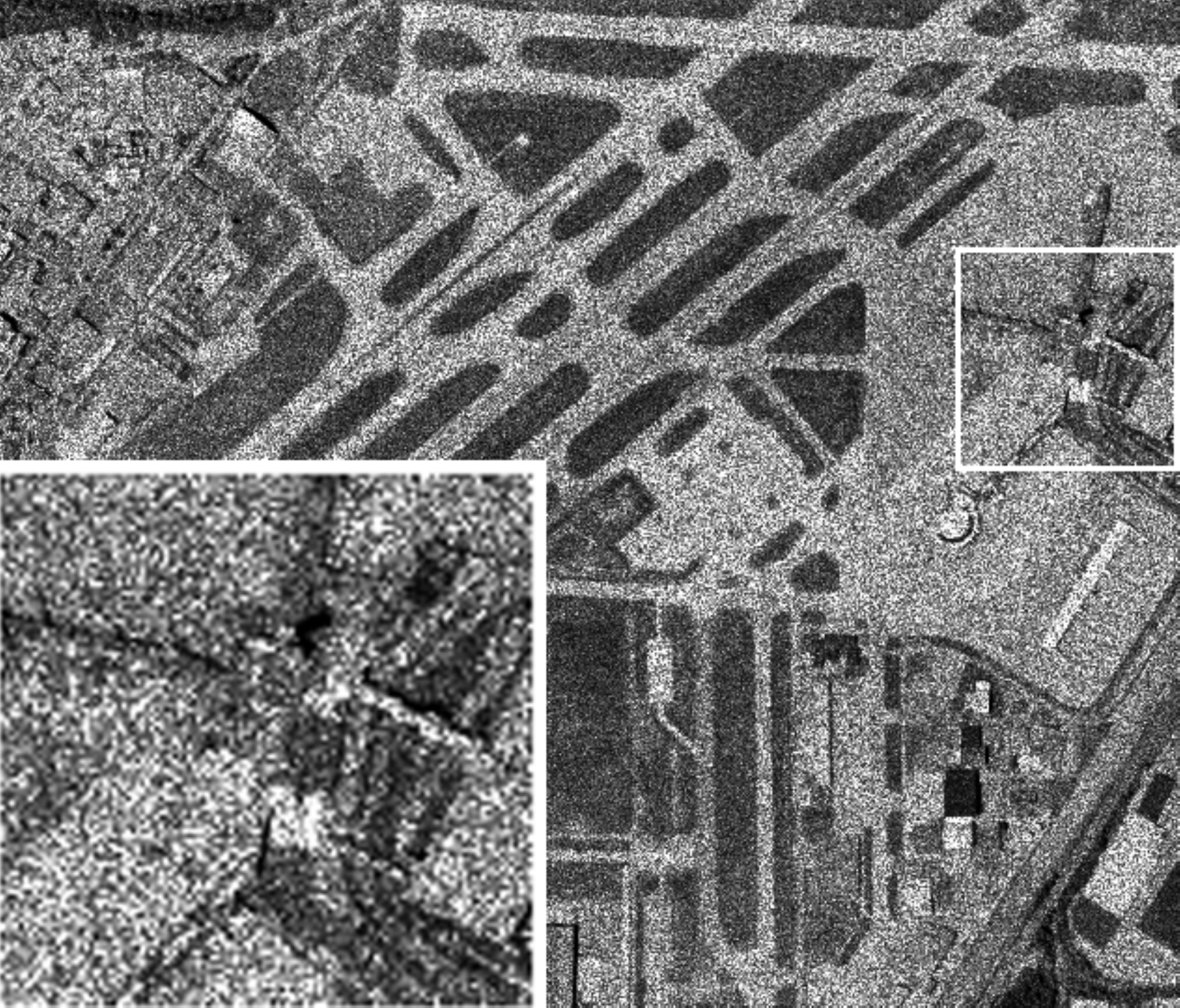}
    \caption{Noisy image (L=5)}
    \end{subfigure}
	\begin{subfigure}[t]{.1942\textwidth}
    \includegraphics[width=3.245cm]{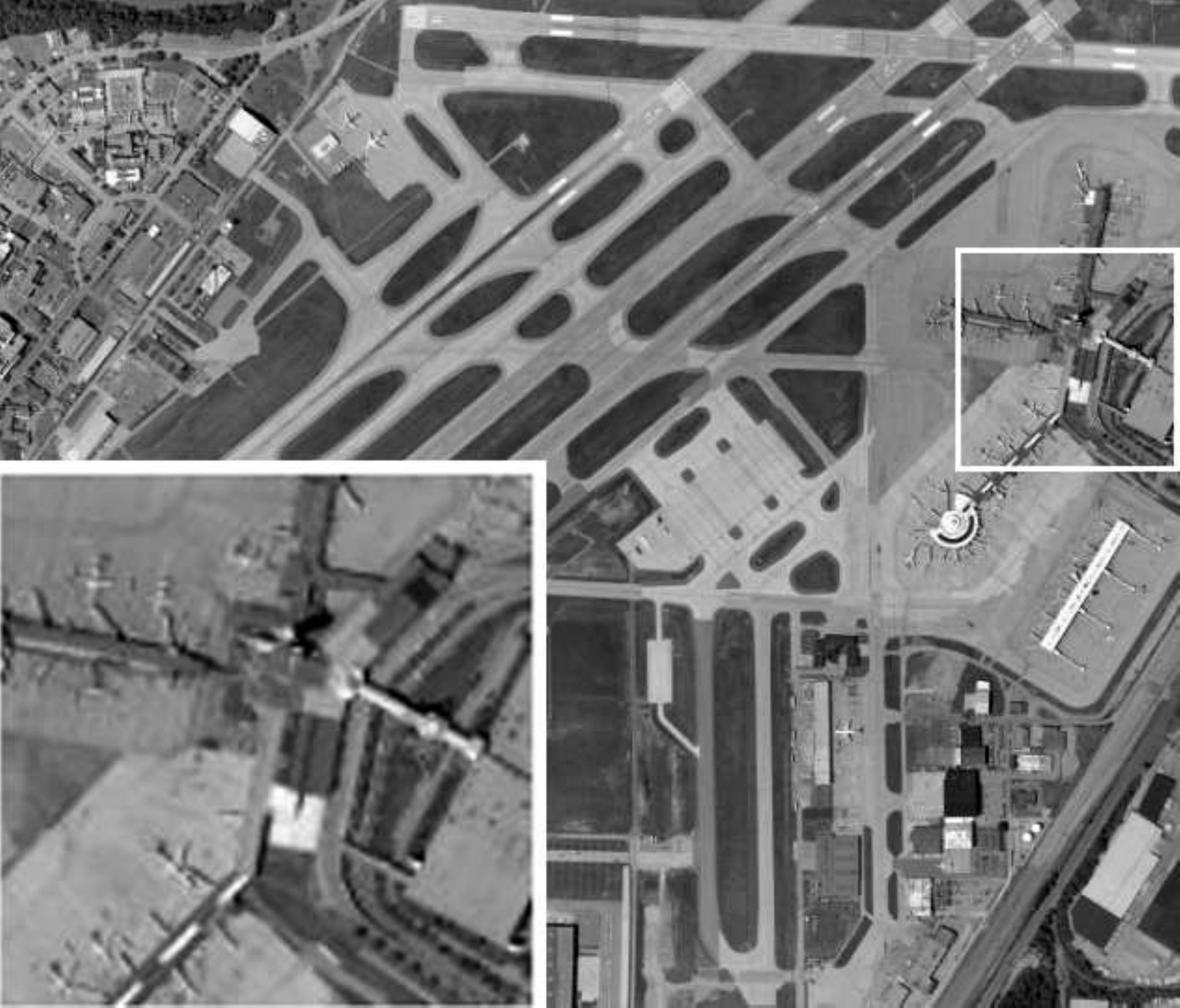}
    \caption{Ground truth}
    \end{subfigure}
    \begin{subfigure}[t]{.1942\textwidth}
    \includegraphics[width=3.245cm]{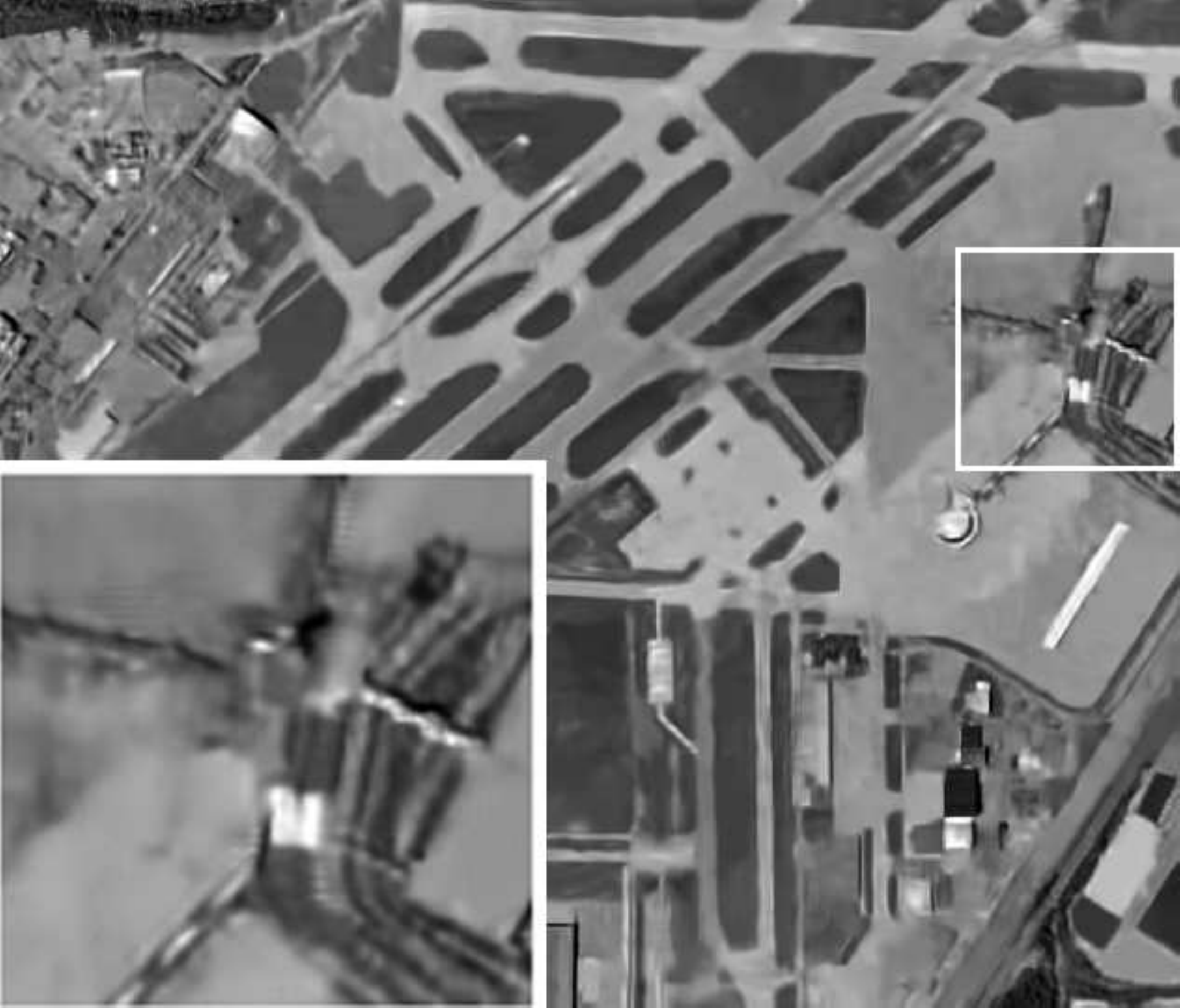}
    \caption{Alg 1}
    \end{subfigure}
    \begin{subfigure}[t]{.1942\textwidth}
    \includegraphics[width=3.245cm]{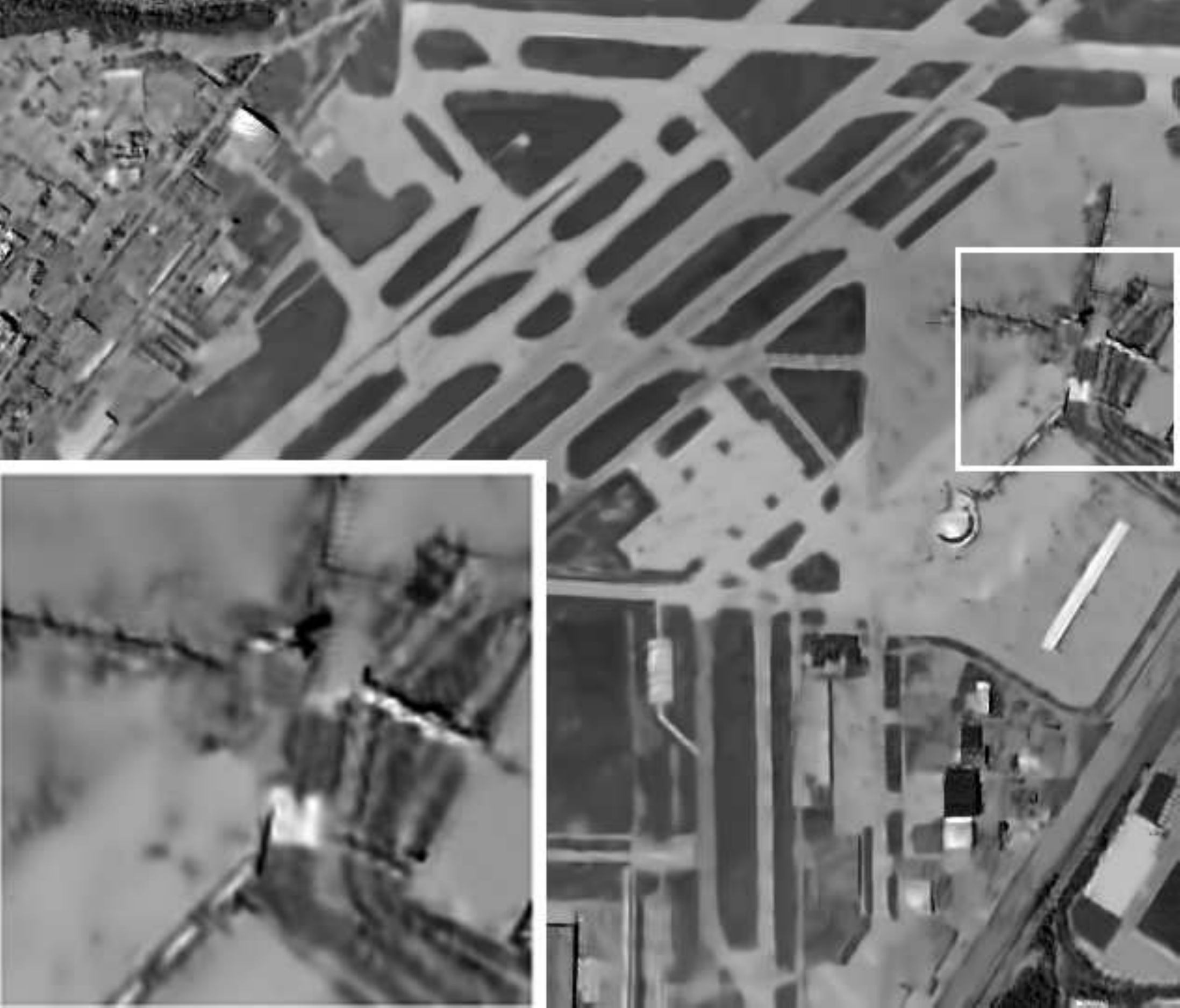}
    \caption{Alg 2}
    \end{subfigure}
    \begin{subfigure}[t]{.1942\textwidth}
    \includegraphics[width=3.245cm]{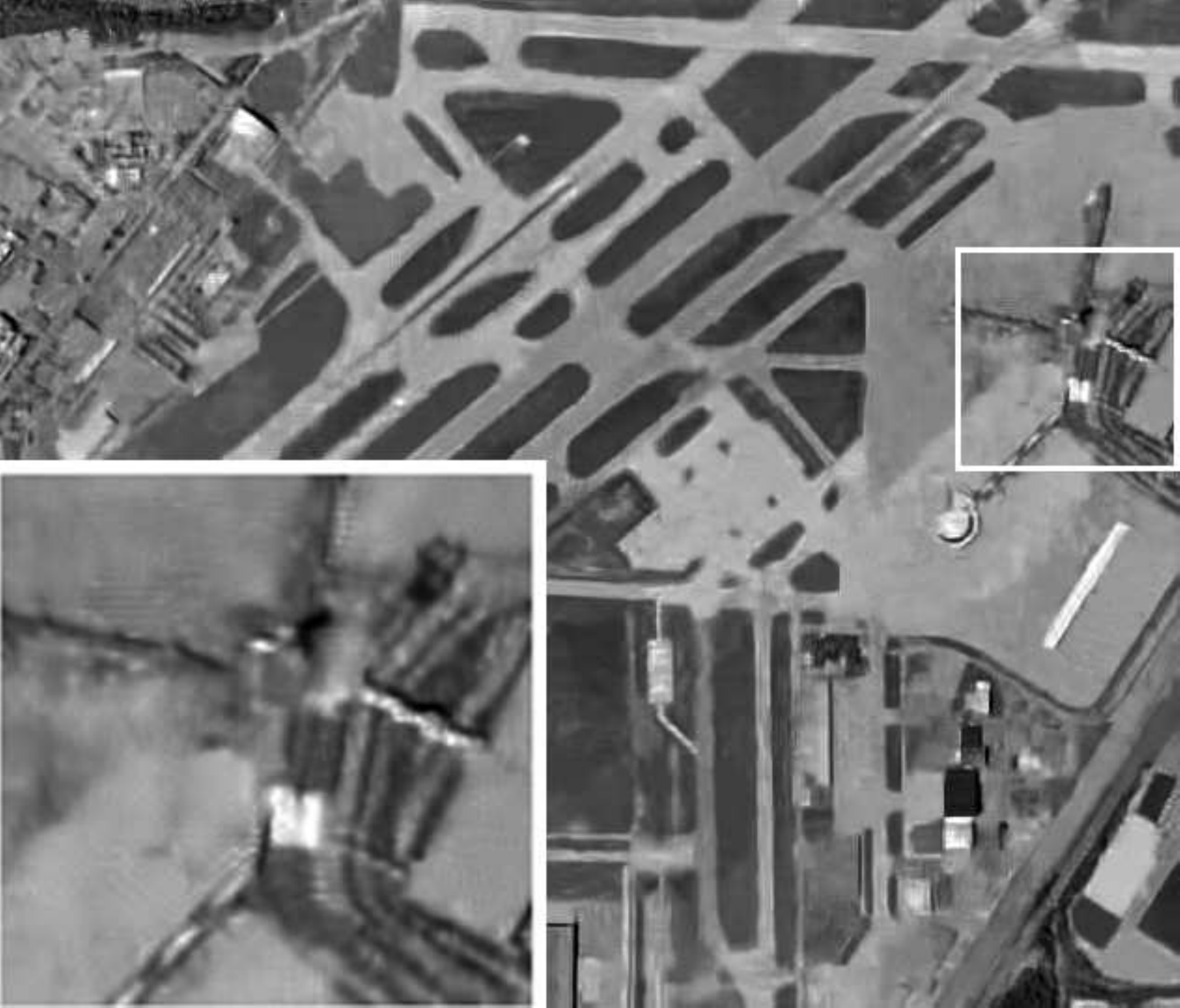}
	\caption{SAR-BM3D}
    \end{subfigure}     \\
    \begin{subfigure}[t]{.1942\textwidth}
    \includegraphics[width=3.245cm]{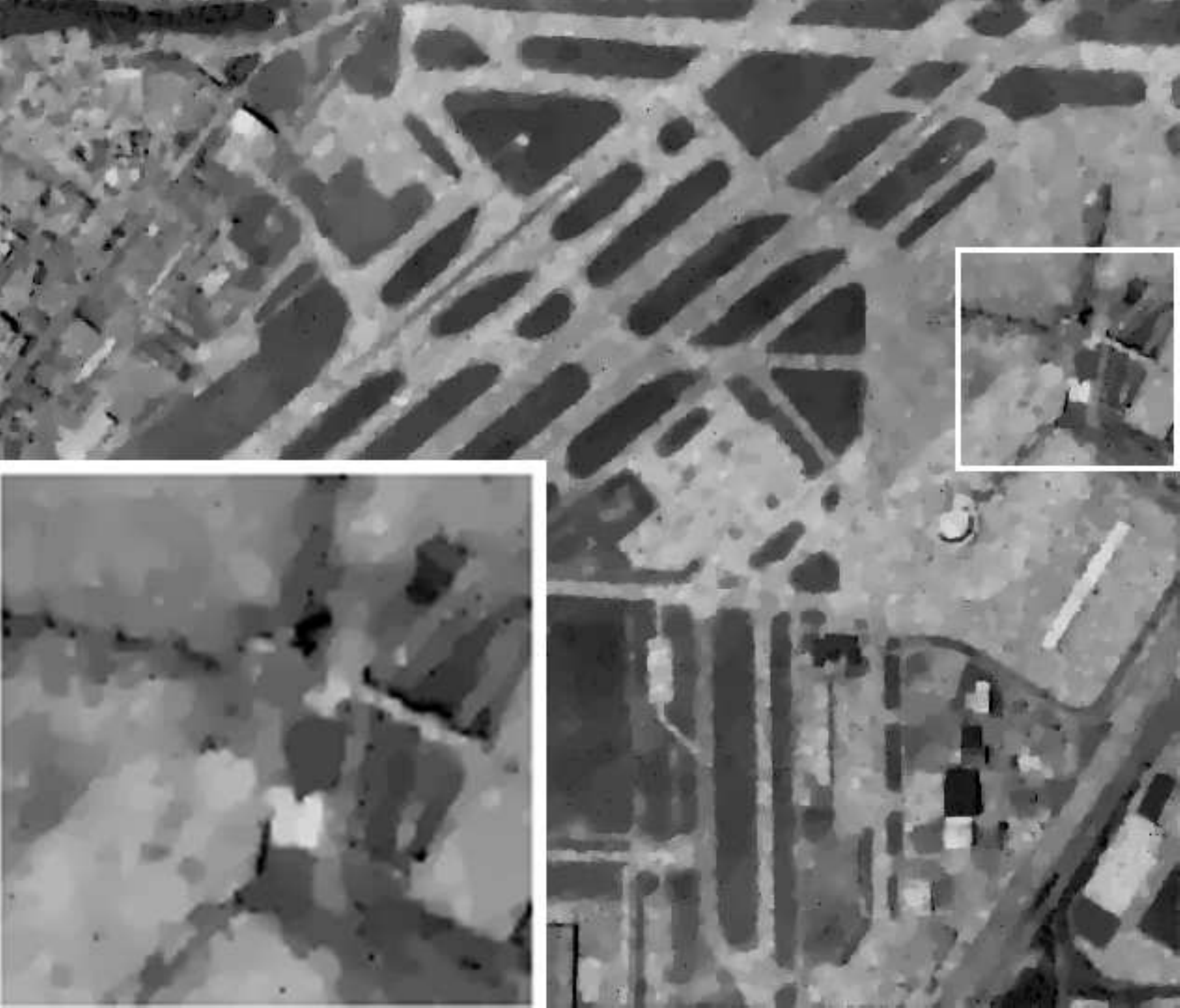}
	\caption{DZ}
    \end{subfigure}
    \begin{subfigure}[t]{.1942\textwidth}
    \includegraphics[width=3.245cm]{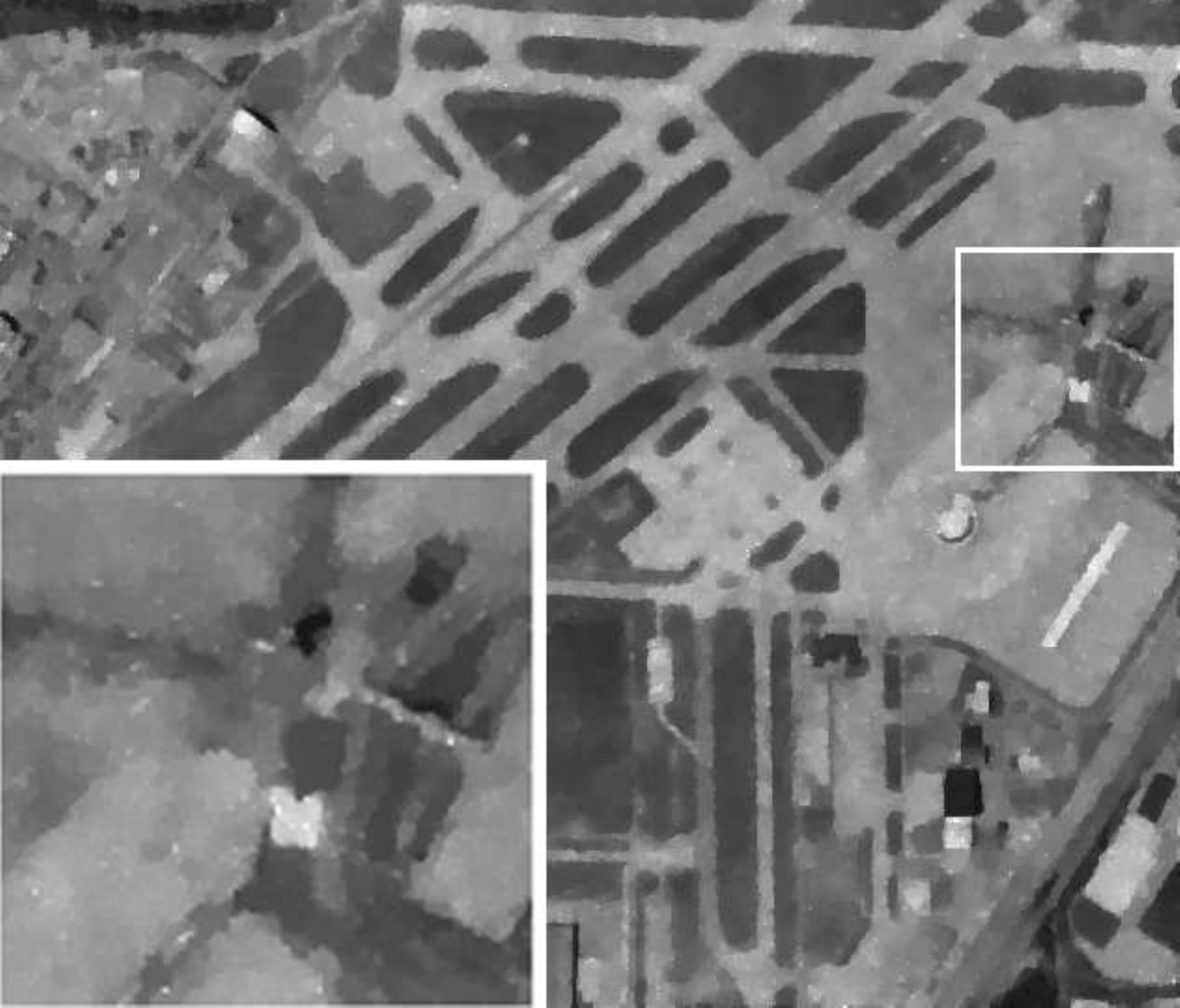}
	\caption{HNW}
    \end{subfigure}
    \begin{subfigure}[t]{.1942\textwidth}
    \includegraphics[width=3.245cm]{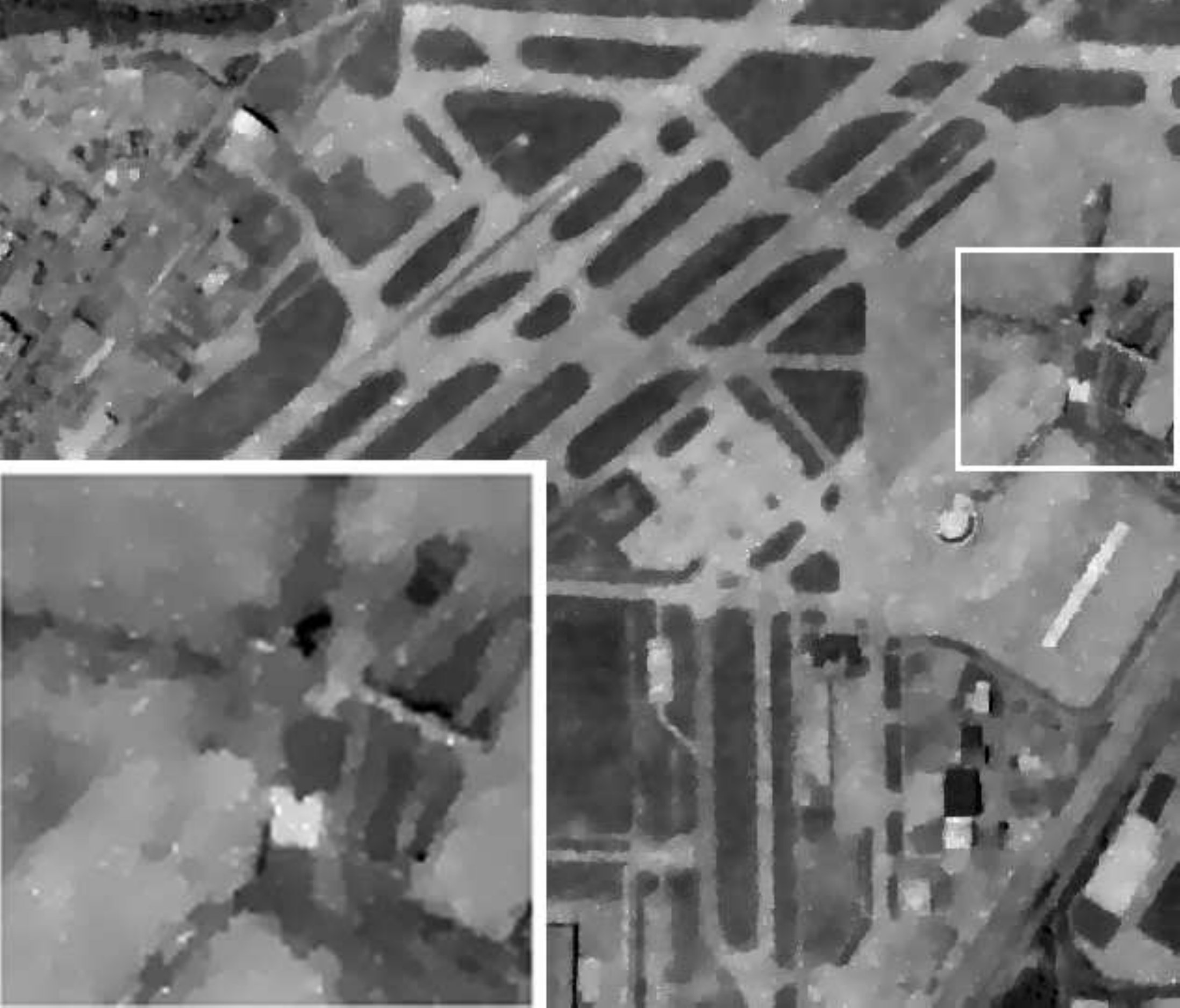}
    \caption{I-DIV}
    \end{subfigure}
    \begin{subfigure}[t]{.1942\textwidth}
    \includegraphics[width=3.245cm]{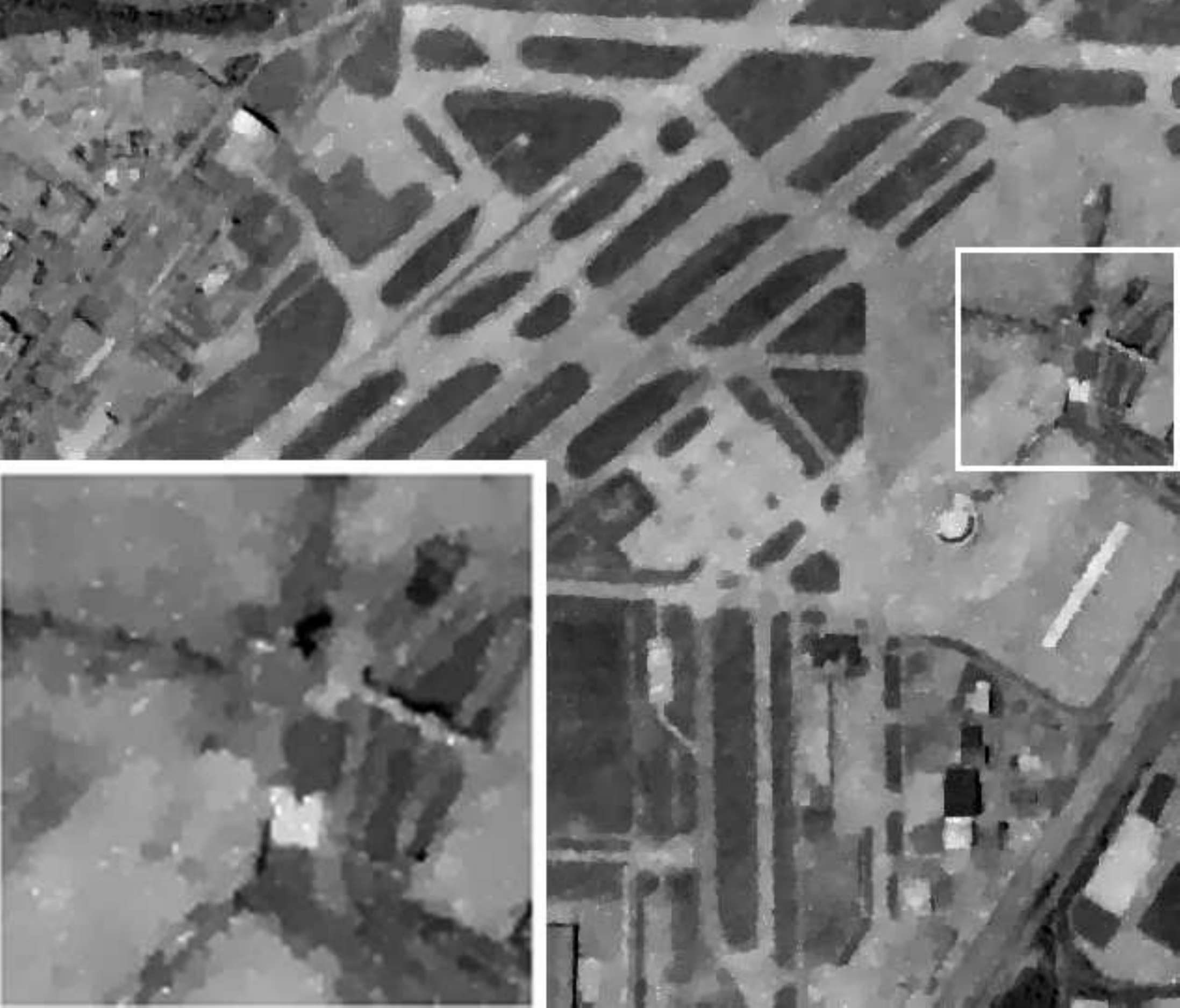}
    \caption{TwL-4V}
    \end{subfigure}
    \begin{subfigure}[t]{.1942\textwidth}
    \includegraphics[width=3.245cm]{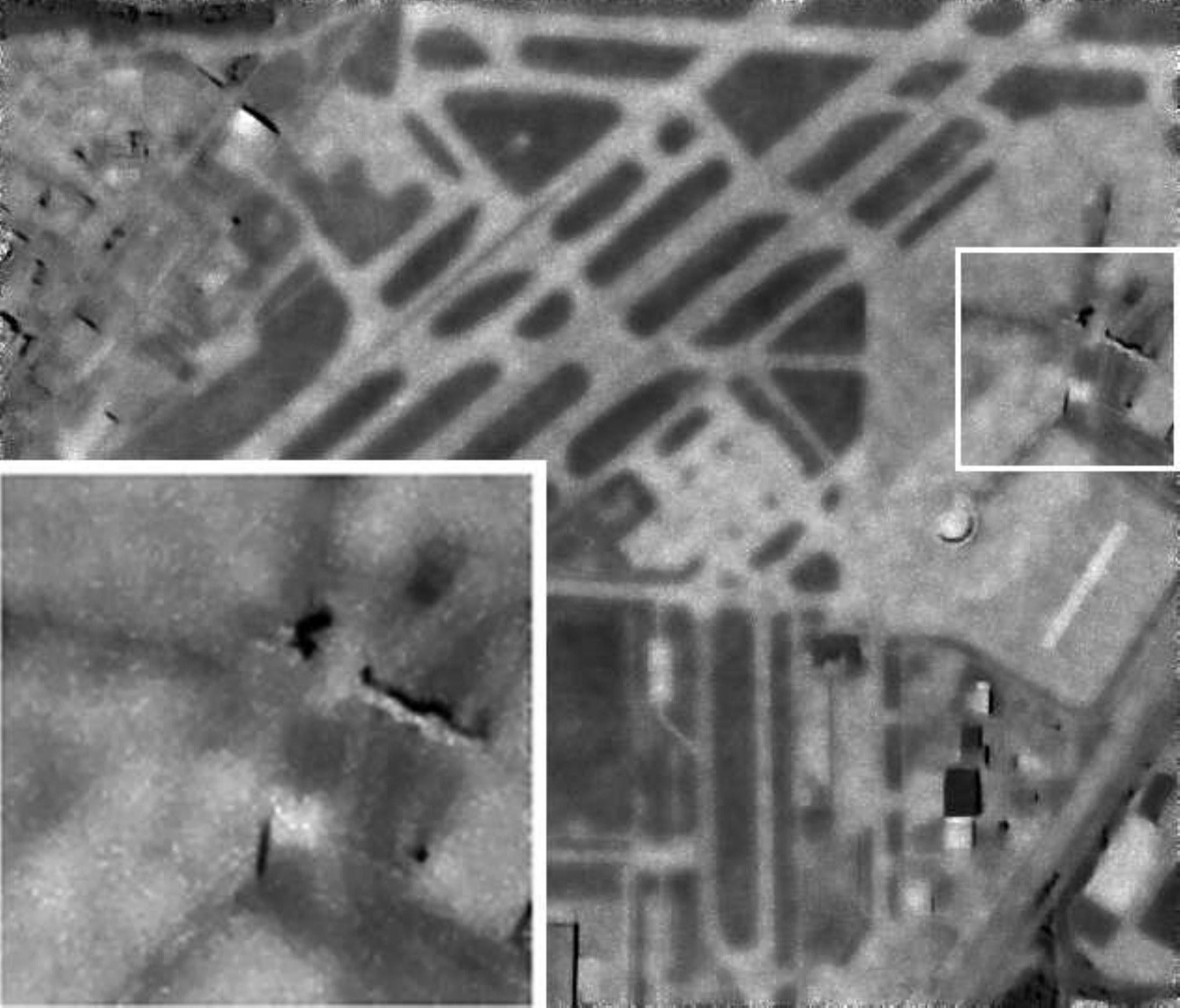}
	\caption{Dictionary}
    \end{subfigure}   	
	\caption{Comparison of denoised images restored from ``Remote 3" at noise level $L=5$ by different methods. The (PSNR, SSIM) values for each denoised image: (c) Alg 1 (25.80dB, 0.7427); (d) Alg 2 (25.83dB, 0.7460); (e) SAR-BM3D (25.65dB, 0.7316); (f) DZ (24.45dB, 0.6745); (g) HNW (23.79dB, 0.6562); (h) I-DIV (24.10dB, 0.6713); (i) TwL-4V (24.30dB, 0.6737); (j) Dictionary (22.70dB, 0.5736).}\label{fig:Remote3}
\end{figure}

Table~\ref{Table:Remote} reports the PSNR and SSIM values of the  denoised images tested on three remote sensing images. Algorithm~\ref{Alg:Theoretical} and Algorithm~\ref{Alg:Practical}  achieve great performance in PSNR and SSIM values over other methods. For example, Algorithm~\ref{Alg:Theoretical} outperforms the benchmark SAR-BM3D method by 0.11-0.28dB,  0.06-0.19dB  and 0.06-0.15dB in PSNR when $L=1$, $L=3$ and $L=5$, respectively; and it outperforms the other traditional methods by 0.76-1.57dB,  0.94-2.93dB  and 0.86-3.62dB in PSNR when $L=1$, $L=3$ and $L=5$, respectively. Algorithm~\ref{Alg:Practical} is also comparable to Algorithm~\ref{Alg:Theoretical} and the SAR-BM3D method.

Figure~\ref{fig:Remote1}-\ref{fig:Remote3} present the denoised images by different methods tested on ``Remote~1" at noise level $L=1$, ``Remote 2" at $L=3$ and ``Remote 3" at $L=5$. Algorithm~\ref{Alg:Theoretical}, Algorithm~\ref{Alg:Practical}  and the benchmark SAR-BM3D method achieve significantly better visual quality over other methods. For example, they reconstruct buildings, roads and patterns with fine edges and textures.

\subsection{Numerical results tested on real SAR images}

In this experiment, we use real SAR images images ``SAR 1" of size  $370\times 370$ and ``SAR 2" of size $350\times 350$ as shown in Figure~\ref{fig:SAR1}(a) and  Figure~\ref{fig:SAR2}(a), respectively.
\begin{figure}[htbp]
	\centering
	\begin{subfigure}[t]{.1942\textwidth}
	\includegraphics[width=3.245cm]{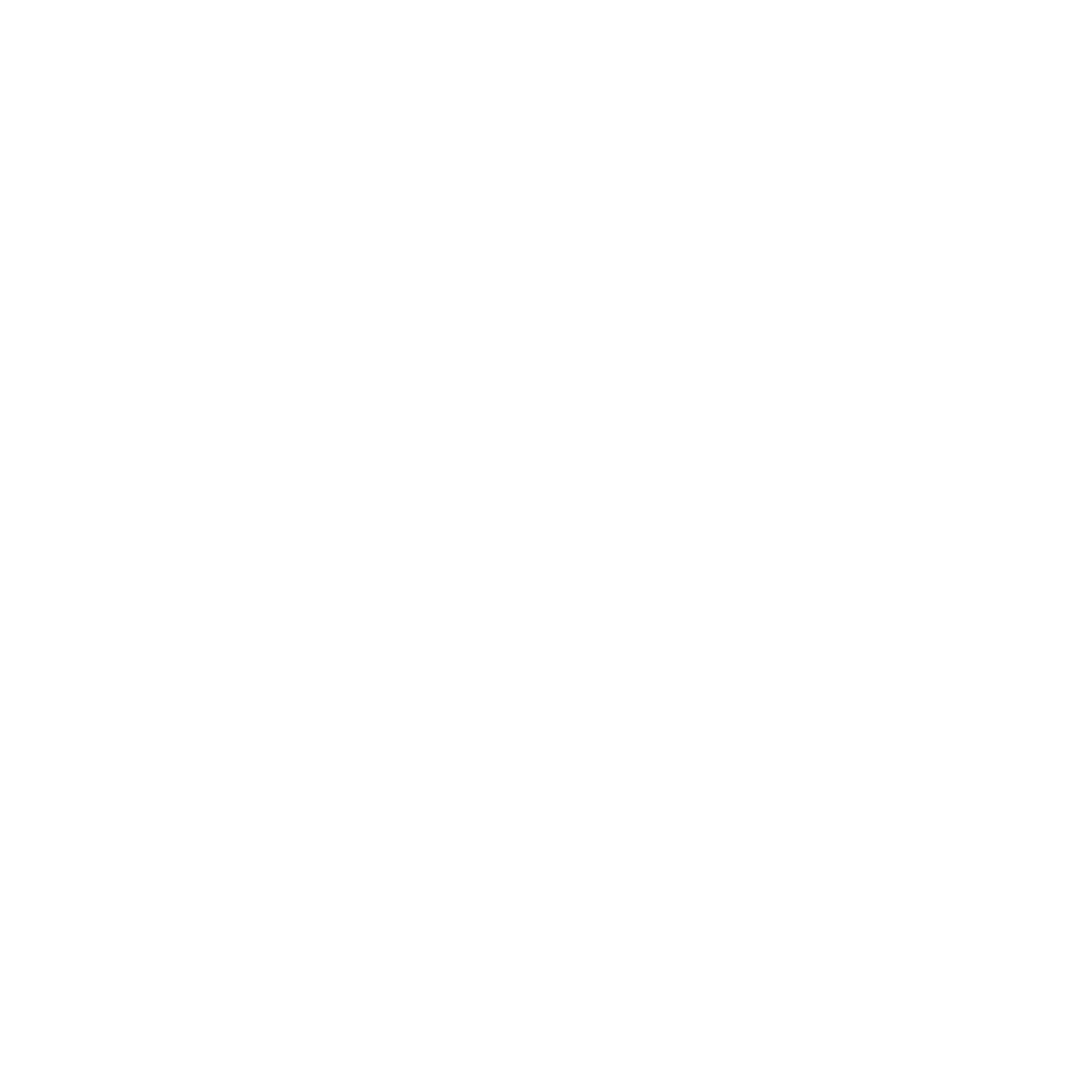}
	\end{subfigure}
   	\begin{subfigure}[t]{.1942\textwidth}
    \includegraphics[width=3.245cm]{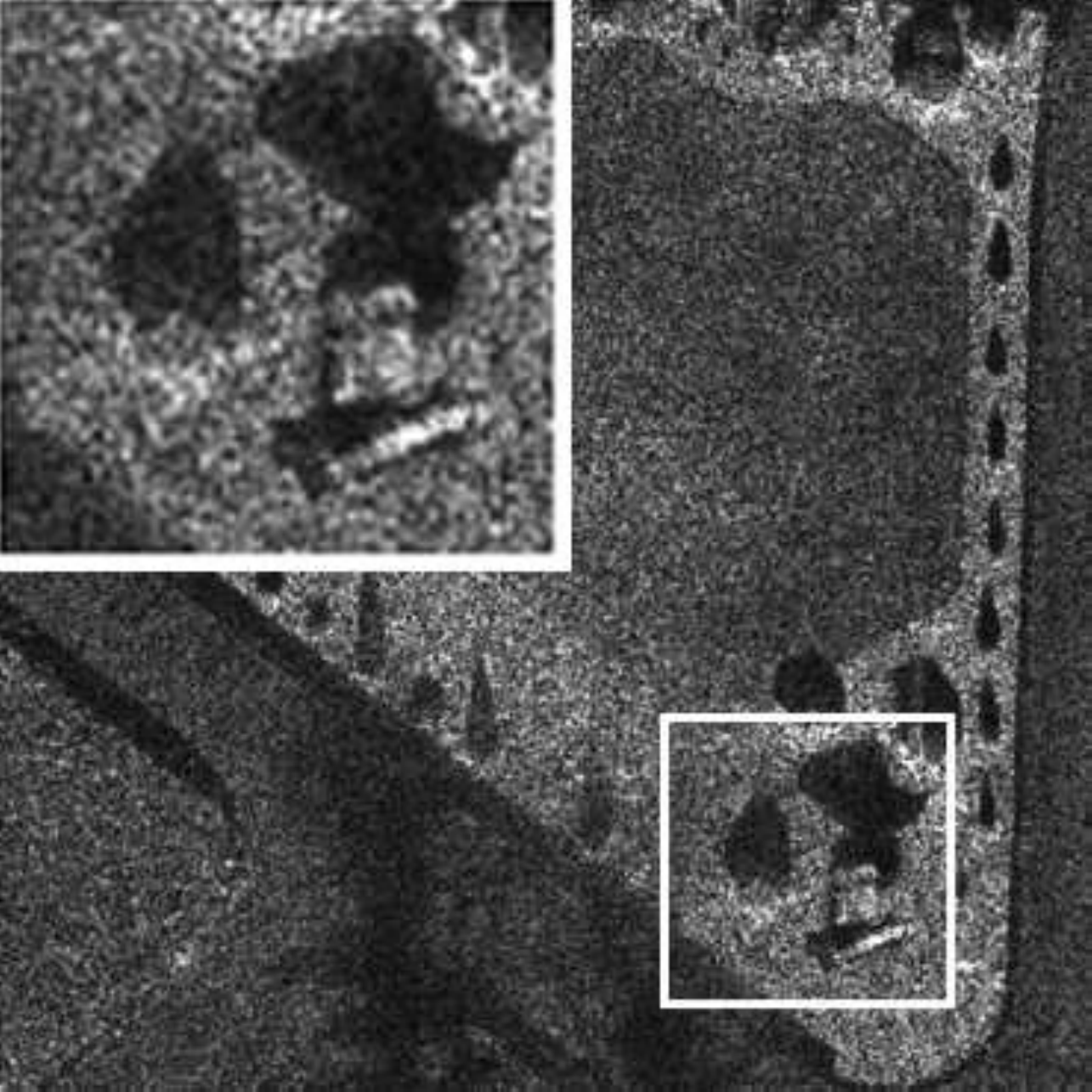}
    \caption{SAR 1}
    \end{subfigure}
    \begin{subfigure}[t]{.1942\textwidth}
    \includegraphics[width=3.245cm]{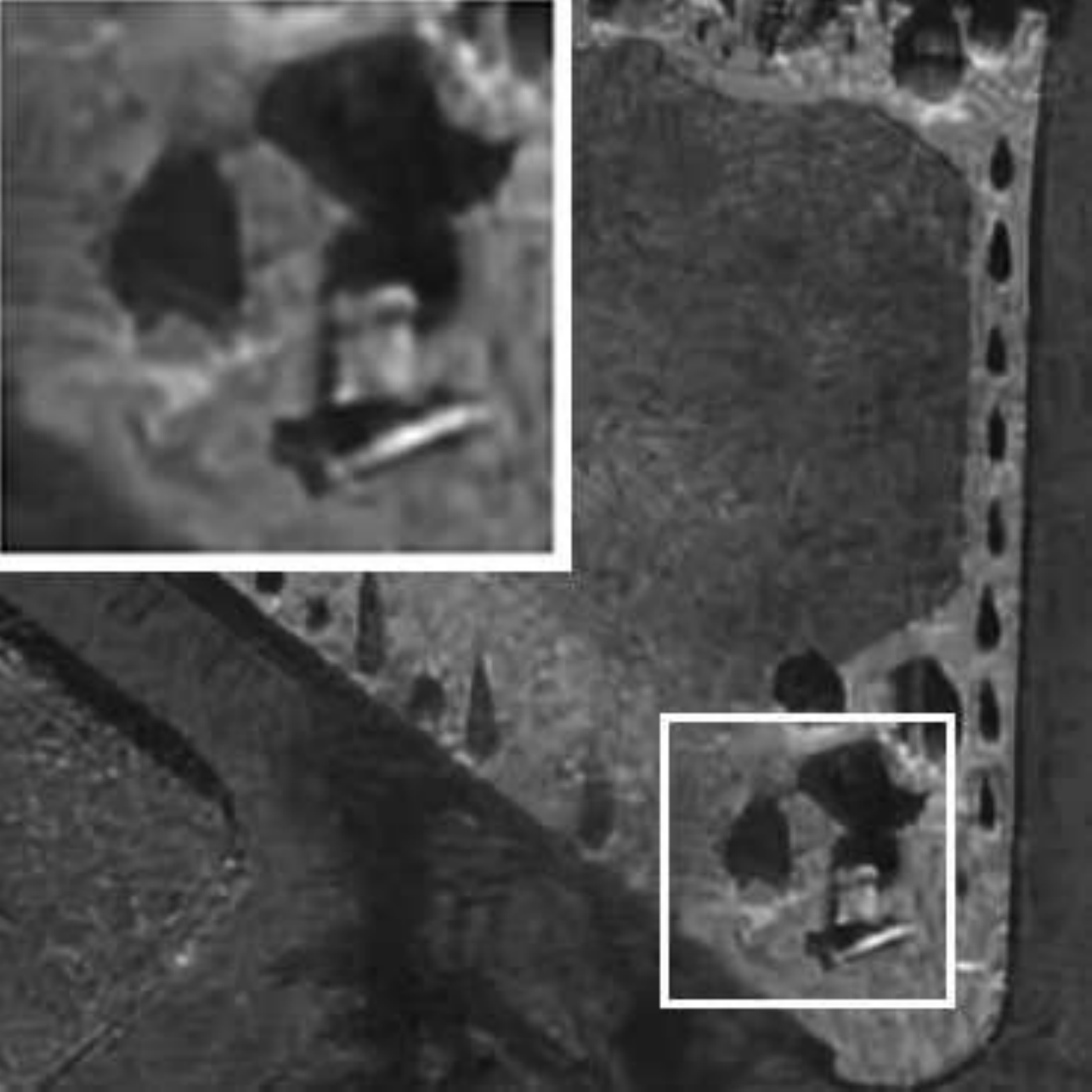}
    \caption{Alg 1}
    \end{subfigure}
    \begin{subfigure}[t]{.1942\textwidth}
    \includegraphics[width=3.245cm]{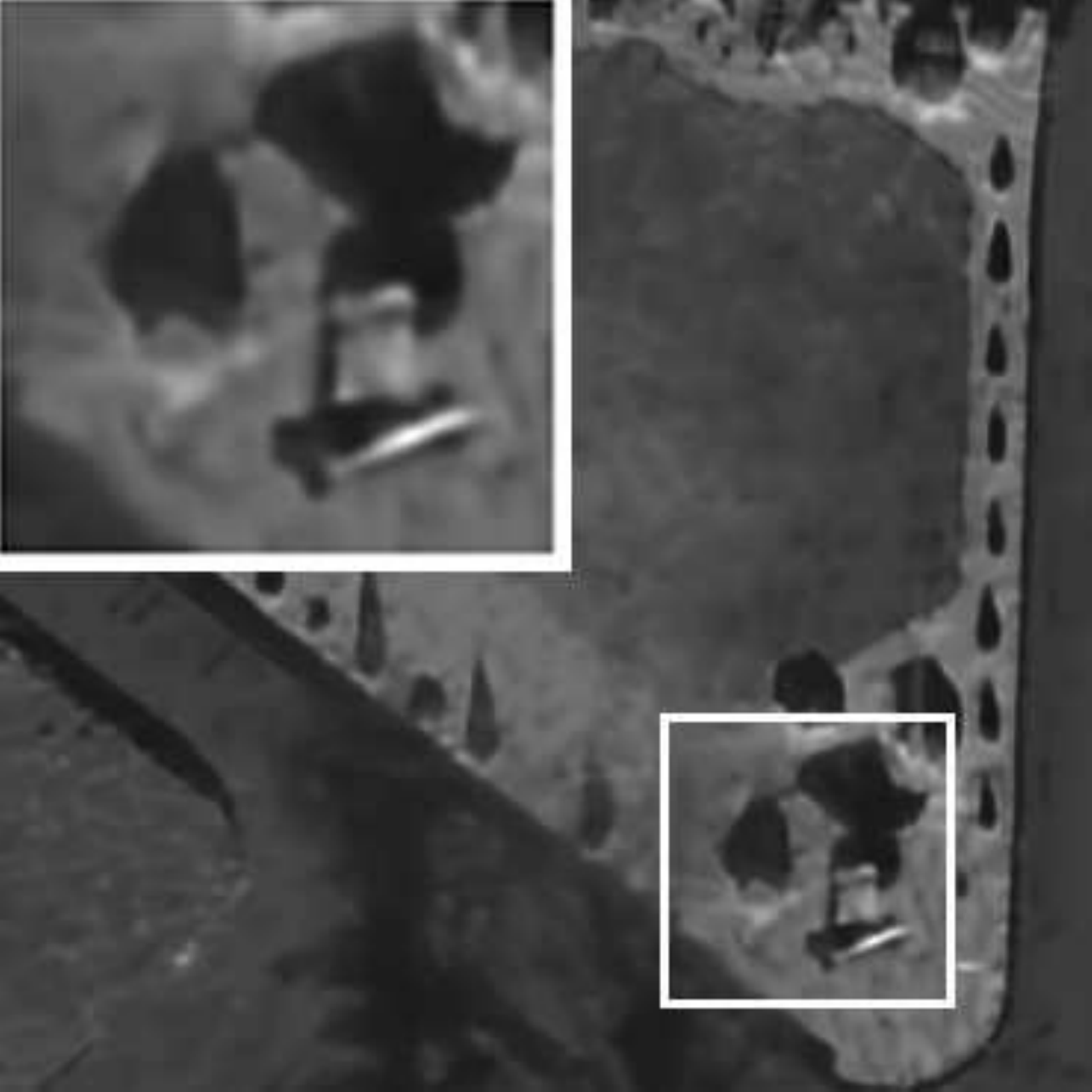}
    \caption{Alg 2}
    \end{subfigure}
    \begin{subfigure}[t]{.1942\textwidth}
    \includegraphics[width=3.245cm]{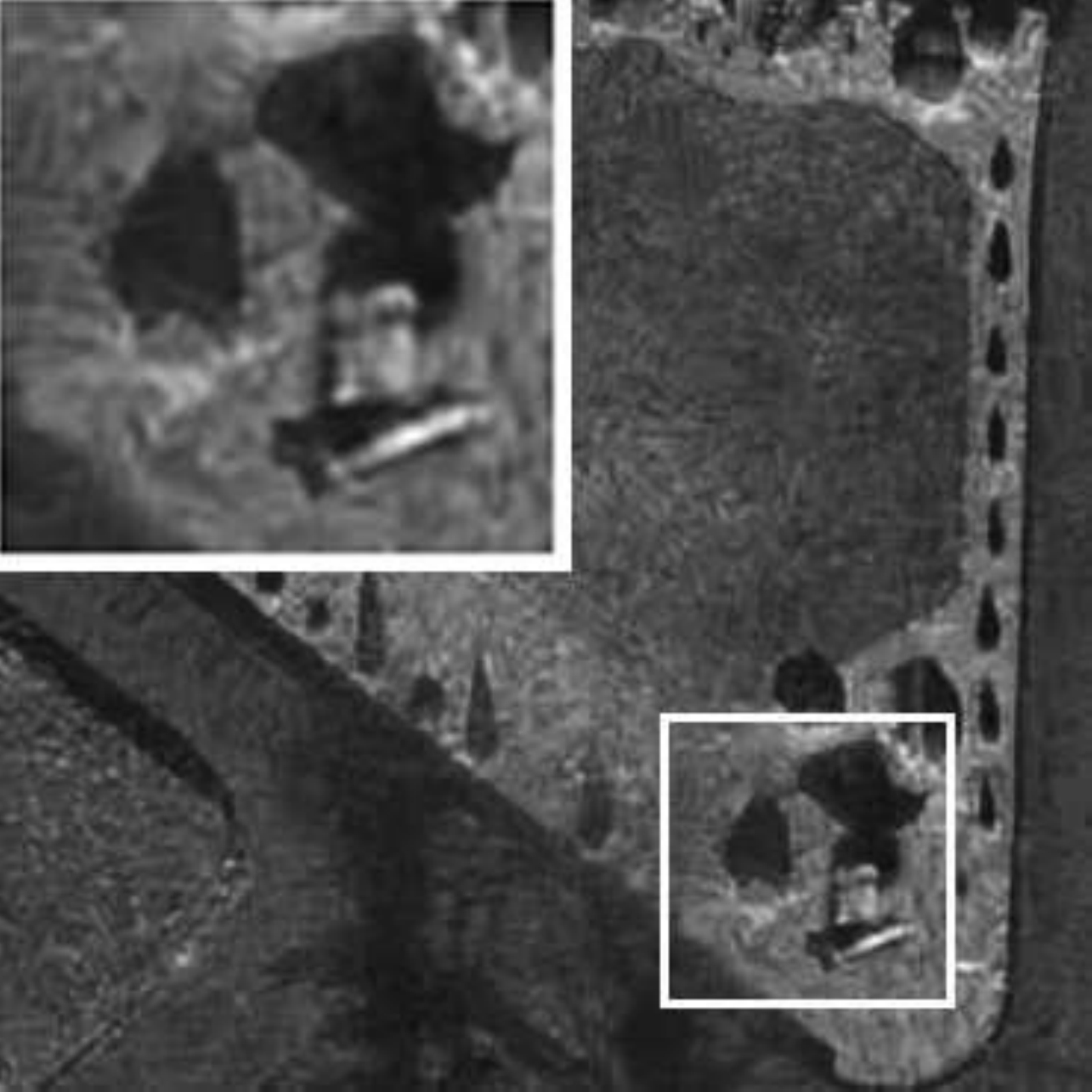}
	\caption{SAR-BM3D}
    \end{subfigure}    \\
    \begin{subfigure}[t]{.1942\textwidth}
    \includegraphics[width=3.245cm]{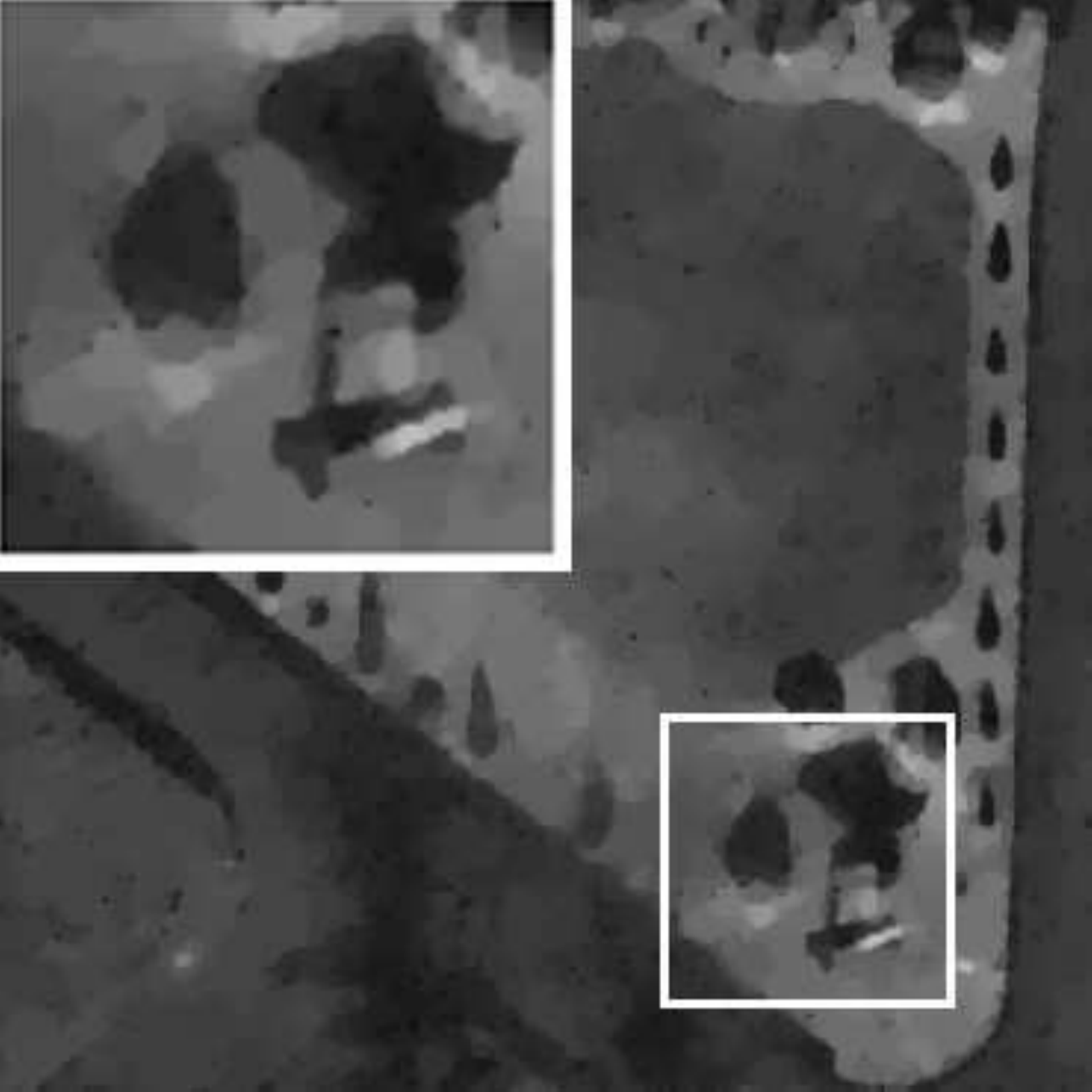}
	\caption{DZ}
    \end{subfigure}	
    \begin{subfigure}[t]{.1942\textwidth}
    \includegraphics[width=3.245cm]{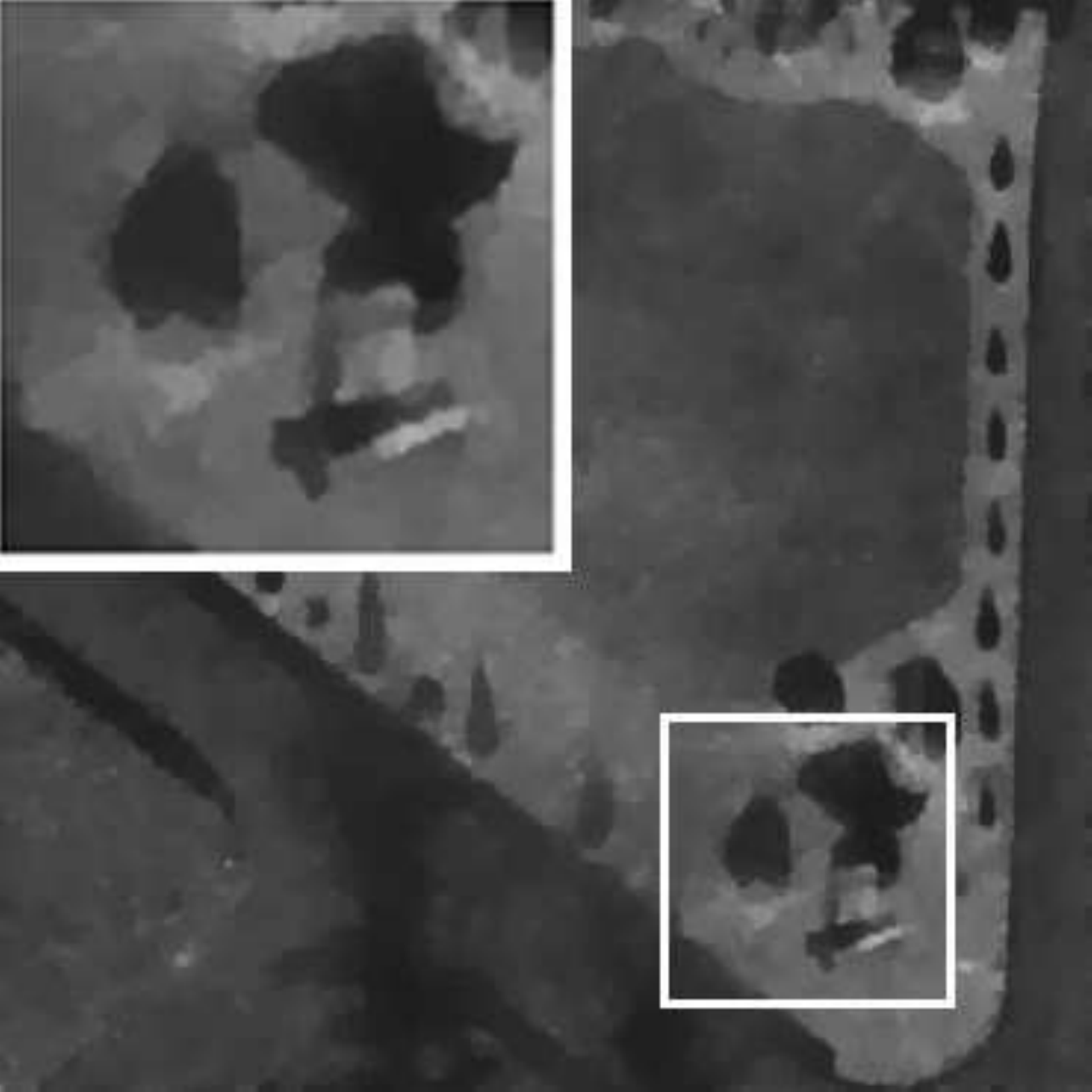}
	\caption{HNW}
    \end{subfigure}
    \begin{subfigure}[t]{.1942\textwidth}
    \includegraphics[width=3.245cm]{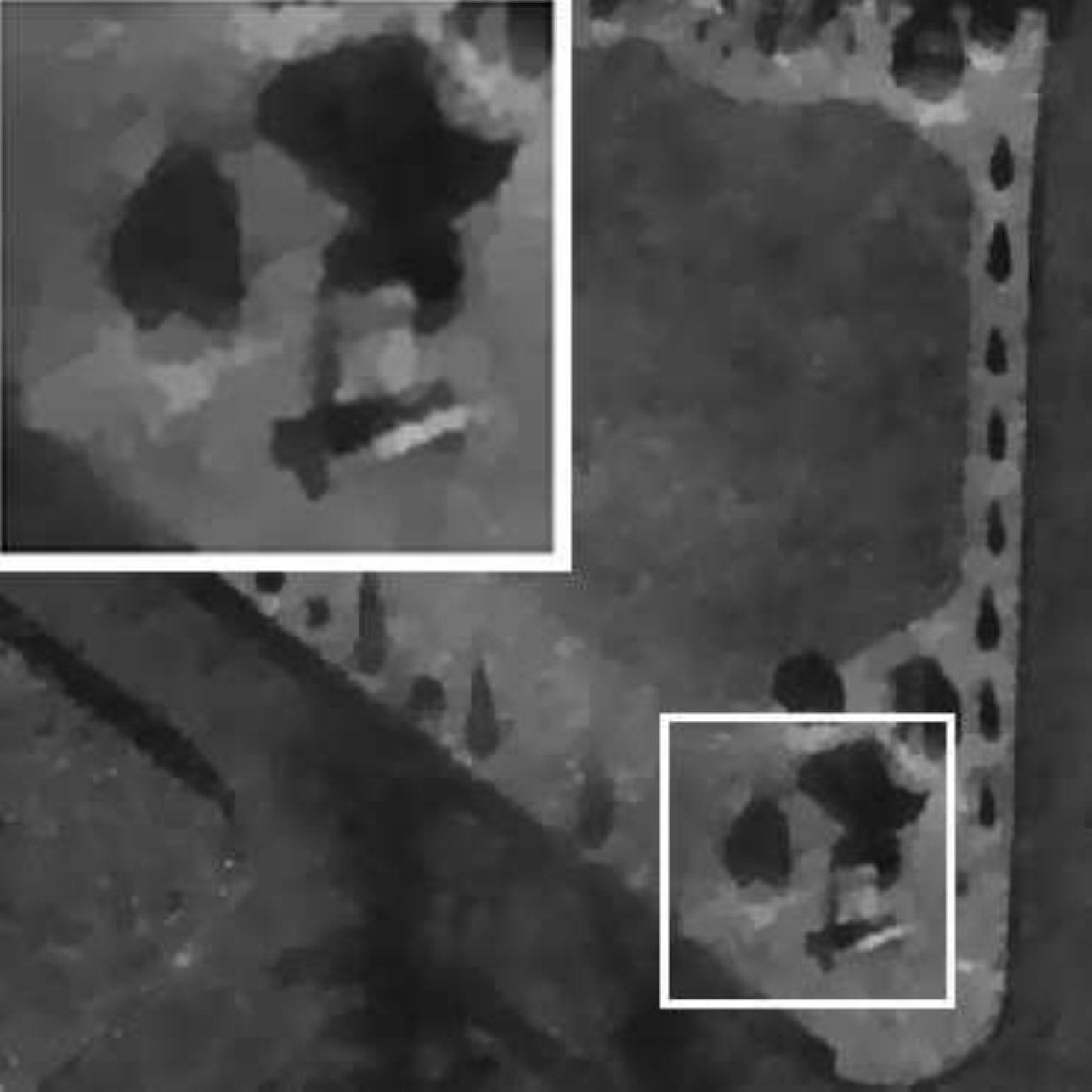}
    \caption{I-DIV}
    \end{subfigure}
    \begin{subfigure}[t]{.1942\textwidth}
    \includegraphics[width=3.245cm]{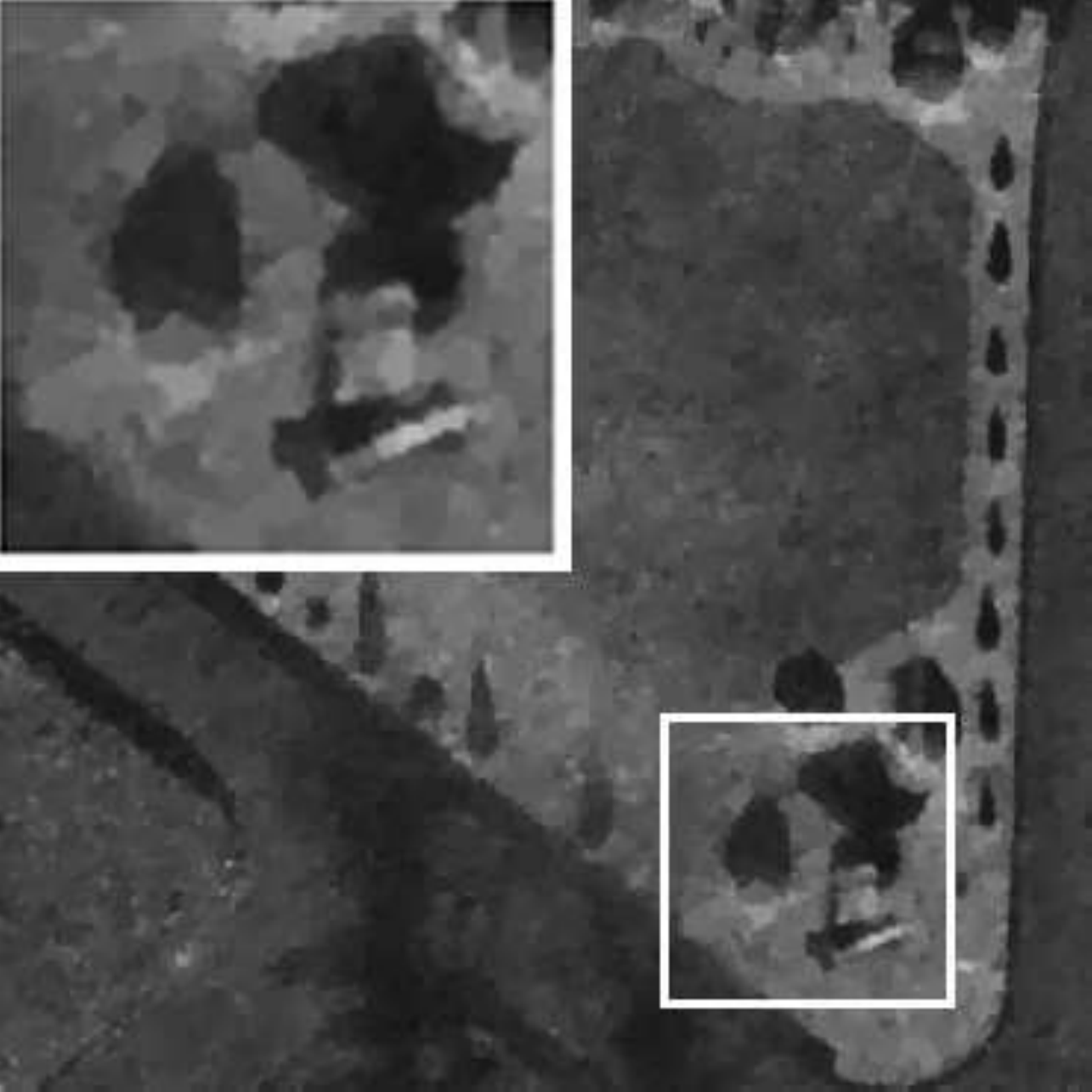}
    \caption{TwL-4V}
    \end{subfigure}
    \begin{subfigure}[t]{.1942\textwidth}
    \includegraphics[width=3.245cm]{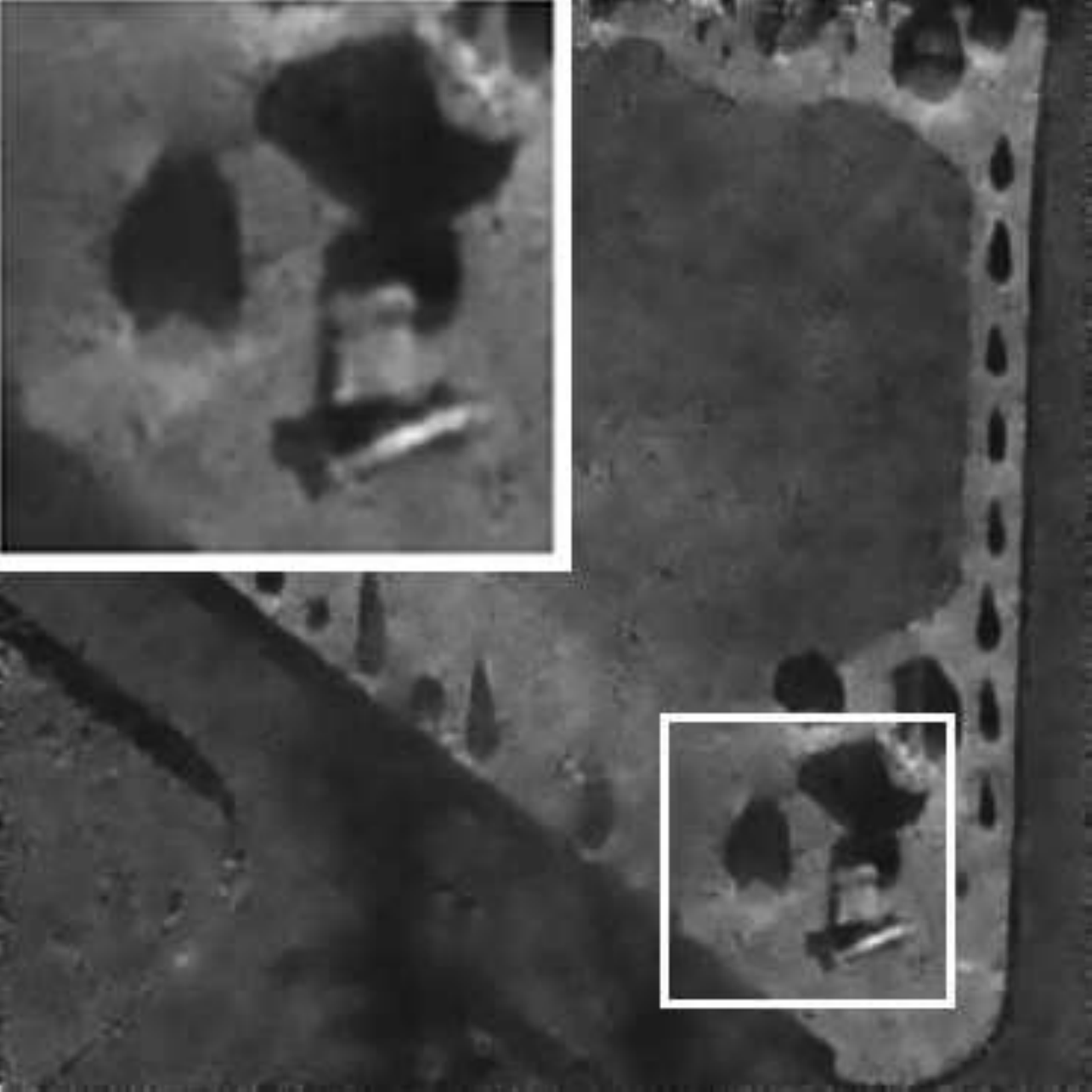}
	\caption{Dictionary}
    \end{subfigure}   	
	\caption{Comparison of denoised images restored from ``SAR 1" by different methods.}\label{fig:SAR1}
\end{figure}

\begin{figure}[htbp]
	\centering
	\begin{subfigure}[t]{.1942\textwidth}
    \includegraphics[width=3.245cm]{empty.pdf}
    \end{subfigure}
	\begin{subfigure}[t]{.1942\textwidth}
    \includegraphics[width=3.245cm]{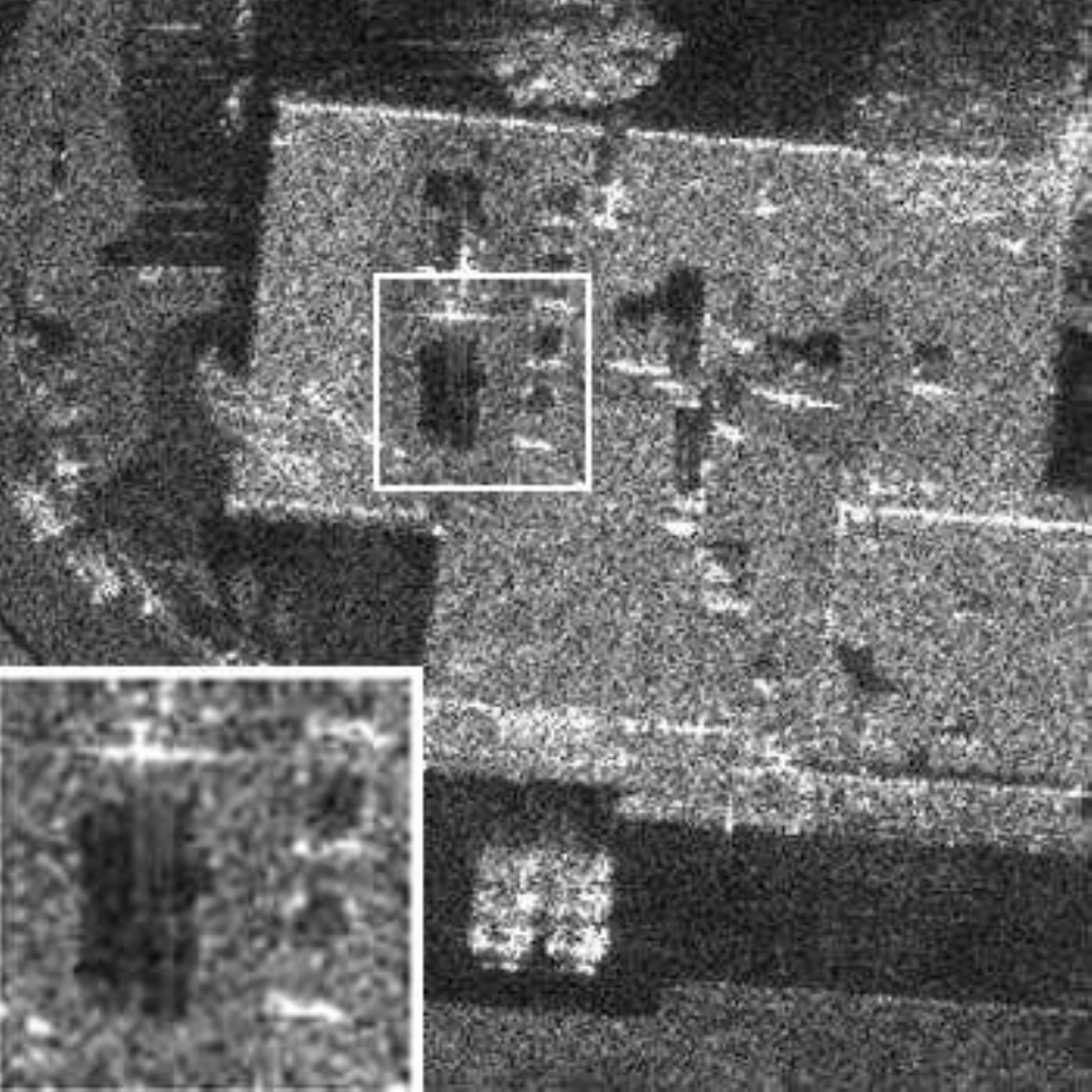}
    \caption{SAR 2}
    \end{subfigure}
    \begin{subfigure}[t]{.1942\textwidth}
    \includegraphics[width=3.245cm]{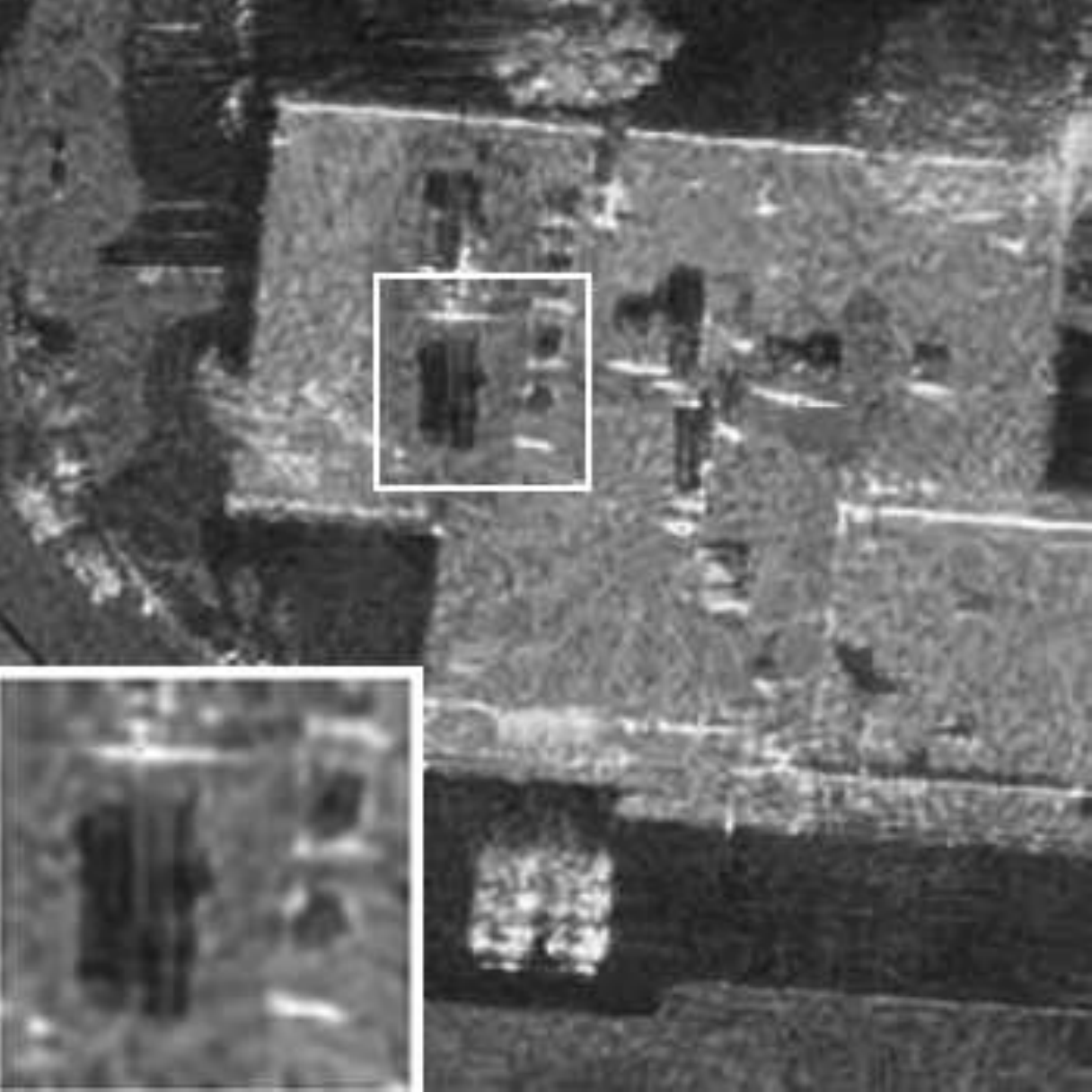}
    \caption{Alg 1}
    \end{subfigure}
    \begin{subfigure}[t]{.1942\textwidth}
    \includegraphics[width=3.245cm]{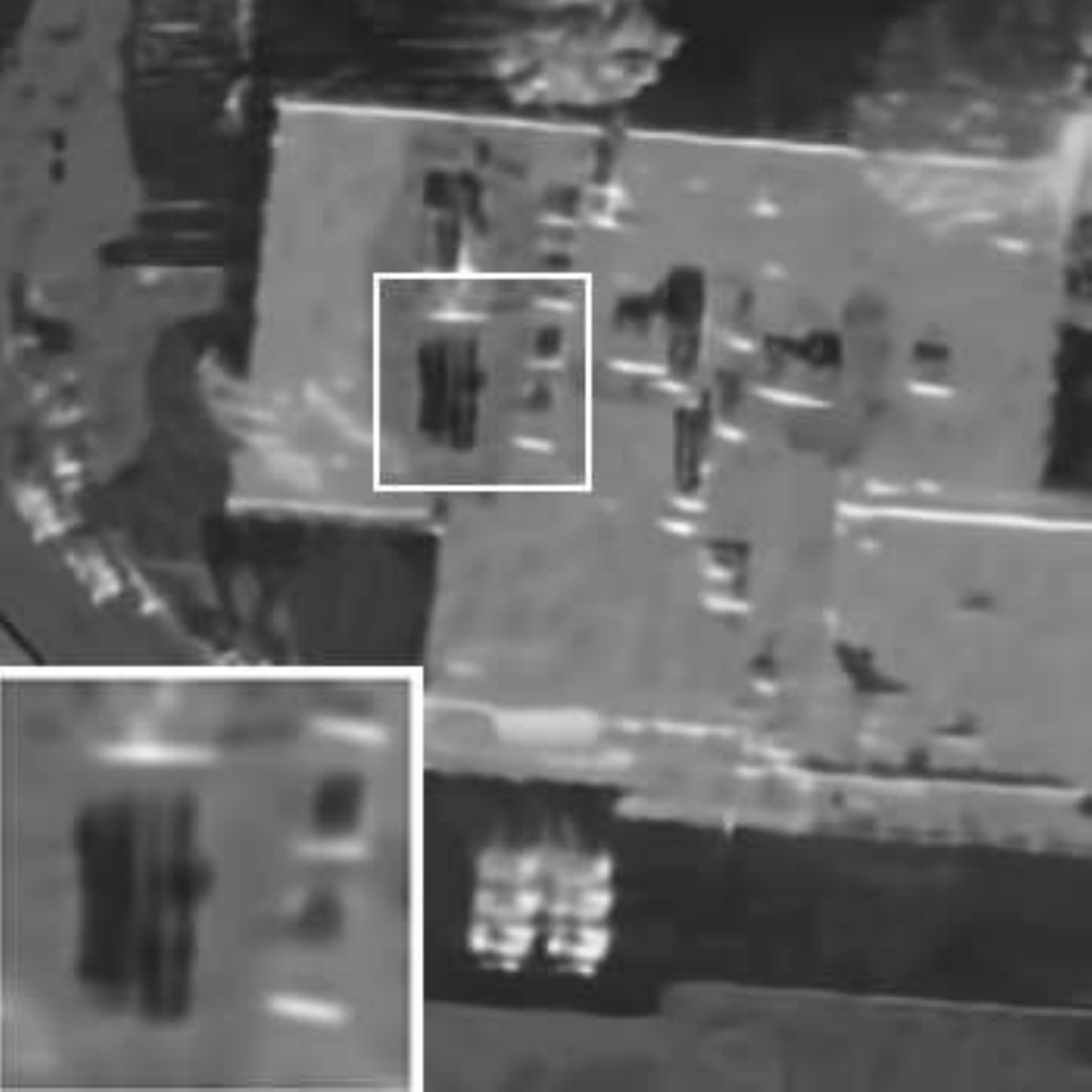}
    \caption{Alg 2}
    \end{subfigure}
    \begin{subfigure}[t]{.1942\textwidth}
    \includegraphics[width=3.245cm]{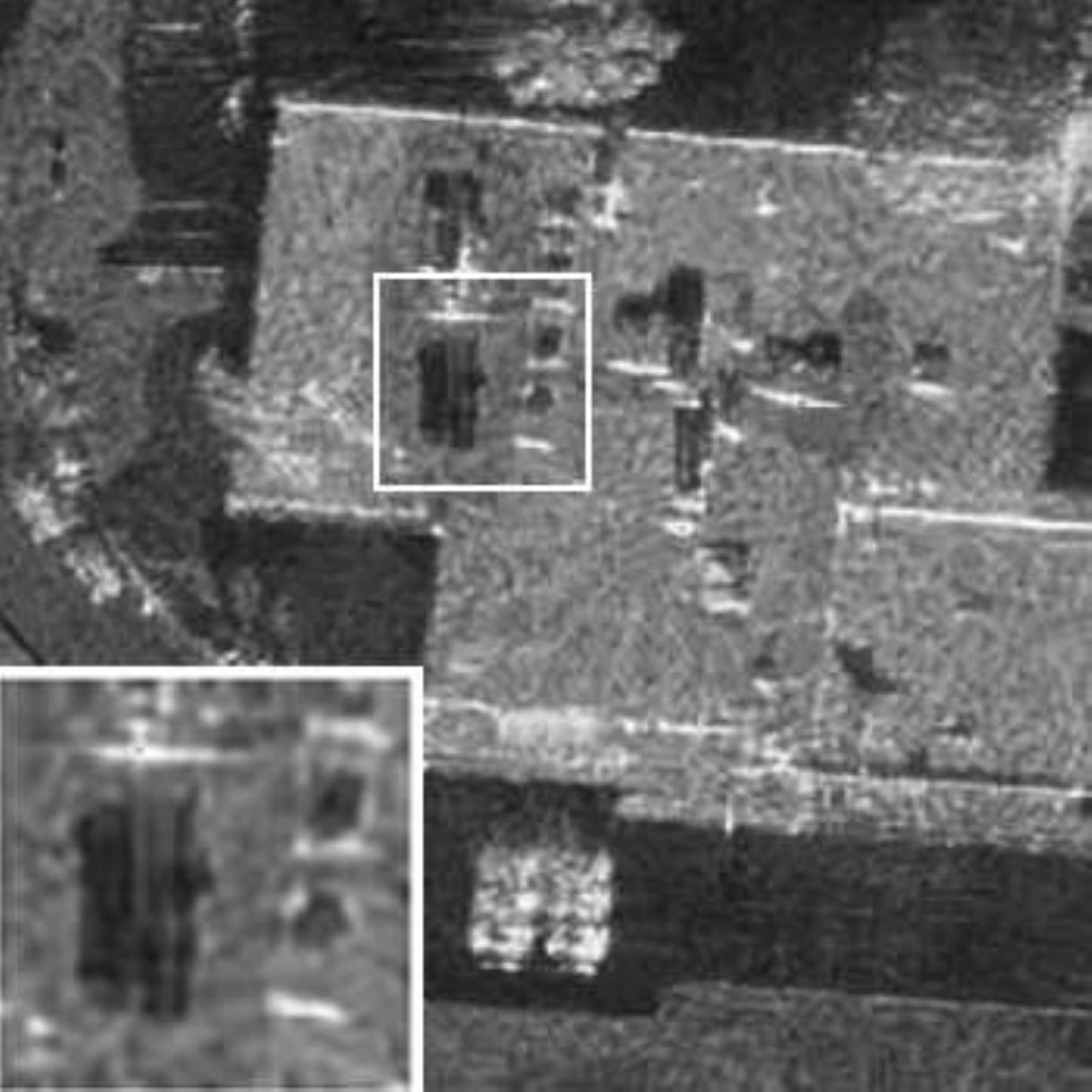}
	\caption{SAR-BM3D}
    \end{subfigure}    \\
    \begin{subfigure}[t]{.1942\textwidth}
    \includegraphics[width=3.245cm]{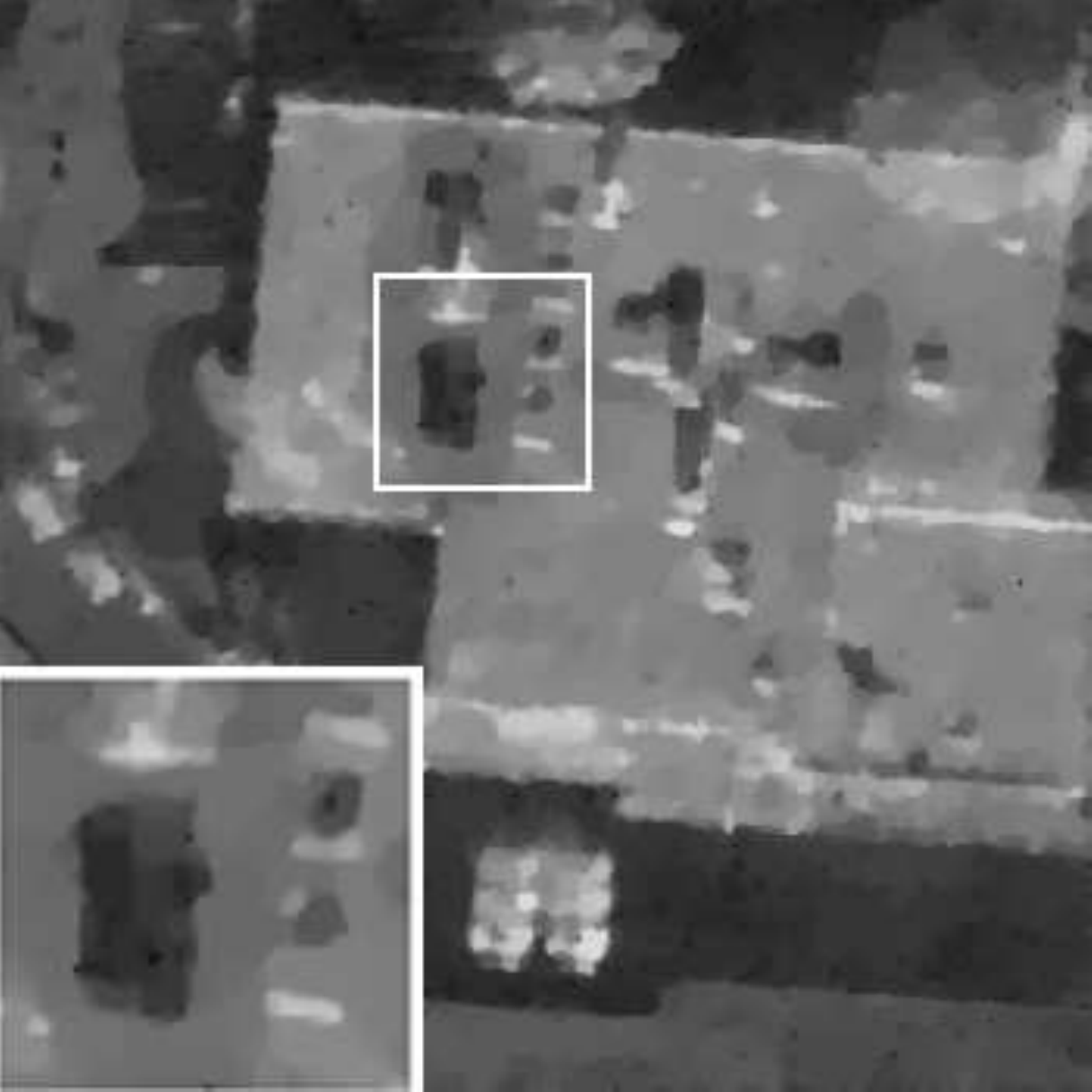}
	\caption{DZ}
    \end{subfigure}
    \begin{subfigure}[t]{.1942\textwidth}
    \includegraphics[width=3.245cm]{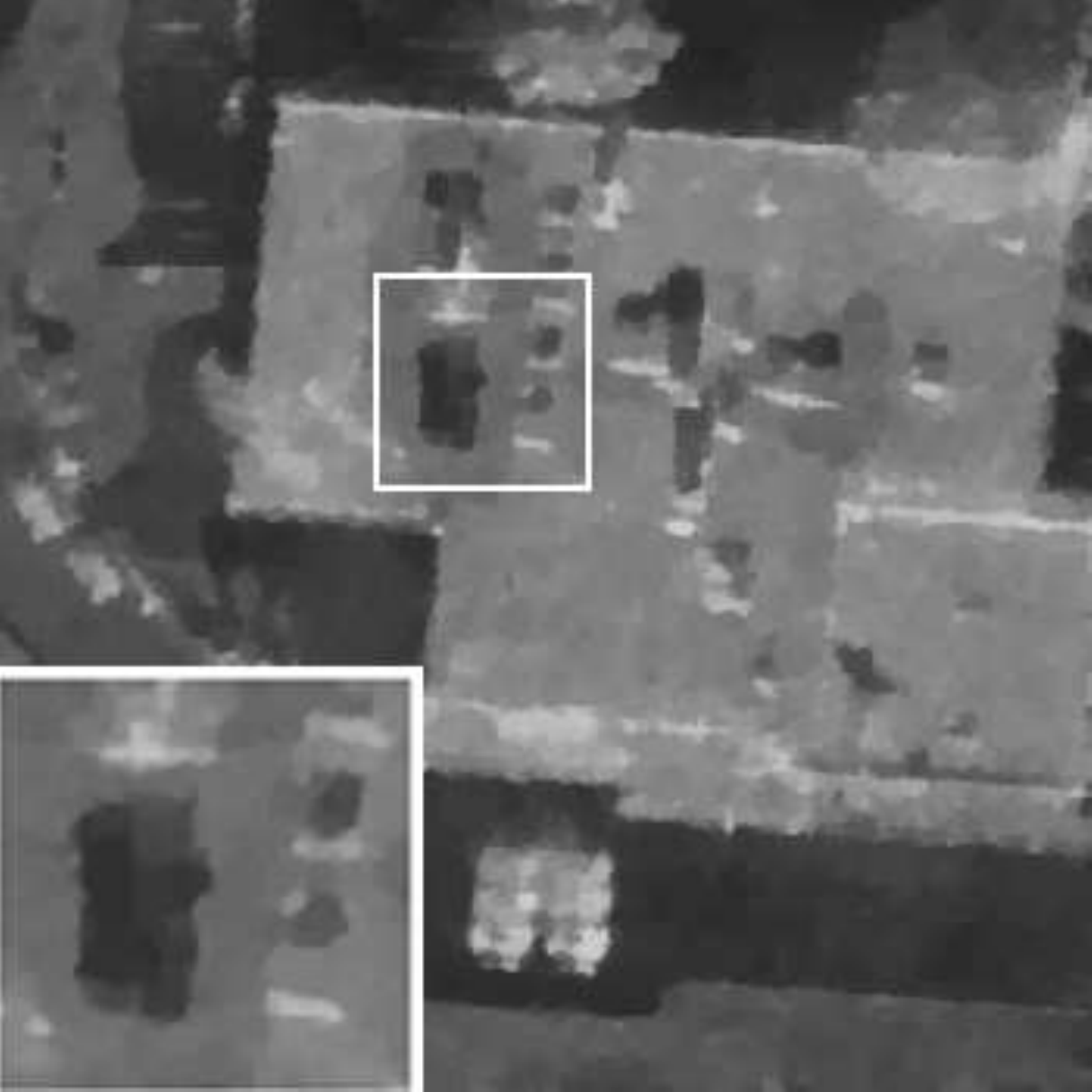}
	\caption{HNW}
    \end{subfigure}
    \begin{subfigure}[t]{.1942\textwidth}
    \includegraphics[width=3.245cm]{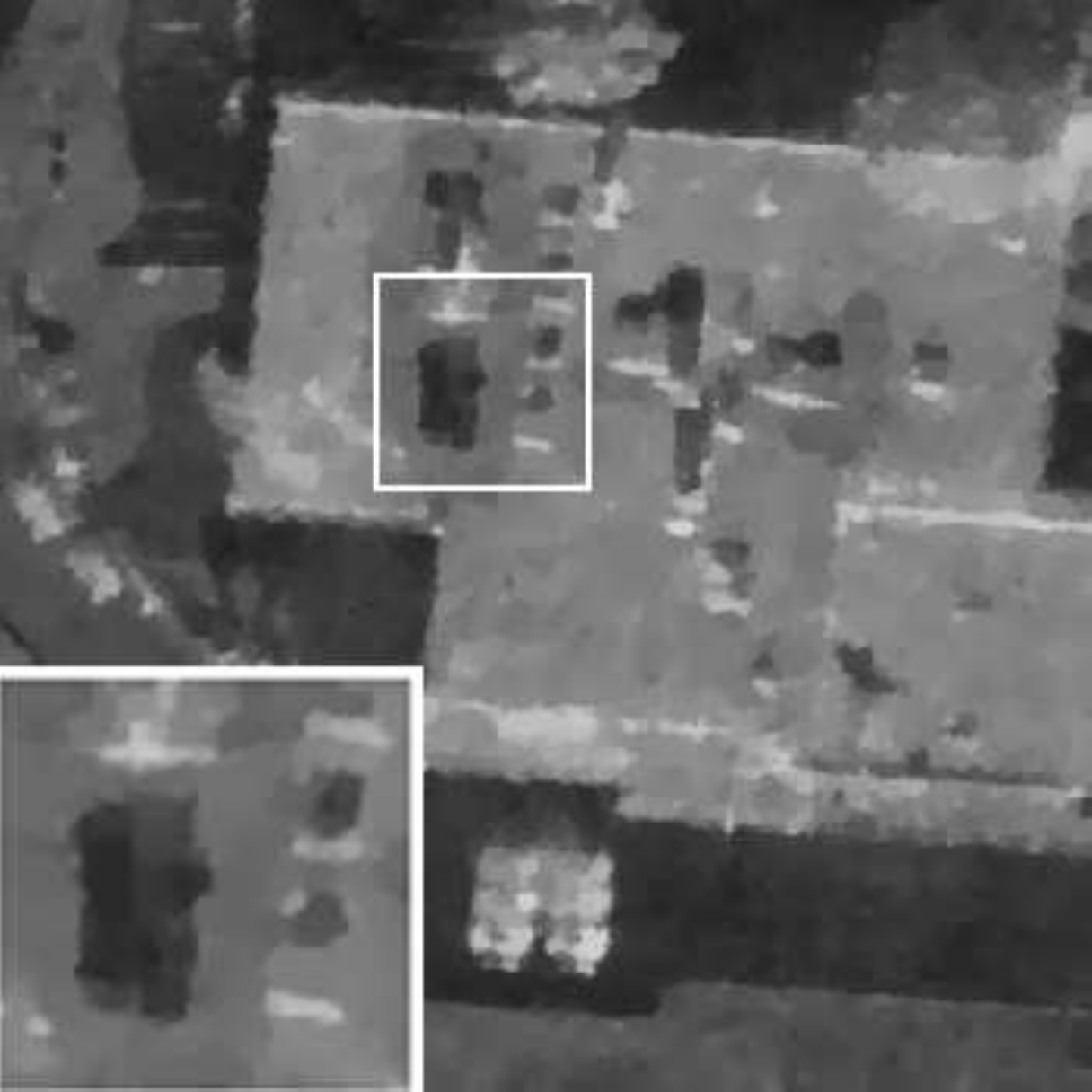}
    \caption{I-DIV}
    \end{subfigure}
    \begin{subfigure}[t]{.1942\textwidth}
    \includegraphics[width=3.245cm]{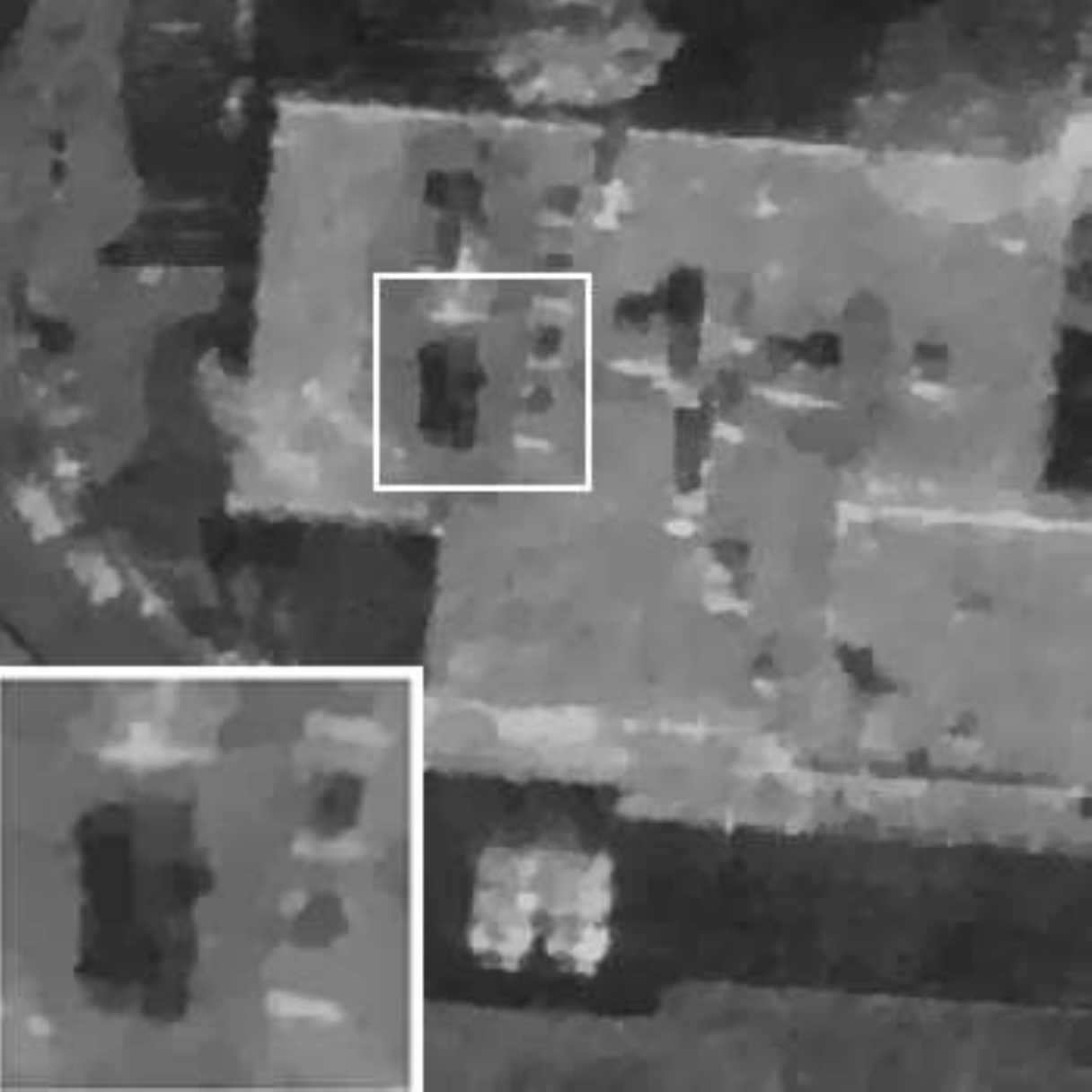}
    \caption{TwL-4V}
    \end{subfigure}
    \begin{subfigure}[t]{.1942\textwidth}
    \includegraphics[width=3.245cm]{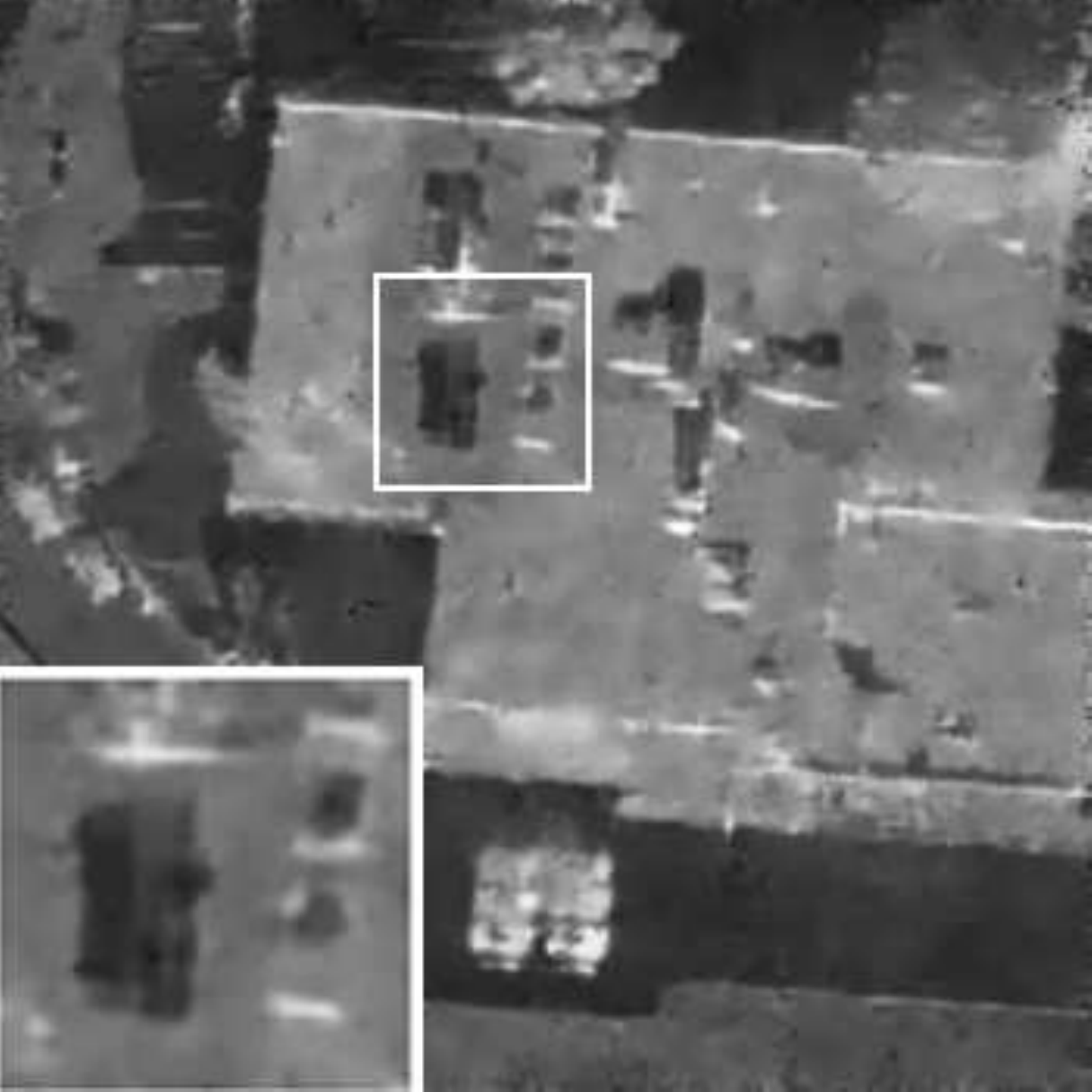}
	\caption{Dictionary}
    \end{subfigure}   	

	\caption{Comparison of denoised images restored from ``SAR 2" by different methods.}\label{fig:SAR2}
\end{figure}

Figure~\ref{fig:SAR1} and Figure~\ref{fig:SAR2} demonstrate that Algorithm~\ref{Alg:Theoretical} and Algorithm~\ref{Alg:Practical}  achieve better denoising performance than other methods. For example, they reconstruct  more local structures and smooth textures than the DZ method, the HNW method, the I-DIV method, the TwL-4V method and  the learned dictionary method; and they remove more noise and generate fewer artifacts than the benchmark SAR-BM3D method.

In addition to the visual quality comparison on the denoised images, we can also receive guidance by computing the equivalent number of looks (ENL) and analyzing the ratio images for different methods.

The ENL of an estimated image $\hat{\bm{u}}\in \mathbb{R}^N$  measures the multiplicative noise reduction in homogeneous regions and is defined as
\begin{displaymath}
\text{ENL}=\frac{\mu_{\hat{\bm{u}}}^2}{\sigma_{\hat{\bm{u}}}^2},
\end{displaymath}
where $\mu_{\hat{\bm{u}}}$ is the average intensity of the selected area and $\sigma_{\hat{\bm{u}}}^2$ is its variance.

For computing the ENL values, two homogeneous regions are respectively selected from ``SAR 1" and ``SAR 2", as indicated by the white boxes in Figure~\ref{fig:SAR1_Ratio}(a) and Figure~\ref{fig:SAR2_Ratio}(a). Table~\ref{Table:SAR} presents the ENL values for different methods. The SAR-BM3D method has the lowest ENL values compared to other methods, which indicates that the multiplicative noise is not effectively reduced or there exist some artifacts in the estimated image. The other methods have relatively large ENL values, which indicates that the multiplicative noise is well removed or the estimated image is over-smooth.

\begin{table}[htbp]
\centering
\caption{ENL values of desnoised images restored from real SAR images by different methods.}\label{Table:SAR}
{\small
\begin{tabular}{lcccccccccc}
\hline
\textbf{Image} & \textbf{Region} & \textbf{Noisy} &\textbf{Alg 1}  &\textbf{Alg 2} & \textbf{SAR-}&\textbf{DZ}& \textbf{HNW}&\textbf{I-DIV}  & \textbf{TwL-} & \textbf{Dict} \\
 &  & & && \textbf{BM3D}&&&  & \textbf{4V} & \\
\hline
SAR 1 & Left & 9.46 & 63.09 & 289.17& 42.84  & 745.14 & 521.24 & 306.90 & 117.48 & 183.57 \\
& Right & 10.65 & 82.59 & 203.24 & 57.36  & 333.18 & 360.12 & 276.04 & 144.54 & 175.49\\
&&&&&&&&&\\
SAR 2 & Left & 22.64 & 97.02 & 894.03 & 91.92  & 1008.50 &816.97 & 501.22 & 579.26 & 336.58\\
& Right & 21.91 & 96.58 & 734.47& 91.03 & 985.64 &740.27 & 566.30 & 724.76 & 444.01  \\
\hline
\end{tabular}
}\end{table}

The pointwise ratio between the real SAR image  $\bm{u}\in \mathbb{R}^N$ and the estimated image $\hat{\bm{u}}\in \mathbb{R}^N$ simulates the multiplicative noise that has been removed by the given method and is defined as
\begin{displaymath}
\text{Ratio}=\frac{\bm{u}}{\hat{\bm{u}}}.
\end{displaymath}
The ratio images for different methods are presented in Figure~\ref{fig:SAR1_Ratio} and Figure~\ref{fig:SAR2_Ratio}. The ratio images for Algorithm~\ref{Alg:Theoretical}, Algorithm~\ref{Alg:Practical}  and the SAR-BM3D method present almost random speckle, which is matched with the expected statistics. On the contrary, the ratio images for the other methods still contain some geometric structures such as edges and details correlated to the real SAR images, which indicates that those methods have removed some valuable information besides of noise.

\begin{figure}[htbp]
	\centering
	\begin{subfigure}[t]{.1942\textwidth}
    \includegraphics[width=3.245cm]{empty.pdf}
    \end{subfigure}
  	\begin{subfigure}[t]{.1942\textwidth}
  	\includegraphics[width=3.245cm]{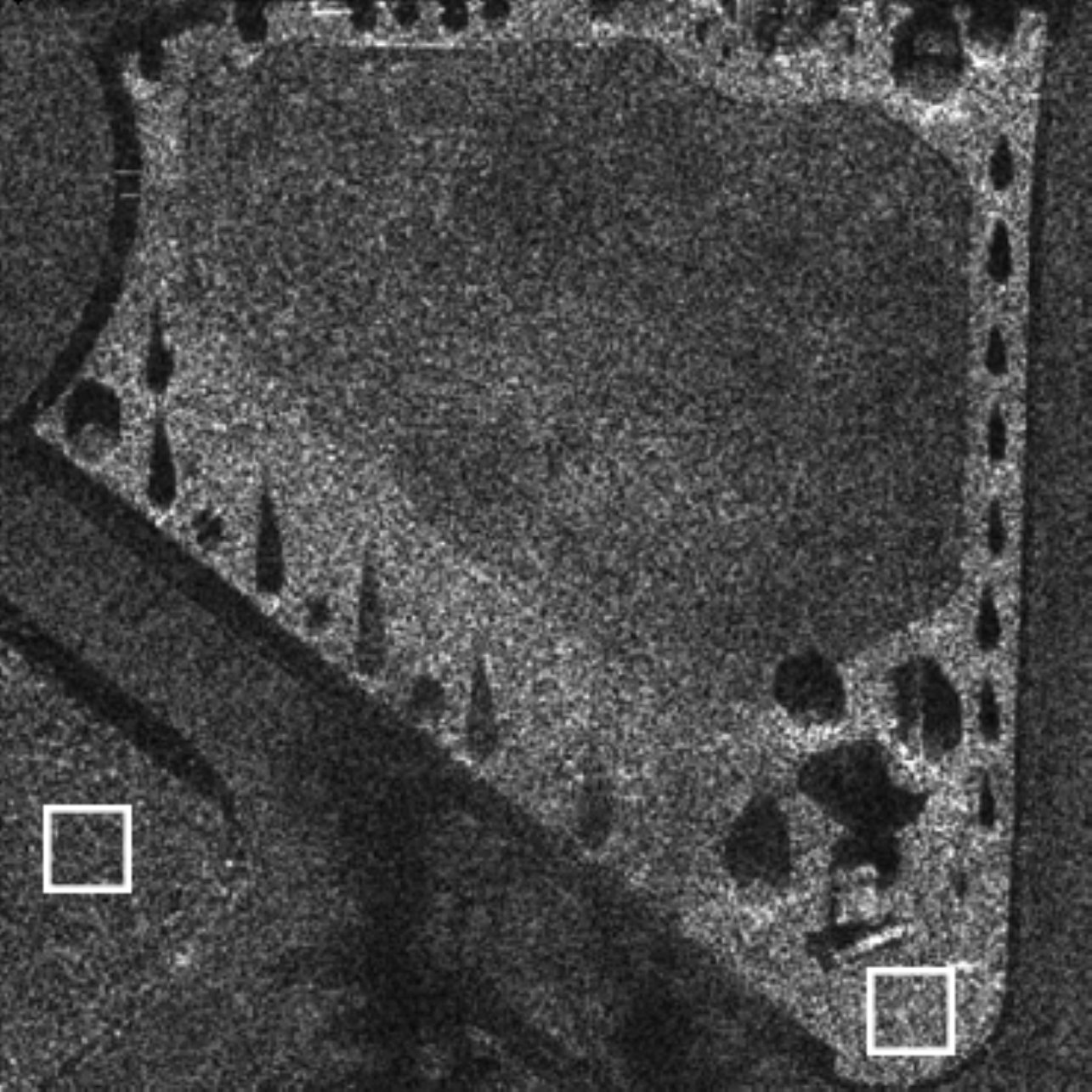}
    \caption{SAR 1}
    \end{subfigure}
    \begin{subfigure}[t]{.1942\textwidth}
    \includegraphics[width=3.245cm]{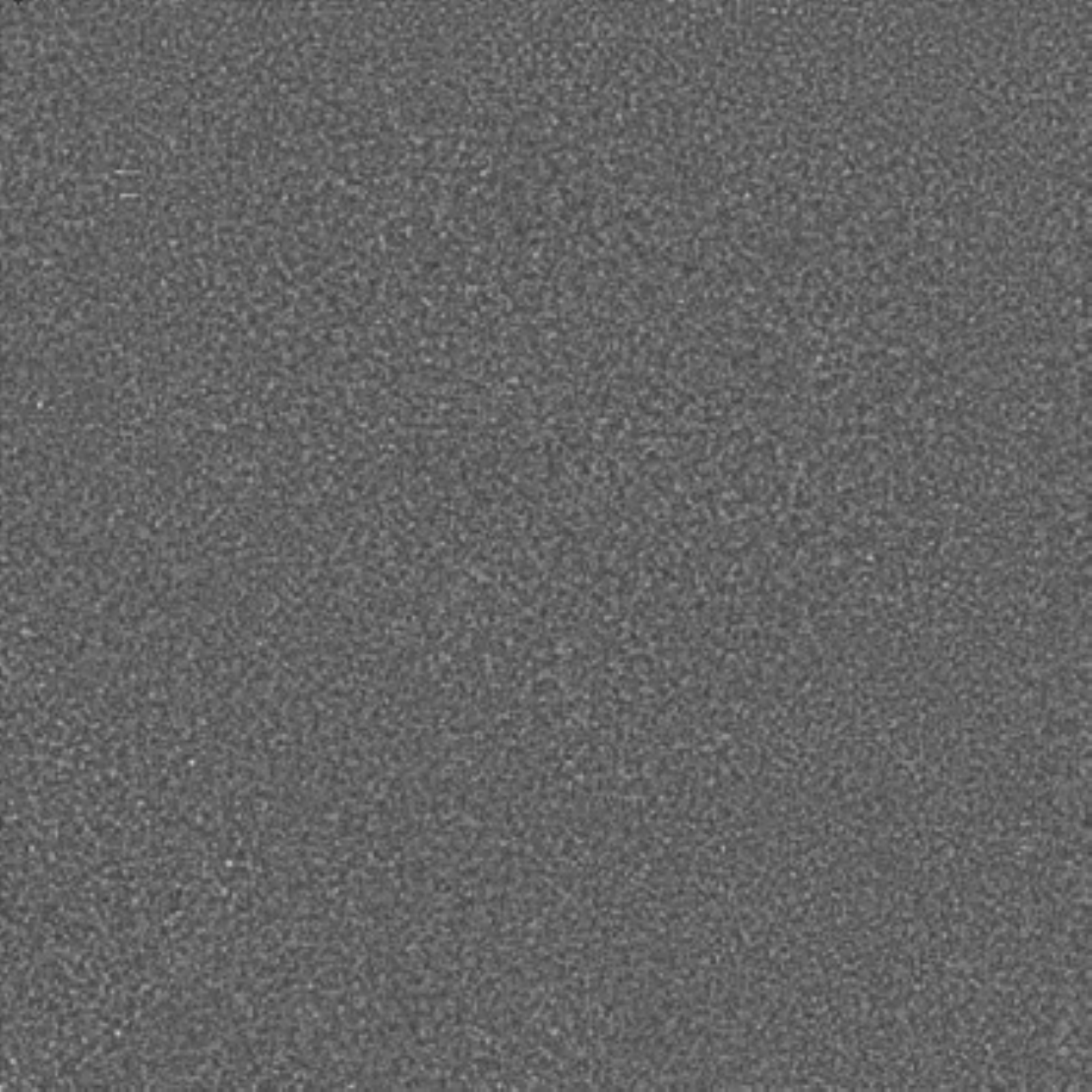}
    \caption{Alg 1}
    \end{subfigure}
    \begin{subfigure}[t]{.1942\textwidth}
    \includegraphics[width=3.245cm]{SAR1_Ratio_Estimated.pdf}
    \caption{Alg 2}
    \end{subfigure}
    \begin{subfigure}[t]{.1942\textwidth}
    \includegraphics[width=3.245cm]{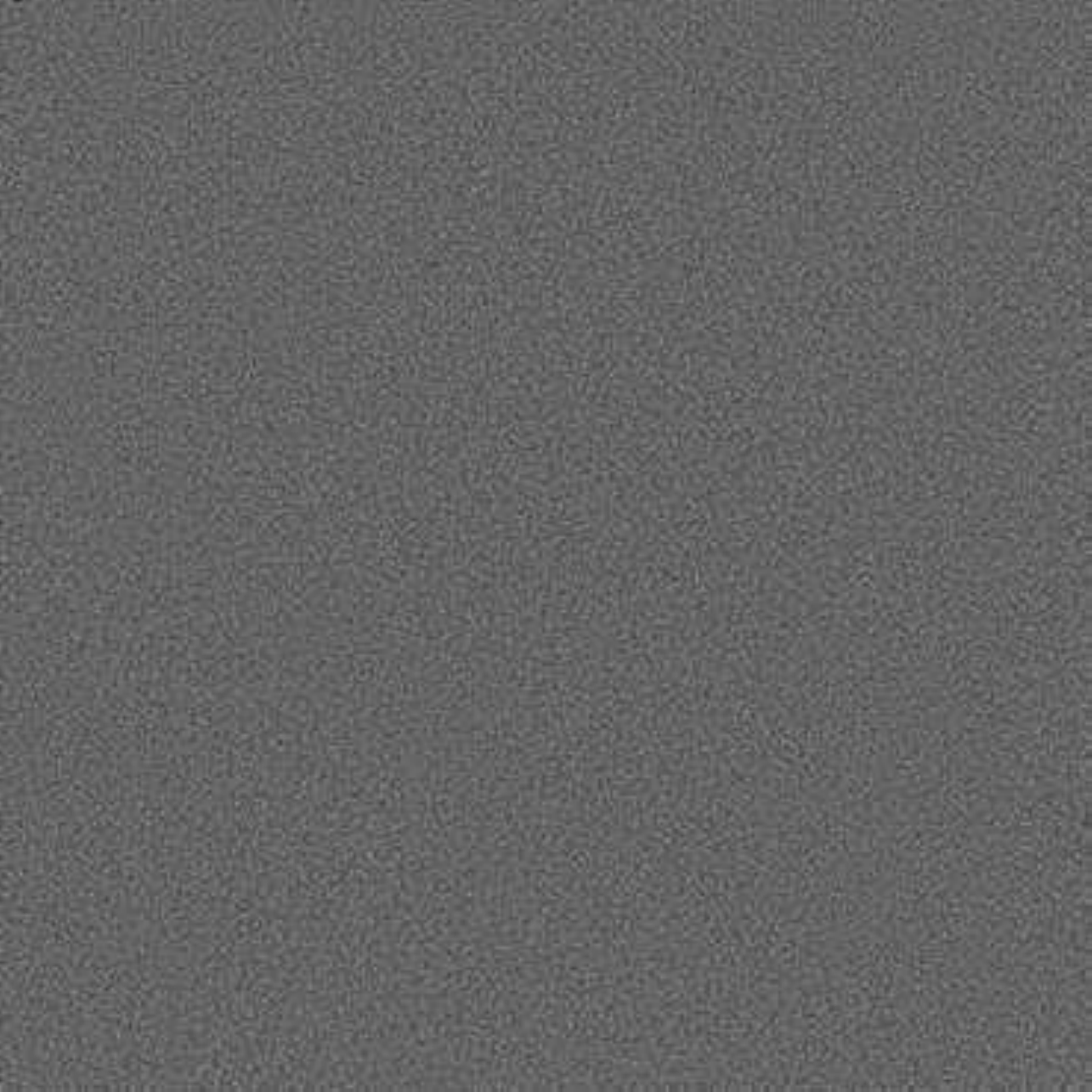}
	\caption{SAR-BM3D}
    \end{subfigure}    \\
    \begin{subfigure}[t]{.1942\textwidth}
    \includegraphics[width=3.245cm]{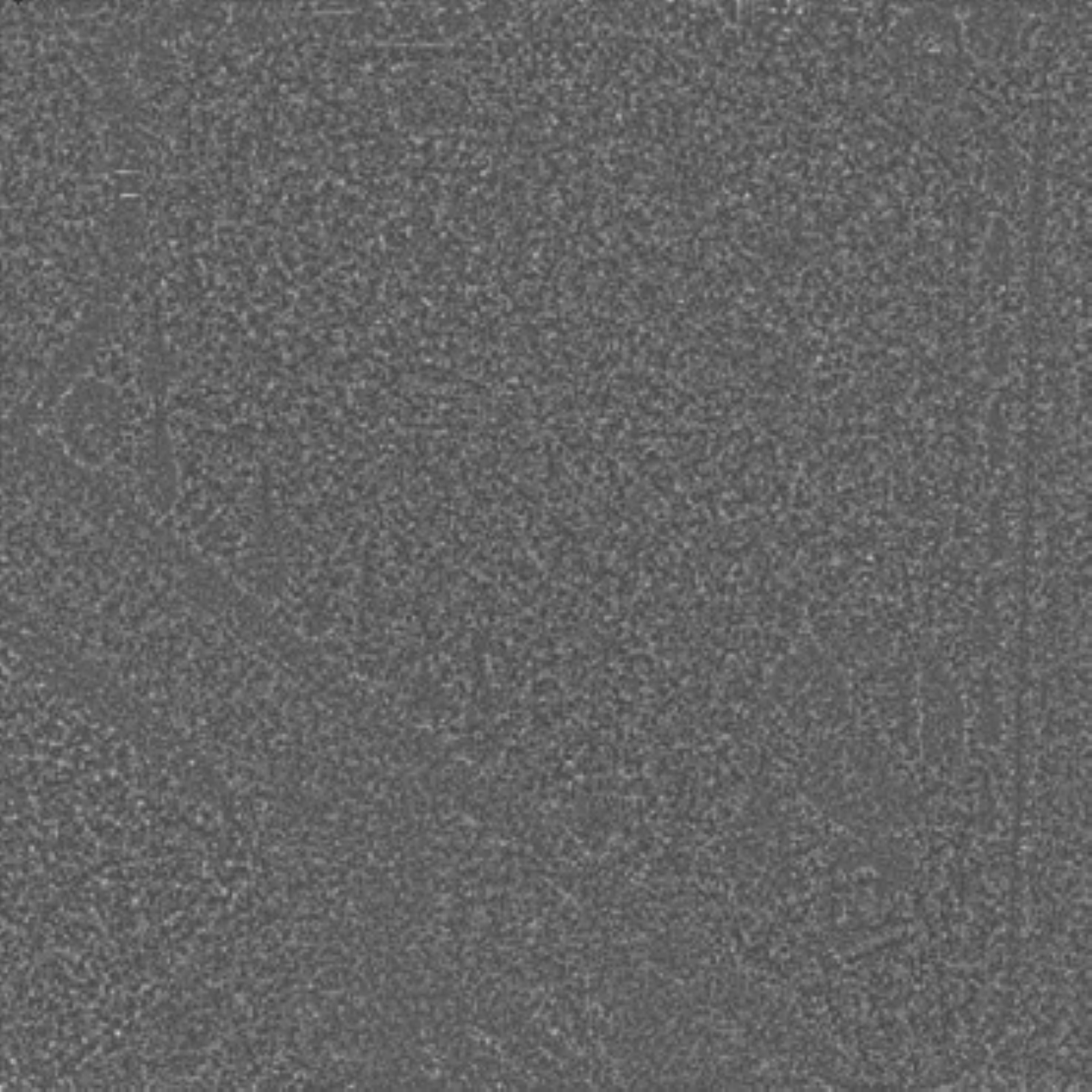}
	\caption{DZ}
    \end{subfigure}
    \begin{subfigure}[t]{.1942\textwidth}
    \includegraphics[width=3.245cm]{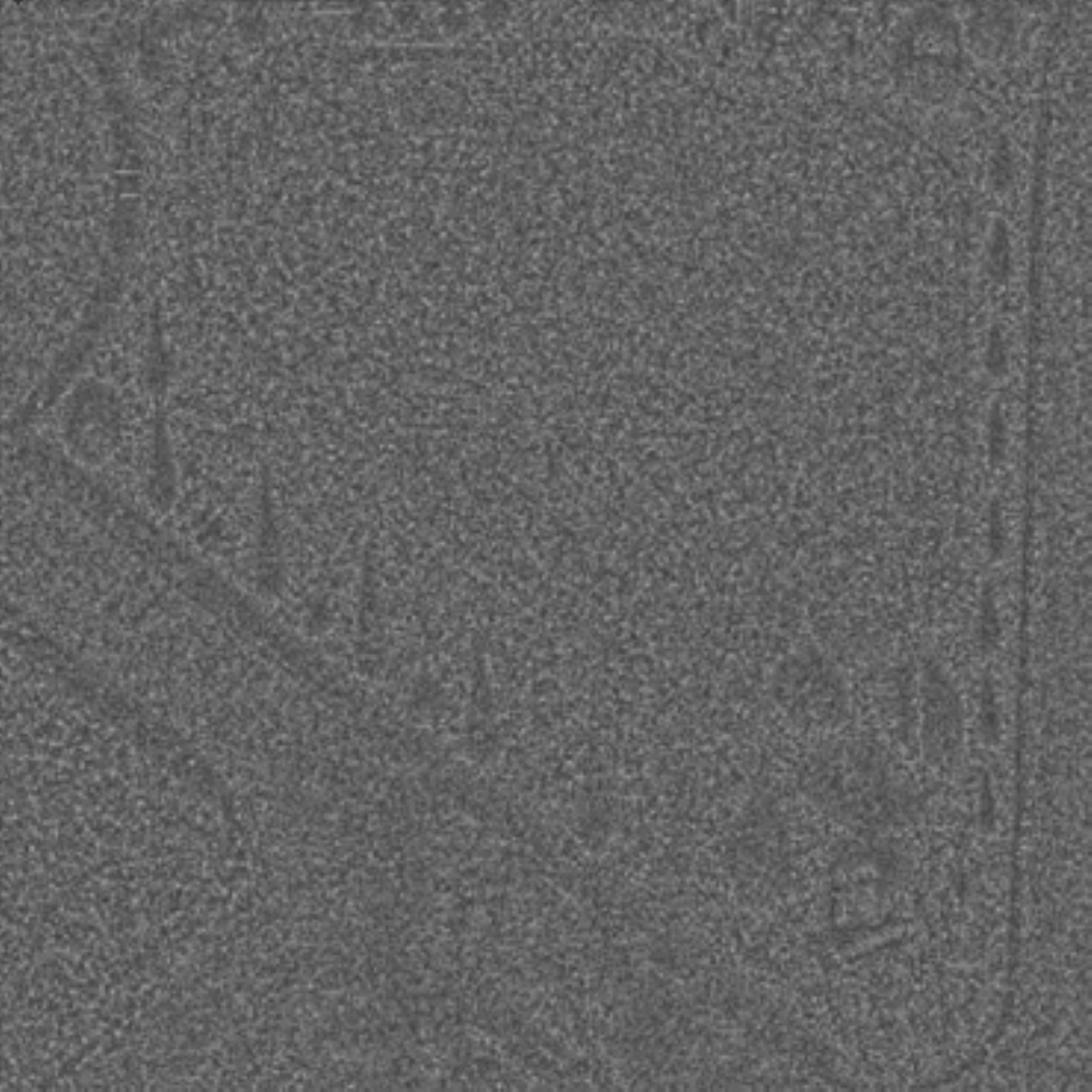}
	\caption{HNW}
    \end{subfigure}
    \begin{subfigure}[t]{.1942\textwidth}
    \includegraphics[width=3.245cm]{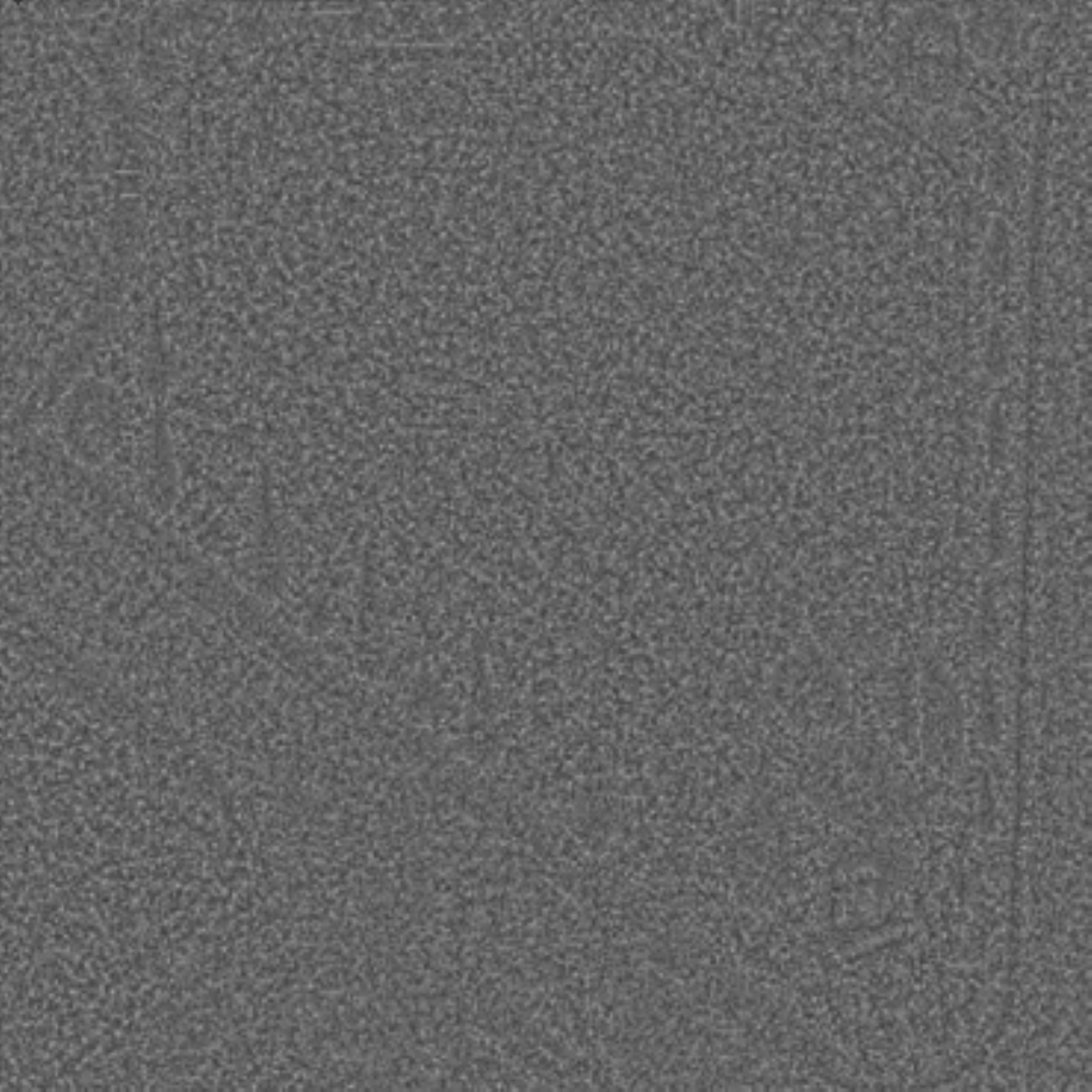}
    \caption{I-DIV}
    \end{subfigure}
    \begin{subfigure}[t]{.1942\textwidth}
    \includegraphics[width=3.245cm]{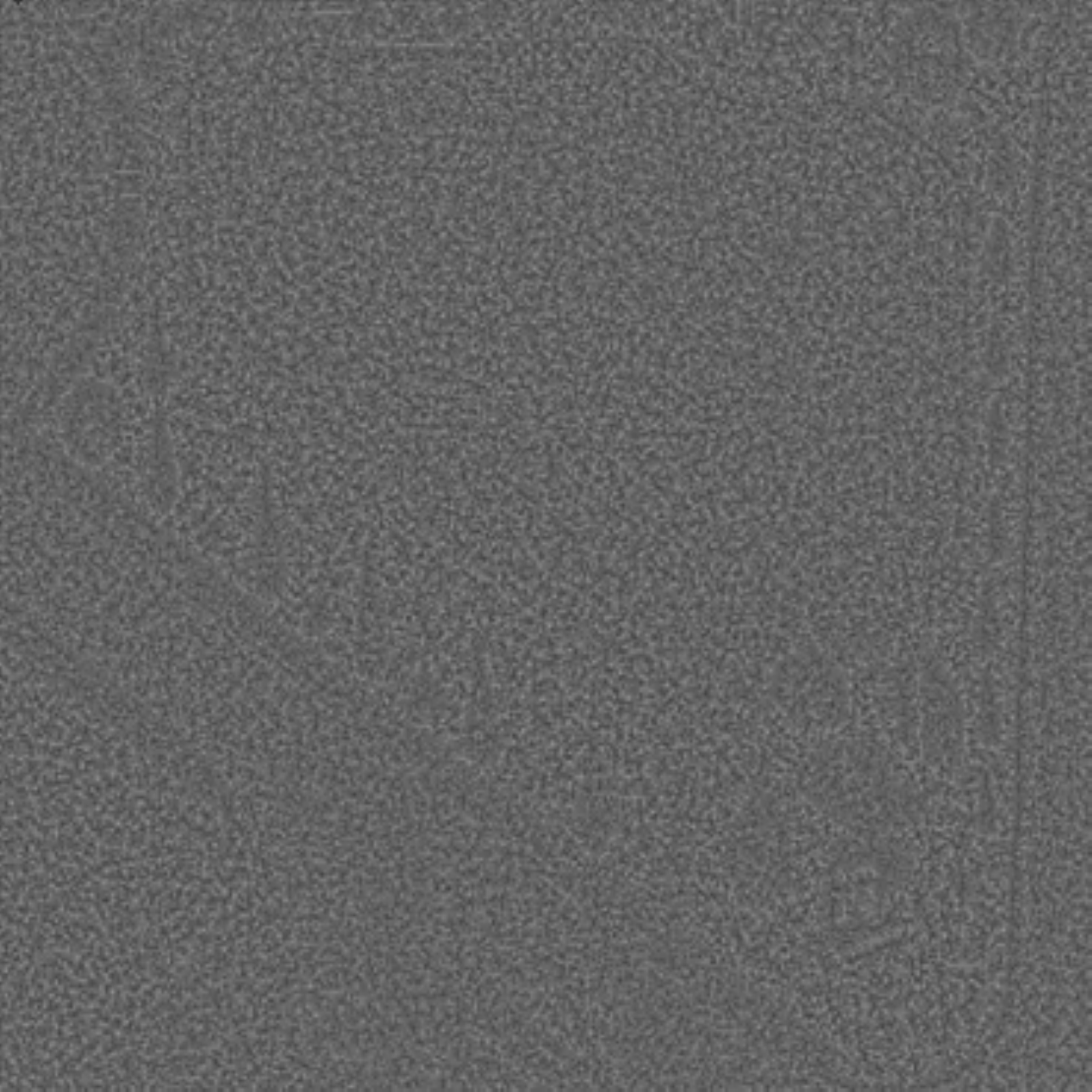}
    \caption{TwL-4V}
    \end{subfigure}
    \begin{subfigure}[t]{.1942\textwidth}
    \includegraphics[width=3.245cm]{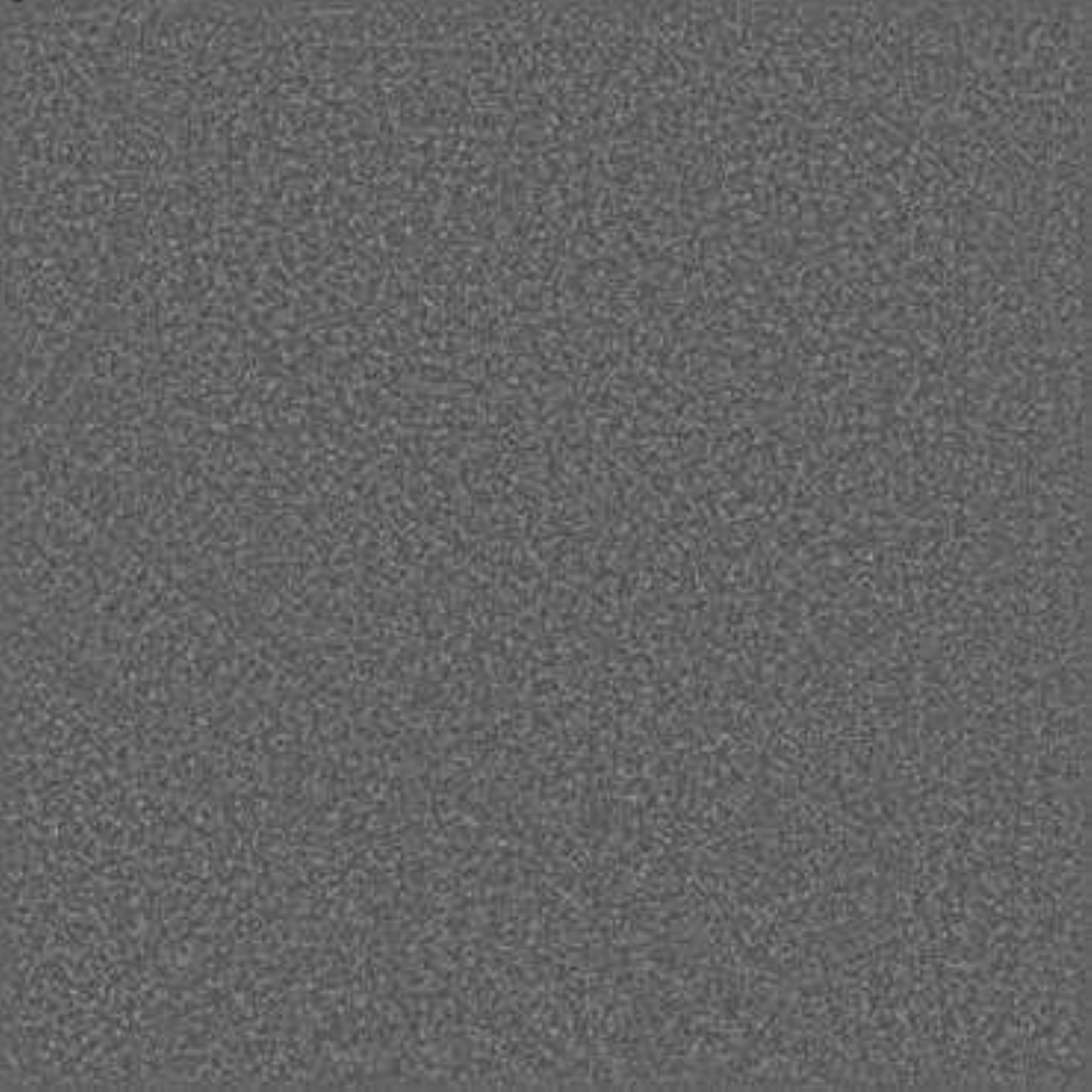}
	\caption{Dictionary}
    \end{subfigure}   	
	\caption{Comparison of the ratio images between ``SAR 1" and the estimated images by different methods.}\label{fig:SAR1_Ratio}
\end{figure}

\begin{figure}[htbp]
	\centering
	\begin{subfigure}[t]{.1942\textwidth}
    \includegraphics[width=3.245cm]{empty.pdf}
    \end{subfigure}
    \begin{subfigure}[t]{.1942\textwidth}
    \includegraphics[width=3.245cm]{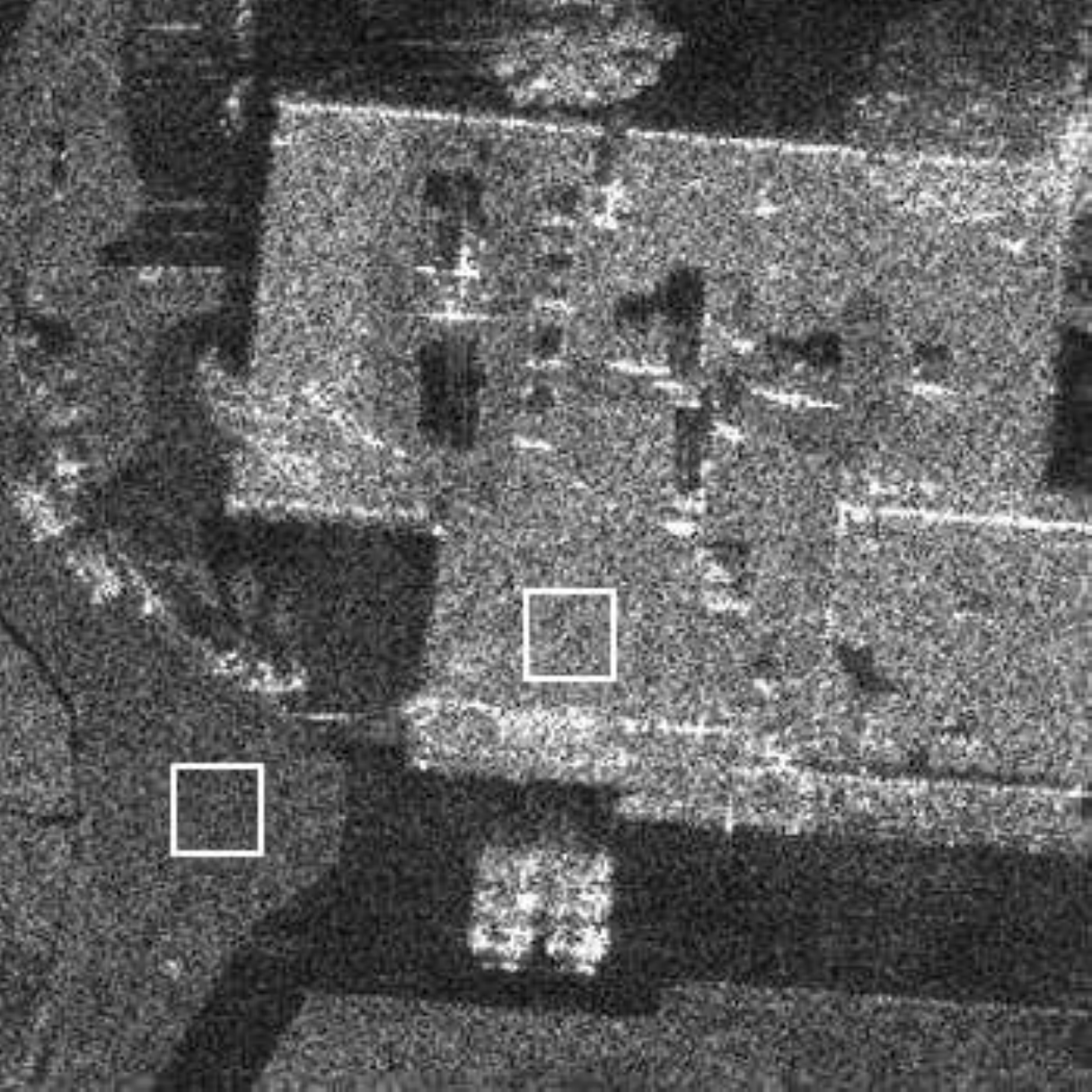}
    \caption{SAR 2}
    \end{subfigure}
    \begin{subfigure}[t]{.1942\textwidth}
    \includegraphics[width=3.245cm]{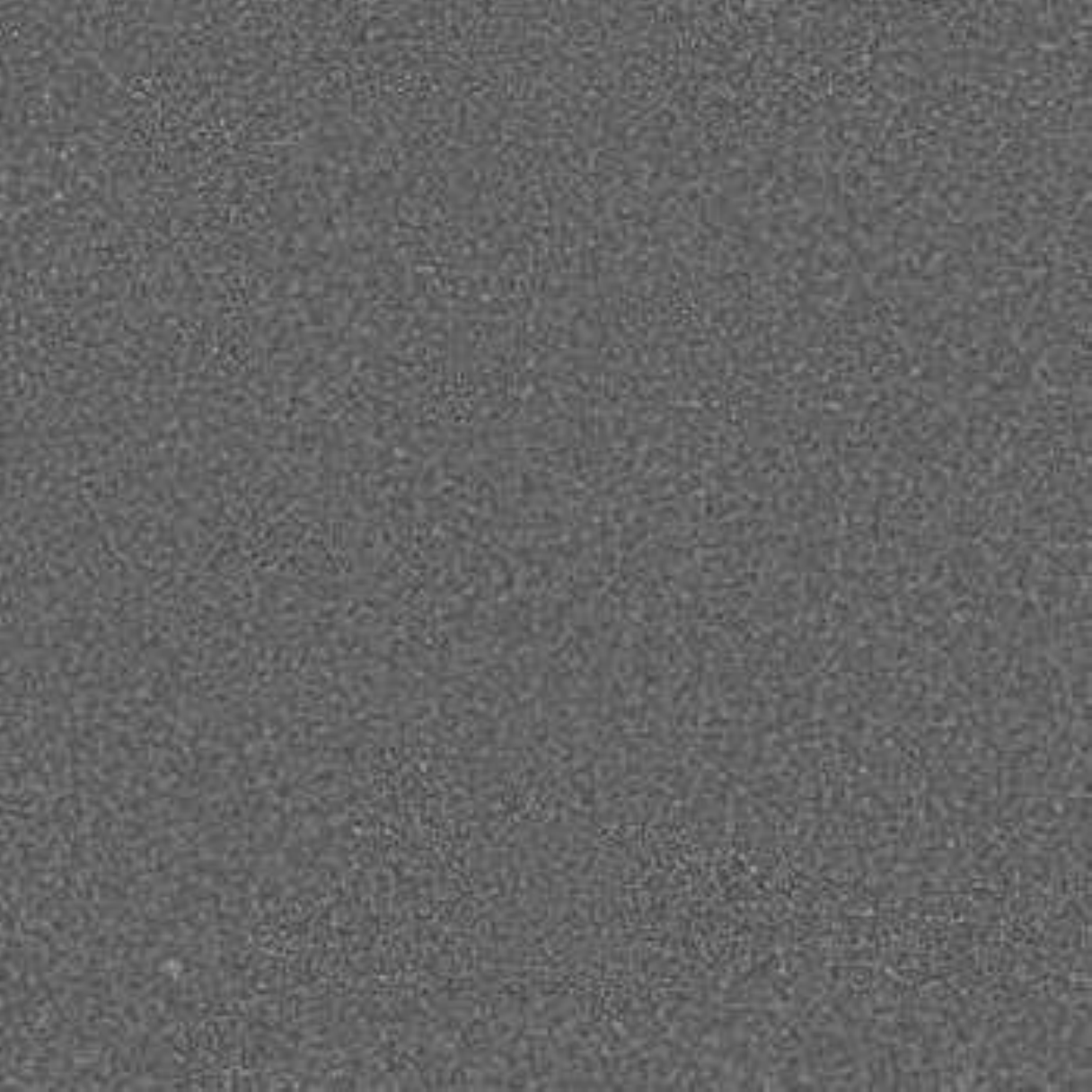}
    \caption{Alg 1}
    \end{subfigure}
    \begin{subfigure}[t]{.1942\textwidth}
    \includegraphics[width=3.245cm]{SAR6_Ratio_Estimated.pdf}
    \caption{Alg 2}
    \end{subfigure}
    \begin{subfigure}[t]{.1942\textwidth}
    \includegraphics[width=3.245cm]{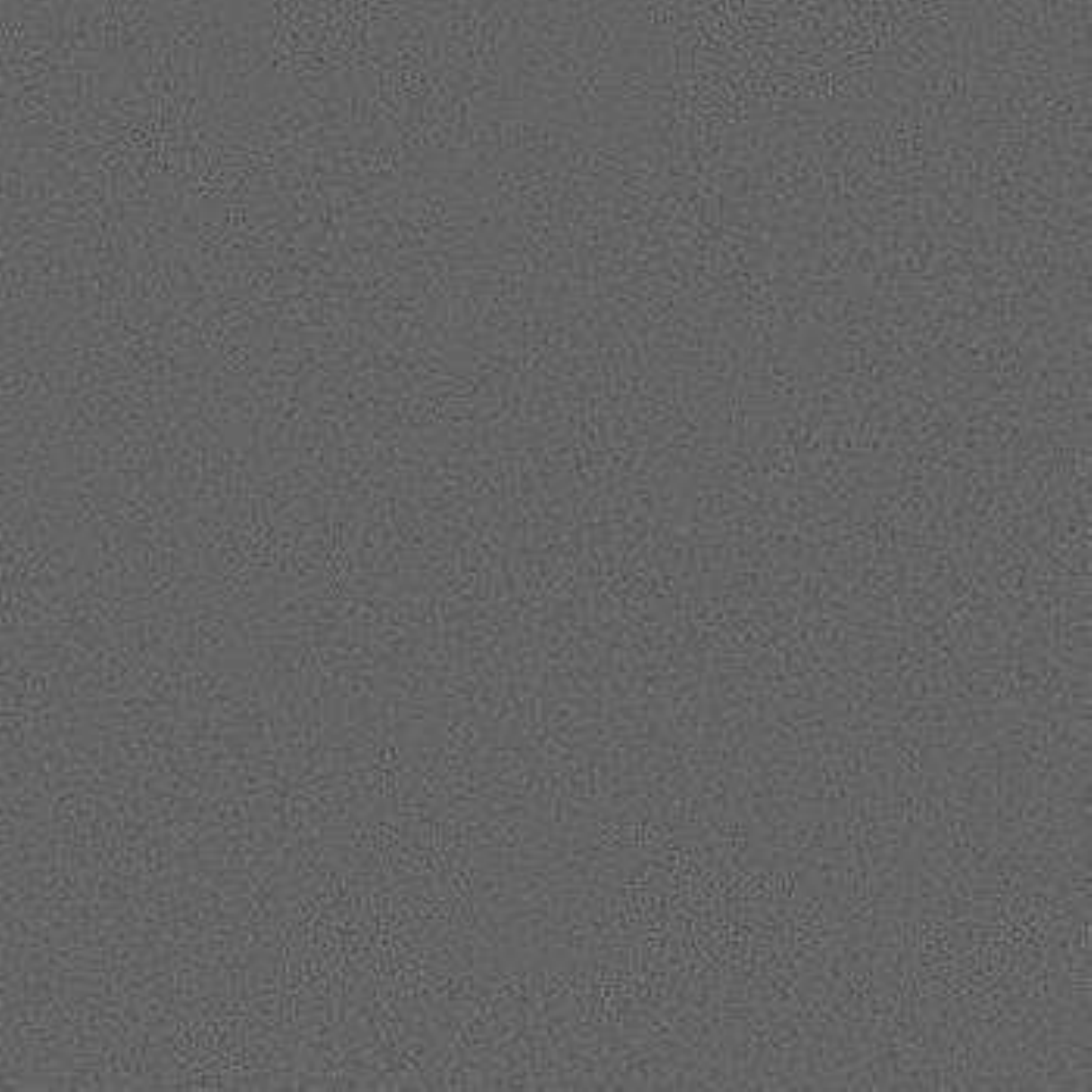}
	\caption{SAR-BM3D}
    \end{subfigure}    \\
    \begin{subfigure}[t]{.1942\textwidth}
    \includegraphics[width=3.245cm]{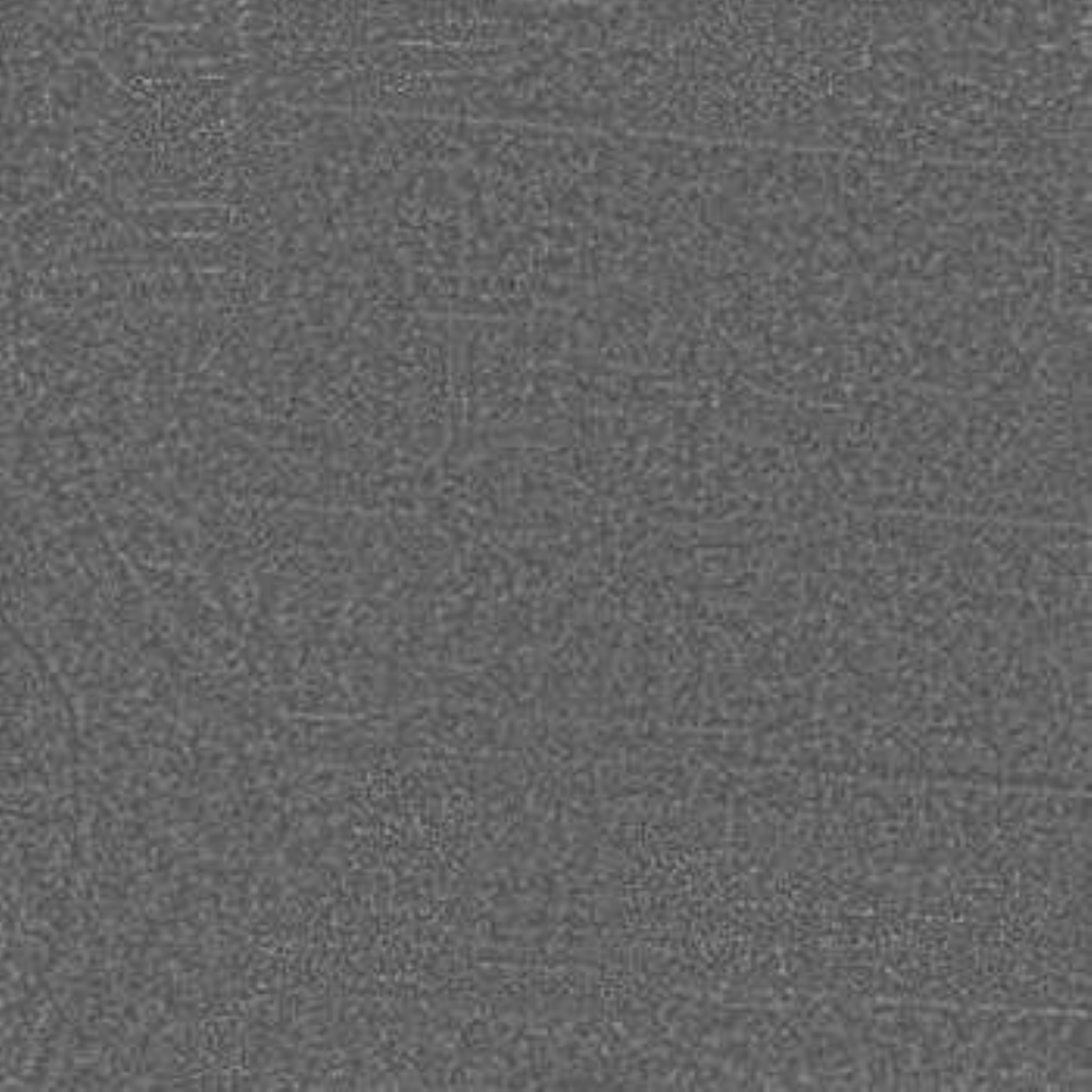}
	\caption{DZ}
    \end{subfigure}
    \begin{subfigure}[t]{.1942\textwidth}
    \includegraphics[width=3.245cm]{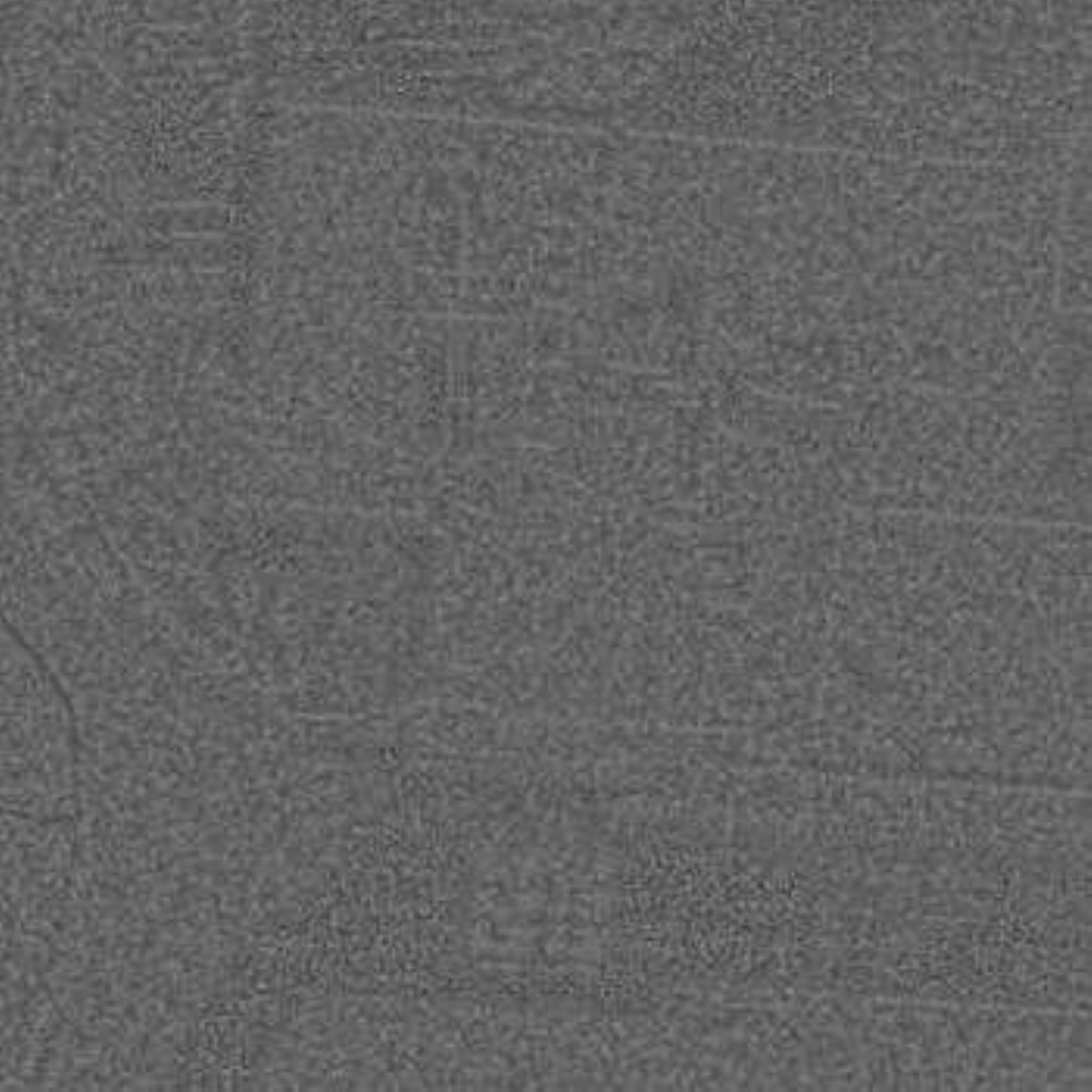}
	\caption{HNW}
    \end{subfigure}
    \begin{subfigure}[t]{.1942\textwidth}
    \includegraphics[width=3.245cm]{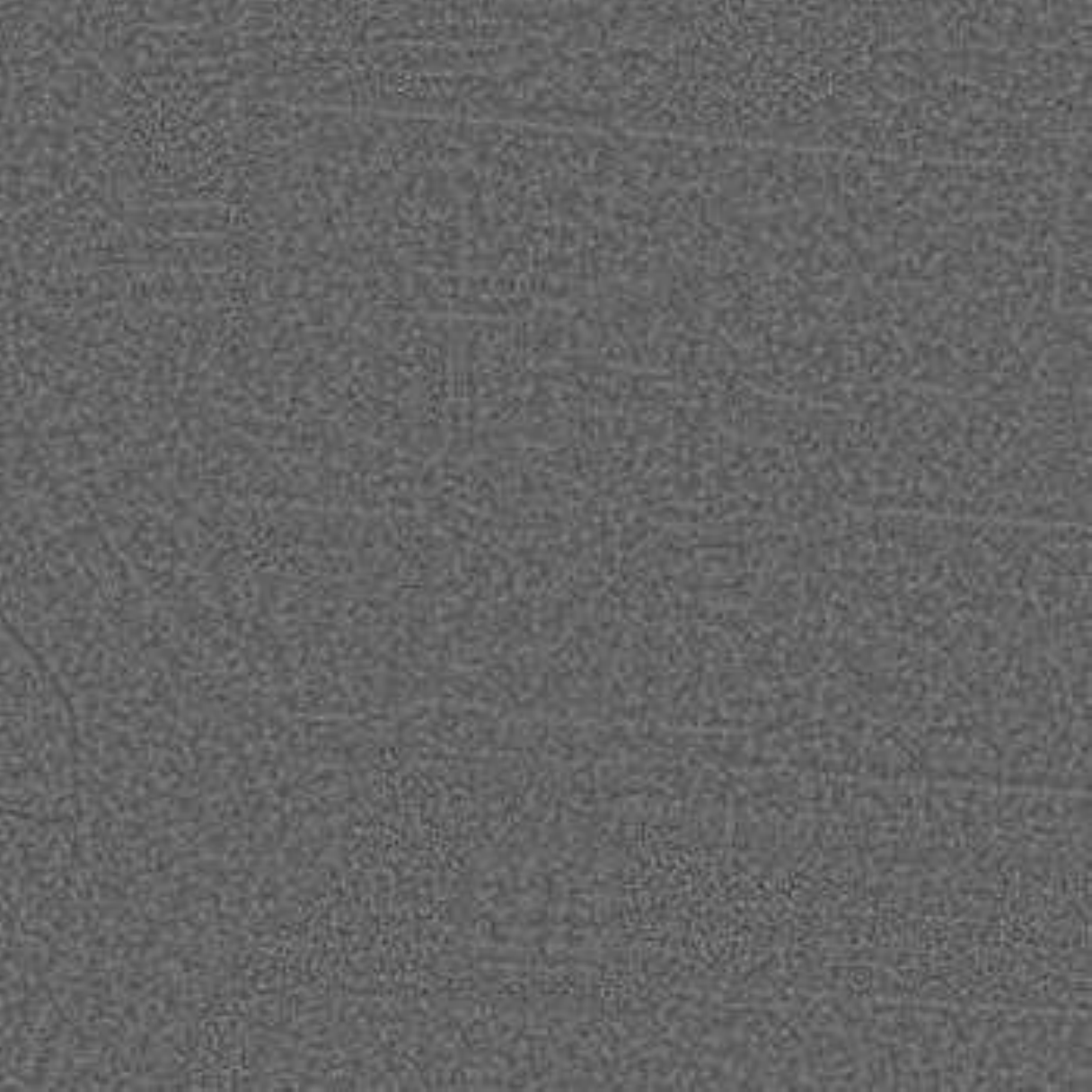}
    \caption{I-DIV}
    \end{subfigure}
    \begin{subfigure}[t]{.1942\textwidth}
    \includegraphics[width=3.245cm]{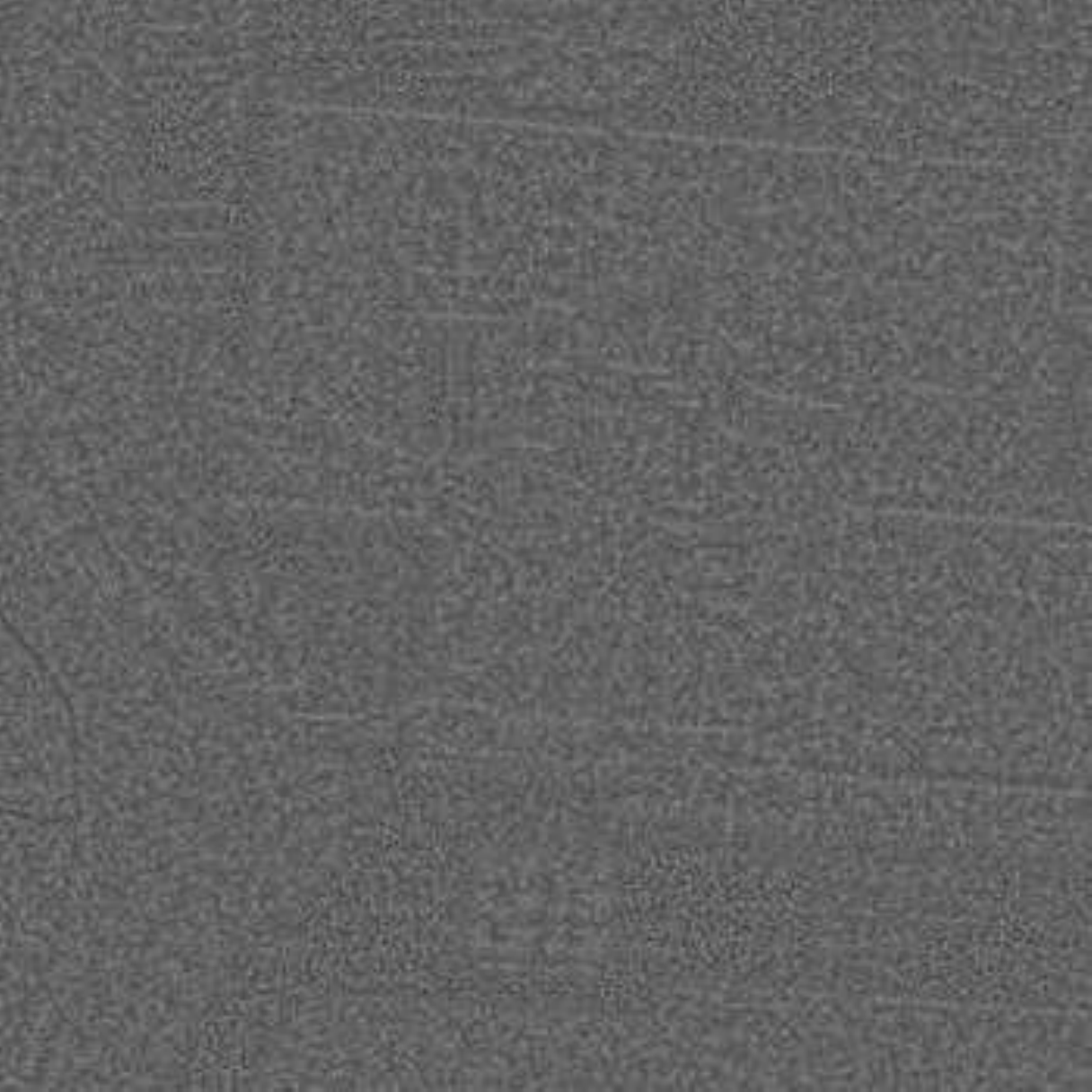}
    \caption{TwL-4V}
    \end{subfigure}
    \begin{subfigure}[t]{.1942\textwidth}
    \includegraphics[width=3.245cm]{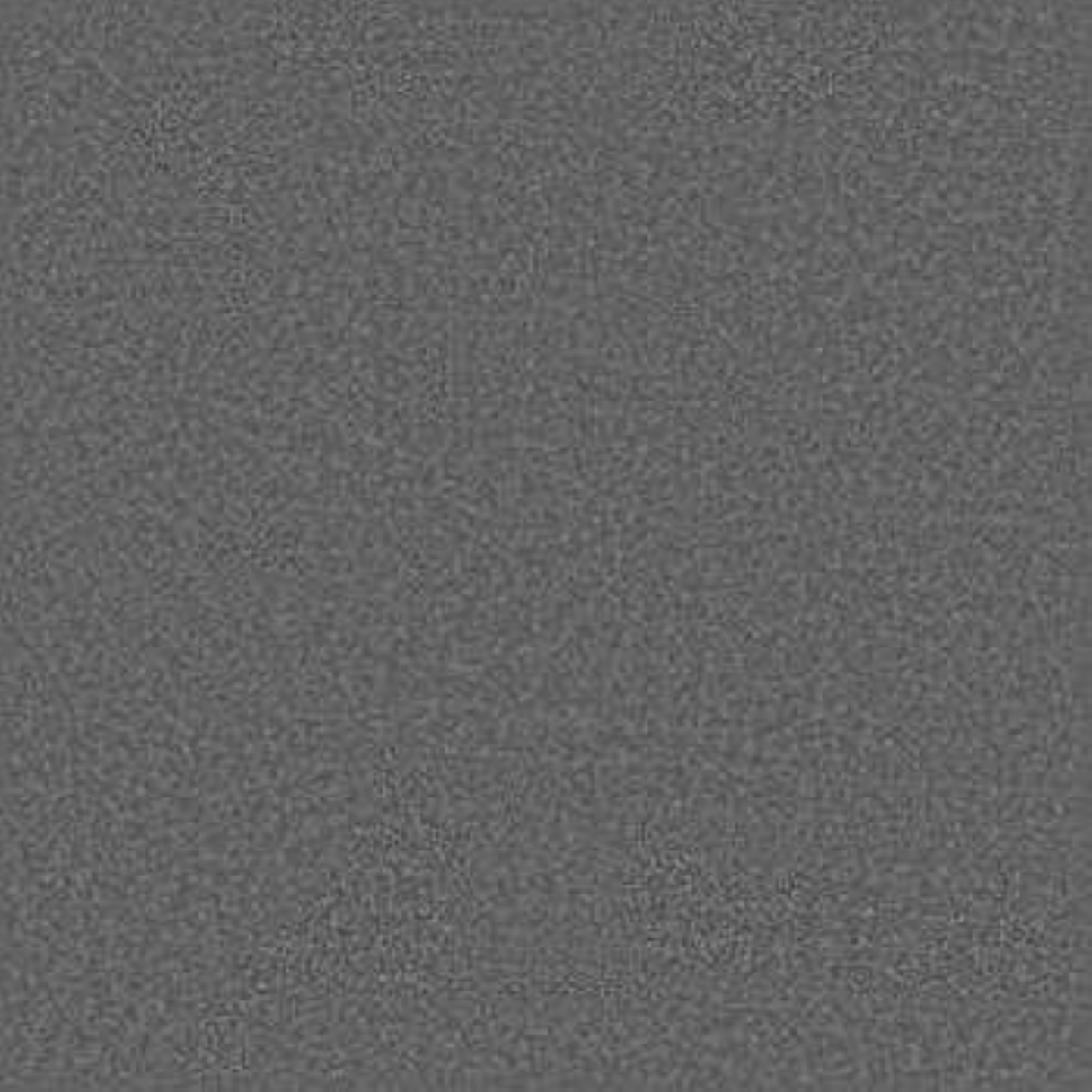}
	\caption{Dictionary}
    \end{subfigure}   	
	\caption{Comparison of the ratio images between ``SAR 2" and the estimated images by different methods.}\label{fig:SAR2_Ratio}
\end{figure}

\section{Conclusions}
\label{Section:Conclusions}

We have proposed an effective method for multiplicative noise removal. The proposed method consists of a nonlocal low-rank model, which exploits the low-rank prior of nonlocal similar patch matrices, and the PARM iterative algorithm, which solves the nonconvex nonsmooth optimization problem resulting from the proposed model. We have established the global convergence of the sequence generated by the PARM algorithm to a critical point of the nonconvex nonsmooth objective function of the resulting optimization problem. Numerical results have demonstrated that the proposed method with a theoretical convergence guarantee outperforms several existing methods including the state-of-the-art SAR-BM3D method.

\nocite{lou2010image,chan2014two}

\bibliographystyle{siam}
\bibliography{references}

\end{document}